\documentclass[11pt]{amsart}

\usepackage{graphicx}
\usepackage[english]{babel}
\usepackage[T1]{fontenc}
\usepackage[latin1]{inputenc}
\usepackage{amsfonts}
\usepackage{amssymb}
\usepackage{amsthm}
\usepackage{amsmath}
\usepackage{tikz}
\usepackage{float}
\usepackage{enumerate}
\usepackage{accents}
\usepackage{mathtools}
\usepackage{mathrsfs}
\usepackage{comment}
\usepackage{afterpage}
\usepackage{bbm}
\usepackage{dsfont}
\usepackage{stmaryrd}
\usepackage{hyperref}
\usepackage{mleftright}
\usepackage{subfig}
\usepackage{soul}
\usepackage{tikz-cd} 
\usepackage[foot]{amsaddr}
\usepackage{fullpage}

\theoremstyle{plain}
\newtheorem{thm}{Theorem}[section]
\newtheorem{lem}[thm]{Lemma}
\newtheorem{prop}[thm]{Proposition}
\newtheorem{cor}[thm]{Corollary}
\newtheorem*{thm*}{Theorem}
\newtheorem*{prop*}{Proposition}
\newtheorem*{cor*}{Corollary}
\newtheorem{thmintro}{Theorem}

\newtheorem{corintro}[thmintro]{Corollary}
\newtheorem{propintro}[thmintro]{Proposition}

\theoremstyle{definition}
\newtheorem{defn}[thm]{Definition}
\newtheorem{ex}[thm]{Example}
\newtheorem{rmk}[thm]{Remark}
\newtheorem*{rmk*}{Remark}
\newtheorem{conj}{Conjecture}
\newtheorem{quest}[conj]{Question}
\newtheorem*{quest*}{Question}

\newtheorem{ass}[thm]{Assumption}
\newtheorem*{defn*}{Definition}

\newcommand{\acts}{\curvearrowright}
\newcommand{\ra}{\rightarrow}
\newcommand{\Ra}{\Rightarrow}
\newcommand{\LRa}{\Leftrightarrow}
\newcommand{\cu}{\subseteq}

\newcommand{\wt}{\widetilde}

\newcommand{\x}{\times}
\renewcommand{\o}{\circ}
\newcommand{\id}{\mathrm{id}}
\newcommand{\mbb}{\mathbb}
\newcommand{\mc}{\mathcal}
\newcommand{\mf}{\mathfrak}
\newcommand{\mscr}{\mathscr}

\newcommand{\R}{\mathbb{R}}
\newcommand{\Z}{\mathbb{Z}}
\newcommand{\N}{\mathbb{N}}

\newcommand{\s}{\sigma}
\newcommand{\eps}{\epsilon}
\newcommand{\Om}{\Omega}
\newcommand{\om}{\omega}
\renewcommand{\l}{\lambda}
\renewcommand{\L}{\Lambda}
\newcommand{\g}{\gamma}
\newcommand{\G}{\Gamma}

\newcommand{\X}{\mc{X}}
\newcommand{\A}{\mc{A}}
\newcommand{\W}{\mc{W}}
\newcommand{\T}{\mc{T}}
\newcommand{\CAT}{{\rm CAT(0)}}
\DeclareMathOperator{\Fix}{Fix}
  
\DeclareMathOperator{\hull}{Hull} 
\DeclareMathOperator{\diam}{diam} 
\DeclareMathOperator{\lk}{lk}
\DeclareMathOperator{\St}{st}
\DeclareMathOperator{\isom}{Isom}
\DeclareMathOperator{\homeo}{Homeo}
\DeclareMathOperator{\aut}{Aut}
\DeclareMathOperator{\out}{Out}

\DeclareMathOperator{\rk}{rk}

\newcommand{\cc}{\mf{c}}
\newcommand{\res}{\mathrm{res}}
\newcommand{\C}{\mc{C}}

\begin{document}

\title{Coarse-median preserving automorphisms} %\\ of special groups % cubulated?
\author[E. Fioravanti]{Elia Fioravanti}\address{Universit\"at Bonn, Bonn, Germany}\email{ fioravan@math.uni-bonn.de} 
%\author[E. Fioravanti]{Elia Fioravanti}\address{Max Planck Institute for Mathematics, Bonn, Germany}\email{fioravanti@mpim-bonn.mpg.de} 

\begin{abstract}
This paper has three main goals. 

First, we study fixed subgroups of automorphisms of right-angled Artin and Coxeter groups. If $\varphi$ is an untwisted automorphism of a RAAG, or an arbitrary automorphism of a RACG, we prove that $\Fix\varphi$ is finitely generated and undistorted. Up to replacing $\varphi$ with a power, we show that $\Fix\varphi$ is quasi-convex with respect to the standard word metric. This implies that $\Fix\varphi$ is separable and a special group in the sense of Haglund--Wise.

By contrast, there exist ``twisted'' automorphisms of RAAGs for which $\Fix\varphi$ is undistorted but not of type $F$ (hence not special), of type $F$ but distorted, or even infinitely generated.

Secondly, we introduce the notion of ``coarse-median preserving'' automorphism of a coarse median group, which plays a key role in the above results. We show that automorphisms of RAAGs are coarse-median preserving if and only if they are untwisted. On the other hand, all automorphisms of Gromov-hyperbolic groups and right-angled Coxeter groups are coarse-median preserving. These facts also yield new or more elementary proofs of Nielsen realisation for RAAGs and RACGs.

Finally, we show that, for every special group $G$ (in the sense of Haglund--Wise), every infinite-order, coarse-median preserving outer automorphism of $G$ can be realised as a homothety of a finite-rank median space $X$ equipped with a ``moderate'' isometric $G$--action. This generalises the classical result, due to Paulin, that every infinite-order outer automorphism of a hyperbolic group $H$ projectively stabilises a small $H$--tree.
\end{abstract}

\maketitle

\tableofcontents

\addtocontents{toc}{\protect\setcounter{tocdepth}{1}}

\section{Introduction.}

This paper is inspired by the following, at first sight unrelated, questions. 

\begin{quest}\label{Q1}
Given a finitely generated group $G$ and $\varphi\in\aut G$, what is the structure of the subgroup of fixed points $\Fix\varphi\leq G$?
\end{quest}

\begin{quest}\label{Q2}
Given a finitely generated group $G$ and $\varphi\in\aut G$, when can we realise $\varphi$ as a homothety of a non-positively curved metric space $X$ equipped with a ``nice'' $G$--action by isometries?
\end{quest} 

Our motivation comes from the theory of automorphisms of free groups. When $G=F_n$, a complete answer to Question~\ref{Q1} was first conjectured by Peter Scott in 1978, and later proved by Bestvina and Handel \cite{BH-Ann} (after work, among others, by \cite{Dyer-Scott,Jaco-Shalen,Gersten-bull,Culler-finiteorder,Gersten-proc,Goldstein-Turner,Gersten-fix,Cooper-fix,Cohen-Lustig}): 
\smallskip
\begin{center}
\emph{``for every $\varphi\in\aut F_n$, the fixed subgroup $\Fix\varphi\leq F_n$ is generated by at most $n$ elements''}.
\end{center} 
\smallskip
In particular, $\Fix\varphi$ is finitely generated, free, and quasi-convex in $F_n$.

Bestvina and Handel's proof is based on the extension of several ideas of Nielsen--Thurston theory from surfaces to graphs. Specifically, every homotopy equivalence between finite graphs is homotopic to a \emph{(relative) train track map} \cite{BH-Ann,BFH1}. This result is also a key ingredient in providing the following answer to Question~\ref{Q2} \cite{GJLL}: 
\smallskip
\begin{center}
\emph{``for every $\varphi\in\aut F_n$, there exists an action by homotheties $F_n\rtimes_{\varphi}\Z\acts T$, where $T$ is an $\R$--tree and the restriction $F_n\acts T$ is isometric, minimal, and with trivial arc-stabilisers''}.
\end{center} 
\smallskip
If $\varphi$ is exponentially growing, then $F_n\acts T$ has dense orbits and $\Fix\varphi$ is elliptic.

We are interested in Question~\ref{Q2} because of its connections to Question~\ref{Q1}. Indeed, if one admits the existence of an $F_n$--tree as above, it is possible to give more elementary proofs of the Scott conjecture, which are completely independent of the complicated machinery of train tracks \cite{GLL} and instead rely on an ``index theory'' for $F_n$--trees \cite{Gaboriau-Levitt}.

More generally, a satisfactory answer to Question~\ref{Q2} was obtained by Paulin for all \emph{Gromov-hyperbolic groups} $G$ \cite{Paulin-ENS}. If $\phi\in\out G$ has infinite order, then it can be similarly realised as a homothety of a \emph{small} $G$--tree, i.e.\ an $\R$--tree with a minimal isometric $G$--action such that no $G$--stabiliser of an arc contains a copy of the free group $F_2$. 

Paulin's proof is abstract in nature, but his result can be pictured quite concretely in the case when $G=\pi_1(S)$ for a closed surface $S$: Thurston showed that $\phi$ is induced by a homeomorphism of $S$ that preserves a projective measured singular foliation on $S$ \cite{Thur-BullAMS}; %(since $\phi\in\out G$ has infinite order)
the $\R$--tree $T$ can then be constructed by lifting this singular foliation to the universal cover $\wt S$ and considering its leaf space.

It is natural to wonder if the above discussion is specific to hyperbolic groups. This might be suggested by the fact that automorphism groups of one-ended hyperbolic groups can essentially be understood in terms of mapping class groups of finite-type surfaces \cite{Sela-II,Levitt-hyp}, for which Nielsen--Thurston theory is available.  

\medskip
In recent years, the study of outer automorphisms of groups other than $\pi_1(S)$ and $F_n$ has gained significant traction. The groups $\out\A_{\G}$ --- where $\A_{\G}$ is a right-angled Artin group (RAAG) --- are particularly appealing in this context, as they can exhibit a variety of interesting behaviours ranging between the extremal cases of $\out F_n$ and $\out\Z^n={\rm GL}_n\Z$.

One may look at the large body of work on $\out F_n$ hoping to extract a blueprint that will direct the study of the groups $\out\A_{\G}$. This has proved a successful approach in some cases, remarkably with the definition of analogues of Outer Space \cite{CSV,BCV} and its consequences for the study of homological properties. However, there are limits to such analogies: in practice, techniques that are taylored to general RAAGs and based on induction on the complexity of the graph $\G$ seem to provide the most effective approach to many problems \cite{CV1,CV2,Guirardel-Sale,Day-Wade,Day-Sale-Wade}.

Our aim is to investigate Questions~\ref{Q1} and~\ref{Q2} when $G$ is a RAAG or, more generally, a cocompactly cubulated group. These are just two of the many basic questions that have been fully solved for $\out F_n$, but have so far remained out of the limelight for the groups $\out\A_{\G}$.

One quickly realises that it is necessary to impose some restrictions on $\varphi\in\aut\A_{\G}$ if the two questions are to be fruitfully addressed. To begin with, it is not hard to construct automorphisms of $F_2\x\Z$ whose fixed subgroup is infinitely generated (Example~\ref{mother of all evil}), which would prevent us from relying on the tools of geometric group theory in relation to Question~\ref{Q1}. In addition, when $G=\Z^n$, it should heuristically always be possible to equivariantly collapse the space $X$ in Question~\ref{Q2} to a copy of $\R$, which forces $\varphi\in {\rm GL}_n\Z$ to have a positive eigenvalue.

We choose to consider the subgroup of \emph{untwisted automorphisms} $U(\A_{\G})\leq\aut\A_{\G}$, which was introduced by Day in \cite{Day-peak} (with the name of ``long-range automorphisms'') and further studied by Charney, Stambaugh and Vogtmann \cite{CSV} and Hensel and Kielak \cite{Hensel-Kielak}. This can be defined as the subgroup generated by a certain subset of the Laurence--Servatius generators for $\aut\A_{\G}$ \cite{Laurence,Servatius}, excluding generators that ``resemble'' too closely elements of ${\rm GL}_n\Z$. 

The subgroup $U(\A_{\G})\leq\aut\A_{\G}$ displays stronger similarities to $\aut F_n$ and often makes up a large portion of the entire group $\aut\A_{\G}$. For instance, $U(F_n)=\aut F_n$ and $U(\A_{\G})$ always contains the kernel of the homomorphism $\aut\A_{\G}\ra{\rm GL}_n\Z$ induced by the $(\aut\A_{\G})$--action on the abelianisation of $\A_{\G}$.

\medskip
Our first result is a novel, \emph{coarse geometric} characterisation of untwisted automorphisms. This will play a fundamental role in addressing both Questions~\ref{Q1} and~\ref{Q2} in the rest of the paper.

Recall that every right-angled Artin group $\A_{\G}$ is equipped with a median operator $\mu\colon\A_{\G}^3\ra\A_{\G}$ coming from the fact that $\A_{\G}$ is naturally identified with the $0$--skeleton of a $\CAT$ cube complex (the universal cover of its Salvetti complex) \cite{Chepoi}. Thus, one can consider those automorphisms of $\A_{\G}$ with respect to which $\mu$ is coarsely equivariant. 

More generally, it makes sense to study such automorphisms for any \emph{coarse median group} $(G,\mu)$. This remarkably broad class of groups was introduced by Bowditch in \cite{Bow-cm} and contains all Gromov-hyperbolic groups, as well as all groups admitting a geometric action on a $\CAT$ cube complex, and all \emph{hierarchically hyperbolic groups} in the sense of \cite[Definition~1.21]{HHS2}.

\begin{defn*}
An automorphism $\varphi$ of a coarse median group $(G,\mu)$ is \emph{coarse-median preserving}\footnote{This terminology is motivated in Subsection~\ref{coarse median structures sect}, see Remark~\ref{motivation for cmp terminology}.} \emph{(CMP)} if there exists a constant $C\geq 0$ such that:
\[\varphi(\mu(g_1,g_2,g_3))\approx_C \mu(\varphi(g_1),\varphi(g_2),\varphi(g_3)),\ \forall g_1,g_2,g_3\in G,\]
where ``$x\approx_C y$'' means ``$d(x,y)\leq C$'' with respect to some fixed word metric $d$ on $G$.
\end{defn*} 

It is easy to see that CMP automorphisms form a subgroup of $\aut G$ containing all inner automorphisms\footnote{Here it is important that our definition of \emph{coarse median group} (Definition~\ref{coarse median group defn}) is slightly stronger than Bowditch's original definition \cite{Bow-cm}, in that we require $\mu$ to be coarsely $G$--equivariant. The difference between the two notions is analogous to the distinction between \emph{hierarchically hyperbolic groups} and groups that are just a \emph{hierarchically hyperbolic space}.}. Thus, it makes sense to speak of CMP \emph{outer} automorphisms, as this property does not depend on the specific lift to $\aut G$.

It turns out that, in the setting of right-angled Artin groups, CMP automorphisms coincide with untwisted automorphisms, perhaps explaining the closer analogy between $U(\A_{\G})$ and $\aut F_n$. In particular, every element of $\aut F_n$ is CMP, while only a finite subgroup of $\aut\Z^n$ is CMP.  

More precisely, we have the following. We endow right-angled Artin/Coxeter groups with the coarse median structure induced by the action on the universal cover of the Salvetti/Davis complex.

\begin{propintro}\label{cmp prop intro}
\begin{enumerate}
\item[]
\item All automorphisms of hyperbolic groups are CMP.
\item All automorphisms of right-angled Coxeter groups are CMP.
\item Automorphisms of right-angled Artin groups are CMP if and only if they are untwisted.
\end{enumerate}
\end{propintro}

Part~(1) is due to the fact that hyperbolic groups admit a unique coarse median structure, which was shown in \cite{NWZ1} (see Example~\ref{hyp cmp ex} below). 
% it can also be shown very easily using the Morse lemma, but this is a better proof
That CMP automorphisms of RAAGs are untwisted can be easily deduced from the proof, due to Laurence, that elementary automorphisms generate the automorphism group \cite{Laurence}. We prove the rest of Proposition~\ref{cmp prop intro} in Subsection~\ref{of RAAGs sect}. 

Our first result on Question~\ref{Q1} applies to all CMP automorphisms of \emph{cocompactly cubulated} groups, i.e.\ those groups that admit a proper cocompact action on a $\CAT$ cube complex. 

We remark that, in addition to Proposition~\ref{cmp prop intro}, examples of CMP automorphisms of cubulated groups are provided by \cite[Theorem~E]{Fio10e}, which characterises when a generalised Dehn twist preserves the coarse median structure induced by the cubulation.

\begin{thmintro}\label{cmp und intro}
Let $G$ be a cocompactly cubulated group, with the induced coarse median structure. If $\varphi\in\aut G$ is coarse-median preserving, then: 
\begin{enumerate}
\item $\Fix\varphi$ is finitely generated and undistorted in $G$;
\item $\Fix\varphi$ is itself cocompactly cubulated.
\end{enumerate}
\end{thmintro}

Both parts of this result fail badly for ``twisted'' automorphisms of right-angled Artin groups. For every finite graph $\G$, there exist automorphisms $\psi\in\aut(\A_{\G}\x\Z)$ with $\Fix\psi=BB_{\G}\x\Z$, where $BB_{\G}\leq\A_{\G}$ denotes the Bestvina--Brady subgroup \cite{Bestvina-Brady} (see Example~\ref{mother of all evil}). When finitely generated, $BB_{\G}$ is quadratically distorted in $\A_{\G}$ as soon as $\A_{\G}$ is directly irreducible and non-cyclic \cite{Tran-distortion}. Even when $\Fix\psi$ is finitely generated and undistorted, one can ensure that $\Fix\psi$ not be of type $F$, which implies that $\Fix\psi$ is not cocompactly cubulated. These examples can be easily extended to RAAGs that do not split as products.

We emphasise that the cubulation of $\Fix\varphi$ provided by Theorem~\ref{cmp und intro} does not arise from a \emph{convex} subcomplex of the cubulation of $G$ in general, but just from a \emph{median subalgebra} of it (see Subsection~\ref{median algebras subsect} for a definition). In fact, the subgroup $\Fix\varphi$ need not be \emph{quasi-convex} in $G$, as can be observed for the automorphism $\varphi\in\aut\Z^2$ that swaps the standard generators, where $\Fix\varphi$ is the diagonal subgroup of $\Z^2$.

Nevertheless, in many situations, $\Fix\varphi$ does turn out to be quasi-convex in the ambient group. We prove this fact in the context of right-angled Artin and Coxeter groups, where it has the remarkable consequence that $\Fix\varphi$ is a \emph{separable} subgroup \cite[Corollary~7.9]{Haglund-Wise-GAFA}.

\begin{thmintro}\label{U_0 cc intro}
Consider the right-angled Artin group $\A_{\G}$ or the right-angled Coxeter group $\W_{\G}$. There are finite-index subgroups $U_0(\A_{\G})\leq U(\A_{\G})$ and $\aut_0\W_{\G}\leq\aut\W_{\G}$ such that, for any automorphism $\varphi$ lying in either of these subgroups:
\begin{enumerate}
\item $\Fix\varphi$ is quasi-convex in $\A_{\G}$ or $\W_{\G}$ with respect to their standard word metric, i.e.\ geodesics in their standard Cayley graph with endpoints in $\Fix\varphi$ stay uniformly close to $\Fix\varphi$;
\item in particular, $\Fix\varphi$ is separable and it is a special group in the Haglund--Wise sense.
\end{enumerate}
\end{thmintro}

For the experts, the finite-index subgroups in Theorem~\ref{U_0 cc intro} are generated by the elementary automorphisms known as inversions, folds and partial conjugations (see Subsection~\ref{of RAAGs sect} and Remark~\ref{0 finite index}). Quasi-convexity of $\Fix\varphi$ can alternatively be characterised saying that $\Fix\varphi$ acts properly and cocompactly on a convex subcomplex of the universal cover of the Salvetti/Davis complex, or, again, in coarse median terms (see Definition~\ref{qc defn}, Remark~\ref{right-angled qc rmk} and Lemma~\ref{equiv cc}).

In light of Theorem~\ref{U_0 cc intro}, it is only natural to wonder what isomorphism types of special groups can arise as $\Fix\varphi$, and whether their complexity can be bounded in any way in terms of the ambient group, in the spirit of Scott's conjecture. We only provide a very partial result on these questions (Corollary~\ref{intro fix graph}), leaving a more detailed treatment for later work. The main proof ingredient, which we believe of independent interest, is the following construction of $U_0(\A_{\G})$--invariant Bass--Serre trees for most right-angled Artin groups.

\begin{propintro}\label{intro invariant splitting}
Let $\A_{\G}$ be directly irreducible, freely irreducible and non-cyclic. Then there exists an amalgamated product splitting $\A_{\G}=\A_+\ast_{\A_0}\A_-$, with $\A_{\pm}$ and $\A_0$ parabolic subgroups of $\A_{\G}$, such that the corresponding Bass--Serre tree $\A_{\G}\acts T$ is $U_0(\A_{\G})$--invariant. That is: for every $\varphi\in U_0(\A_{\G})$, there exists an isometry $f\colon T\ra T$ satisfying $f\o g=\varphi(g)\o f$ for all $g\in\A_{\G}$.
\end{propintro}

\begin{corintro}\label{intro fix graph}
Consider a right-angled Artin group $\A_{\G}$ and $\varphi\in U_0(\A_{\G})$. 
\begin{enumerate}
\item If $\A_{\G}$ splits as a direct product $\A_1\x\A_2$, then $\varphi(\A_i)=\A_i$ and $\Fix\varphi=\Fix\varphi|_{\A_1}\x\Fix\varphi|_{\A_2}$.
\item If $\A_{\G}$ is directly irreducible, then the subgroup $\Fix\varphi\leq\A_{\G}$ splits as a (possibly trivial) finite graph of groups with vertex and edge groups of the form $\Fix\varphi|_P$, for proper parabolic subgroups $P\leq\A_{\G}$ with $\varphi(P)=P$ and $\varphi|_{P}\in U_0(P)$.
\end{enumerate}
\end{corintro}

The same two results hold for right-angled Coxeter groups $\W_{\G}$ and automorphisms $\varphi\in\aut_0\W_{\G}$.

\medskip
We now turn to Question~\ref{Q2}, which is the second main focus of the paper. Recall that Paulin showed that, for every Gromov-hyperbolic group $G$, every infinite-order element of $\out G$ can be realised as a homothety of a small, isometric $G$--tree \cite{Paulin-ENS}.
% Paulin's proof seems to contain a significant gap in the first paragraph of page 157 (the fixed point can be 0). But this can be fixed using instead the Lefschetz fixed point theorem as we do here (though it seems to remain a problem for the more general statement about amenable subgroups of Out(G) that appears in Paulin's paper).

Our main result on Question~\ref{Q2}, generalises Paulin's theorem to CMP automorphisms of \emph{special groups} $G$, in the Haglund--Wise sense \cite{Haglund-Wise-GAFA,Sageev-notes}. This is a broad class of groups including right-angled Artin groups, finite-index subgroups of right-angled Coxeter groups, as well as free and surface groups and a number of other hyperbolic examples.

\begin{figure} 
\begin{tikzpicture}
\draw[fill] (0,-1) -- (0,1);
\draw[fill] (0,-1) -- (-1,-0.5);
\draw[fill] (0,-1) -- (1,-0.5);
\draw[fill] (-1,-0.5) -- (-1,0.5);
\draw[fill] (1,-0.5) -- (1,0.5);
\draw[fill] (0,1) -- (-1,0.5);
\draw[fill] (0,1) -- (1,0.5);
\draw[fill] (0,0) circle [radius=0.04cm];
\draw[fill] (0,1) circle [radius=0.04cm];
\draw[fill] (0,-1) circle [radius=0.04cm];
\draw[fill] (-1,0.5) circle [radius=0.04cm];
\draw[fill] (-1,-0.5) circle [radius=0.04cm];
\draw[fill] (1,0.5) circle [radius=0.04cm];
\draw[fill] (1,-0.5) circle [radius=0.04cm];
\end{tikzpicture}
\caption{}
\label{intro fig} 
\end{figure}

Note that \emph{small} $G$--actions on $\R$--trees are not the right notion to consider in this context. Indeed, if a special group $G$ has a small action on an $\R$--tree $T$, then every arc stabiliser is free abelian and the work of Rips and Bestvina--Feighn implies that $G$ splits over an abelian subgroup \cite[Theorem~9.5]{BF-stable}. However, there exist special groups that admit an infinite-order CMP outer automorphism, but do not split over any abelian subgroup (e.g.\ the RAAG $\A_{\G}$ with $\G$ as in Figure~\ref{intro fig}, by \cite{Groves-Hull-little}).
% stability follows from smallness here
% a virtually abelian subgroup of a RAAG is necessarily free abelian

In fact, due to the lack of hyperbolicity, it is reasonable to expect that $\R$--trees will need to be replaced by higher-dimensional analogues.

The correct setting seems to be provided by the simultaneous generalisation of $\R$--trees and $\CAT$ cube complexes known as \emph{median spaces}. These are those metric spaces $(X,d)$ such that, for all $x_1,x_2,x_3\in X$, there exists a unique point $m(x_1,x_2,x_3)$ (known as their \emph{median}) satisfying:
\[d(x_i,x_j)=d(x_i,m(x_1,x_2,x_3))+d(m(x_1,x_2,x_3),x_j),\ \forall 1\leq i<j\leq 3.\]
A connected median space X is said to have \emph{rank $\leq r$} if all its locally compact subsets have topological dimension $\leq r$. Rank--$1$ connected median spaces are precisely $\R$--trees. 

The following is our main result on Question~\ref{Q2} (a more general statement for infinite abelian subgroups of $\out G$ is Theorem~\ref{homothety 2}). Note that, although higher-rank median spaces are never non-positively curved, they always admit a canonical, bi-Lipschitz equivalent $\CAT$ metric\footnote{The reader should keep in mind the case of $\R^n$, where the $\ell^1$ metric is median and the Euclidean metric is $\CAT$.} \cite{Bow4}.

\begin{thmintro}\label{Q2 thm intro}
Let $G$ be the fundamental group of a compact special cube complex. Suppose $G$ has trivial centre. Let $\phi\in\out G$ be infinite-order and coarse-median preserving. Then:
\begin{enumerate}
\item there is a geodesic, finite-rank median space $X$ and an action by homotheties $G\rtimes_{\phi}\Z\acts X$;
\item the restriction $G\acts X$ is isometric, minimal, with unbounded orbits, and ``moderate'';
\item if $\varphi\in\aut G$ represents $\phi$, then the subgroup $\Fix\varphi\leq G$ fixes a point of $X$;
\item if $\phi$ and $\phi^{-1}$ are sub-exponentially growing, then the action $G\rtimes_{\phi}\Z\acts X$ is isometric.
\end{enumerate}
\end{thmintro}
% Q: Is $G\acts X$ cocompact?
% If we want to extend the theorem to virtually special groups, we should rather use groups with a virtually special cubulation. The problem is that, if we extend the coarse median to a finite-index supergroup, I don't think it will stay coarsely G--invariant in general (i.e. it will remain a coarse median group according to Bowditch, but not according to us).

As for actions on $\R$--trees, we say that $G\acts X$ is \emph{minimal} if $X$ does not contain any proper, $G$--invariant convex subsets. We propose the notion of ``moderate'' action on a median space as a higher-rank generalisation of the notion of small action on an $\R$--tree.

\begin{defn*}[Moderate actions]
Let $G$ be a group and $X$ be a median space.
\begin{enumerate} 
\item A \emph{$k$--cube} in $X$ is a median subalgebra $C\cu X$ isomorphic to the product $\{0,1\}^k$.
\item An isometric action $G\acts X$ is \emph{moderate} if, for every $k\geq 1$ and every $k$--cube $C\cu X$, the subgroup of $G$ fixing $C$ pointwise contains a copy of $\Z^k$ in its centraliser.
\end{enumerate}
\end{defn*}

Any $2$--element subset of $X$ is a $1$--cube. Thus, if $G$ is hyperbolic and $G\acts X$ is moderate, the intersection of any two point-stabilisers must be virtually cyclic. In particular, if $G$ is torsion-free hyperbolic and $T$ is an $\R$--tree, then the action $G\acts T$ is moderate if and only if it is small. We remark that, when $G$ is hyperbolic, the space $X$ provided by Theorem~\ref{Q2 thm intro} is indeed an $\R$--tree.

We would like to emphasise that Theorem~\ref{Q2 thm intro} does not provide any \emph{lower} bounds to the rank of the median space $X$. In particular, we still do not have an answer to the following:

\begin{quest}\label{trees Q}
\begin{enumerate}
\item[]
\item Can we always take the median space $X$ in Theorem~\ref{Q2 thm intro} to be an $\R$--tree?
\item If $G$ is a directly and freely irreducible RAAG, can we even take $X$ to be a simplicial tree?
\end{enumerate}
\end{quest}

We have seen that, when $\A_{\G}$ is directly and freely irreducible, Proposition~\ref{intro invariant splitting} yields a $U_0(\A_{\G})$--invariant simplicial $\A_{\G}$--tree. However, it remains unclear if such a simplicial tree can always be taken to be moderate and, more importantly, if it can be constructed so that $\Fix\varphi$ is elliptic.

\medskip
We conclude this overview by highlighting two more results. These fall outside the main purpose of this text, but they are almost immediate consequences of the techniques used in this paper and we find them of independent interest. We prove them at the end of Subsection~\ref{fg sect}.

Recall that the property of being cocompactly cubulated does not, in general, pass to finite-index \emph{overgroups}. Many examples of this are provided by crystallographic groups \cite{Hagen-cryst}: for instance, the $(3,3,3)$ triangle group has $\Z^2$ as a finite-index subgroup, but it is not itself cocompactly cubulated. 

The following is a criterion for cubulating finite-index overgroups. Its proof is loosely inspired by the idea of Guirardel cores \cite{Guirardel-core,Hagen-Wilton}, but it requires none of the technical machinery. Instead, it is a simple consequence of Proposition~\ref{approx subalgebras} (or the earlier \cite[Proposition~4.1]{Bow-hulls}).

\begin{corintro}\label{cmp fi overgroups}
Let $G$ be a group with a cocompactly cubulated finite-index subgroup $H$. Suppose that the coarse median structure on $G$ induced by the cubulation of $H$ is $G$--invariant (it is automatically $H$--invariant). Then $G$ is cocompactly cubulated.
\end{corintro}

Along with Proposition~\ref{cmp prop intro}, the previous corollary implies the following version of Nielsen realisation for automorphisms of right-angled Artin and Coxeter groups.

\begin{corintro}[Nielsen realisation for RA$\ast$Gs]\label{Nielsen intro}
Consider one of the following two settings:
\begin{enumerate}
\item a centreless right-angled Artin group $G=\A_{\G}$ and a finite subgroup $F\leq\out\A_{\G}$ contained in the projection to outer automorphisms of the untwisted subgroup $U(\A_{\G})\leq\aut\A_{\G}$;
\item a centreless right-angled Coxeter group $G=\W_{\G}$ and any finite subgroup $F\leq\out\W_{\G}$.
\end{enumerate}
Then $F$ can be realised as a group of automorphisms of a compact, non-positively curved, cube (orbi-)complex $Q$ with $G=\pi_1Q$.
\end{corintro}

Part~(2) is new, while part~(1) is originally due to Hensel and Kielak \cite{Hensel-Kielak}. When $F\leq U_0(\A_{\G})$, they constructed $Q$ quite explicitly via a glueing construction, ensuring that $\dim Q=\dim\X_{\G}$. By comparison, our approach does not offer much control on dimension (except $\dim Q\leq \#F\cdot\dim\X_{\G}$), but it provides a much more elementary proof of the existence of some $Q$. 

We expect our complex $Q$ to be special, but this would require additional arguments in the proof (the only delicate point being lack of inter-osculations). We also think it should be possible to ``trim'' $Q$ into having the optimal dimension $\dim\X_{\G}$ by relying on the ``panel collapse'' procedure of Hagen and Touikan \cite{Hagen-Touikan} (or small variations thereof), but the details seem too technical to be discussed here.

\subsection{On the proof of Theorems~\ref{cmp und intro} and~\ref{U_0 cc intro}}

The two theorems are proved in Section~\ref{fix cmp sect} under the aliases of Theorem~\ref{AMS thm} and Corollaries~\ref{fix RAAG cc} and~\ref{fix RACG cc}.

Regarding Theorem~\ref{cmp und intro}, the starting observation is that $\Fix\varphi$ is an \emph{approximate median subalgebra} of the group $G$ (see Definition~\ref{approx subalg defn} and Lemma~\ref{fix approximate subalgebra}). Fixing a proper cocompact action on a $\CAT$ cube complex $G\acts \mc{Z}$, the proof then takes place in three steps.
\begin{enumerate}
\item If a subgroup $H\leq G$ is an approximate median subalgebra, then $H$ is finitely generated (Proposition~\ref{AMS fg}). We prove this by relying on a straightforward adaptation of an argument due to Paulin in the context of hyperbolic groups \cite{Paulin-Fourier}. Paulin's argument is itself a generalisation of Cooper's proof in the case when the group $G$ is free \cite{Cooper-fix} (a result originally due to Gersten from the early 80s \cite{Gersten-fix}). 
\item Approximate median subalgebras of $\CAT$ cube complexes are always at finite Hausdorff distance from actual median subalgebras (Proposition~\ref{approx subalgebras} or \cite[Proposition~4.1]{Bow-hulls}).
\item Applying the previous step to $H$--orbits in $\mc{Z}$, we obtain an $H$--invariant median subalgebra $M\cu\mc{Z}^{(0)}$ such that $H\acts M$ is cofinite. Along with the fact that $H$ is finitely generated, this yields a cocompact cubulation that quasi-isometrically embeds into $\mc{Z}$ (Lemma~\ref{cubical undistortion criterion}), though not necessarily as a convex subcomplex.
\end{enumerate}

A similar strategy gives a new proof of W.\ Neumann's result that fixed subgroups of automorphisms of hyperbolic groups are quasi-convex \cite{Neumann-fix,Minasyan-Osin-fix}. Indeed, recall that, although not all hyperbolic groups are cocompactly cubulated, they are all coarse median, and all their automorphisms $\varphi$ are CMP by Proposition~\ref{cmp prop intro}. It is easy to see that all coarsely connected, approximate median subalgebras of hyperbolic spaces are quasi-convex. 
% because intervals are ``thin''
As above, this implies that $\Fix\varphi$ is quasi-convex.

When dealing with non-hyperbolic groups, quasi-convexity is significantly harder to ensure and the proof of Theorem~\ref{U_0 cc intro} requires additional work. Namely, assuming that $\varphi\in U_0(\A_{\G})$ or $\varphi\in\aut_0\W_{\G}$, we need to show that $(\Fix\varphi)$--orbits in the Salvetti complex $\X_{\G}$ or Davis complex $\mc{Y}_{\G}$ are quasi-convex (in the coarse median sense, see Definition~\ref{qc defn}, Remark~\ref{right-angled qc rmk} and Lemma~\ref{equiv cc}).

The proof of this is based on a quasi-convexity criterion for median subalgebras of $\CAT$ cube complexes (Proposition~\ref{wqc->qc}). The most important ingredients are the fact that $\X_{\G}$ and $\mc{Y}_{\G}$ do not contain ``infinite staircases'' (Subsection~\ref{staircases sect}), and certain properties that distinguish elements of $U_0(\A_{\G})$ and $\aut_0\W_{\G}$ from more general CMP automorphisms in $U(\A_{\G})$ and $\aut\W_{\G}$ (Lemmas~\ref{orthogonals are preserved} and~\ref{few labels outside}).

We conclude by mentioning that other important tools for the study of undistortion and quasi-convexity of subgroups of cubulated groups were recently developed by Beeker--Lazarovich and Dani--Levcovitz, based on extensions of the classical machinery of \emph{Stallings folds} \cite{Stallings-folds,Stallings-arboreal} from graphs to higher-dimensional cube complexes (see in particular \cite{BL-res}, \cite[Theorem~1.2(2)]{BL-folds}, \cite[Theorem~A]{Dani-Levcovitz}). These techniques play no role in our arguments, but it is possible that they can be used to give alternative proofs of certain special cases of Theorems~\ref{cmp und intro} and~\ref{U_0 cc intro}.

\subsection{On the proof of Theorem~\ref{Q2 thm intro}}

Keeping the case of $\out F_n$ in mind, as described e.g.\ in \cite[Section~2]{GJLL}, there are two main obstacles to overcome.

\begin{enumerate}
\item[(a)] No good analogue of \emph{(relative) train track maps} is available to represent homotopy equivalences between non-positively curved cube complexes.
\item[(b)] It is not known if (isometric) actions on finite-rank median spaces are completely determined by their length function. There are results of this type for actions on $\R$--trees \cite{Culler-Morgan} and cube complexes \cite{BF1,BF2}, but their extension to a general median setting would require some significantly new ideas.
\end{enumerate}

The proof of Theorem~\ref{Q2 thm intro} is made up of two main steps, which we now describe. In this sketch, we restrict our attention to the construction of the homothetic action $G\rtimes_{\phi}\Z\acts X$ (parts~(1) and~(2) of the theorem). Parts~(3) and~(4) follow, respectively, from parts~(1) and~(2) of Remark~\ref{addendum to Q2 thm}.

Let $G$ be a special group, let $\mc{Z}$ be a $\CAT$ cube complex, and let $\rho\colon G\ra\aut\mc{Z}$ be the homomorphism corresponding to a proper, cocompact, cospecial action $G\acts\mc{Z}$. Equip $G$ with the coarse median structure arising from $\mc{Z}$. Let $\varphi\in\aut G$ be a coarse-median preserving automorphism projecting to an infinite-order element of $\out G$. 

\smallskip
{\bf Step~1:} \emph{there exist a finite-rank median space $X$, an isometric action $G\acts X$ with unbounded orbits, and a homeomorphism $H\colon X\ra X$ satisfying $H\o g=\varphi(g)\o H$ for all $g\in G$.}

\smallskip
In order to prove this, we consider the sequence of homomorphisms $\rho_n:=\rho\o\varphi^n$ and the sequence of $G$--actions on cube complexes $G\acts\mc{Z}_n$ that they induce. We then fix a non-principal ultrafilter $\om$, choose basepoints $p_n\in\mc{Z}_n$ and scaling factors $\eps_n>0$, and consider the ultralimit:
\[ (X,p):=\lim_{\om}(\eps_n\mc{Z}_n,p_n).\]
This is easily seen to be a finite-rank median space and, for a suitable choice of $p_n$ and $\l_n$, the actions $G\acts\mc{Z}_n$ converge to an isometric action $G\acts X$ with unbounded orbits.

So far this is just a classical Bestvina--Paulin construction \cite{Bestvina-Duke,Paulin-IM}. The actual subtleties lie in the definition of the map $H\colon X\ra X$. By the Milnor--Schwarz lemma, there exists a quasi-isometry $h\colon\mc{Z}\ra\mc{Z}$ satisfying $h\o g=\varphi(g)\o h$ for all $g\in G$. We would like to define $H$ as the ultralimit of the corresponding sequence of quasi-isometries $\mc{Z}_n\ra\mc{Z}_n$, but this might displace the basepoint $p\in X$ by an infinite amount.

In order to rule out this eventuality, we rely on an argument similar to the one used in \cite{Paulin-ENS} for hyperbolic groups. On closer inspection, Paulin's argument only requires the following property, which is satisfied by non-elementary hyperbolic groups.

\begin{defn*}
Let $G$ be a infinite group with a (fixed) Cayley graph $(\mc{G},d)$. We say that $G$ is \emph{uniformly non-elementary (UNE)} if there exists a constant $c>0$ with the following property. For every finite generating set $S\cu G$ and for all $x,y\in\mc{G}$, we have:
\[d(x,y)\leq c\cdot\max_{s\in S}[d(x,sx)+d(y,sy)].\]
\end{defn*}

The important part of this definition is that the constant $c$ \emph{does not} depend on the generating set $S$. Note that the UNE property is independent of the specific choice of $\mc{G}$ (cf.\ Definition~\ref{UNE defn}). 

Our main contribution to Step~1 is the proof of the following fact (Corollary~\ref{UNE cor}), which is potentially of independent interest.

\begin{thmintro}\label{special UNE thm intro}
Let $G$ be the fundamental group of a compact special cube complex. If $G$ has trivial centre, then $G$ is uniformly non-elementary.
\end{thmintro}

Now, let $m\colon X^3\ra X$ denote the median operator of the median space $X$. The fact that $\varphi\in\aut G$ is coarse-median preserving easily implies that the homeomorphism $H\colon X\ra X$ arising from the above construction satisfies $H(m(x,y,z))=m(H(x),H(y),H(z))$ for all $x,y,z\in X$. However, $H$ needs not be a homothety at this stage.

\smallskip
{\bf Step~2:} \emph{there exists a $G$--invariant (pseudo-)metric $\eta\colon X\x X\ra[0,+\infty)$ such that $(X,\eta)$ is a median space with the same median operator $m$, and $H$ is a homothety with respect to $\eta$.}

\smallskip
Since $H\colon X\ra X$ preserves the median operator $m$, there is an action of $H$ on the space of all $G$--invariant median pseudo-metrics on $X$ that induce $m$. More precisely, we show that $H$ gives a homeomorphism of a certain space of (projectivised) median pseudo-metrics on $X$, and that the latter is a compact \emph{absolute retract (AR)}. The existence of the required pseudo-metric $\eta$ then follows from the Lefschetz fixed point theorem for homeomorphisms of compact ANRs. This is discussed mainly in Subsections~\ref{metrics on algebras} and~\ref{meaty sect} (see especially Corollaries~\ref{homothety 1} and~\ref{from sub-ultras to X new cor}).

Once the pseudo-metric $\eta$ is obtained, we can pass to the quotient metric space to obtain a genuine median space.

\subsection{Further questions.} 

We would like to highlight four questions raised by our results. 

As mentioned earlier, every hyperbolic group admits a unique \emph{coarse median structure} (Definition~\ref{structure defn}). At the opposite end of the spectrum, any RAAG for which $U(\A_{\G})$ has infinite index in $\aut\A_{\G}$ will admit infinitely many $\A_{\G}$--invariant coarse median structures. 

Right-angled Coxeter groups $\W_{\G}$ seem to place themselves in-between these two extremal situations: they can admit infinitely many distinct coarse median structures --- e.g.\ because every RAAG is a finite-index subgroup of a RACG \cite{Davis-Janusz} --- but it is not clear which of these structures are $\W_{\G}$--invariant. For instance, Proposition~\ref{cmp prop intro}(2) implies that all Coxeter generating sets of $\W_{\G}$ give rise to the same coarse median structure (which fails for Artin generating sets of $\A_{\G}$).

\begin{quest}
Does each RACG $\W_{\G}$ have only \emph{finitely many} $\W_{\G}$--invariant coarse median structures?
\end{quest}

As an example of why one might expect this kind of rigidity, we suggest looking at the difference between the RAAG $\Z^n$ and the RACG $(D_{\infty})^n$, where $D_{\infty}$ is the infinite dihedral group. The space of $\Z^n$--invariant coarse median structures on $\Z^n$ (equivalently: on $\R^n$) is uncountable, simply because it is endowed with a natural $GL_n\R$--action and we can consider the orbit of the standard structure. 
% the stabiliser is the subgroup of diagonal matrices
However, of the structures in this orbit, only finitely many are $(D_{\infty})^n$--invariant.
% because only diagonal matrices commute with all diagonal matrices that have only +1s on the diagonal except for one -1 

The second question naturally arises from Theorem~\ref{U_0 cc intro} and was already mentioned above:

\begin{quest}\label{cc quest}
Consider $\varphi\in U_0(\A_{\G})$ or $\varphi\in\aut_0\W_{\G}$. 
\begin{enumerate}
\item What isomorphism types of special groups can arise as $\Fix\varphi$ for some choice of $\varphi$ and $\G$? When $\varphi\in U_0(\A_{\G})$, is $\Fix\varphi$ itself a right-angled Artin group?
\item Can we bound the ``complexity'' of $\Fix\varphi$ in terms of $\#\G^{(0)}$, in the spirit of Scott's conjecture? %What is the right notion of complexity?
\end{enumerate}
\end{quest}
% fixing the RAAG, are there only finitely many isomorphism types of fixed subgroups of elements of U_0?

Regarding part~(1) of Question~\ref{cc quest}, note that every RAAG can arise as the fixed subgroup of some element of $U_0(\A_{\G})$, simply because we can always take $\varphi=\id$. One can easily construct more elaborate examples using this observation as a starting point.

One can also wonder about fixed subgroups of automorphisms of general coarse median groups $G$. By Lemma~\ref{fix approximate subalgebra}, this reduces to understanding subgroups that are \emph{approximate median subalgebras} (Definition~\ref{approx subalg defn}). We study these subgroups when $G$ is cocompactly cubulated (Theorem~\ref{AMS thm}), but some of our arguments should work more generally (especially the proof of Proposition~\ref{AMS fg}).

\begin{quest}
Let $(G,\mu)$ be a finite-rank coarse median group. Let a subgroup $H\leq G$ be an approximate median subalgebra. 
\begin{enumerate}
\item Is $H$ finitely generated? 
\item Is $H$ undistorted? Which properties of $G$ does $H$ retain?
\end{enumerate}
\end{quest}

For instance, when $G$ is hierarchically hyperbolic, I do not know if $H$ must be finitely generated. However, assuming that it is, the second part of the question has a positive answer: $H$ is undistorted and hierarchically hyperbolic. This is evident from Bowditch's axioms (B1)--(B10) for (weak) hierarchically hyperbolic spaces \cite[Section~7]{Bow-hulls} and the coarse median characterisation of hierarchy paths \cite[Theorem~1.1]{Bow-hulls}.

We emphasise that our definition of \emph{coarse median group} (Definition~\ref{coarse median group defn}) is slightly stronger than Bowditch's original definition \cite{Bow-cm}, in that we require $\mu$ to be coarsely $G$--equivariant.

Our last question regards UNE groups. It is clear that UNE groups have finite centre, and it is not hard to show that non-elementary hyperbolic groups are UNE. All other examples of UNE groups that we are aware of are provided by Theorem~\ref{special UNE thm intro}.

Are there other interesting examples or non-examples of UNE groups? Given the proof of Theorem~\ref{special UNE thm intro}, a positive answer to the following seems likely:

\begin{quest}
Are hierarchically hyperbolic groups with finite centre UNE?
\end{quest}

{\bf Outline of the paper.} Section~\ref{prelim sect} mostly contains background material on median algebras, cube complexes and coarse median groups. An exception is Subsection~\ref{cores sect}, which reviews some of the results of \cite{Fio10b}. The latter will be helpful, mostly in Sections~\ref{compatible sect} and~\ref{ultra sect}, for some of the more technical arguments in the proof of Theorem~\ref{Q2 thm intro}.

In Section~\ref{cc whole sect}, we consider cocompactly cubulated groups $G$ and study a notion of \emph{convex-co\-com\-pact\-ness} for subgroups of $G$, which is a special instance of quasi-convexity in coarse median spaces (Definition~\ref{qc defn}). Subsection~\ref{label-irreducible sect} studies cyclic, convex-cocompact subgroups of RAAGs (whose generators we name \emph{label-irreducible}). Subsection~\ref{of RAAGs sect} contains the proof of Proposition~\ref{cmp prop intro}.

Section~\ref{fix cmp sect} is concerned with fixed subgroups of CMP automorphisms. First, Subsections~\ref{approx subalg sect} and~\ref{fg sect} are devoted to the proof of Theorem~\ref{cmp und intro}. Then Subsection~\ref{staircases sect} studies staircases in cube complexes, allowing us to formulate a quasi-convexity criterion for median subalgebras in Subsection~\ref{wqc sect}. Finally, Subsection~\ref{fix cc sect} restricts to Salvetti and Davis complexes, proving Theorem~\ref{U_0 cc intro}.

Section~\ref{invariant splittings sect} is completely independent from the subsequent part of the paper and can be safely skipped. It only contains the proof of Proposition~\ref{intro invariant splitting} and Corollary~\ref{intro fix graph}.

Finally, Sections~\ref{compatible sect} and~\ref{ultra sect} are the most technical parts of the paper and they contain the bulk of the proof of Theorem~\ref{Q2 thm intro}. In Section~\ref{compatible sect}, we consider group actions on finite-rank median algebras and develop a criterion for the existence of a (projectively) invariant metric (as required for Step~2 of the proof sketch for Theorem~\ref{Q2 thm intro}). In Section~\ref{ultra sect}, we study ultralimits of actions on Salvetti complexes, in order to obtain the properties needed to apply the results of Section~\ref{compatible sect}. Theorems~\ref{Q2 thm intro} and~\ref{special UNE thm intro} are proved in Subsection~\ref{meaty sect}.

\smallskip
{\bf Acknowledgments.} I would like to thank Mark Hagen for noticing that a result similar to Proposition~\ref{approx subalgebras} already appeared in \cite{Bow-hulls}, Camille Horbez for mentioning that ARs have good fixed-point properties, Robert Kropholler for pointing me to \cite{Tran-distortion} and \cite{Hagen-cryst}, Alessandro Sisto for discussions related to Definition~\ref{UNE defn}, and Ric Wade for suggesting a major simplification to my original argument for Proposition~\ref{cmp prop intro} in the case of RAAGs. I thank Jason Behrstock, Corey Bregman, Ruth Charney, Ashot Minasyan and Karen Vogtmann for other interesting conversations, and the referee for their many useful comments.

I am grateful to Ursula Hamenst\"adt and the Max Planck Institute for Mathematics in Bonn for their hospitality and financial support while part of this work was being completed.

\addtocontents{toc}{\protect\setcounter{tocdepth}{2}}
\section{Preliminaries.}\label{prelim sect}

\subsection{Frequent notation and identities.}\label{identities sect}

Throughout the paper, all groups will be equipped with the discrete topology. Thus, we will refer to \emph{properly discontinuous} actions on topological spaces simply as \emph{proper} actions.

If $G$ is a group and $F\cu G$ is a subset, we denote by $\langle F\rangle$ the subgroup of $G$ generated by $F$. We denote by $Z_G(F)$ the \emph{centraliser} of the subset $F$, i.e.\ the subgroup of elements of $G$ commuting with all elements of $F$.

If $(X,d)$ is a metric space, $A\cu X$ is a subset, and $R\geq 0$ is a real number, we denote by $\mc{N}_R(A)$ the closed $R$--neighbourhood of $A$. If $x,y\in X$, we write $x\approx_R y$ with the meaning of $d(x,y)\leq R$.

\smallskip
Consider a group action on a set $G\acts X$. If $\eta$ is a $G$--invariant pseudo-metric on $X$, we write, for every $x\in X$, $g\in G$, and $F\cu G$:
\begin{align*}
\ell(g,\eta)=\inf_{x\in X}\eta(x,gx), & & \tau_F^{\eta}(x)=\max_{f\in F}\eta(x,fx), & & \overline{\tau}_F^{\eta}=\inf_{x\in X}\tau_F^{\eta}(x).
\end{align*}
When $X$ is a metric space and we do not name its metric explicitly, we also write: $\ell(g,X)$, $\tau_F^X$, $\overline\tau_F^X$. If $X$ is equipped with several $G$--actions originating from homomorphisms $\rho_n\colon G\ra\isom X$, we will write $\ell(g,\rho_n)$, $\tau_F^{\rho_n}$, $\overline\tau_F^{\rho_n}$ in order to avoid confusion.

If $S\cu G$ is a finite generating set, we denote by $|\cdot|_S$ and $\|\cdot\|_S$ the associated \emph{word length} and \emph{conjugacy length}, respectively:
\begin{align*}
|g|_S&=\inf \{k \mid g=s_1\cdot\ldots\cdot s_k,\ s_i\in S^{\pm}\}, & \|g\|_S&=\inf_{h\in G} |hgh^{-1}|_S.
\end{align*}

The following useful identities will be repeatedly used in this text. We consider a $G$--action on a set $X$, a $G$--invariant pseudo-metric $\eta$, a point $x\in X$, and finite generating sets $S,S_1,S_2\cu G$:
\begin{align*}
\eta(x,gx)&\leq |g|_S\cdot\tau_S^{\eta}(x), & \ell(g,\eta)&\leq \|g\|_S\cdot\overline\tau^{\eta}_S;
\end{align*}
\begin{align*}
\tau_{S_1}^{\eta}(x)&\leq |S_1|_{S_2}\cdot\tau_{S_2}^{\eta}(x), & \text{ where we have defined: \hspace{.2cm} } |S_1|_{S_2}&:=\max_{s\in S_1}|s|_{S_2}.
\end{align*}

%\begin{rmk}\label{translation length vs conjugacy length}
%Let $G\acts X$ be a proper, cocompact action by isometries on a geodesic metric space. Then, for every finite generating set $S\cu G$, there exists a constant $c(S)\geq 1$ such that, for every $g\in G$, we have:
%\[\tfrac{1}{c(S)}\cdot\|g\|_S\leq\ell(g,X)\leq c(S)\cdot\|g\|_S.\]
%The top inequality follows from the above identities if we take $c(S)\geq\overline\tau_S^X$. It suffices to prove the bottom inequality for a specific generating set. This is a standard observation using the generating set provided by \cite[Theorem~I.8.10(1)]{BH}.
%\end{rmk}

\subsection{Median algebras.}\label{median algebras subsect}

In this and the next section, we only fix notation and prove a few simple facts that do not appear elsewhere in the literature. For a comprehensive introduction to median algebras and median spaces, the reader can consult \cite[Sections~(2)--(4)]{CDH}, \cite[Sections~(4)--(6)]{Bow-cm} and \cite[Section~2]{Fio1}.

A median algebra is a pair $(M,m)$, where $M$ is a set and $m\colon M^3\ra M$ is a map satisfying, for all $a,b,c,x\in M$:
\begin{align*}
m(a,a,b)&=a, & m(a,b,c)&=m(b,c,a)=m(b,a,c), & m(m(a,x,b),x,c)&=m(a,x,m(b,x,c)).
\end{align*}
The third identity, usually known as the \emph{4--point condition}, is sometimes replaced by a different identity involving 5 points (for instance, in \cite{Roller,CDH,Bow-cm,Fio1}). The equivalence of the two conditions \cite{Kolibiar-Marcisova,Bandelt-Hedlikova} is quite nontrivial, but not required in the rest of the paper.

A map $\phi\colon M\ra N$ between median algebras is a \emph{median morphism} if, for all $x,y,z\in M$, we have $\phi(m(x,y,z))=m(\phi(x),\phi(y),\phi(z))$. We denote by $\aut M$ the group of median automorphisms of $M$. Throughout the paper, all group actions on median algebras will be by (median) automorphisms, unless stated otherwise.

A subset $S\cu M$ is a \emph{median subalgebra} if $m(S\x S\x S)\cu S$. A subset $C\cu M$ is \emph{convex} if $m(C\x C\x M)\cu C$. Helly's lemma states that any finite family of pairwise-intersecting convex subsets of $M$ has nonempty intersection \cite[Theorem~2.2]{Roller}. We say that $C$ is \emph{gate-convex} if it admits a \emph{gate-projection}, i.e.\ a map $\pi_C\colon M\ra C$ with the property that $m(z,\pi_C(z),x)=\pi_C(z)$ for all $x\in C$ and $z\in M$. Gate-convex subsets are convex, and convex subsets are median subalgebras. Each gate-convex subset admits a unique gate-projection, and gate-projections are median morphisms.

The \emph{interval} $I(x,y)$ between points $x,y\in M$ is defined as the set $\{z\in M\mid m(x,y,z)=z\}$. Note that $I(x,y)$ is gate-convex with projection given by the map $z\mapsto m(x,y,z)$. Intervals can be used to give an alternative description of convexity: a subset $C\cu M$ is convex if and only if $I(x,y)\cu C$ for all $x,y\in C$.

A \emph{halfspace} is a subset $\mf{h}\cu M$ such that both $\mf{h}$ and $\mf{h}^*:=M\setminus\mf{h}$ are convex and nonempty. A \emph{wall} is a set of the form $\mf{w}=\{\mf{h},\mf{h}^*\}$, where $\mf{h}$ and $\mf{h}^*$ are halfspaces. We say that $\mf{w}$ is the wall \emph{bounding} $\mf{h}$, and that $\mf{h}$ and $\mf{h}^*$ are the halfspaces \emph{associated} to $\mf{w}$. 

Two halfspaces $\mf{h}_1,\mf{h}_2$ are \emph{transverse} if all four intersections $\mf{h}_1\cap\mf{h}_2$, $\mf{h}_1^*\cap\mf{h}_2$, $\mf{h}_1\cap\mf{h}_2^*$, $\mf{h}_1^*\cap\mf{h}_2^*$ are nonempty. If $\mf{w}_1$ and $\mf{w}_2$ are the walls bounding $\mf{h}_1$ and $\mf{h}_2$, we also say that $\mf{w}_1$ is \emph{transverse} to $\mf{w}_2$ and $\mf{h}_2$. If $\mc{U}$ and $\mc{V}$ are sets of walls or halfspaces, we say that $\mc{U}$ and $\mc{V}$ are \emph{transverse} if every element of $\mc{U}$ is transverse to every element of $\mc{V}$. If $\mc{H}$ is a set of halfspaces, we write $\mc{H}^*:=\{\mf{h}^*\mid\mf{h}\in\mc{H}\}$.

We denote by $\mscr{W}(M)$ and $\mscr{H}(M)$, respectively, the set of all walls and all halfspaces of $M$. Given subsets $A,B\cu M$, we write:
\begin{align*}
\mscr{H}(A|B)&=\{\mf{h}\in\mscr{H}(M)\mid A\cu\mf{h}^*,\ B\cu\mf{h}\}, & \mscr{W}(A|B)&=\{\mf{w}\in\mscr{W}(M)\mid \mf{w}\cap\mscr{H}(A|B)\neq\emptyset\}.
\end{align*}
If $\mf{w}_1,\mf{w}_2$ are walls bounding disjoint halfspaces $\mf{h}_1,\mf{h}_2$, we set $\mscr{W}(\mf{w}_1|\mf{w}_2):=\mscr{W}(\mf{h}_1|\mf{h}_2)\setminus\{\mf{w}_1,\mf{w}_2\}$. 

If $A,B\cu M$ are nonempty, then $\mscr{H}(A|B)$ admits minimal elements under inclusion. This follows from Zorn's lemma since, for every totally ordered subset $\mscr{C}\cu\mscr{H}(A|B)$, the intersection of all halfspaces in $\mscr{C}$ is again a halfspace in $\mscr{H}(A|B)$. Note that any two minimal elements $\mf{h}_1,\mf{h}_2\in\mscr{H}(A|B)$ are transverse, since $\mf{h}_1\cap\mf{h}_2$ and $\mf{h}_1^*\cap\mf{h}_2^*$ are nonempty and there is no inclusion relation between $\mf{h}_1$ and $\mf{h}_2$.

If $\mf{w}\in\mscr{W}(A|B)$, we say that the wall $\mf{w}$ \emph{separates} $A$ and $B$. Any two disjoint convex subsets of $M$ are separated by at least one wall \cite[Theorem~2.8]{Roller}; in particular, distinct points of $M$ are always separated by a wall. 

Given a subset $A\cu M$, we also introduce:
\begin{align*}
\mscr{H}_A(M):=&\{\mf{h}\in \mscr{H}(M) \mid \mf{h}\cap A\neq\emptyset,\ \mf{h}^*\cap A\neq\emptyset\}, & \mscr{W}_A(M):=&\{\mf{w}\in\mscr{W}(M) \mid \mf{w}\cu\mscr{H}_A(M)\}.
\end{align*}
Equivalently, a wall $\mf{w}$ lies in $\mscr{W}_A(M)$ if and only if it separates two points of $A$.

\begin{rmk}\label{median homo rmk}
If $\mc{U}\cu\mscr{H}(M)$ and $\mc{V}\cu\mscr{H}(N)$ are subsets, we say that a map $\phi\colon\mc{U}\ra\mc{V}$ is a \emph{morphism of pocsets} if, for all $\mf{h},\mf{k}\in\mc{U}$ with $\mf{h}\cu\mf{k}$, we have $\phi(\mf{h})\cu\phi(\mf{k})$ and $\phi(\mf{h}^*)=\phi(\mf{h})^*$.

Every median morphism $\phi\colon M\ra N$ induces a morphism of pocsets $\phi^*\colon\mscr{H}_{\phi(M)}(N)\ra\mscr{H}(M)$ defined by $\phi^*(\mf{h})=\phi^{-1}(\mf{h})$. When $\phi\colon M\ra N$ is surjective, we obtain a map $\phi^*\colon\mscr{H}(N)\ra\mscr{H}(M)$ that is injective and preserves transversality.
\end{rmk}

\begin{rmk}\label{halfspaces of subsets} 
\begin{enumerate}
\item[] 
\item If $S\cu M$ is a subalgebra, we have a map $\res_C\colon\mscr{H}_C(M)\ra\mscr{H}(C)$ given by $\res_C(\mf{h})=\mf{h}\cap C$. This is a morphism of pocsets and, by \cite[Lemma~6.5]{Bow-cm}, it is a surjection.
\item If $C\cu M$ is convex, then the map $\res_C$ is also injective and it preserves transversality. In particular, the sets $\mscr{H}(C)$ and $\mscr{H}_C(M)$ are naturally identified in this case.

Indeed, if $\mf{h},\mf{k}\in\mscr{H}_C(M)$ are intersecting halfspaces, Helly's lemma guarantees that $\mf{h}\cap C$ and $\mf{k}\cap C$ intersect too. Moreover, we have $\mf{h}=\mf{k}$ if and only if $\mf{h}\cap\mf{k}^*$ and $\mf{h}^*\cap\mf{k}$ are empty.
\item If $C$ is gate-convex with projection $\pi_C$, then $\res_C\o\pi_C^*=\id_{\mscr{H}(C)}$ and $\pi_C^*\o\res_C=\id_{\mscr{H}_C(M)}$.
\end{enumerate}
\end{rmk}

If $C_1,C_2\cu M$ are gate-convex subsets with gate-projections $\pi_1,\pi_2$, then $\mscr{H}(x|C_i)=\mscr{H}(x|\pi_i(x))$ for all $x\in M$. We say that $x_1\in C_1$ and $x_2\in C_2$ are a \emph{pair of gates} if $\pi_2(x_1)=x_2$ and $\pi_1(x_2)=x_1$. Pairs of gates always exist and satisfy $\mscr{H}(x_1|x_2)=\mscr{H}(C_1|C_2)$.

The \emph{standard $k$--cube} is the finite set $\{0,1\}^k$ equipped with the median operator $m$ determined by a majority vote on each coordinate. A subset $S\cu M$ is a \emph{$k$--cube} if it is a median subalgebra isomorphic to the standard $k$--cube. In particular, any subset of $M$ with cardinality $2$ is a $1$--cube.

\begin{rmk}
An important example of median algebra is provided by the $0$--skeleton of any $\CAT$ cube complex $X$ \cite{Chepoi}. The vertex set of any $k$--cell of $X$ is a $k$--cube in the above sense, but the converse does not hold. For instance, in the standard tiling of $\R^n$, every set of the form $\{a_1,b_1\}\x\dots\x\{a_n,b_n\}$ with $a_i<b_i$ is a $k$--cube according to the above notion. To avoid confusion, when dealing with cube complexes we will refer to $k$--cubes in $X^{(0)}$ as \emph{generalised $k$--cubes}.
\end{rmk}

The \emph{rank} of $M$, denoted $\rk M$, is the largest cardinality of a set of pairwise-transverse walls of $M$. Equivalently, $\rk M$ is the supremum of the integers $k$ such that $M$ contains a $k$--cube (assuming $\rk M$ is at most countable). See \cite[Proposition~6.2]{Bow-cm}. We will be exclusively interested in median algebras of finite rank.

We will need the following criterion, which summarises Lemmas~2.9 and~2.11 in \cite{Fio10b}. If $\mc{H}\cu\mscr{H}(M)$, we denote by $\bigcap\mc{H}\cu M$ the intersection of all halfspaces in $\mc{H}$.

\begin{lem}\label{useful criterion}
Let $M$ be a finite-rank median algebra. 
Partially order $\mscr{H}(M)$ by inclusion.
\begin{enumerate}
\item Let $\mc{H}\cu\mscr{H}(M)$ be a set of pairwise intersecting halfspaces. Suppose that every chain in $\mc{H}$ admits a lower bound in $\mc{H}$. Then $\bigcap\mc{H}$ is a nonempty convex subset of $M$.
\item A convex subset $C\cu M$ is gate-convex if and only if there does not exist a chain $\mscr{C}\cu\mscr{H}_C(M)$ such that $\bigcap\mscr{C}$ is nonempty and disjoint from $C$.
\end{enumerate}
\end{lem}

If $A\cu M$ is a subset, we denote by $\langle A\rangle$ the median subalgebra generated by $A$, i.e.\ the smallest subalgebra of $M$ containing $A$. We also denote by $\hull A$ the smallest convex subset of $M$ that contains $A$; this coincides with the intersection of all halfspaces of $M$ that contain $A$. 

The sets $\langle A\rangle$ and $\hull A$ are best understood in terms of the following operators:
\begin{align*}
\mc{M}(A)=\mc{M}^1(A)&:=m(A\x A\x A), & \mc{M}^{n+1}(A)&:=\mc{M}(\mc{M}^n(A)); \\
\mc{J}(A)=\mc{J}^1(A)&:=m(A\x A\x M)=\bigcup_{x,y\in A}I(x,y), & \mc{J}^{n+1}(A)&:=\mc{J}(\mc{J}^n(A)).
\end{align*}
It is clear that $\hull A=\bigcup_{n\geq 1}\mc{J}^n(A)$ and $\langle A\rangle=\bigcup_{n\geq 1}\mc{M}^n(A)$. 

\begin{rmk}\label{Bow J}
When $\rk M=r$ is finite, \cite[Lemma~6.4]{Bow-cm} shows that already $\mc{J}^r(A)=\hull A$. A similar result holds for $\langle A\rangle$ and the operator $\mc{M}$ (see Proposition~\ref{gen subalg prop} below), but its proof will require considerable work. 
\end{rmk}

If $M_1$ and $M_2$ are median algebras, we denote by $M_1\x M_2$ their \emph{product}. This is the median algebra with underlying set $M_1\x M_2$ and the only median operator for which both coordinate projections are median morphisms. 

The set $\mscr{W}(M_1\x M_2)$ is naturally partitioned into two transverse subsets $\mscr{W}_1$ and $\mscr{W}_2$. A wall lies in $\mscr{W}_1$ if and only if it separates two points in one (equivalently, every) fibre $M_1\x\{\ast\}$; halfspaces associated to walls in $\mscr{W}_1$ are unions of fibres $\{\ast\}\x M_2$. The set $\mscr{W}_2$ is defined similarly, swapping the roles played by the two indices. Since all fibres are gate-convex in $M_1\x M_2$, Remark~\ref{halfspaces of subsets} gives natural identifications between $\mscr{W}_i$ and $\mscr{W}(M_i)$.

In finite rank, product splittings can be completely characterised in terms of walls. The following is \cite[Lemma~2.12]{Fio10b} (also see \cite[Lemma~2.5]{CS} in the special case of cube complexes).

\begin{lem}\label{product lem}
For a finite-rank median algebra $M$, the following are equivalent:
\begin{enumerate}
\item $M$ splits as a product of median algebras $M_1\x M_2$, where neither $M_i$ is a singleton;
\item there exists a partition $\mscr{W}(M)=\mscr{W}_1\sqcup\mscr{W}_2$, where the $\mscr{W}_i$ are nonempty and transverse.
\end{enumerate}
When this happens, the set $\mscr{W}_i$ is identified with $\mscr{W}(M_i)$ as described above.
\end{lem}

\subsection{Compatible metrics on median algebras.}\label{compatible subsec}

A metric space $(X,d)$ is a \emph{median space} if, for all $x_1,x_2,x_3\in X$, there exists a unique point $m(x_1,x_2,x_3)\in X$ such that
\[d(x_i,x_j)=d(x_i,m(x_1,x_2,x_3))+d(m(x_1,x_2,x_3),x_j)\]
for all $1\leq i<j\leq 3$. In this case, the map $m\colon X^3\ra X$ gives a median algebra $(X,m)$.

\begin{rmk}[Rank of median spaces]
We define the \emph{rank} of $X$ as the rank of the underlying median algebra $(X,m)$. If $X$ is a connected median space, then this notion of rank coincides with the supremum of the topological dimensions of the locally compact subsets of $X$. The latter is the definition of rank that we used in the Introduction. One inequality follows from Theorem~2.2 and Lemma~7.6 in \cite{Bow-cm}, while the other from \cite[Proposition~5.6]{Bow4}.
\end{rmk}

For the purposes of this paper, it is convenient to think of median spaces in terms of the following notion. Let $M$ be a median algebra.

\begin{defn}\label{compatible defn}
A pseudo-metric $\eta\colon M\x M\ra[0,+\infty)$ is \emph{compatible} if, for every $x,y,z\in M$:
\[\eta(x,y)=\eta(x,m(x,y,z))+\eta(m(x,y,z)+y).\]
\end{defn}

Thus, we can equivalently define median spaces as pairs $(M,d)$, where $M$ is a median algebra and $d$ is a compatible metric on $M$.

We write $\mc{D}(M)$ and $\mc{PD}(M)$, respectively, for the sets of all compatible metrics and all compatible pseudo-metrics on $M$. In the presence of a group action $G\acts M$, we write $\mc{D}^G(M)$ and $\mc{PD}^G(M)$ for the subsets of $G$--invariant (pseudo-)metrics (or just $\mc{D}^g(M)$ and $\mc{PD}^g(M)$ if $G=\langle g\rangle$).

To avoid confusion, we will normally denote compatible \emph{metrics} by the letter $\delta$, and general compatible \emph{pseudo-metrics} by the letter $\eta$.

Consider a gate-convex subset $C\cu M$ and its gate-projection $\pi_C\colon M\ra C$. For every pseudo-metric $\eta\in\mc{PD}(M)$, the maps $\pi_C\colon M\ra C$ and $m\colon M^3\ra M$ are $1$--Lipschitz, in the sense that:
\begin{align*}
\eta(\pi_C(x),\pi_C(y))&\leq\eta(x,y), & \eta(m(x,y,z),m(x',y',z'))&\leq\eta(x,x')+\eta(y,y')+\eta(z,z').
\end{align*}
This can be proved as in Lemma~2.13 and Corollary~2.15 of \cite{CDH}. In addition, gate-projections are nearest-point projections, in the sense that $\eta(x,\pi_C(x))=\eta(x,C)$ for all $x\in M$.

If $\delta\in\mc{D}(M)$ and $(M,\delta)$ is complete, then a subset $C\cu M$ is gate-convex if and only if it is convex and closed in the topology induced by $\delta$ (see \cite[Lemma~2.13]{CDH}).

If $M$ is the $0$--skeleton of a $\CAT$ cube complex $X$, then a natural compatible metric on $M$ is given by the restriction of the \emph{combinatorial metric} on $X$: this is just the intrinsic path metric of the $1$--skeleton of $X$. All cube complexes in this paper will be implicitly endowed with their combinatorial metric, rather than the $\CAT$ metric. All geodesics will be assumed to be \emph{combinatorial geodesics}.

\begin{rmk}\label{from pseudo-metrics to measures}
A \emph{halfspace-interval} is a set of the form $\mscr{H}(x|y)\cu\mscr{H}(M)$ for $x,y\in M$. Let $\mscr{B}(M)\cu 2^{\mscr{H}(M)}$ denote the $\s$--algebra generated by halfspace-intervals. We say that a subset $\mc{H}\cu\mscr{H}(M)$ is $\mscr{B}$--measurable if it lies in $\mscr{B}(M)$.

Every $\eta\in\mc{PD}(M)$ induces a measure $\nu_{\eta}$ on $\mscr{B}(M)$ such that $\nu_{\eta}(\mscr{H}(x|y))=\eta(x,y)$ for all $x,y\in M$ (see e.g.\ \cite[Theorem~5.1]{CDH}). If $\eta\in\mc{PD}^G(M)$, then $\nu_{\eta}$ is $G$--invariant. 
\end{rmk}

\begin{lem}\label{about J}
Let $(X,d)$ be a median space. Let $A\cu X$ be a subset such that $\mc{J}(A)\cu\mc{N}_R(A)$ for some $R\geq 0$. Then, for every $D\geq 0$, we have:
\[\mc{J}(\mc{N}_D(A))\cu\mc{N}_{2D+R}(A).\]
In addition, if $\rk X=r$, we have $\hull A\cu\mc{N}_{2^rR}(A)$.
\end{lem}
\begin{proof}
If $z\in\mc{J}(\mc{N}_D(A))$, there exist $x,y\in\mc{N}_D(A)$ and $z\in I(x,y)$. Consider points $x',y'\in A$ with $d(x,x'),d(y,y')\leq D$. Set $z'=m(x',y',z)$. Since $z'\in\mc{J}(A)$, we have $d(z',A)\leq R$. Furthermore:
\[d(z,z')=d(m(x,y,z),m(x',y',z))\leq d(x,x')+d(y,y')\leq 2D.\]
In conclusion, $d(z,A)\leq d(z,z')+d(z',A)\leq 2D+R$, as required.

Proceeding by induction, it is straightforward to obtain $\mc{J}^i(A)\cu\mc{N}_{(2^i-1)R}(A)$ for every $i\geq 0$. If $\rk X=r$, we have $\hull A=\mc{J}^r(A)$ by Remark~\ref{Bow J}, hence $\hull A\cu\mc{N}_{(2^r-1)R}(A)\cu\mc{N}_{2^rR}(A)$.
\end{proof}

\subsection{Convex cores in median algebras.}\label{cores sect}

In this subsection, we collect a few facts proved in \cite{Fio10b} extending the notion of ``essential core'' \cite[Section~3]{CS} from actions on cube complexes to general actions on finite-rank median algebras (even with no invariant metric or topology). These results will only play a role in the proofs of Theorems~\ref{Q2 thm intro} and~\ref{special UNE thm intro} (especially in Sections~\ref{compatible sect} and~\ref{ultra sect}). The reader only interested in the other results mentioned in the Introduction can safely read this subsection with $\CAT$ cube complexes in mind, just to familiarise themselves with our notation. 

Let $M$ be a median algebra of finite rank $r$.

\begin{defn}
We say that $g\in\aut M$ acts:
\begin{enumerate}
\item[(1')] \emph{non-transversely} if there does not exist a wall $\mf{w}\in\mscr{W}(X)$ such that $\mf{w}$ and $g\mf{w}$ are transverse;
\item[(2')] \emph{stably without inversions} if there do not exist $n\in\Z$ and $\mf{h}\in\mscr{H}(X)$ with $g^n\mf{h}=\mf{h}^*$.
\end{enumerate}
An action $G\acts M$ by automorphisms is:
\begin{enumerate}
\item \emph{non-transverse} if every $g\in G$ acts non-transversely;
\item \emph{without wall inversions} if every $g\in G$ acts stably without inversions;
\item \emph{essential} if, for every $\mf{h}\in\mscr{H}(M)$, there exists $g\in G$ with $g\mf{h}\subsetneq\mf{h}$.
\end{enumerate}
\end{defn}

\begin{rmk}\label{connected inversions}
If there exists $\delta\in\mc{D}^G(M)$ such that $(M,\delta)$ is connected, then $G\acts M$ is without wall inversions. This follows from \cite[Proposition~B]{Fio1} when $(M,\delta)$ is complete, and from \cite[Remark~4.3]{Fio10b} in general.
\end{rmk}

Keeping the notation of \cite{Fio10b}, each action $G\acts M$ determines sets of halfspaces:
\begin{align*}
\mc{H}_1(G):=&\{\mf{h}\in\mscr{H}(M) \mid \exists g\in G \text{ such that } g\mf{h}\subsetneq\mf{h}\} \\
\overline{\mc{H}}_{1/2}(G):=&\{\mf{h}\in\mscr{H}(M)\setminus\mc{H}_1(G) \mid \exists g\in G \text{ such that $g\mf{h}^*\cap\mf{h}^*=\emptyset$ and $g\mf{h}\neq\mf{h}^*$}\} \\
\overline{\mc{H}}_0(G):=&\{\mf{h}\in\mscr{H}(M) \mid \forall g\in G \text{ either $g\mf{h}\in\{\mf{h},\mf{h}^*\}$ or $g\mf{h}$ and $\mf{h}$ are transverse}\}.
\end{align*}
As observed in \cite[Subsection~3.1]{Fio10b}, we have a $G$--invariant partition:
\[\mscr{H}(M)=\overline{\mc{H}}_0(G)\sqcup\mc{H}_1(G)\sqcup\overline{\mc{H}}_{1/2}(G)\sqcup\overline{\mc{H}}_{1/2}(G)^*.\]
We write $\W_1(G)$ and $\W_0(G)$ for the sets of walls bounding the halfspaces in $\mc{H}_1(G)$ and $\overline{\mc{H}}_0(G)$.

\begin{defn}
The \emph{reduced core} $\overline{\C}(G)$ is the intersection of all halfspaces lying in $\overline{\mc{H}}_{1/2}(G)$. 
\end{defn}

We adopt the convention that $\overline{\C}(G)=M$ when $\overline{\mc{H}}_{1/2}(G)$ is empty. We will write $\overline{\C}(G,M)$ (and $\mc{H}_{\bullet}(G,M)$, $\W_{\bullet}(G,M)$) if it is necessary to specify the ambient median algebra. We just write $\overline{\C}(g)$ (and $\mc{H}_{\bullet}(g)$, $\W_{\bullet}(g)$) if $G=\langle g\rangle$. 

\begin{thm}[\cite{Fio10b}]\label{all from 10b}
Let $G$ be finitely generated and let $G\acts M$ be without wall inversions.
\begin{enumerate}
\item The reduced core $\overline{\C}(G)$ is nonempty, $G$--invariant and convex.
\end{enumerate}
Suppose in addition that $\mc{D}^G(M)\neq\emptyset$.
\begin{enumerate}
\setcounter{enumi}{1}
\item There is a $G$--fixed point in $M$ if and only if $\mc{H}_1(G)=\emptyset$.
\item The sets $\mc{W}_1(G)$ and $\mc{W}_0(G)$ are transverse and $\mscr{W}_{\overline{\C}(G)}(M)=\mc{W}_0(G)\sqcup\mc{W}_1(G)$.
\item The resulting partition of $\mscr{W}(\overline{\C}(G))$ gives a product splitting $\overline{\C}(G)=\overline{\C}_0(G)\x\overline{\C}_1(G)$. The normaliser of the image of $G$ in $\aut M$ leaves $\overline{\C}(G)$ invariant, preserving the two factors. The action $G\acts\overline{\C}_1(G)$ is essential, while $G\acts\overline{\C}_0(G)$ fixes a point. 
\end{enumerate}
\end{thm}
\begin{proof}
We just refer the reader to the relevant statements in \cite{Fio10b}. Part~(1) follows from Theorem~3.17(2). The two implications in part~(2) are obtained from Proposition~3.23(2) and Lemma~4.5(1), respectively. Part~(3) is a consequence of Lemma~4.5 and Lemma~3.22(2). Finally, part~(4) follows from Remark~3.16 and the previous parts. 
\end{proof}

\begin{rmk}\label{essential core rmk}
If $G$ acts on a $\CAT$ cube complex $X$ and $M=X^{(0)}$, then the action $G\acts\overline{\C}_1(G)$ in part~(4) of Theorem~\ref{all from 10b} can be easily identified as the \emph{$G$--essential core} of Caprace and Sageev (cf.\ \cite[Section~3.3]{CS}). In particular, note that Theorem~\ref{all from 10b} strengthens \cite[Proposition~3.5]{CS}, showing that the $G$--essential core always embeds $G$--equivariantly as a convex subcomplex of $X$. 
\end{rmk}

\begin{thm}\label{all from 10b -2}
If $g\in\aut M$ acts non-transversely and stably without inversions, then:
\begin{enumerate}
\item the reduced core $\overline{\C}(g)$ is gate-convex;
\item for every $x\in M$ and every $\eta\in\mc{PD}^g(M)$, we have $\eta(x,gx)=\ell(g,\eta)+2\eta(x,\overline{\mc{C}}(g))$.
\end{enumerate}
\end{thm}
\begin{proof}
Part~(1) is \cite[Proposition~3.36]{Fio10b} and part~(2) is \cite[Proposition~4.9(3)]{Fio10b}.
\end{proof}

Note that $\overline{\C}(G)$ is not gate-convex in general, even when $G\acts M$ is an isometric action of a finitely generated free group on a complete $\R$--tree. See \cite[Example~3.37]{Fio10b}.

\begin{rmk}
Part~(2) of Theorem~\ref{all from 10b -2} implies that, if $\delta\in\mc{D}^g(M)$ and $(M,\delta)$ is a geodesic space, then $g$ is \emph{semisimple}: either $g$ fixes a point of $M$ or $g$ translates along a $\langle g\rangle$--invariant geodesic.
\end{rmk}

The next two remarks will only be needed in Section~\ref{ultra sect}.

\begin{rmk}\label{from 10b rmk}
Let $g\in\aut M$ act non-transversely and stably without inversions, with $\mc{D}^g(M)\neq\emptyset$.
\begin{enumerate}
\item Each $\mf{h}\in\mc{H}_1(g)$ satisfies $\bigcap_{n\in\Z}g^n\mf{h}=\emptyset$ (see \cite[Lemma~4.5(1)]{Fio10b}). 
\item A halfspace $\mf{h}$ lies in $\mf{h}\in\overline{\mc{H}}_0(g)$ if and only if $g\mf{h}=\mf{h}$, and it lies in $\mc{H}_1(g)$ if and only if either $g\mf{h}\subsetneq\mf{h}$ or $g\mf{h}\supsetneq\mf{h}$. This follows from Remarks~3.33 and~3.34 in \cite{Fio10b}, after observing that $\mc{H}_1(g)\cu\mscr{H}_{\overline{\C}(g)}(M)$ (e.g.\ by part~(1) of this remark).

\item Let $N\cu M$ be a $\langle g\rangle$--invariant median subalgebra. By Remark~\ref{halfspaces of subsets}, intersecting the halfspaces of $M$ with $N$, we obtain a surjective restriction map $\res_N\colon\mscr{H}_N(M)\ra\mscr{H}(N)$. Parts~(1) and~(2) show that:
\begin{itemize}
\item if $\mf{h}\in\overline{\mc{H}}_0(g,M)\cap\mscr{H}_N(M)$, then $g\cdot\res_N(\mf{h})=\res_N(\mf{h})$ and $\res_N(\mf{h})\in\overline{\mc{H}}_0(g,N)$;
\item if $\mf{h}\in\overline{\mc{H}}_{1/2}(g,M)\cap\mscr{H}_N(M)$, then either $\res_N(\mf{h})\in\overline{\mc{H}}_{1/2}(g,N)$ or $g\cdot\res_N(\mf{h})=\res_N(\mf{h})^*$;
\item we have $\mc{H}_1(g,M)\cu\mscr{H}_N(M)$ and $\res_N(\mc{H}_1(g,M))=\mc{H}_1(g,N)$. 
% \cu is clear, the other inclusion is due to the previous two items of the itemize
\end{itemize}
\end{enumerate}
\end{rmk}

\begin{rmk}\label{length from fundamental domains}
Let $g\in\aut M$ act non-transversely and stably without inversions. Let $\nu_{\eta}$ be the measure introduced in Remark~\ref{from pseudo-metrics to measures}. Part~(2) of Theorem~\ref{all from 10b -2} shows that $\ell(g,\eta)=\nu_{\eta}(\mscr{H}(x|gx))$ for any $x\in\overline{\C}(g)$. In view of parts~(1) and~(2) of Remark~\ref{from 10b rmk}, the set $\mscr{H}(x|gx)\sqcup\mscr{H}(gx|x)$ is a $\mscr{B}$--measurable fundamental domain for the action $\langle g\rangle\acts\mc{H}_1(g)$. It follows that, for \emph{any} fundamental domain $\Om\in\mscr{B}(M)$ for the action $\langle g\rangle\acts\mc{H}_1(g)$, we have $\ell(g,\eta)=\tfrac{1}{2}\nu_{\eta}(\Om)$.
\end{rmk}

\subsection{Two constructions involving cube complexes.}\label{CCC prelims}

\subsubsection{Restriction quotients.}

Restriction quotients of $\CAT$ cube complexes were originally introduced in \cite[p.\ 860]{CS}. Our interest is due to the fact that the Salvetti blowups and collapses from \cite{CSV} are a particular instance of this construction, which can actually be phrased purely in median-algebra terms. This is mainly needed in the proof of Proposition~\ref{cmp prop intro}(3) in Subsection~\ref{of RAAGs sect}, though it will also be useful in Subsections~\ref{cc sect} and~\ref{ultra-Salvetti sect}.

A map $f\colon X\ra Y$ between cube complexes is said to be \emph{cubical} if, on every cube $c\cu X$, it factors as a projection of $c$ onto one of its faces, followed by an isomorphism onto a cube of $Y$.

Let $X$ be a $\CAT$ cube complex. The \emph{carrier} of a hyperplane $\mf{w}\in\mscr{W}(X)$ is the smallest convex subcomplex of $X$ that contains all edges crossing $\mf{w}$. It naturally splits as a product $C\x[0,1]$, where $C\x\{0\}$ and $C\x\{1\}$ are convex subcomplexes of $X$ on the two sides of $\mf{w}$. 

Given a hyperplane $\mf{w}\in\mscr{W}(X)$, we can construct a new $\CAT$ cube complex $Y$ by \emph{collapsing} $\mf{w}$: we remove from $X$ the interior of the carrier $C\x(0,1)$ and we identify the isomorphic subcomplexes $C\x\{0\}$ and $C\x\{1\}$. The natural collapse map $X\ra Y$ is a cubical map.

Now, consider a set of hyperplanes $\mc{U}\cu\mscr{W}(X)$. The \emph{restriction quotient} of $X$ determined by $\mc{U}$ is the $\CAT$ cube complex $X(\mc{U})$ obtained by collapsing all hyperplanes in $\mscr{W}(X)\setminus\mc{U}$ (which usually involves infinitely many collapses). It has one vertex for every connected component of the complement in $X$ of the union of the hyperplanes in $\mc{U}$, with two vertices joined by an edge exactly when the corresponding components are separated by a single element of $\mc{U}$. Let $\pi_{\mc{U}}\colon X\ra X(\mc{U})$ be the natural collapse, which is again a cubical map.

If $G\acts X$ is an action and the subset $\mc{U}\cu\mscr{W}(X)$ is $G$--invariant, then the restriction quotient $X(\mc{U})$ is also equipped with a natural $G$--action and the collapse map $\pi_{\mc{U}}$ is $G$--equivariant.

\begin{prop}\label{restriction quotient prop}
Consider $\CAT$ cube complexes $X,Y$ and a surjective cubical map $\pi\colon X\ra Y$. Then the following are equivalent:
\begin{enumerate}
\item there exists a subset $\mc{U}\cu\mscr{W}(X)$ and an isomorphism $Y\cong X(\mc{U})$ with respect to which $\pi$ corresponds to the natural collapse $\pi_{\mc{U}}\colon X\ra X(\mc{U})$;
\item for every vertex $v\in Y$, the preimage $\pi^{-1}(v)$ is a convex subcomplex of $X$;
\item the restriction $\pi\colon X^{(0)}\ra Y^{(0)}$ is a median morphism.
\end{enumerate}
If $X$ and $Y$ are equipped with $G$--actions and $\pi$ is $G$--equivariant, then the set $\mc{U}$ is $G$--invariant.
\end{prop}
\begin{proof}
The equivalence of (1) and (2) was shown in \cite[Theorem~4.4]{Huang-Kleiner}. Fibres of median morphisms between median algebras are always convex, so (3) implies (2). Finally, (1)$\Ra$(3) can be shown by observing that single hyperplane-collapses are median morphisms.
\end{proof}

\subsubsection{Roller boundaries.}

In two proofs (Proposition~\ref{AMS fg} and, briefly, Lemma~\ref{friendly label-irreducibles}), we will need the notion of Roller boundary of a $\CAT$ cube complex $X$, denoted $\partial X$. We list here the (well-known) properties that we will use. 

The $0$--skeleton of any $\CAT$ cube complex $X$ has a natural structure of median algebra (see \cite[Theorem~6.1]{Chepoi} and \cite[Theorem~10.3]{Roller}). The $\ell^1$--metric on $X$, denoted $d$, is a compatible metric in the sense of Definition~\ref{compatible defn}. Thus, the pair $(X^{(0)},d)$ is a median space. The notions of ``halfspace'' and ``wall'' coincide with the usual notion of halfspace and hyperplane in $\CAT$ cube complexes. Thus, we write $\mscr{W}(X)$ and $\mscr{H}(X)$ with the meaning of $\mscr{W}(X^{(0)})$ and $\mscr{H}(X^{(0)})$.

We can embed $X^{(0)}\hookrightarrow 2^{\mscr{H}(X)}$ by mapping each vertex $v$ to the subset $\s_v\cu\mscr{H}(X)$ of halfspaces that contain it. This is a median morphism if we endow $2^{\mscr{H}(X)}$ with the structure of median algebra given by:
\[m(\s_1,\s_2,\s_3)=(\s_1\cap\s_2)\cup(\s_2\cap\s_3)\cup(\s_3\cap\s_1).\]
The space $2^{\mscr{H}(X)}$ is compact with the product topology, and we can consider the closure $\overline X$ of $X^{(0)}$ inside it. We define the \emph{Roller boundary} $\partial X$ as the set $\overline X\setminus X^{(0)}$.

For us, the only important facts will be:
\begin{enumerate}
\item The subset $\overline X=X\sqcup\partial X\cu 2^{\mscr{H}(X)}$ is a median subalgebra and $X^{(0)}$ is convex in $\overline X$.
\item The median $m\colon\overline X^3\ra\overline X$ is continuous with respect to the topology that $\overline X$ inherits from $2^{\mscr{H}(X)}$. With this topology, $\overline X$ is compact and totally disconnected. If $X$ is locally finite, the subset $X^{(0)}\cu\overline X$ is discrete. 
\item If $\mf{h}\in\mscr{H}(X)$, its closure $\overline{\mf{h}}$ inside $\overline X$ is gate-convex. In fact, $\overline{\mf{h}}$ and $\overline{\mf{h}^*}$ are complementary halfspaces of the median algebra $\overline X$. The gate-projection $\pi_{\mf{h}}\colon\overline X\ra\overline{\mf{h}}$ takes $X^{(0)}$ to $\mf{h}$.
\item Two halfspaces $\mf{h},\mf{k}\in\mscr{H}(X)$ are said to be \emph{strongly separated} if $\mf{h}\cap\mf{k}=\emptyset$ and no halfspace of $X$ is transverse to both $\mf{h}$ and $\mf{k}$ \cite{Behrstock-Charney}. If $\mf{h}$ and $\mf{k}$ are strongly separated, then the gate-projection $\pi_{\mf{h}}\colon\overline X\ra\overline{\mf{h}}$ maps $\overline{\mf{k}}$ to a single point.
\end{enumerate}
The reader can consult \cite[Subsections~2.3--2.4]{Fernos} and \cite[Theorem~4.14]{Fio1} for more details on Facts~(1)--(3). Fact~(4) follows e.g.\ from Corollary~2.22 and Lemma~2.23 in \cite{Fio3}.

\subsection{Coarse median structures.}\label{coarse median structures sect}

Coarse median spaces were introduced by Bowditch in \cite{Bow-cm}. We present the following equivalent definition from \cite{NWZ1}.

\begin{defn}\label{coarse median space defn}
Let $X$ be a metric space. A \emph{coarse median} on $X$ is a map $\mu\colon X^3\ra X$ for which there exists a constant $C\geq 0$ such that, for all $a,b,c,x\in X$, we have:
\begin{enumerate}
\item $\mu(a,a,b)=a$ and $\mu(a,b,c)=\mu(b,c,a)=\mu(b,a,c)$;
\item $\mu(\mu(a,x,b),x,c)\approx_C\mu(a,x,\mu(b,x,c))$;
\item $d(\mu(a,b,c),\mu(x,b,c))\leq Cd(a,x)+C$.
\end{enumerate}
\end{defn}

Note that part~(2) of the definition is an approximate version of the 4--point condition, from our definition of median algebras at the beginning of Subsection~\ref{median algebras subsect}.

There is an appropriate notion of \emph{rank} also for coarse median spaces. Since this notion will play no significant role in our paper (except when we briefly mention it at the end of Subsection~\ref{BP sect}), we simply refer the reader to \cite{Bow-cm,NWZ1,NWZ2} for more details.

The following notion of \emph{coarse median structure} is different from the one in \cite[Definition~2.8]{NWZ2}, but it is hard to imagine this being cause for confusion.

\begin{defn}\label{structure defn}
Two coarse medians $\mu_1,\mu_2\colon X^3\ra X$ are \emph{at bounded distance} if there exists a constant $C\geq 0$ such that $\mu_1(x,y,z)\approx_C\mu_2(x,y,z)$ for all $x,y,z\in X$. A \emph{coarse median structure} on $X$ is an equivalence class $[\mu]$ of coarse medians pairwise at bounded distance. A \emph{coarse median space} is a pair $(X,[\mu])$ where $X$ is a metric space and $[\mu]$ is a coarse median structure on it.
\end{defn}

\begin{rmk}\label{QIs act on CMs}
Let $f\colon X\ra Y$ be a quasi-isometry with a coarse inverse denoted $f^{-1}\colon Y\ra X$. If $\mu\colon X^3\ra X$ is a coarse median on $X$, then
\[(f_*\mu)(x,y,z):=f(\mu(f^{-1}(x),f^{-1}(y),f^{-1}(z)))\]
is a coarse median on $Y$. If $[\mu_1]=[\mu_2]$, then $[f_*\mu_1]=[f_*\mu_2]$.

If $QI(X)$ is the group of quasi-isometries $X\ra X$ up to bounded distance (as defined e.g.\ in \cite[Definition~8.22]{DK}), the above defines a natural left action of $QI(X)$ on the set of coarse median structures on $X$.
\end{rmk}

\begin{defn}\label{coarse median group defn}
A \emph{coarse median} group is a pair $(G,[\mu])$ where $G$ is a finitely generated group equipped with a word metric and $[\mu]$ is a $G$--invariant coarse median structure on $G$.
\end{defn}

The requirement that $[\mu]$ be $G$--invariant can be equivalently stated as follows: for each $g\in G$, there exists a constant $C(g)\geq 0$ such that $g\mu(g_1,g_2,g_3)\approx_{C(g)}\mu(gg_1,gg_2,gg_3)$ for all $g_1,g_2,g_3\in G$.

Note that Definition~\ref{coarse median group defn} is stronger than Bowditch's original definition from \cite{Bow-cm}, which did not ask for $[\mu]$ to be $G$--invariant. Definition~\ref{coarse median group defn} is better suited to our needs in this paper, but it is not QI--invariant or even commensurability-invariant (unlike Bowditch's).

These two definitions of coarse median group parallel the notions of HHS and HHG from \cite{HHS1,HHS2}. Namely, every hierarchically hyperbolic group is a coarse median group in the sense of Definition~\ref{coarse median group defn}, while any group that admits a structure of hierarchically hyperbolic space is coarse median in the sense of Bowditch \cite{Bow-largescale} (we will simply refer to these as ``groups with a coarse median structure'').

\begin{rmk}\label{motivation for cmp terminology}
If $G$ is finitely generated, any group automorphism $\varphi\colon G\ra G$ is bi-Lipschitz with respect to any word metric on $G$. The resulting homomorphism $\aut G\ra QI(G)$ defines an $(\aut G)$--action on the set of coarse median structures on $G$ that takes $G$--invariant structures to $G$--invariant structures. 
% I guess $\aut G\ra QI(G)$ is injective precisely when G has the R_{\infty} property? (all twisted conjugacy classes are infinite)
If $(G,[\mu])$ is a coarse median group, then every inner automorphism of $G$ fixes $[\mu]$, and we obtain an action of $\out G$ on the $(\aut G)$--orbit of $[\mu]$.
\end{rmk}

\begin{defn}
Let $(G,[\mu])$ be a coarse median group. We say that $\phi\in\out G$ (or $\varphi\in\aut G$) is \emph{coarse-median preserving} if it fixes $[\mu]$. We denote by $\out(G,[\mu])\leq\out G$ and $\aut(G,[\mu])\leq\aut G$ the subgroups of coarse-median preserving automorphisms.
\end{defn}

Thus $\varphi\in\aut G$ is coarse-median preserving exactly when, fixing a word metric on $G$, there exists a constant $C\geq 0$ such that, for all $g_i\in G$:
\[\varphi(\mu(g_1,g_2,g_3))\approx_C\mu(\varphi(g_1),\varphi(g_2),\varphi(g_3)).\]

\begin{rmk}\label{induced cms}
Let $G\acts X$ be a proper cocompact action on a $\CAT$ cube complex. Any orbit map $o\colon G\ra X$ is a quasi-isometry that can be used to pull back the median operator $m\colon X^3\ra X$ to a coarse median structure $[\mu_X]:=o^{-1}_*[m]$ on $G$. 
% it might not be median anymore since G--orbits need not be themselves convex
It is straightforward to check that $[\mu_X]$ is independent of all choices involved (though the notation is slightly improper, as $[\mu_X]$ does depend on the specific $G$--action on $X$). We refer to $[\mu_X]$ as the \emph{coarse median structure induced by $G\acts X$}.

Let us write $gx$ for the action of $g\in G$ on $x\in X$ according to $G\acts X$. Then, every $\varphi\in\aut G$ gives rise to a twisted $G$--action on $X$, which we denote by $G\acts X^{\varphi}$ and is defined as $g\cdot x=\varphi^{-1}(g)x$. Note that $\varphi_*[\mu_X]=[\mu_{X^{\varphi}}]$ and thus $\varphi\out(G,[\mu_X])\varphi^{-1}=\out(G,[\mu_{X^{\varphi}}])$.

Each of the structures $[\mu_{X^{\varphi}}]$ is $G$--invariant. In particular, $(G,[\mu_X])$ is a coarse median group.
\end{rmk}

\begin{ex}\label{hyp cmp ex}
Every geodesic Gromov-hyperbolic space $X$ is equipped with a natural coarse median structure $[\mu]$ represented by the operators $\mu$ that map each triple $(x,y,z)$ to an approximate incentre for a geodesic triangle with vertices $x,y,z$ (cf.\ \cite[Section~3]{Bow-cm}). In fact, by \cite[Theorem~4.2]{NWZ1}, this is the only coarse median structure that $X$ can be endowed with. It follows that $[\mu]$ is preserved by every quasi-isometry of $X$.

In particular, all automorphisms of Gromov-hyperbolic groups are coarse-median preserving. Alternatively, it is not hard to prove this last fact directly, relying on the Morse lemma and the observation that group automorphisms are quasi-isometries with respect to any word metric.
\end{ex}

\begin{ex}\label{Z^n cmp ex}
Equipping $\Z^n$ with the median operator $\mu$ associated to its $\ell^1$ metric, we obtain a coarse median group $(\Z^n,[\mu])$. An automorphism $\varphi\in\aut\Z^n={\rm GL}_n\Z$ is coarse-median preserving if and only if it lies in the signed permutation group $O(n,\Z)\leq {\rm GL}_n\Z$ (i.e.\ if it can be realised as an automorphism of the standard tiling of $\R^n$ by unit cubes). This will follow from Proposition~\ref{cmp prop intro}(3) once we prove it in Subsection~\ref{of RAAGs sect} (though it also is easily shown by hand).
\end{ex}

We conclude this subsection with the definitions of \emph{quasi-convex} subsets and \emph{approximate median subalgebras}, which will play an important role in Sections~\ref{cc whole sect} and~\ref{fix cmp sect}.

\begin{defn}\label{qc defn}
Let $(X,[\mu])$ be a coarse median space. A subset $A\cu X$ is \emph{quasi-convex} if there exists $R\geq 0$ such that $\mu(A\x A\x X)\cu\mc{N}_R(A)$.
\end{defn}

This notion is clearly independent of the chosen representative $\mu$ of the structure $[\mu]$. Moreover, by Definition~\ref{coarse median space defn}(3), if subsets $A$ and $B$ have finite Hausdorff distance, then $A$ is quasi-convex if and only if $B$ is. 

By Remark~\ref{hyp cmp ex}, Definition~\ref{qc defn} extends the usual notion of quasi-convexity in hyperbolic spaces. The next remark shows that this is also the notion of quasi-convexity appearing in the statement of Theorem~\ref{U_0 cc intro}. We will discuss in Subsection~\ref{cc sect} other equivalent notions of quasi-convexity in (non-hyperbolic) cube complexes.

\begin{rmk}\label{right-angled qc rmk}
Let $G$ be a right-angled Artin/Coxeter group. Let $G\acts X$ be the action on the universal cover of the Salvetti/Davis complex and let $[\mu_X]$ be the induced coarse median structure on $G$, as in Remark~\ref{induced cms}. Recall that, for a subset $A\cu X^{(0)}$, the set $\mc{J}(A)=\mu_X(A\x A\x X)$ is the union of all geodesics joining points of $A$.

Since the standard Cayley graph of $G$ is precisely the $1$--skeleton of $X$, a subgroup $H\leq G$ is quasi-convex as defined in the statement of Theorem~\ref{U_0 cc intro} if and only if we have $\mc{J}(H\cdot x)\cu\mc{N}_R(H\cdot x)$ for some $x\in X$ and $R\geq 0$. This is clearly equivalent to quasi-convexity of $H$ with respect to the coarse median structure $[\mu_X]$.
\end{rmk}

\begin{rmk}\label{qc in median spaces}
If $X$ is a finite-rank median space, then a subset $A\cu X$ is quasi-convex if and only if $d_{\rm Haus}(A,\hull A)<+\infty$. This follows from Lemma~\ref{about J}.
\end{rmk}

A similar, weaker notion is that of \emph{approximate median subalgebra}.

\begin{defn}\label{approx subalg defn}
Let $(X,[\mu])$ be a coarse median space. A subset $A\cu X$ is an \emph{approximate median subalgebra} if there exists $R\geq 0$ such that $\mu(A\x A\x A)\cu\mc{N}_R(A)$.
\end{defn}

Again, the definition only depends on the structure $[\mu]$ and passes on to all subsets of $X$ at finite Hausdorff distance from $A$. An analogue of Remark~\ref{qc in median spaces} also holds, but it is more complicated and will be discussed in Subsection~\ref{approx subalg sect}.

If $\varphi$ is a coarse-median preserving automorphism of a coarse median group $(G,[\mu])$, the fixed subgroup $\Fix\varphi\leq G$ is in general not quasi-convex (for instance, consider the automorphism of $\Z^2$ that swaps the standard generators). However, it is always an approximate median subalgebra, as the next two lemmas show. This will be important in the proof of Theorem~\ref{cmp und intro}.

\begin{lem}\label{displacement vs fix dist}
Let $G$ be a finitely generated group and let $d$ be a word metric on $G$. For every $\varphi\in\aut G$, there exist functions $\zeta_1,\zeta_2\colon\N\ra\R_{>0}$, with $\zeta_1$ linear, such that, for every $g\in G$:
\[\zeta_1\big(d(g,\varphi(g))\big)\leq d(g,\Fix\varphi) \leq \zeta_2\big(d(g,\varphi(g))\big).\]
\end{lem}
\begin{proof}
For the first inequality, note that $\varphi\colon G\ra G$ is $C$--bi-Lipschitz with respect to $d$, for some constant $C\geq 0$. If $g'\in\Fix\varphi$ is an element closest to $g$, we have:
\[d(g,\varphi(g))\leq d(g,g')+d(\varphi(g'),\varphi(g))\leq (1+C)\cdot d(g,g')=(1+C)\cdot d(g,\Fix\varphi).\]
Thus, we can take $\zeta_1(t):=t/(1+C)$.

Regarding the second inequality, suppose for the sake of contradiction that there does not exist a function $\zeta_2$ so that it is satisfied. Then, there exist elements $g_n\in G$ with $d(g_n,\Fix\varphi)\ra+\infty$, but $d(g_n,\varphi(g_n))\leq D$ for some $D\geq 0$. Passing to a subsequence, we can assume that $\varphi(g_n)=g_nx$ for some $x\in G$ and all $n$. Thus $g_ng_m^{-1}\in\Fix\varphi$, hence $d(g_n,\Fix\varphi)=d(g_m,\Fix\varphi)$ for all $n,m\geq 0$, contradicting the fact that the distances $d(g_n,\Fix\varphi)$ diverge.
\end{proof}

\begin{lem}\label{fix approximate subalgebra}
Let $(G,[\mu])$ be a coarse median group. If $\varphi\in\aut(G,[\mu])$, then $\Fix\varphi\leq G$ is an approximate median subalgebra.
\end{lem}
\begin{proof}
Since $\varphi\in\aut(G,[\mu])$, there is a constant $C$ such that:
\[\varphi(\mu(x,y,z))\approx_C\mu(\varphi(x),\varphi(y),\varphi(z)),\ \forall x,y,z\in G.\]
Thus, if $x,y,z\in\Fix\varphi$, we have $\varphi(\mu(x,y,z))\approx_C\mu(x,y,z)$. Lemma~\ref{displacement vs fix dist} gives a constant $C'$ such that $d(\mu(x,y,z),\Fix\varphi)\leq C'$ for all $x,y,z\in\Fix\varphi$, as required.
\end{proof}

\subsection{UNE actions and groups.}

The following (seemingly novel) notion will play an important role in the proof of Theorem~\ref{Q2 thm intro}, especially in Subsections~\ref{metrics on algebras},~\ref{BP sect} and~\ref{meaty sect}.

\begin{defn}\label{UNE defn}
Let $G$ be a finitely generated group and let $(X,d)$ be a (pseudo-)metric space.
\begin{enumerate}
\item An isometric action $G\acts X$
% we don't want to add "with unbounded orbits" because it would mess with Definition~\ref{WNE defn}, though it would make the name "UNE" more reasonable
is \emph{uniformly non-elementary (UNE)} if there exists a constant $c>0$ with the following property. For every finite generating set $S\cu G$ and for all $x,y\in X$:
\[d(x,y)\leq c\cdot[\tau_S^d(x)+\tau_S^d(y)].\] 
% could add "+c" to make it QI invariant, but it messes up the notion we need for median algebras
% in any case, this does not matter for groups, see Remark below
We say that $G\acts X$ is \emph{$c$--uniformly non-elementary ($c$--UNE)} when we need to specify $c$.
\item An infinite group $G$ is \emph{UNE} if it admits a UNE, proper, cocompact action on a geodesic metric space. 
\end{enumerate}
\end{defn}

The previous definition differs slightly from the one given in the introduction, but it is easily seen to be equivalent.

\begin{rmk}\label{every UNE}
If $G$ is infinite and an action $G\acts X$ is proper and cocompact, then there exists $\eps>0$ such that, for every generating set $S\cu G$ and every $x\in X$, we have $\tau_S^d(x)\geq\eps$. % can move x into a compact set by applying an element of G and conjugating the generating set

Along with the Milnor--Schwarz lemma, this can be used to show that a group is UNE if and only if \emph{every} proper, cocompact action on a geodesic space is UNE. Equivalently, if the action of $G$ on its locally finite Cayley graphs is UNE.
\end{rmk}

\begin{ex}\label{UNE example}
\begin{enumerate}
\item[]
\item Non-elementary hyperbolic groups are UNE (for instance, this is implicitly shown in the last two paragraphs of the proof of \cite[Lemme~3.1]{Paulin-ENS}).
\item Fundamental groups of compact special cube complexes with finite centre are UNE. We will obtain this in Corollary~\ref{UNE cor}.
\item UNE groups have finite centre.
\end{enumerate}
\end{ex}

\section{Cubical convex-cocompactness.}\label{cc whole sect}

This section is devoted to \emph{convex-cocompact} subgroups of cocompactly cubulated groups (Definition~\ref{cc defn}). First, in Subsection~\ref{cc sect}, we discuss the relationship between convex-cocompactness and coarse median quasi-convexity. Then, Subsection~\ref{label-irreducible sect} discusses basic properties of \emph{cyclic}, convex-cocompact subgroups of RAAGs. Finally, Proposition~\ref{cmp prop intro} is proved in Subsection~\ref{of RAAGs sect}.

The reader that is not interested in the proofs of Theorems~\ref{Q2 thm intro} and~\ref{special UNE thm intro} can safely skip Subsection~\ref{more on raag cc}, which is devoted to some of the finer properties of convex-cocompact subgroups of RAAGs and is more technical. Its results will only be needed in Section~\ref{ultra sect}.

\subsection{Cubical convex-cocompactness in general.}\label{cc sect}

Let $G\acts X$ be a proper cocompact action on a $\CAT$ cube complex. In particular, $X$ is finite-dimensional and locally finite.

\begin{defn}\label{cc defn}
A subgroup $H\leq G$ is \emph{convex-cocompact} in $G\acts X$ if there exists an $H$--invariant, convex subcomplex $C\cu X$ that is acted upon cocompactly by $H$.
\end{defn}

Despite the similarity in terminology, we emphasise that the above is much weaker than the notion of ``boundary convex-cocompactness'' due to Cordes and Durham \cite{Cordes-Durham}. For instance, all convex-cocompact subgroups of RAAGs are free if we consider the Cordes--Durham notion \cite{Koberda-Mangahas-Taylor}, whereas every special group is a convex-cocompact subgroup of some RAAG acting on its Salvetti complex according to Definition~\ref{cc defn} (see \cite{Haglund-Wise-GAFA}).

Let $[\mu_X]$ be the coarse median structure on $G$ induced by $G\acts X$ as in Remark~\ref{induced cms}. Recall that quasi-convex subsets of coarse median spaces were introduced in Definition~\ref{qc defn}. For the notion of $H$--essential core, see Remark~\ref{essential core rmk} or \cite[Section~3.3]{CS}.

The following is just a restating of some well-known facts. The equivalence of the first two parts is due to Haglund (see \cite[Theorem~H]{Haglund-GD} and \cite{Sageev-Wise-core}).

\begin{lem}\label{equiv cc}
The following are equivalent for a subgroup $H\leq G$:
\begin{enumerate}
\item $H$ is convex-cocompact in $G\acts X$;
\item $H$ is quasi-convex in $(G,[\mu_X])$;
\item $H$ is finitely generated and acts cocompactly on the $H$--essential core of $H\acts X$.
% without finite generation, I guess we could have a product of an H-essential action and one action where every element of H is elliptic (but no global fixed point) 
\end{enumerate}
\end{lem}
\begin{proof}
Let us begin with the equivalence of (1) and (2). Picking a vertex $v\in X$, condition~(2) holds if and only if there exists a constant $R'$ such that $m(H\cdot v,H\cdot v, G\cdot v)\cu\mc{N}_{R'}(H\cdot v)$. Since $G$ acts cocompactly and $m$ is $1$--Lipschitz in each component, this is equivalent to the existence of $R''$ with:
\[\mc{J}(H\cdot v)=m(H\cdot v,H\cdot v,X)\cu\mc{N}_{R''}(H\cdot v).\] 
It is clear that this holds when (1) is satisfied, so (1)$\Ra$(2).

Conversely, if (2) holds, then $H\cdot v$ is quasi-convex in $X$ and Remark~\ref{qc in median spaces} implies that $\hull(H\cdot v)$ is at finite Hausdorff distance from $H\cdot v$. Since $X$ is locally finite, this means that $H$ acts cocompactly on $\hull(H\cdot v)$, hence $H$ is convex-cocompact.

We now show the equivalence of (1) and (3). First, if $C\cu X$ is convex and $H$--invariant, the $H$--essential core of $H\acts X$ is a restriction quotient of $C$ (as defined in Subsection~\ref{CCC prelims}). Thus, if $H$ acts cocompactly on $C$, it also acts cocompactly on the $H$--essential core. Moreover, the action $H\acts C$ is proper and cocompact, which implies that $H$ is finitely generated. This proves (1)$\Ra$(3).

Conversely, let $X'$ be the cubical subdivision. Since $H$ is finitely generated and $H\acts X'$ has no inversions, the essential core of $H\acts X'$ embeds $H$--equivariantly as a convex subcomplex of $X'$ (see Remark~\ref{essential core rmk}). This shows that (3)$\Ra$(1).
\end{proof}

Recalling that automorphisms of $G$ are bi-Lipschitz with respect to word metrics on $G$, the equivalence of (1) and (2) in Lemma~\ref{equiv cc} has the following straightforward consequence:

\begin{cor}\label{cmp preserve cc}
If $\varphi\in\aut(G,[\mu_X])$, then a subgroup $H\leq G$ is convex-cocompact in $G\acts X$ if and only if $\varphi(H)$ is.
\end{cor}

\begin{ex}
If $G$ is Gromov-hyperbolic, then a subgroup $H\leq G$ is convex-cocompact in $G\acts X$ if and only if $H$ is quasi-convex in $G$ (again since (1)$\LRa$(2) in Lemma~\ref{equiv cc}). In particular, the notion of convex-cocompactness is independent of the chosen cubulation of $G$ in this case. A quick look at the standard cubulation of $\Z^2$ immediately shows that the latter does not hold in general.
\end{ex}

\subsection{Label-irreducible elements in RAAGs.}\label{label-irreducible sect}

This subsection studies convex-cocompact \emph{cyclic} subgroups of right-angled Artin groups. Let $\G$ be a finite simplicial graph. Let $\A=\A_{\G}$ be a RAAG and $\X=\X_{\G}$ the universal cover of its Salvetti complex. Set $r=\dim\X$. 

The Cayley graph of $\A$ corresponding to the standard generating set $\G^{(0)}$ is naturally identified with the $1$--skeleton of the $\CAT$ cube complex $\X$. Thus, every edge of $\X$ is labelled by a vertex of $\G$. Observing that edges crossing the same hyperplane have the same label, we obtain a map $\g\colon\mscr{W}(\X)\ra\G^{(0)}$. 

We can apply the discussion in Subsection~\ref{cores sect} to the standard action $\A\acts\X$ (or, to be precise, the action on the $0$--skeleton of $\X$). Every element of $\A$ acts non-transversely and stably without inversions. For every $g\in\A\setminus\{1\}$, the reduced core $\overline{\C}(g)$ is the union of all axes of $g$.

A hyperplane of $\X$ lies in $\W_1(g)$ if and only if it is crossed by one (equivalently, all) axis of $g$. Hyperplanes lie in $\W_0(g)$ when they are preserved by $g$; equivalently, when they are transverse to all elements of $\W_1(g)$, or, again, when they separate two axes of $g$. 

The factor $\overline{\C}_1(g)$ is $\langle g\rangle$--equivariantly isomorphic to the convex hull in $\X$ of any axis of $g$. The factor $\overline{\C}_0(g)$ is fixed pointwise by $g$ and it is isomorphic to $\X_{\Lambda}$, where $\Lambda\cu\G$ is the maximal subgraph all of whose vertices are joined by an edge to all vertices in $\g(\mc{W}_1(g))$.

For a simplicial graph $\Delta$, we denote by $\Delta^o$ the \emph{opposite} of $\Delta$. This the graph that has the same vertex set as $\Delta$ and an edge between two vertices exactly when they are not connected by an edge in $\Delta$.

\begin{defn}\label{label-irreducible defn}
Consider $g\in\A\setminus\{1\}$.
\begin{enumerate}
\item We define $\G(g):=\g(\W_1(g))\cu\G^{(0)}$. These are precisely the standard generators of $\A$ that appear in the cyclically reduced words representing elements conjugate to $g$.
\item We say that $g$ is \emph{label-irreducible} if the full subgraph of $\G$ spanned by $\G(g)$ does not split as a nontrivial join (i.e.\ its opposite graph is connected). Equivalently, $g$ is contracting \cite{Charney-Sultan} within a parabolic subgroup of $\A$.
\end{enumerate}
\end{defn}

Two label-irreducible elements $g,h\in\A$ are \emph{independent} if $\langle g,h\rangle\not\simeq\Z$. If $g,h$ are independent and commute, then $\langle g,h\rangle\simeq\Z^2$. We will also use the following result of Servatius, see e.g.\ \cite[Proposition~III.1]{Servatius}.

\begin{lem}\label{independent LI lem}
If $g,h\in\A$ are commuting, independent, label-irreducible elements, then every vertex of $\G(g)$ is joined to every vertex of $\G(h)$ by an edge of $\G$.
\end{lem}

To each element $g\in\A$, we can associate a canonical collection of label-irreducible elements $g_1,\dots,g_k$, called the \emph{label-irreducible components} of $g$, as shown in the next result.

\begin{lem}[Label-irreducible components]\label{label-irreducible decomposition}
For every element $g\in\A$, the following hold.
\begin{enumerate}
\item We can write $g=g_1\cdot\ldots\cdot g_k$ for pairwise-commuting, pairwise-independent label-irreducibles $g_i\in \A$. In addition, $0\leq k\leq r$ and the $g_i$ are unique up to permutation.
\item The sets $\W_1(g_i)$ are transverse to each other and $\W_1(g_i)\cu\W_0(g_j)$ for $i\neq j$. In addition: 
\begin{align*}
\W_1(g)&=\W_1(g_1)\sqcup\dots\sqcup\W_1(g_k), & \ell(g,\X)&=\ell(g_1,\X)+\dots+\ell(g_k,\X), \\
\overline{\C}_1(g)&\simeq\overline{\C}_1(g_1)\x\dots\x\overline{\C}_1(g_k), & \overline{\C}(g)&=\overline{\C}(g_1)\cap\dots\cap\overline{\C}(g_k).
\end{align*}
\item Centralisers satisfy $Z_{\A}(g)=Z_{\A}(g_1)\cap\dots\cap Z_{\A}(g_k)$. Moreover, $Z_{\A}(g)$ splits as the direct product of a parabolic subgroup of $\A$ and a copy of $\Z^k$ freely generated by roots of $g_1,\dots,g_k$.
\end{enumerate}
\end{lem}
\begin{proof}
Since label-irreducibility is invariant under taking conjugates, we assume throughout the proof that $g$ is cyclically reduced. If $g$ is the identity, we can simply take $k=0$ and the entire lemma holds trivially. Suppose instead that $g\neq 1$.

We begin with part~(1). Let $\L_1,\dots,\L_k$ be the connected components of the subgraph of $\G^o$ spanned by $\G(g)$. In $\G$, every vertex of $\L_i$ is joined by an edge to every vertex of $\L_j$ with $j\neq i$. Thus, permuting the letters in a word representing $g$, we can write $g=g_1\cdot\ldots\cdot g_k$, where each $g_i$ is cyclically reduced and $\G(g_i)=\L_i^{(0)}$. The elements $g_i$ commute pairwise and, since each $\L_i$ is connected, they are all label-irreducible. It is clear that $g_i$ and $g_j$ are independent for $i\neq j$.

Uniqueness of the $g_i$ up to permutations follows from the fact that, by Lemma~\ref{independent LI lem}, $\G(g_1),\dots,\G(g_k)$ must coincide with the vertex sets of $\L_1,\dots,\L_k$ in any such decomposition of $g$. Furthermore, choosing a vertex from each $\L_i$, we obtain a $k$--clique in $\G$, so $k\leq r$. This proves part~(1).

We now prove part~(2). Since $g_i$ is cyclically reduced, there exists a (combinatorial) axis $\alpha_i\cu\X_{\G(g_i)}$ through the identity. Note that the product $\X_{\G(g_1)}\x\dots\x\X_{\G(g_k)}\cu\X$ is preserved by all $g_i$ and that each $g_i$ leaves invariant every hyperplane of $\X_{\G(g_j)}$ for all $j\neq i$. Thus, the sets $\W_1(g_i)$ are transverse to each other and $\W_1(g_i)\cu\W_0(g_j)$ for $i\neq j$. The equality $\W_1(g)=\W_1(g_1)\sqcup\dots\sqcup\W_1(g_k)$ now follows by observing that $\alpha_1\x\dots\x\alpha_k$ contains an axis of $g$. The product splitting of $\overline{\C}_1(g)$ can be deduced from the transverse partition of $\W_1(g)$ using Lemma~\ref{product lem}.

If $\Om_i$ is a fundamental domain for the $\langle g_i\rangle$--action on $\W_1(g_i)$, the previous paragraph implies that $\Om_1\sqcup\dots\sqcup\Om_k$ is a fundamental domain for the $\langle g\rangle$--action on $\W_1(g)$. Taking cardinalities, this shows that $\ell(g,\X)=\ell(g_1,\X)+\dots+\ell(g_k,\X)$. Finally, the characterisation of $\overline{\C}(g)$ can be deduced from the fact that this is the set of points of $\X$ that realise the translation length $\ell(g,\X)$.

We conclude with part~(3). The inclusion $Z_{\A}(g_1)\cap\dots\cap Z_{\A}(g_k)\leq Z_{\A}(g)$ is clear. Conversely, if $h\in\A$ commutes with $g$, uniqueness in part~(1) implies that the elements $hg_ih^{-1}$ coincide with the $g_i$ up to permutation. Since $\G(hg_ih^{-1})=\G(g_i)$, it follows that $hg_ih^{-1}=g_i$ for each $i$. Hence $h\in Z_{\A}(g_1)\cap\dots\cap Z_{\A}(g_k)$, as required. The last statement is Servatius' Centraliser Theorem from \cite[Section~III]{Servatius}. 
\end{proof}

\begin{rmk}\label{noetherian centralisers}
For every $H\leq\A$, there exists a finite subset $F\cu H$ such that $Z_{\A}(H)=Z_{\A}(F)$. 

Indeed, we have observed in Lemma~\ref{label-irreducible decomposition}(3) that the centraliser of every element of $\A$ splits as a product of a free abelian group and a parabolic subgroup of $\A$. It follows that every descending chain of centralisers of subsets of $\A$ eventually stabilises, since this is true of chains of parabolics.
\end{rmk}

We conclude this subsection by showing that label-irreducibles are precisely those elements $g\in\A$ such that the subgroup $\langle g\rangle$ is convex-cocompact in $\A\acts\X$. After a couple of preliminary results, this is shown below in Lemma~\ref{label-irreducibles are cc 1}.

\begin{lem}\label{diameter of antigraph}
Every connected full subgraph $\Lambda\cu\G^o$ has diameter $\leq 2r-1$. 
\end{lem}
\begin{proof}
Suppose towards a contradiction that there exist vertices $x,y\in\Lambda$ and a shortest path $\s\cu\Lambda$ joining them, with $\s$ made up of $2r$ edges. Let $\s_i$ be the $i$--th vertex of $\G^o$ met by $\s$, with $\s_0=x$ and $\s_{2r}=y$. Since $\s$ is shortest and $\Lambda$ is full, no two of the $r+1$ vertices $\s_0,\s_2,\dots,\s_{2r}$ are joined by an edge of $\G^o$. Thus, they form an $(r+1)$--clique in $\G$, contradicting the fact that $r=\dim\X$.
\end{proof}

\begin{lem}\label{irreducible elements new}
Let $g\in\A$ be label-irreducible. Then, for every $\mf{u}\in\W_1(g)$, there exists a point $x\in\overline{\C}(g)$ such that $\mscr{W}(x|gx)\cu\mscr{W}(\mf{u}|g^{4r-2}\mf{u})$. In particular, $\g(\mscr{W}(\mf{u}|g^{4r-2}\mf{u}))=\G(g)$.
\end{lem}
\begin{proof}
Pick a point $y$ on an axis of $g$ so that $\mf{u}\in\mscr{W}(y|gy)$. Set $x=g^{2r-1}y$ and consider a hyperplane $\mf{w}\in\mscr{W}(x|gx)$.  Since $g$ is label-irreducible, the full subgraph of $\G^o$ spanned by $\G(g)$ is connected. By Lemma~\ref{diameter of antigraph}, there exists a sequence $\s_0=\g(\mf{u}),\s_1,\dots,\s_k=\g(\mf{w})$ of vertices in $\G(g)$ such that $k\leq 2r-1$ and consecutive $\s_i$ are not joined by an edge of $\G$. Set $\s_j=\s_k$ for $k<j\leq 2r-1$.

For $0\leq i\leq 2r-1$, pick a hyperplane $\mf{w}_i\in\mscr{W}(g^iy|g^{i+1}y)$ with $\g(\mf{w}_i)=\s_i$, making sure that $\mf{w}_0=\mf{u}$ and $\mf{w}_{2r-1}=\mf{w}$. Since $\s_i$ and $\s_{i+1}$ are not joined by an edge, the hyperplanes $\mf{w}_i$ and $\mf{w}_{i+1}$ are not transverse. Since these hyperplanes are all crossed by an axis of $g$, we conclude that each $\mf{w}_i$ separates the $\mf{w}_j$ with $j<i$ from those with $j>i$. In particular, $\mf{u}$ and $\mf{w}$ are not transverse. 

The same argument shows that $\mf{w}$ and $g^{4r-2}\mf{u}$ are not transverse, hence $\mf{w}\in\mscr{W}(\mf{u}|g^{4r-2}\mf{u})$. Since $\mf{w}\in\mscr{W}(x|gx)$ was arbitrary, we have shown that $\mscr{W}(x|gx)\cu\mscr{W}(\mf{u}|g^{4r-2}\mf{u})$. 
\end{proof}

\begin{lem}\label{label-irreducibles are cc 1}
\begin{enumerate}
\item[]
\item If $g$ is label-irreducible and $\alpha\cu\X$ is an axis, then $d_{\rm Haus}(\alpha,\hull\alpha)\leq (8r-4)\ell(g,\X)$.
\item An element $g\in\A\setminus\{1\}$ is label-irreducible if and only if $\langle g\rangle$ is convex-cocompact in $\A\acts\X$.
\end{enumerate}
\end{lem}
\begin{proof}
Assuming part~(1), we first prove part~(2). Using the third characterisation of convex-cocompactness in Lemma~\ref{equiv cc} and the fact that $\overline{\C}_1(g)$ is equivariantly isomorphic to $\hull\alpha$, part~(1) shows that label-irreducible elements are convex-cocompact. Conversely, if $g$ is not label-irreducible, the nontrivial splitting of $\overline{\C}_1(g)$ provided by Lemma~\ref{label-irreducible decomposition}(2) implies that $\langle g\rangle$ cannot act cocompactly on $\overline{\C}_1(g)$.

Let us now prove part~(1). Considering a point $p\in\hull\alpha$, it is enough to obtain the inequality $d(p,\alpha)\leq (8r-4)\ell(g,\X)$. 

Every element of $\mscr{H}_{\hull\alpha}(\X)$ intersects $\alpha$ in a sub-ray. Let $\mc{H}_+$ be the subset of halfspaces intersecting $\alpha$ in a positive sub-ray (i.e.\ containing all points $g^nz$ with $n\geq 0$, for a suitable choice of $z\in\alpha$). Any two maximal halfspaces lying in $\mc{H}_+$ and not containing $p$ are transverse. It follows that there are only finitely many such maximal halfspaces, which we denote by $\mf{h}_1,\dots,\mf{h}_k$. 

A negative sub-ray of $\alpha$ is contained in $\mf{h}_1^*\cap\dots\cap\mf{h}_k^*$, so we can pick a point $x\in\alpha\cap\mf{h}_1^*\cap\dots\cap\mf{h}_k^*$. In particular, $x$ does not lie in any halfspaces of $\mc{H}_+$ that do not contain $p$; hence $\mscr{H}(x|p)\cu\mc{H}_+$. Let $y\in\alpha$ be the point with $d(x,p)=d(x,y)$ and $\mscr{H}(x|y)\cu\mc{H}_+$. Setting $m=m(x,p,y)$, we note that every $\mf{j}\in\mscr{H}(m|p)$ is transverse to every $\mf{k}\in\mscr{H}(m|y)$. Indeed, $m\in\mf{j}^*\cap\mf{k}^*$, $p\in\mf{j}\cap\mf{k}^*$ and $y\in\mf{j}^*\cap\mf{k}$, while $\mf{j}\cap\mf{k}$ is nonempty because $\mf{j}$ and $\mf{k}$ both lie in $\mc{H}_+$.

Now, suppose for the sake of contradiction that $d(p,y)>(8r-4)\ell(g,\X)$. Since we chose $y$ with $d(x,p)=d(x,y)$, we have $d(p,m)=d(m,y)>(4r-2)\ell(g,\X)$. Note that $\mscr{W}(p|m)\cu\mscr{W}_{\hull\alpha}(\X)=\W_1(g)$, a set on which $\langle g^{4r-2}\rangle$ acts with exactly $(4r-2)\ell(g,\X)$ orbits. Thus, there exists a hyperplane $\mf{u}\in\mscr{W}(p|m)$ such that $g^{4r-2}\mf{u}\in\mscr{W}(p|m)$. Lemma~\ref{irreducible elements new} implies that $\g(\mscr{W}(p|m))=\G(g)$. Similarly, we obtain $\g(\mscr{W}(m|y))=\G(g)$. This contradicts the fact that $\mscr{W}(p|m)$ is transverse to $\mscr{W}(m|y)$.
\end{proof}

\subsection{More on convex-cocompactness in RAAGs.}\label{more on raag cc}

The results in this subsection will only be used in Section~\ref{ultra sect} and can be skipped by the reader uninterested in the proof of Theorems~\ref{Q2 thm intro} and~\ref{special UNE thm intro}.

First, we discuss additional properties of label-irreducible elements of RAAGs. Our aim is obtaining \emph{uniform} control on the extent to which axes of distinct label-irreducibles can track each other. The main result here is Lemma~\ref{friendly label-irreducibles}, along with its direct consequence Corollary~\ref{label-irreducible cor}. Both results will be fundamental building blocks in the proof that centreless special groups are UNE.

Then, in the second part of the subsection, we study general convex-cocompact subgroups of RAAGs, proving only a couple of simple properties related to label-irreducible components.

\subsubsection{Additional properties of label-irreducible elements.}

We maintain the notation introduced at the beginning of Subsection~\ref{label-irreducible sect}. Recall that $r=\dim\X$.

Recall that the carrier of a hyperplane $\mf{w}\in\mscr{W}(\X)$ is the smallest convex subcomplex of $\X$ that contains all edges crossing $\mf{w}$. A hyperplane of $\X$ separates two points in the carrier of $\mf{w}$ if and only if it is either equal or transverse to $\mf{w}$. If two hyperplanes $\mf{u},\mf{w}$ have intersecting carriers, then they are transverse if and only if $\g(\mf{u})$ and $\g(\mf{w})$ are joined by an edge of $\G$.

\begin{lem}\label{common argument}
If $g,h\in \A$ and $\G(g)\cu\g\big(\mscr{W}_{\overline{\C}(g)}(\X)\cap\mscr{W}_{\overline{\C}(h)}(\X)\big)$, then $\overline{\C}(g)\cap\overline{\C}(h)\neq\emptyset$.
\end{lem}
\begin{proof}
Suppose for the sake of contradiction that $\overline{\C}(g)$ and $\overline{\C}(h)$ are disjoint. Then there exists a hyperplane $\mf{v}$ separating them, which we pick so that the carrier of $\mf{v}$ intersects $\overline{\C}(g)$. This guarantees that $g$ admits an axis $\alpha$ that intersects the carrier of $\mf{v}$.

Since $\mf{v}$ separates $\overline{\C}(g)$ and $\overline{\C}(h)$, it is transverse to $\mscr{W}_{\overline{\C}(g)}(\X)\cap\mscr{W}_{\overline{\C}(h)}(\X)$, so the vertex $\g(\mf{v})$ is connected by an edge of $\G$ to all elements of $\G(g)$. Observing that all hyperplanes crossed by $\alpha$ are labelled by elements of $\G(g)$ and recalling that $\alpha$ intersects the carrier of $\mf{v}$, we deduce that all hyperplanes crossed by $\alpha$ are transverse to $\mf{v}$. In other words, $\mf{v}$ is transverse to $\W_1(g)$, hence $\mf{v}\in\W_0(g)\cu\mscr{W}_{\overline{\C}(g)}(\X)$. This is the required contradiction.
\end{proof}

\begin{lem}\label{friendly label-irreducibles}
Let $g,h\in \A$ be label-irreducible. If there exist hyperplanes $\mf{u},\mf{w}\in\mscr{W}(\X)$ such that $\{\mf{u},g^{4r}\mf{u},\mf{w},h^{4r}\mf{w}\}\cu\W_1(g)\cap\W_1(h)$, then $\langle g,h\rangle\simeq\Z$. 
\end{lem}
\begin{proof}
The proof will consist of three steps.

\smallskip
{\bf Step~1:} \emph{we can assume that $1\in \A\cong \X^{(0)}$ lies in $\overline{\C}(g)\cap\overline{\C}(h)$, and that $\G(g)=\G(h)=\G^{(0)}$}.

\smallskip \noindent
Since $\W_1(g)\cap\W_1(h)$ contains any hyperplane separating two of its elements, we have $\mscr{W}(\mf{u}|g^{4r-2}\mf{u})\cu\W_1(g)\cap\W_1(h)$. Lemma~\ref{irreducible elements new} yields:
\[\G(g)=\g\left(\mscr{W}(\mf{u}|g^{4r-2}\mf{u})\right)\cu\g\left(\W_1(g)\cap\W_1(h)\right)\cu\G(h).\]
One the one hand, this allows us to apply Lemma~\ref{common argument} and deduce that $\overline{\C}(g)\cap\overline{\C}(h)\neq\emptyset$. On the other, this shows that $\G(g)\cu\G(h)$ and the inclusion $\G(h)\cu\G(g)$ is obtained similarly, so $\G(g)=\G(h)$. 

Conjugating $g$ and $h$ by any $x\in\overline{\C}(g)\cap\overline{\C}(h)$, we can assume that $1\in\overline{\C}(g)\cap\overline{\C}(h)$. Equivalently, $g$ and $h$ are cyclically reduced, so they lie in the parabolic subgroup $\mc{A}_{\G(g)}=\mc{A}_{\G(h)}\leq\A$. Replacing $\A$ with $\mc{A}_{\G(g)}$ does not alter the properties in the statement of the lemma, so we can assume that $\G(g)=\G(h)=\G^{(0)}$.

\smallskip
{\bf Step~2:} \emph{Assume without loss of generality that $\ell(g,\X)\leq\ell(h,\X)$. Possibly replacing $g$ and $h$ with their inverses and conjugating them, there exists a geodesic $\s\cu \X$ from $1$ to $g$ such that:
\begin{itemize}
\item the union $\rho:=\bigcup_{i\geq 0}g^i\s$ is a ray and contains $h$ and $h^2$ (viewing $1,g,h,h^2$ as vertices of $\X$);
\item if $\tau\cu\rho$ is the arc joining $1$ to $h$, then $h\cdot\tau$ is the arc of $\rho$ joining $h$ to $h^2$.
\end{itemize}
}
%\emph{we can assume that there exists a (combinatorial) geodesic $\s\cu \X$ from $1$ to $g$ such that $h$ and $h^2$ belong to the ray $\bigcup_{i\geq 0}g^i\s$ (viewing $1,g,h,h^2$ as vertices of $\X$).}

\smallskip \noindent 
Let $\mf{k}\in\mscr{H}(\X)$ be a halfspace bounded by $h^{4r-2}\mf{w}\in\W_1(g)\cap\W_1(h)$. Possibly replacing $g$ and/or $h$ with their inverses, we have $g\mf{k}\subsetneq\mf{k}$ and $h\mf{k}\subsetneq\mf{k}$. Since $\G^{(0)}=\G(h)$, Lemma~\ref{irreducible elements new} shows that $\mf{w}$ and $h^{4r-2}\mf{w}$ are strongly separated in $\X$. 
 
The sub-ray contained in $\mf{k}^*$ of any (combinatorial) axis of $g$ defines a point $\xi$ in the Roller boundary $\partial \X$ such that $g\xi=\xi$ and $\xi\in h^{-4r+2}\mf{k}^*$ (recall that this halfspace is bounded by $\mf{w}$). Similarly, there exists $\eta\in\partial \X$ with $h\eta=\eta$ and $\eta\in h^{-4r+2}\mf{k}^*$. Since the halfspaces $h^{-4r+2}\mf{k}^*$ and $\mf{k}$ are strongly separated, the gate-projections of $\xi$ and $\eta$ to $\mf{k}$ coincide and they are a vertex $x\in\overline{\C}(g)\cap\overline{\C}(h)$. Conjugating $g$ and $h$ by $x$, we can assume that $x=1$. 

Label $\mf{k}_1\supsetneq\mf{k}_2\supsetneq\dots\supsetneq\mf{k}_m$ the elements of $\mscr{H}(1|h^2)$ bounded by hyperplanes with label $\g(\mf{w})$. Set $\mf{k}_0:=\mf{k}$ and observe that $\mf{k}_m=h^2\mf{k}$, which is bounded by $h^{4r}\mf{w}\in\W_1(g)\cap\W_1(h)$. In conclusion:
\[\xi,\eta\not\in h^{-4r+2}\mf{k}\supsetneq\mf{k}=\mf{k}_0\supsetneq\mf{k}_1\supsetneq\dots\supsetneq\mf{k}_m=h^2\mf{k}.\]

Note that that the hyperplanes bounding the $\mf{k}_i$ all lie in $\W_1(g)\cap\W_1(h)$. Since $1\in\overline{\C}(g)\cap\overline{\C}(h)$, there exist an axis of $h$ and an axis of $g$ each crossing all hyperplanes bounding the $\mf{k}_i$. Hence there exist $1\leq t\leq s$ such that $g\mf{k}_j=\mf{k}_{j+t}$ for all $0\leq j\leq m-t$, and $h\mf{k}_i=\mf{k}_{i+s}$ for all $0\leq i\leq m-s$.

Let $x_i$ be the gate-projection of $x=1$ to $\mf{k}_i$. Note that this is also the gate-projection to $\mf{k}_i$ of $\xi$ and $\eta$. Since $g\xi=\xi$ and $h\eta=\eta$, we must have $gx_j=x_{j+t}$ and $hx_i=x_{i+s}$ for all $1\leq j\leq m-t$ and $1\leq i\leq m-s$. In particular, since $x_0=1$, we have $h=x_s$, $h^2=x_{2s}=x_m$ and $g=x_t$.

Observe that each $x_i$ is also the gate-projection to $\mf{k}_i$ of each $x_j$ with $j<i$. Thus, we can construct a (combinatorial) geodesic $\s$ from $1$ to $g$ by concatenating arbitrary geodesics $\s_j$ from $x_j$ to $x_{j+1}$ for $0\leq j<t$. The union $\rho=\bigcup_{i\geq 0}g^i\s$ is a ray since $1\in\overline{\C}(g)$. Let $k,l\geq 1$ be the integers with $0\leq s-kt<t$ and $0\leq 2s-lt<t$. Since $\s$ contains the points $g^{-k}h=x_{s-kt}$ and $g^{-l}h^2=x_{2s-lt}$, it is clear that $h$ and $h^2$ lie on the ray $\rho$.

Finally, note that we can choose the geodesics $\s_j$ so that the following compatibility condition is satisfied: \emph{whenever there exist $f\in\A$ and $0\leq i,j<t$ with $fx_i=x_j$ and $fx_{i+1}=x_{j+1}$, we have $f\s_i=\s_j$}. This is possible because the action $\A\acts\X$ is free and so the element $f$ is uniquely determined by $i$ and $j$ (when it exists). Now, given $0\leq j<s$, the arc of the ray $\rho$ joining $x_{s+j}$ to $x_{s+j+1}$ is precisely $g^{a_j}\s_{b_j}$, where $s+j=a_jt+b_j$ and $0\leq b_j<t$. The element $g^{-a_j}h$ maps $x_j$ and $x_{j+1}$ to $x_{b_j}$ and $x_{b_j+1}$, so it takes $\s_j$ to $\s_{b_j}$ by our construction. 
% $\s_j$ has not exactly been defined (as we can have $j>t$), but it's clear what this means
Thus $h\s_j=g^{a_j}\s_{b_j}$ is contained in $\rho$, for every $0\leq j<s$. This proves the second condition in the statement of Step~2.

\smallskip
{\bf Step~3:} \emph{we have $\langle g,h\rangle\simeq\Z$.}

\smallskip \noindent 
Let $S\cong\G^{(0)}$ be the standard generating set of $\A$. Let $F(S)$ be the free group freely generated by $S$, and let $\pi\colon F(S)\ra \A$ be the surjective homomorphism that takes each generator of $F(S)$ to the corresponding standard generator of $\A$. Let $w_g\in F(S)$ be the word spelled by the labels of the edges met moving from $1$ to $g$ along the geodesic $\s$. Let $w_h\in F(S)$ be the word spelled moving from $1$ to $h$ along the ray $\rho=\bigcup_{i\geq 0}g^i\s$. It is clear that $\pi(w_g)=g$ and $\pi(w_h)=h$. 

From Step~2, we have $w_h=w_g^pa$, for some $p\geq 1$ and an initial subword $a$ of $w_g$, and $w_h^2=w_g^{p+1}ab$, for some word $b$ such that $w_g^{p+1}ab$ is reduced in $F(S)$. It follows that $w_g^paw_g^pa=w_g^{p+1}ab$ in $F(S)$, where both sides of the equality are reduced words. Looking at the first $((p+1)|w_g|+|a|)$ letters on the left, we deduce that $aw_g=w_ga$. Hence $\langle w_g,w_h\rangle=\langle w_g, a\rangle$ is a cyclic subgroup of $F(S)$. We conclude that $\langle g,h\rangle=\pi\left(\langle w_g,w_h\rangle\right)\simeq\Z$.
\end{proof}

\begin{cor}\label{label-irreducible cor}
Consider two elements $g,h\in\A$. Suppose that $g$ is label-irreducible. Assume in addition that {\bf one} of the following conditions is satisfied.
\begin{itemize}
\item There exists $\mf{w}\in\W_1(g)$ such that $h$ preserves $\mf{w}$ and $g^{4r}\mf{w}$. 
\item There exist hyperplanes $\mf{u},\mf{w}\in\mscr{W}(\X)$ with $\{\mf{u},\mf{w},h^{4r}\mf{u},g^{4r}\mf{w}\}\cu\W_1(g)\cap\W_1(h)$.
\end{itemize}
Then $g$ and $h$ commute in $\A$.
\end{cor}
\begin{proof}
Assume first that there exists $\mf{w}\in\W_1(g)$ such that $\mf{w}$ and $g^{4r}\mf{w}$ are preserved by $h$. Then $\{\mf{w},g^{4r}\mf{w}\}=\{\mf{w},(hgh^{-1})^{4r}\mf{w}\}\cu\W_1(g)\cap\W_1(hgh^{-1})$. Since $g$ and $hgh^{-1}$ are label-irreducible, Lemma~\ref{friendly label-irreducibles} implies that $\langle g,hgh^{-1}\rangle\simeq\Z$. Observing that $\ell(g,\X)=\ell(hgh^{-1},\X)$, we deduce that $hgh^{-1}$ must coincide with either $g$ or $g^{-1}$. The second option cannot occur in a right-angled Artin group, hence $hgh^{-1}=g$, as required.

Suppose now that there exist hyperplanes $\mf{u},\mf{w}$ with $\{\mf{u},\mf{w},h^{4r}\mf{u},g^{4r}\mf{w}\}\cu\W_1(g)\cap\W_1(h)$. In light of Lemma~\ref{label-irreducible decomposition}(2), there exist (possibly equal) irreducible components $h_1,h_2$ of $h$, such that $\{\mf{u},g^{4r}\mf{u}\}\cu\W_1(g)\cap\W_1(h_1)$ and $\{\mf{w},h^{4r}\mf{w}\}=\{\mf{w},h_2^{4r}\mf{w}\}\cu\W_1(g)\cap\W_1(h_2)$. 

Since $g$ is label-irreducible and $\g(\mscr{W}(\mf{u}|g^{4r}\mf{u}))=\G(g)$ by Lemma~\ref{irreducible elements new}, no element of $\W_1(g)$ can be transverse to both $\mf{u}$ and $g^{4r}\mf{u}$. Hence $h_1=h_2$, otherwise $\mc{W}_1(h_1)$ and $\mc{W}_1(h_2)$ would be transverse. Thus $\{\mf{u},g^{4r}\mf{u},\mf{w},h_2^{4r}\mf{w}\}\cu\W_1(g)\cap\W_1(h_2)$ and Lemma~\ref{friendly label-irreducibles} yields $\langle g,h_2\rangle\simeq\Z$. Now, a power of $g$ coincides with a power of $h_2$, hence it commutes with $h$. It follows that $g$ and $h$ commute.
\end{proof}

We conclude with the following lemma, which is actually independent from the notion of label-irreducibility and from the discussion in the rest of this subsection, albeit in a similar spirit.

\begin{lem}\label{another commutation criterion new}
Consider elements $g_1,g_2\in\A$ and vertices $x_1,x_2\in\X$ such that the two sets $\mscr{W}(x_i|g_ix_i)$ are transverse. Then $g_1$ and $g_2$ commute and we have $\W_1(g_i)\cu\W_0(g_j)$ for $i\neq j$.
\end{lem}
\begin{proof}
We begin with the following observation.

\smallskip
{\bf Claim:} \emph{for $\mf{w}\in\mscr{W}(\X)$ and $x,y\in\X$, the hyperplane $\mf{w}$ is transverse to $\mscr{W}(x|y)$ if and only if every vertex in the set $\g(\mscr{W}(x|y))$ is joined by an edge of $\G$ to every vertex in the set $\{\g(\mf{w})\}\cup\g(\mscr{W}(x|\mf{w}))$.}

\smallskip\noindent
\emph{Proof of Claim.}
Suppose first that $\mf{w}$ is transverse to $\mscr{W}(x|y)$. Then, since every hyperplane in $\mscr{W}(\mf{w}|x)$ is disjoint from $\mf{w}$, we have $\mscr{W}(\mf{w}|x)\cu\mscr{W}(\mf{w}|x,y)$. Since $\mscr{W}(x|y)$ is transverse to $\mscr{W}(\mf{w}|x,y)$, it is also transverse to $\{\mf{w}\}\cup\mscr{W}(\mf{w}|x)$. It follows that every vertex in $\g(\mscr{W}(x|y))$ is joined by an edge to every vertex in $\{\g(\mf{w})\}\cup\g(\mscr{W}(x|\mf{w}))$, as required.

Suppose now instead that $\mf{w}$ is disjoint from a hyperplane $\mf{u}\in\mscr{W}(x|y)$. Choosing $\mf{u}$ so that it is closest to $\mf{w}$, we can assume that no hyperplane of $\mscr{W}(x|y)$ separates $\mf{w}$ and $\mf{u}$. If the carriers of $\mf{w}$ and $\mf{u}$ intersect, then $\g(\mf{w})$ and $\g(\mf{u})$ cannot be joined by an edge, as $\mf{u}$ and $\mf{w}$ are disjoint. Otherwise, there exists a hyperplane $\mf{v}\in\mscr{W}(\mf{u}|\mf{w})$ such that its carrier intersects the carrier of $\mf{u}$; in particular, $\g(\mf{u})$ and $\g(\mf{v})$ are not joined by an edge. Since $\mf{v}$ does not separate $x$ and $y$, we must have $\mf{v}\in\mscr{W}(\mf{w}|x,y)$, hence $\g(\mf{v})\in\g(\mscr{W}(x|\mf{w}))$ as required.
\hfill$\blacksquare$ 

\smallskip
Consider for a moment $g\in\A$, $x\in\X$ and $n\geq 1$. Since $\mscr{W}(x|g^nx)$ is contained in the union $\mscr{W}(x|gx)\cup\dots\cup\mscr{W}(g^{n-1}x|g^nx)$, we have $\g(\mscr{W}(x|g^nx))\cu\g(\mscr{W}(x|gx))$. Thus, the Claim implies that a hyperplane $\mf{w}$ is transverse to $\mscr{W}(x|gx)$ if and only if it is transverse to $\bigcup_{n\in\Z}\mscr{W}(x|g^nx)$.

Now, consider the situation in the statement of the lemma. If $x_i'$ is the gate-projection of $x_i$ to $\overline{\C}(g_i)$, we have $\mscr{W}(x_i'|g_ix_i')\cu\mscr{W}(x_i|g_ix_i)$ and $\W_1(g_i)=\bigcup_{n\in\Z}\mscr{W}(x_i'|g_i^nx_i')$. It follows that the sets $\W_1(g_1)$ and $\W_1(g_2)$ are transverse, or, equivalently, $\W_1(g_i)\cu\W_0(g_j)$ for $i\neq j$. This implies that $g_1$ and $g_2$ commute (for instance, by decomposing $g_i$ into label-irreducible components as in Lemma~\ref{label-irreducible decomposition} and applying Corollary~\ref{label-irreducible cor}).
\end{proof}

\subsubsection{Convex-cocompact subgroups of RAAGs.}

Again, we keep the notation from Subsection~\ref{label-irreducible sect}. We will simply say that a subgroup $G\leq\A$ is \emph{convex-cocompact} when $G$ is convex-cocompact for the action $\A\acts\X$ (in the sense of Definition~\ref{cc defn}).

\begin{lem}\label{label-irreducibles don't escape}
Let $G\leq\A$ be convex-cocompact. If $g\in G$ and $g=a_1\cdot\ldots\cdot a_k$ is its decomposition into label-irreducible components $a_i\in\A$, then there exists $m\geq 1$ such that all $a_i^m$ lie in $G$.
\end{lem}
\begin{proof}
Let $A\leq G$ be a free abelian subgroup containing a power of $g$, such that no finite-index subgroup of $A$ is contained in a free abelian subgroup of $G$ of higher rank. Since $G$ is convex-cocompact, Theorem~3.6 in \cite{WW} shows that there exists a convex, $A$--invariant, $A$--cocompact subcomplex $Y\cu\X$ that splits as a product $L_1\x\dots\x L_p$, where $A\simeq\Z^p$ and each $L_i$ is a quasi-line. Replacing each $L_i$ with a subcomplex, we can assume that all quasi-lines are $A$--essential.

Note that $Y$ must contain an axis of $g$ in $\X$, hence its convex hull, which is isomorphic to:
\[\overline{\C}_1(g)=\overline{\C}_1(a_1)\x\dots\x\overline{\C}_1(a_k).\] 
Since each $a_i$ is label-irreducible, Lemma~\ref{label-irreducibles are cc 1} shows that $\overline{\C}_1(a_i)$ is an irreducible quasi-line. Up to permuting the factors of $Y$, we can thus assume that $L_i\simeq\overline{\C}_1(a_i)$ for $1\leq i\leq k$, where $k\leq p$.

Since the $L_i$ are locally finite, none of the groups $\aut L_i$ contains subgroups isomorphic to $\Z^2$. It follows that every projection of $\Z^p\simeq A\leq\prod_i\aut L_i$ to a product of $(p-1)$ factors must have nontrivial kernel. Equivalently, there exist elements $h_i\in A$ such that $h_i$ acts loxodromically on $L_i$, and fixes pointwise each $L_j$ with $j\neq i$. For each $1\leq i\leq k$, the elements $h_i$ and $a_i$ stabilise a common copy of $L_i\simeq\overline{\C}_1(a_i)$ inside $Y$, and act freely and cocompactly on it. It follows that $h_i$ and $a_i$ are commensurable, hence a power of $a_i$ lies in $A\leq G$. This concludes the proof.
\end{proof}

The exponent $m$ in Lemma~\ref{label-irreducibles don't escape} can be chosen independently of $g\in G$ due to the following.

\begin{rmk}\label{uniform exponent}
Suppose that $G\leq\A$ is convex-cocompact and, more precisely, that there exists a $G$--invariant, convex subcomplex $Y\cu\X$ such that the action $G\acts Y^{(0)}$ has $q$ orbits. Then, for every $g\in\A$ such that $\langle g\rangle\cap G\neq\{1\}$, there exists $1\leq k\leq q$ such that $g^k\in G$.

Indeed, consider $N\geq 1$ such that $g^N\in G$. Since $Y$ is $G$--invariant and acted upon without inversions, it contains an axis $\alpha$ for $g^N$ \cite{Haglund}. Every axis of a power of $g$ is, in fact, also an axis of $g$ (this property is specific to the action $\A\acts\X$). Thus, picking any $x\in\alpha$, we have $g^ix\in Y$ for all $i\in\Z$. Hence there exist $0\leq i<j\leq q$ such that $g^ix$ and $g^jx$ are in the same $G$--orbit. Since $\A$ acts freely on $\X$, we have $g^{j-i}\in G$ and $0<j-i\leq q$.
\end{rmk}

\subsection{CMP automorphisms of right-angled groups.}\label{of RAAGs sect}

This subsection is devoted to the proof of Proposition~\ref{cmp prop intro}. Automorphisms of hyperbolic groups were already discussed in Example~\ref{hyp cmp ex}, so we are only concerned with right-angled Artin/Coxeter groups.

Let $\G$ be a finite simplicial graph. Let $\A=\A_{\G}$ and $\W=\W_{\G}$ be, respectively, the right-angled Artin group and the right-angled Coxeter group defined by $\G$.

We identify with $\G^{(0)}$ the standard generating sets of $\A$ and $\W$. The standard Cayley graphs of $\A$ and $\W$ are $1$--skeleta of $\CAT$ cube complexes: the universal covers of the Salvetti and Davis complex, respectively. Thus, $\A$ and $\W$ are each endowed with a natural median operator $\mu_{\G}$.

\begin{rmk}
We have $g\cdot\mu_{\G}(x,y,z)=\mu_{\G}(gx,gy,gz)$ for all elements $g,x,y,z$ in $\A$ or $\W$. This implies that $(\A,[\mu_{\G}])$ and $(\W,[\mu_{\G}])$ are coarse median groups, in the sense of Definition~\ref{coarse median group defn}.
\end{rmk}

Unlike hyperbolic groups, $\A$ and $\W$ can admit infinitely many different coarse median structures. For this reason, we will never omit the subscript in $\mu_{\G}$, in order to emphasise that this is the coarse median structure provided by the \emph{standard} generating set of $\A$ or $\W$. Other Artin/Coxeter generating sets can a priori give different coarse median structures; it will be a consequence of Proposition~\ref{cmp prop intro}(2), that this does not actually happen in the Coxeter case.

It was shown by Laurence, Servatius and Corredor--Gutierrez that $\aut\A$ and $\aut\W$ are generated by finitely many \emph{elementary automorphisms} \cite{Servatius,Laurence,Corredor-Gutierrez}. These take the same form in both cases.
\begin{itemize}
\item \emph{Graph automorphisms}. Every automorphism of the graph $\G$ gives a permutation of the standard generating sets that defines an automorphism of $\A$ and $\W$.
\item \emph{Inversions} $\iota_v$ for each $v\in\G^{(0)}$. We have $\iota_v(v)=v^{-1}$ and $\iota_v(u)=u$ for all $u\in\G^{(0)}\setminus\{v\}$.
\item \emph{Partial conjugations} $\kappa_{w,C}$ for $w\in\G^{(0)}$ and a connected component $C$ of $\G\setminus\St w$. We have $\kappa_{w,C}(u)=w^{-1}uw$ if $u\in C^{(0)}$ and $\kappa_{w,C}(u)=u$ if $u\in\G^{(0)}\setminus C$.
%N.B.: we are not requiring the star to disconnect $\G$: this ensures that all inner automorphisms are untwisted
\item \emph{Transvections} $\tau_{v,w}$ for $v,w\in\G^{(0)}$ with $\lk v\cu\St w$. They are defined by $\tau_{v,w}(v)=vw$ and $\tau_{v,w}(u)=u$ for all $u\in\G^{(0)}\setminus\{v\}$. 

We refer to $\tau_{v,w}$ as a \emph{fold} if $v$ and $w$ are not joined by an edge (equivalently, $\lk v\cu\lk w$), and as a \emph{twist} if $v$ and $w$ are joined by an edge (equivalently, $\St v\cu\St w$).
\end{itemize}

\begin{rmk}\label{inversions cmp}
Graph automorphisms and inversions can be realised as automorphisms of the Salvetti/Davis complex, so they preserve the operator $\mu_{\G}$ (hence the coarse median structure $[\mu_{\G}]$).
\end{rmk}

In the case of right-angled Artin groups, the following class of automorphisms was introduced by Day \cite{Day-peak} and Charney, Stambaugh and Vogtmann \cite{CSV}.

\begin{defn}\label{U(A) defn}
An automorphism $\varphi\in\aut\A$ is \emph{untwisted} if it lies in the subgroup $U(\A)\leq\aut\A$ generated by graph automorphisms, inversions, partial conjugations and folds. % this contains all inner automorphisms, since they are products of partial conjugations
\end{defn}

We now proceed to prove parts~(2) and~(3) of Proposition~\ref{cmp prop intro}. We will treat separately the Coxeter and Artin case, as the simplest arguments appear to be quite different in spirit. Still, in both situations, the following basic observation is important.

\begin{rmk}\label{restriction quotient cmp rmk}
Let a group $G$ act properly and cocompactly on two $\CAT$ cube complexes $X$ and $Y$. Let $[\mu_X]$ and $[\mu_Y]$ be the induced coarse median structures on $G$. If there exists an equivariant restriction-quotient map $\pi\colon X\ra Y$, then $[\mu_X]=[\mu_Y]$. This is immediate from the third characterisation of restriction quotients in Proposition~\ref{restriction quotient prop}.
\end{rmk}

\subsubsection{The Coxeter case.} 

Here our aim is to prove that all elements of $\aut\W$ preserve the coarse median structure $[\mu_{\G}]$. We will achieve this by showing that all elementary automorphisms of $\W$ restrict to graph automorphisms on certain finite-index Coxeter subgroups of $\W$. This guarantees that they are all coarse-median preserving.

Given a vertex $w\in\G$, we denote by $\Delta(\G,w)$ the \emph{double} of $\G\setminus\{w\}$ along $\lk w$. More precisely, $\Delta(\G,w)$ is obtained from two disjoint copies of the graph $\G\setminus\{w\}$ by identifying the two subgraphs corresponding to $\lk w$. We continue to denote by $\lk w$ the resulting subgraph of $\Delta(\G,w)$, even though $w$ does not appear in $\Delta(\G,w)$ and so this is not the link of any vertex of $\Delta(\G,w)$.

Let $\alpha_w\colon\W\ra\Z/2\Z$ be the homomorphism that maps $w$ to the generator of $\Z/2\Z$, and all other standard generators of $\W$ to the identity.

\begin{lem}\label{RACG kernels}
Consider a vertex $w\in\G$. Then:
\begin{enumerate}
\item $\ker\alpha_w$ is generated by $\{x\mid x\in\G\setminus\{w\}\}\sqcup\{w^{-1}yw \mid y\in\G\setminus\St w\}$;
\item this is a Coxeter generating set giving an isomorphism between $\ker\alpha_w$ and $\W_{\Delta(\G,w)}$;
\item the coarse median structure $[\mu_{\Delta(\G,w)}]$ induced on $\ker\alpha_w$ by this isomorphism with $\W_{\Delta(\G,w)}$ coincides with the restriction of the coarse median structure $[\mu_{\G}]$ on $\W$.
\end{enumerate}
\end{lem}
\begin{proof}
The first two parts are a straightforward application of the normal form for words in Coxeter groups (see e.g.\ \cite[Chapter~3.4]{Davis}). We instead focus on part~(3).

Let $\mc{Y}_{\G}$ and $\mc{Y}_{\Delta(\G,w)}$ be the universal covers of the Davis complexes of $\W$ and $\W_{\Delta(\G,w)}$. We aim to show that, under the above identification between $\ker\alpha_w$ and $\W_{\Delta(\G,w)}$, the standard action $\W_{\Delta(\G,w)}\acts\mc{Y}_{\Delta(\G,w)}$ is a restriction quotient of the action $\ker\alpha_w\acts\mc{Y}_{\G}$. This proves the lemma, since, by Remark~\ref{restriction quotient cmp rmk}, the two actions then induce the same coarse median structure on $\ker\alpha_w$.

First, if $\Om\cu\mc{Y}_{\G}$ is a fundamental domain for the $\W$--action, note that $\Om\cup w\Om$ is a fundamental domain for $\ker\alpha_w\acts\mc{Y}_{\G}$. In addition, observe that $\ker\alpha_w$ contains the entire $\W$--stabiliser of a hyperplane $\mf{u}\in\mscr{W}(\mc{Y}_{\G})$ precisely when $\g(\mf{u})\not\in\St w$. Thus, the orbit $\W\cdot\mf{u}$ is made up of two $(\ker\alpha_w)$--orbits of hyperplanes when $\g(\mf{u})\not\in\St w$, while it is a single $(\ker\alpha_w)$--orbit when $\g(\mf{u})\in\St w$. Combining these two observations, the reader should convince themselves that, starting with the action $\ker\alpha_w\acts\mc{Y}_{\G}$ and collapsing the single orbit of hyperplanes $\mf{u}$ with $\g(\mf{u})=w$, we obtain precisely the action $\W_{\Delta(\alpha,w)}\acts\mc{Y}_{\Delta(\alpha,w)}$, as required.
\end{proof}

\begin{prop}
All automorphisms of $\W$ preserve the coarse median structure $[\mu_{\G}]$.
\end{prop}
\begin{proof}
Recall that $\aut\W$ is generated by graph automorphisms, partial conjugations and transvections, as defined above. We have already noticed in Remark~\ref{inversions cmp} that graph automorphisms are coarse-median preserving. We make two additional observations.
\begin{itemize}
\item \emph{Every partial conjugation $\kappa_{w,C}$ preserves the subgroup $\ker\alpha_w\leq\W$. The restriction of $\kappa_{w,C}$ to $\ker\alpha_w$ is a graph automorphism with respect to the identification $\ker\alpha_w\simeq\W_{\Delta(\G,w)}$ constructed in Lemma~\ref{RACG kernels}.}

Indeed, the connected component $C\cu\G\setminus\St w$ gets doubled to two connected components $C',C''$ of the graph $\Delta(\G,w)\setminus\lk w$. These two subgraphs correspond to the sets of generators $C^{(0)}$ and $w^{-1}C^{(0)}w$ for $\ker\alpha_w$. The automorphism $\kappa_{w,C}$ swaps these two sets of generators, while fixing all generators of $\ker\alpha_w$ corresponding to vertices of $\Delta(\G,w)\setminus(C'\cup C'')$. This is realised by an automorphism of the graph $\Delta(\G,w)$.
\item \emph{Every transvection $\tau_{v,w}$ preserves $\ker\alpha_v\leq\W$. The restriction of $\tau_{v,w}$ to $\ker\alpha_v$ is a product of partial conjugations with respect to the identification $\ker\alpha_v\simeq\W_{\Delta(\G,v)}$.}

Indeed, if $x\in\G\setminus\{v\}$, we have $\tau_{v,w}(x)=x$ and $\tau_{v,w}(v^{-1}xv)=w^{-1}(v^{-1}xv)w$. Setting $A:=\G\setminus\St v\cu\G$, we have $\Delta(\G,v)=\lk v\sqcup A'\sqcup A''$, where the subgraphs $A',A''$ correspond to the subsets $A^{(0)},v^{-1}A^{(0)}v\cu\W$, respectively. Let $w'\in A'$ be the vertex originating from $w\in\G$. The set $A''$ is a union of connected components of $\Delta(\G,v)\setminus\lk v$. Since $\lk v\cu\St w$ in $\G$, we have $\lk v\cu\St w'$ in $\Delta(\G,v)$, hence $A''$ is also a union of connected components of $\Delta(\G,v)\setminus\St w'$. The composition of the partial conjugations $\kappa_{w',C}\in\aut\W_{\Delta(\G,v)}$, as $C$ ranges through these connected components, is precisely the restriction of $\tau_{v,w}$ to $\ker\alpha_v$.
\end{itemize}

Now, in view of Lemma~\ref{RACG kernels} and Remark~\ref{inversions cmp}, partial conjugations and transvections each preserve the restriction of the coarse median structure $[\mu_{\G}]$ to a finite-index subgroup of $\W$. Since finite-index subgroups are coarsely dense in $\W$, this implies that these automorphisms actually preserve $[\mu_{\G}]$ itself, proving the proposition.
\end{proof}

\subsubsection{The Artin case.}

We begin by showing that \emph{untwisted} automorphisms of $\A$ preserve the coarse median structure $[\mu_{\G}]$. I present a proof that was suggested to me by Ric Wade, as it is much simpler than my original brute-force argument.

The main ingredient is the construction of Salvetti blowups from the work of Charney, Stambaugh and Vogtmann \cite{CSV}, which we record in the following lemma. Restriction quotients were discussed in Subsection~\ref{CCC prelims}.

\begin{lem}\label{blowup lem}
Let $\varphi\in U(\A_{\G})$ be a fold or partial conjugation. Then there exists a proper cocompact action on a $\CAT$ cube complex $\A_{\G}\acts Z$ and two restriction-quotient maps $\pi_1,\pi_2\colon Z\ra\X_{\G}$ such that, for all $g\in\A_{\G}$, we have $\pi_1\o g=g\o\pi_1$ and $\pi_2\o g=\varphi(g)\o\pi_2$.
\end{lem}
\begin{proof}
This holds more generally when $\varphi$ is a \emph{$\G$--Whitehead automorphism}, as defined at the beginning of \cite[Section~2.3]{CSV}. Our statement is a straightforward rephrasing of \cite[Lemma~3.2]{CSV} in terms of universal covers.
\end{proof}

\begin{cor}
Automorphisms in $U(\A)$ preserve the coarse median structure $[\mu_{\G}]$.
\end{cor}
\begin{proof}
Let $\varphi\in U(\A_{\G})$ be a fold or partial conjugation. Let the action $\A_{\G}\acts Z$ and the maps $\pi_1,\pi_2\colon Z\ra\X_{\G}$ be as provided by Lemma~\ref{blowup lem}, and let $[\mu_Z]$ be the coarse median structure on $\A_{\G}$ induced by $Z$. Since $\pi_1$ is $\A_{\G}$--equivariant, Remark~\ref{restriction quotient cmp rmk} guarantees that $[\mu_{\G}]=[\mu_Z]$. 

On the other hand, $\pi_2$ becomes $\A_{\G}$--equivariant if we endow $\X_{\G}$ with the $\varphi$--twisted action: using the notation from Remark~\ref{induced cms}, this corresponds to replacing $\X_{\G}$ with $\X_{\G}^{\varphi^{-1}}$, which induces the coarse median structure $(\varphi^{-1})_*[\mu_{\G}]$ on $\A_{\G}$. Thus, another application of Remark~\ref{restriction quotient cmp rmk}, yields $(\varphi^{-1})_*[\mu_{\G}]=[\mu_Z]$. We conclude that $\varphi_*[\mu_{\G}]=[\mu_{\G}]$.

This shows that all folds and partial conjugations preserve $[\mu_{\G}]$. Graph automorphisms and inversions are also coarse-median preserving, by Remark~\ref{inversions cmp}. Since these four types of elementary automorphisms generate $U(\A)$, this proves the corollary.
\end{proof}

In order to complete the proof of Proposition~\ref{cmp prop intro}(3), we are left to show that all coarse-median preserving automorphisms of $\A$ are untwisted. This can be easily deduced from the work of Laurence \cite{Laurence}, as we now describe.

\begin{prop}\label{cmp -> untwisted}
If $\varphi\in\aut\A$ preserves the coarse median structure $[\mu_{\G}]$, then $\varphi\in U(\A)$.
\end{prop}
\begin{proof}
In the terminology of \cite[Section~2]{Laurence}, an automorphism $\varphi\in\aut\A$ is \emph{conjugating} if it preserves the conjugacy class of each standard generator $v\in\G$. More generally, $\varphi$ is \emph{simple} if, for every $v\in\G$, the image $\varphi(v)$ is label-irreducible and $v\in\G(\varphi(v))$ (compare \cite[Definition~5.3]{Laurence} and Definition~\ref{label-irreducible defn} in our paper). 

Consider a coarse-median preserving automorphism $\varphi=\varphi_0$. By \cite[Corollary to Lemma~4.5]{Laurence}, there exists a graph automorphism $\psi_1$ such that, setting $\varphi_1:=\psi_1\varphi$, we have $v\in\G(\varphi_1(v))$ for every generator $v\in\G$. Since graph automorphisms are coarse-median preserving, $\varphi_1$ is again coarse-median preserving. By Corollary~\ref{cmp preserve cc} and Lemma~\ref{label-irreducibles are cc 1}(2), the element $\varphi_1(v)$ is label-irreducible for every $v\in\G$. Thus, $\varphi_1$ is simple. 

By the proofs of \cite[Lemma~6.8]{Laurence} and \cite[Corollary to Lemma~6.6]{Laurence}, there exists a product of inversions, folds and partial conjugations $\psi_2$ such that $\varphi_2:=\varphi_1\psi_2$ is conjugating. Finally, by \cite[Theorem~2.2]{Laurence}, the automorphism $\varphi_2$ is a product of partial conjugations. This shows that $\varphi\in U(\A)$, as required.
\end{proof}

\subsubsection{Pure automorphisms.}

We end this subsection by introducing the subgroups $U_0(\A)\leq U(\A)$ and $\aut_0\W\leq\aut\W$ generated by inversions, folds and partial conjugations (no graph automorphisms or twists, in both cases). These are the subgroups appearing in the statements of Theorem~\ref{U_0 cc intro} and Proposition~\ref{intro invariant splitting}, and we will study them further in Subsection~\ref{fix cc sect} and Section~\ref{invariant splittings sect}. For the time being, we limit ourselves to a few quick observations.

%\begin{rmk}\label{intrinsic U_0}
%The argument in the proof of Corollary~\ref{cmp -> untwisted} yields the following characterisation of elements of $U_0(\A)$: \emph{``An automorphism $\varphi\in\aut\A$ lies in $U_0(\A)$ if and only if $\varphi$ is coarse-median preserving and, for every $v\in\G^{(0)}$, we have $v\in\G(\varphi(v))$.''}
% FALSE!!
%\end{rmk}

\begin{rmk}\label{0 finite index}
The subgroups $U_0(\A)\leq U(\A)$ and $\aut_0\W\leq\aut\W$ have finite index. In the Coxeter case, see e.g.\ \cite[Proposition~1.2]{Sale-Susse}. In the Artin case, it suffices to observe that $U_0(\A)$ is normalised by all graph automorphisms, and that the latter generate a finite subgroup of $U(\A)$.
\end{rmk}

\begin{rmk}\label{graph automorphisms in U_0}
Note that, although they do not appear in our chosen generating set for $U_0(\A)$, graph automorphisms of $\A$ can still lie in $U_0(\A)$. Indeed, confusing $\s\in\aut\G$ with the induced $\s\in\aut\A$, we have $\s\in U_0(\A)$ if and only if $\lk\s(v)=\lk v$ for every $v\in\G$.

The ``only if'' part follows from Lemma~\ref{orthogonals are preserved}. For the ``if'' part, it suffices to show that $\s\in U_0(\A)$ whenever $\s$ swaps two vertices of $\G$ with the same link and fixes the rest of $\G$. In this case, $\s$ is a product of 3 folds and 3 inversions, as described at the end of the proof of \cite[Proposition~3.3]{Day-Wade}.
\end{rmk}

\begin{lem}\label{restriction of U_0}
If $\varphi(\A_{\Delta})=\A_{\Delta}$ for a full subgraph $\Delta\cu\G$ and $\varphi\in U_0(\A)$, then $\varphi|_{\A_{\Delta}}\in U_0(\A_{\Delta})$.
\end{lem}
\begin{proof}
We begin with a general observation. As in the proof of Corollary~\ref{cmp -> untwisted}, we can write $\varphi=\s\varphi_1$, where $\s$ is a graph automorphism and $\varphi_1$ is a simple automorphism of $\A$. Moreover, simple automorphisms are products of inversions, folds and partial conjugations, so $\varphi_1\in U_0(\A)$. We conclude that $\s\in U_0(\A)$ and Remark~\ref{graph automorphisms in U_0} shows that $\lk\s(v)=\lk v$ for every $v\in\G$.

If $v\in\Delta$, then $v\in\G(\varphi_1(v))$ because $\varphi_1$ is simple. Thus: 
\[\s(v)\in\s(\G(\varphi_1(v)))=\G(\s\varphi_1(v))=\G(\varphi(v))\cu\Delta.\]
We deduce that $\s(\Delta)=\Delta$, and Remark~\ref{graph automorphisms in U_0} shows that $\s|_{\A_{\Delta}}\in U_0(\A_{\Delta})$. Since $\s$ and $\varphi$ preserve $\A_{\Delta}$, so does $\varphi_1$, and it suffices to show that $\varphi_1|_{\A_{\Delta}}\in U_0(\A_{\Delta})$.

It is clear that $\varphi_1|_{\A_{\Delta}}$ is a simple automorphism of $\A_{\Delta}$, so the fact that $\varphi_1|_{\A_{\Delta}}\in U_0(\A_{\Delta})$ follows again from \cite{Laurence} as in the proof of Corollary~\ref{cmp -> untwisted}. 
\end{proof}

\section{Fixed subgroups of CMP automorphisms.}\label{fix cmp sect}

This section is devoted to fixed subgroups of coarse-median preserving automorphisms of cocompactly cubulated groups. Theorem~\ref{cmp und intro} is proved in Subsections~\ref{approx subalg sect} and~\ref{fg sect}, where we study the properties of those subgroups of cocompactly cubulated groups that are approximate median subalgebras (see Theorem~\ref{AMS thm}). At the end of Subsection~\ref{fg sect}, we also prove Corollaries~\ref{cmp fi overgroups} and~\ref{Nielsen intro}. 

Then in Subsections~\ref{staircases sect} and~\ref{wqc sect}, we develop a quasi-convexity criterion for approximate median subalgebras of cube complexes (Proposition~\ref{wqc->qc}). This is used to prove Theorem~\ref{U_0 cc intro} in Subsection~\ref{fix cc sect} (see Corollaries~\ref{fix RAAG cc} and~\ref{fix RACG cc}).

The reader interested only in Theorems~\ref{Q2 thm intro} and~\ref{special UNE thm intro} can just read the proof that $\Fix\varphi$ is finitely generated (Proposition~\ref{AMS fg}) and skip the rest of this section in its entirety.

\subsection{Approximate median subalgebras.}\label{approx subalg sect}

The goal of this subsection is to show that approximate median subalgebras (Definition~\ref{approx subalg defn}) of median spaces stay close to actual subalgebras. This is an important ingredient in the proofs of Theorem~\ref{cmp und intro} and Corollaries~\ref{cmp fi overgroups} and~\ref{Nielsen intro}, which will be discussed in the next subsection.

Shortly after the first draft of this paper appeared on arXiv, it was pointed out to me by Mark Hagen that a similar result appears in the work of Bowditch \cite[Proposition~4.1]{Bow-hulls}, which I was not aware of. Although Propositions~\ref{approx subalgebras} and~\ref{gen subalg prop} below are more general and our proofs seem different, I want to emphasise that Bowditch's result would suffice for all applications in this paper.

\begin{prop}\label{approx subalgebras}
If $X$ is a finite-rank median space and $A\cu X$ is an approximate median subalgebra, then $d_{\rm Haus}(A,\langle A\rangle)<+\infty$.
\end{prop}

The only focus of this subsection will actually be the next result, which provides an analogue of Remark~\ref{Bow J}. From it, it is straightforward to deduce Proposition~\ref{approx subalgebras} proceeding as in Lemma~\ref{about J}, which we leave to the reader.

\begin{prop}\label{gen subalg prop}
There exists a function $h\colon\N\ra\N$ with the following property. If $M$ is a median algebra of rank $r$ and $A\cu M$ is a subset, then $\langle A\rangle=\mc{M}^{h(r)}(A)$.
\end{prop}

We now obtain a sequence of lemmas leading up to Proposition~\ref{generating median algebras}, which proves Proposition~\ref{gen subalg prop}.

Let $M$ be a median algebra. We denote by $\mscr{M}(M)$ the collection of subsets of $M$ of one of these three forms:
\begin{itemize}
\item $\mf{h}$, where $\mf{h}$ is a halfspace;
\item $\mf{h}\cup\mf{k}$, where $\mf{h}$ and $\mf{k}$ are transverse halfspaces;
\item $\mf{h}\cup\mf{k}$, where $\mf{h}$ and $\mf{k}$ are disjoint halfspaces.
\end{itemize}
Elements of $\mscr{M}(M)$ are to median subalgebras what halfspaces of $M$ are to convex subsets. More precisely, the following is a straightforward characterisation of the median subalgebra generated by a subset $A\cu M$ (see for instance \cite[II.4.25.7]{vdVel}).

\begin{lem}\label{van de Vel}
For every subset $A\cu M$, the median subalgebra $\langle A\rangle\cu M$ is the intersection of all elements of $\mscr{M}(M)$ 
containing $A$.
\end{lem}

We will make repeated use of the following observation, without explicit mention:

\begin{lem}\label{3pwt}
Given points $a,b,c,d\in M$, the three sets $\mscr{W}(a,b|c,d),\mscr{W}(a,c|b,d),\mscr{W}(a,d|b,c)$ are transverse to each other.
\end{lem}

It is also convenient to give a name to the situation in Figure~\ref{pentagonal fig}.

\begin{defn}
An ordered $5$--tuple $(x_1,x_2,x_3,x_4,x_5)\in M^5$ is a \emph{pentagonal configuration} if the $5$ sets 
$\mscr{W}(x_{i-1},x_i,x_{i+1}|x_{i+2},x_{i+3})$ 
are all nonempty (indices are taken mod $5$).
\end{defn}

This requirement is invariant under cyclic permutations of the $5$ points. Also note that, setting 
$\mc{W}_i:=\mscr{W}(x_{i-1},x_i,x_{i+1}|x_{i+2},x_{i+3})$, 
the sets $\mc{W}_i$ and $\mc{W}_{i+1}$ are transverse for all $i$ mod $5$.

\begin{figure} 
\begin{tikzpicture}
\draw[fill] (0,0) -- ({1.5*0.5*cos(36)/cos(36)},{1.5*0.5*sin(36)/cos(36)});
\draw[fill] (0,0) -- ({1.5*0.5*cos(108)/cos(36)},{1.5*0.5*sin(108)/cos(36)});
\draw[fill] (0,0) -- ({1.5*0.5*cos(180)/cos(36)},{1.5*0.5*sin(180)/cos(36)});
\draw[fill] (0,0) -- ({1.5*0.5*cos(-108)/cos(36)},{1.5*0.5*sin(-108)/cos(36)});
\draw[fill] (0,0) -- ({1.5*0.5*cos(-36)/cos(36)},{1.5*0.5*sin(-36)/cos(36)});
\draw[fill] ({1.5*0.5*cos(36)/cos(36)},{1.5*0.5*sin(36)/cos(36)}) -- ({1.5*cos(0)},{1.5*sin(0)});
\draw[fill] ({1.5*0.5*cos(108)/cos(36)},{1.5*0.5*sin(108)/cos(36)}) -- ({1.5*cos(72)},{1.5*sin(72)});
\draw[fill] ({1.5*0.5*cos(180)/cos(36)},{1.5*0.5*sin(180)/cos(36)}) -- ({1.5*cos(144)},{1.5*sin(144)});
\draw[fill] ({1.5*0.5*cos(-108)/cos(36)},{1.5*0.5*sin(-108)/cos(36)}) -- ({1.5*cos(-144)},{1.5*sin(-144)});
\draw[fill] ({1.5*0.5*cos(-36)/cos(36)},{1.5*0.5*sin(-36)/cos(36)}) -- ({1.5*cos(-72)},{1.5*sin(-72)});
\draw[fill] ({1.5*0.5*cos(36)/cos(36)},{1.5*0.5*sin(36)/cos(36)}) -- ({1.5*cos(72)},{1.5*sin(72)});
\draw[fill] ({1.5*0.5*cos(108)/cos(36)},{1.5*0.5*sin(108)/cos(36)}) -- ({1.5*cos(144)},{1.5*sin(144)});
\draw[fill] ({1.5*0.5*cos(180)/cos(36)},{1.5*0.5*sin(180)/cos(36)}) -- ({1.5*cos(-144)},{1.5*sin(-144)});
\draw[fill] ({1.5*0.5*cos(-108)/cos(36)},{1.5*0.5*sin(-108)/cos(36)}) -- ({1.5*cos(-72)},{1.5*sin(-72)});
\draw[fill] ({1.5*0.5*cos(-36)/cos(36)},{1.5*0.5*sin(-36)/cos(36)}) -- ({1.5*cos(0)},{1.5*sin(0)});
\draw[fill] ({1.5*cos(0)},{1.5*sin(0)}) circle [radius=0.04cm];
\draw[fill] ({1.5*cos(72)},{1.5*sin(72)}) circle [radius=0.04cm];
\draw[fill] ({1.5*cos(144)},{1.5*sin(144)}) circle [radius=0.04cm];
\draw[fill] ({1.5*cos(-144)},{1.5*sin(-144)}) circle [radius=0.04cm];
\draw[fill] ({1.5*cos(-72)},{1.5*sin(-72)}) circle [radius=0.04cm];
\node[right] at ({1.5*cos(0)},{1.5*sin(0)}) {$x_1$};
\node[right] at ({1.5*cos(72)},{1.5*sin(72)}) {$x_2$};
\node[left] at ({1.5*cos(144)},{1.5*sin(144)}) {$x_3$};
\node[left] at ({1.5*cos(-144)},{1.5*sin(-144)}) {$x_4$};
\node[right] at ({1.5*cos(-72)},{1.5*sin(-72)}) {$x_5$};
\end{tikzpicture}
\caption{A pentagonal configuration in the $0$--skeleton of a $\CAT$ square complex.}
\label{pentagonal fig} 
\end{figure}
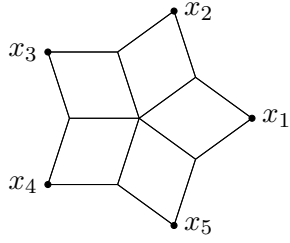

\begin{lem}\label{iterated medians in rank 2}
Suppose that $\rk M\leq 2$. Consider $x\in M$ with $x=m(m(a_1,a_2,a_3),b,c)$ for points $a_i,b,c\in M$. Then one of the following 
happens:
\begin{itemize}
\item there exists $1\leq i\leq 3$ such that $x=m(a_i,b,c)$;
\item there exist $1\leq i<j\leq 3$ such that either $x=m(a_i,a_j,b)$ or $x=m(a_i,a_j,c)$;
\item we have $x=m(a_1,a_2,a_3)$;
\item the points $a_1,a_2,a_3,b,c$ can be ordered to form a pentagonal configuration.
\end{itemize}
\end{lem}
\begin{proof}
Set $n=m(a_1,a_2,a_3)$. Consider the projections $\overline a_i=m(a_i,b,c)$ to the interval $I(b,c)$. Since gate-projections are median morphisms, we have $x=m(\overline a_1,\overline a_2,\overline a_3)$. 

\smallskip
{\bf Claim~1:} \emph{if we are not in the 1st or 3rd case, we can assume that the four sets $\mscr{W}(x|\overline a_1)$, $\mscr{W}(x|\overline a_2)$, $\mscr{W}(x|\overline a_3)$, $\mscr{W}(a_1,a_2|b,c)$ are all nonempty, and that $\mscr{W}(a_1,c|a_2,b)=\emptyset$.}

\smallskip\noindent
\emph{Proof of Claim~1.} If one of the sets $\mscr{W}(x|\overline a_i)$ is empty, then $x=\overline a_i$ and we are in the 1st case. On the other hand, if the sets $\mscr{W}(a_i,a_j|b,c)$ are all empty for $i\neq j$, then we are in the 3rd case. Indeed, since $\mscr{W}(n|b,c)\cu\bigcup_{i<j}\mscr{W}(a_i,a_j|b,c)$, we have $n\in I(b,c)$, hence $x=m(n,b,c)=n=m(a_1,a_2,a_3)$.

Thus, up to permuting the $a_i$, we can assume that $\mscr{W}(a_1,a_2|b,c)\neq\emptyset$. Since this is transverse to the transverse sets $\mscr{W}(a_1,b|a_2,c)$ and $\mscr{W}(a_1,c|a_2,b)$, one of the latter must be empty. Swapping $b$ and $c$ if necessary, we can assume that it is $\mscr{W}(a_1,c|a_2,b)$.
\hfill$\blacksquare$

\smallskip
{\bf Claim~2:} \emph{if we are not in the 4th case either, we can further assume that $\mscr{W}(a_1,a_2,b|a_3,c)=\emptyset$.}

\smallskip\noindent
\emph{Proof of Claim~2.}
Note that the assumptions in Claim~1 are left unchanged if we simultaneously swap $b\leftrightarrow c$ and $a_1\leftrightarrow a_2$. Thus, it suffices to show that we can suppose that at least one of the two sets $\mscr{W}(b,a_1,a_2|c,a_3)$ and $\mscr{W}(a_1,a_2,c|a_3,b)$ is empty.

In order to do so, we assume that $\mscr{W}(b,a_1,a_2|c,a_3)$ and $\mscr{W}(a_1,a_2,c|a_3,b)$ are both nonempty and show that $(b,a_1,a_2,c,a_3)$ is a pentagonal configuration. This places us in the 4th case. 

Since $\mscr{W}(a_1,c|a_2,b)=\emptyset$ and $x=m(\overline a_1,\overline a_2,\overline a_3)$, where $\overline a_i$ is the projection of $a_i$ to $I(b,c)$, we have:
\begin{align*}
\mscr{W}(a_2,c,a_3|b,a_1)&\supseteq\mscr{W}(\overline a_2,\overline a_3|\overline a_1)=\mscr{W}(x|\overline a_1)\neq\emptyset, \\  
\mscr{W}(a_3,b,a_1|a_2,c)&\supseteq\mscr{W}(\overline a_3,\overline a_1|\overline a_2)=\mscr{W}(x|\overline a_2)\neq\emptyset.
\end{align*}
Moreover, since $\mscr{W}(a_1,a_3|b,c)$ is transverse to the nonempty transverse subsets $\mscr{W}(b,a_1,a_2|c,a_3)$ and $\mscr{W}(a_1,a_2,c|a_3,b)$, we have $\mscr{W}(a_1,a_3|b,c)=\emptyset$. 
Hence $\mscr{W}(c,a_3,b|a_1,a_2)=\mscr{W}(c,b|a_1,a_2)\neq\emptyset$.
\hfill$\blacksquare$

\smallskip
{\bf Claim~3:} \emph{under these assumptions, we have $\mscr{W}(x|m(a_1,a_3,c))=\mscr{W}(b,c|a_1,a_3)$.}

\smallskip\noindent
\emph{Proof of Claim~3.}
By the properties of gate-projections, the set $\mscr{W}(b|c)$ does not intersect any of the sets $\mscr{W}(a_i|\overline a_i)$. Thus, since $x=m(\overline a_1,\overline a_2,\overline a_3)$, we must have:
\begin{align*}
\mscr{W}(x|m(a_1,a_3,c))\cap\mscr{W}(b|c)&=\mscr{W}(m(a_1,a_2,a_3)|m(a_1,a_3,c))\cap\mscr{W}(b|c) \\
&=\mscr{W}(a_1|a_3)\cap\mscr{W}(a_2|c)\cap\mscr{W}(b|c) \\
&=\mscr{W}(a_1,a_2,b|a_3,c)\sqcup\mscr{W}(a_2,a_3,b|a_1,c)=\emptyset,
\end{align*}
where we have used Claims~1 and~2 at the very end. Since $x\in I(b,c)$, we have $\mscr{W}(x|b,c)=\emptyset$. Thus:
\[\mscr{W}(x|m(a_1,a_3,c))=\mscr{W}(x,b,c|m(a_1,a_3,c))=\mscr{W}(b,c|a_1,a_3). \] 
\hfill$\blacksquare$

\smallskip
In order to conclude the proof of the lemma, suppose for the sake of contradiction that we are not in the 2nd case, in addition to the assumptions of the claims. Then, Claim~3 implies:
\[\emptyset\neq\mscr{W}(x|m(a_1,a_3,c))=\mscr{W}(b,c|a_1,a_3).\]
On the other hand, since $\mscr{W}(a_1,c|a_2,b)$ and $\mscr{W}(a_1,a_2,b|a_3,c)$ are empty by Claims~1 and~2:
\begin{align*}
\emptyset\neq\mscr{W}(x|\overline a_1)&=\mscr{W}(\overline a_2,\overline a_3|\overline a_1)=\mscr{W}(c,a_2,a_3|b,a_1)\cu\mscr{W}(c,a_3|b,a_1), \\
\emptyset\neq\mscr{W}(x|\overline a_3)&=\mscr{W}(\overline a_1,\overline a_2|\overline a_3)=\mscr{W}(a_1,a_2,c|a_3,b)\cu\mscr{W}(a_1,c|a_3,b).
\end{align*}
Since the three sets $\mscr{W}(b,c|a_1,a_3),\mscr{W}(c,a_3|b,a_1),\mscr{W}(a_1,c|a_3,b)$ are pairwise transverse, this violates the assumption that $\rk M\leq 2$. This proves the lemma.
\end{proof}

\begin{cor}\label{rank 2}
If $T_1,T_2$ are rank--$1$ median algebras, then $\langle A\rangle=\mc{M}(A)$ for all $A\cu T_1\x T_2$.
\end{cor}
\begin{proof}
The product $T_1\x T_2$ does not contain any pentagonal configurations. Otherwise, there would be walls $\mf{w}_1,\mf{w}_2,\mf{w}_3,\mf{w}_4,\mf{w}_5$ with each $\mf{w}_i$ transverse to $\mf{w}_{i+1}$. If $\mf{w}_1$ originates from the factor $T_1$, say, then $\mf{w}_2$ must originate from $T_2$ and, continuing like this, $\mf{w}_5$ again originates from $T_1$. Since $\mf{w}_5$ and $\mf{w}_1$ are transverse, this would contradict the fact that $\rk T_1=1$.

Thus, the 4th case of Lemma~\ref{iterated medians in rank 2} never occurs, hence $\mc{M}^2(A)=\mc{M}(A)$ for all $A\cu T_1\x T_2$. 
\end{proof}

For the next result, let us consider the following functions $f,g,h\colon\N\ra\N$:
\begin{align*} 
f(n)&=2^{2^n}, & g(n)&=1+f\left(\tfrac{n(n-1)}{2}\right), & h(n)&=ng(n)+n.
\end{align*}

\begin{prop}\label{generating median algebras}
Given a median algebra $M$ and a subset $A\cu M$, the following hold.
\begin{enumerate}
\item If $\#A\leq n$, then $\langle A\rangle=\mc{M}^{f(n)}(A)$.
\item If $M$ can be embedded in a product of $d$ rank--$1$ median algebras, then $\langle A\rangle=\mc{M}^{g(d)}(A)$.
\item If $\rk M\leq r$, then $\langle A\rangle=\mc{M}^{h(r)}(A)$.
\end{enumerate}
\end{prop}
\begin{proof}
Part~(1) is immediate from most constructions of the free median algebra on the set $A$; for instance, see \cite[Lemma~4.2]{Bow-cm} and the 
subsequent paragraph.
% for instance: if A has n elements, there are at most (2^n-2) possible halfspaces for <A>, hence at most f(n) ultrafilters.

Regarding part~(2), let us fix an injective median morphism $M\hookrightarrow T_1\x\dots\x T_d$, where the $T_i$ have rank $1$. Let $\pi_{ij}\colon M\ra T_i\x T_j$ denote the composition with the projection to $T_i\x T_j$. Given $x\in M$, Lemma~\ref{van de Vel} implies that 
$x\in\langle A\rangle$ if and only if $\pi_{ij}(x)\in\langle\pi_{ij}(A)\rangle$ for all $1\leq i<j\leq d$. 

Since each $\pi_{ij}$ is a median morphism, Corollary~\ref{rank 2} shows that:
\[\langle\pi_{ij}(A)\rangle=\mc{M}(\pi_{ij}(A))=\pi_{ij}(\mc{M}(A)).\] 
Thus, given $x\in\langle A\rangle$, there exist points $m_{ij}\in\mc{M}(A)$ such that $\pi_{ij}(x)=\pi_{ij}(m_{ij})$. It follows that:
\[x\in\langle\{m_{ij} \mid 1\leq i<j\leq d\}\rangle.\] 
Since there are at most $\tfrac{d(d-1)}{2}$ distinct points $m_{ij}$, part~(1) yields: 
\[\langle\{m_{ij} \mid 1\leq i<j\leq d\}\rangle=\mc{M}^{g(d)-1}(\{m_{ij} \mid 1\leq i<j\leq d\})\cu\mc{M}^{g(d)-1}(\mc{M}(A))=\mc{M}^{g(d)}(A).\]
Hence $\langle A\rangle\cu\mc{M}^{g(d)}(A)$.

Finally, let us prove part~(3). Since $\rk\langle A\rangle\leq\rk M$, we can safely assume that $M=\langle A\rangle$. Consider two points $a,b\in M$ and recall that the gate-projection $\pi_{ab}\colon M\ra I(a,b)$ is given by $\pi_{ab}(x)=m(a,b,x)$. By Dilworth's lemma, the interval $I(a,b)\cu M$ can be embedded in a product of $r$ rank--$1$ median algebras for all $a,b\in M$ (cf.\ \cite[Proposition~1.4]{Bow2}). 

If $B\cu M$ is a subset with $\langle B\rangle=M$ and $a,b\in B$, then part~(2) yields:
\[I(a,b)=\pi_{ab}(M)=\pi_{ab}(\langle 
B\rangle)=\langle\pi_{ab}(B)\rangle=\mc{M}^{g(r)}(\pi_{ab}(B))=\pi_{ab}(\mc{M}^{g(r)}(B))\cu\mc{M}^{g(r)+1}(B).\]
It follows that $\mc{J}(B)\cu\mc{M}^{g(r)+1}(B)$ for every subset $B\cu M$ with $\langle B\rangle=M$. Observing that:
\[\mc{J}^{k+1}(B)=\mc{J}(\mc{J}^k(B))\cu\mc{M}^{g(r)+1}(\mc{J}^k(B)),\] 
we inductively obtain $\mc{J}^m(B)\cu\mc{M}^{m(g(r)+1)}(B)$ for all $m\geq 1$. In particular, by Remark~\ref{Bow J}:
\[\langle A\rangle\cu\hull A=\mc{J}^r(A)\cu\mc{M}^{r(g(r)+1)}(A)=\mc{M}^{h(r)}(A).\]
This concludes the proof of the proposition.
\end{proof}

\begin{rmk}
The bounds provided by Proposition~\ref{generating median algebras} are highly non-sharp. For instance, if $\rk M\leq 2$, a slightly more careful use of Lemma~\ref{iterated medians in rank 2} would show that $\langle A\rangle=\mc{M}^2(A)$ for every $A\cu M$, while the proposition only gives $\langle A\rangle=\mc{M}^{244}(A)$. For the purposes of this paper, we only care that such bounds exist and only depend on the rank of $M$.
\end{rmk}

\subsection{Approximate subalgebras of cubulated groups.}\label{fg sect}

This subsection is devoted to the proof of Theorem~\ref{cmp und intro} and to a few examples of how this result can fail for automorphisms that do not preserve the coarse median structure (Example~\ref{mother of all evil}). Towards the end, we use similar techniques to prove Corollaries~\ref{cmp fi overgroups} and~\ref{Nielsen intro}.

Our main focus will be the following result. Recall that, if $\varphi$ is a coarse-median preserving automorphism of a cocompactly cubulated group, Lemma~\ref{fix approximate subalgebra} guarantees that the subgroup $\Fix\varphi$ is an approximate median subalgebra. Thus, Theorem~\ref{cmp und intro} immediately follows from:

\begin{thm}\label{AMS thm}
Let $G\acts X$ be a proper cocompact action on a $\CAT$ cube complex. Let $[\mu_X]$ be the induced coarse median structure on $G$. If a subgroup $H\leq G$ is an approximate median subalgebra of $(G,[\mu_X])$, then:
\begin{enumerate}
\item $H$ is finitely generated and undistorted in $G$;
\item $H$ admits a proper cocompact action on a $\CAT$ cube complex.
\end{enumerate}
\end{thm}

As a first step, we need to show that the subgroup $H$ in Theorem~\ref{AMS thm} is finitely generated. The proof of this is a straightforward adaptation of an argument due to Cooper and Paulin for fixed subgroups of automorphisms of hyperbolic groups \cite{Cooper-fix,Paulin-Fourier}.

\begin{prop}\label{AMS fg}
Let $G\acts X$ be a proper cocompact action on a $\CAT$ cube complex. If a subgroup $H\leq G$ is an approximate median subalgebra of $(G,[\mu_X])$, then $H$ is finitely generated.
\end{prop}
\begin{proof}
Fix a base vertex $p\in X$. The main observation is the following.

\smallskip
{\bf Claim:} \emph{If $x_n\in H\cdot p$ is a diverging sequence, then there exists an element $h\in H$ such that $d(p,hx_n)<d(p,x_n)$ holds for infinitely many values of $n$.}

\smallskip\noindent
\emph{Proof of Claim.}
Since $H$ is an approximate median subalgebra of $(G,[\mu_X])$, there exists $L\geq 0$ such that all medians of points in $H\cdot p$ lie in the $L$--neighbourhood of $H\cdot p$ in $X$.

Passing to a subsequence, we can assume that the vertices $x_n$ converge to a point in the Roller boundary $\xi\in\partial X$. Recalling that $X^{(0)}$ is discrete in the Roller compactification and that the median map is continuous, there exist integers $M(n)\geq 0$ such that, for every $m\geq M(n)$, we have:
\[m(p,x_n,\xi)=m(p,x_n,x_m).\]
In particular, there exist elements $h_n\in H$ such that $h_np$ is $L$--close to $m(p,x_n,\xi)$. 

Now, since $x_n\ra\xi$, the medians $m(p,x_n,\xi)$ diverge, and so do the points $h_np$. In particular, there exist indices $i<j$ such that:
\[d(p,h_ip)+2L<d(p,h_jp).\] 
Since, for $m\geq M(j)$, the point $h_jp$ is $L$--close to the median $m(p,x_j,x_m)$, we also have:
\[d(p,h_jp)+d(h_jp,x_m)\leq d(p,x_m)+2L.\]

In conclusion, setting $h:=h_ih_j^{-1}$, we obtain, for all $m\geq M(j)$:
\begin{align*}
d(p,hx_m)=d(p,h_ih_j^{-1}x_m)&\leq d(p,h_ip)+d(h_ip,h_ih_j^{-1}x_m)=d(p,h_ip)+d(h_jp,x_m) \\
&\leq d(p,h_ip)+d(p,x_m)-d(p,h_jp)+2L \\
&<d(p,x_m).
\end{align*}
\hfill$\blacksquare$

\smallskip
Now, suppose for the sake of contradiction that $H$ is not finitely generated. Write $H$ as the union of an infinite ascending chain of subgroups $H_1\lneq H_2\lneq\dots$, where $H_{n+1}=\langle H_n,h_{n+1}\rangle$ for some $h_{n+1}\in H$. Possibly replacing $h_{n+1}$, we can assume that the point $x_{n+1}:=h_{n+1}p$ minimises the distance to $p$ within the set $H_nh_{n+1}\cdot p$.

The Claim provides an element $h\in H$ such that $d(p,hx_n)<d(p,x_n)$ occurs infinitely often. Since $h\in H$, there exists $N\geq 0$ such that $h\in H_n$ for all $n\geq N$. This contradicts the fact that $x_{n+1}$ minimises the distance to $p$ within $H_n\cdot x_{n+1}$ for $n\geq N$.
\end{proof}
% Define "almost-median" spaces as those coarse median spaces $(X,\mu)$ for which there is a constant L such that, for all $x,y,z\in X$, the sum $d(x,\mu(x,y,z))+d(\mu(x,y,z),y)$ is at most $d(x,y)+L$. The main examples are CAT(0) cube complexes and Gromov-hyperbolic spaces.
% The proof of the proposition shows, more generally, the following. If a group $H$ acts properly on a *proper* almost-median space $X$ and an orbit is an approximate median subalgebra, then $H$ is finitely generated. 
% We don't need a notion of Roller boundary, just the observation that the image of each map $\mu(x,y,\cdot)$ is finite, and so we can pass to a subsequence of $x_n$ to ensure that $\mu(p,x_n,x_m)$ is constant for $m>n$ if we fix $n$.
% It would be interesting to prove the same result when $X$ is *not proper*. Indeed, by \cite{Hagen-Petyt}[Theorem~3.1], every HHG with a BBF colouring acts properly on a product of hyperbolic spaces with an orbit that is an approximate median subalgebra. So it would follow that every AMS subgroup of such an HHG is finitely generated. 
% Unfortunately, without properness, the (only) problem seems to be the following: how do we show that the sequence $\mu(p,x_n,x_m)$ diverges as $n<m$ diverge? It could be that all medians $\mu(p,x_n,x_m)$ with $n<m$ are the same point. (Especially since $H$--orbits might well not be coarsely connected a priori.)

Along with Propositions~\ref{approx subalgebras} and~\ref{AMS fg}, the following is the only missing ingredient in the proof of Theorem~\ref{AMS thm}. We refer the reader to the proof sketch in the Introduction.

\begin{lem}\label{cubical undistortion criterion}
Let $G\acts X$ be a proper cocompact action on a $\CAT$ cube complex. Consider a subgroup $H\leq G$. Suppose that there exists an $H$--invariant median subalgebra $M\cu X^{(0)}$ such that the action $H\acts M$ is cofinite. Then:
\begin{enumerate}
\item $H$ is finitely generated and undistorted in $G$; 
\item $H$ admits a proper cocompact action on a $\CAT$ cube complex.
\end{enumerate}
\end{lem}
\begin{proof}
Halfspaces and hyperplanes of the cube complex $X$, as usually defined, correspond exactly to halfspaces and hyperplanes of the median algebra $X^{(0)}$. As customary, we write $\mscr{H}(X)$ and $\mscr{W}(X)$ with the meaning of $\mscr{H}(X^{(0)})$ and $\mscr{W}(X^{(0)})$. By Remark~\ref{halfspaces of subsets}(1), we have a natural surjection ${\rm res}_M\colon\mscr{H}_M(X)\ra\mscr{H}(M)$.

Since $H$ acts cofinitely on the subalgebra $M$, it is an approximate median subalgebra of $(G,[\mu_X])$ and Proposition~\ref{AMS fg} implies that $H$ is finitely generated. Thus, every $H$--orbit in $X$ is coarsely connected and, since $H\acts M$ is cofinite, $M$ is coarsely connected as well. It follows that there exists a uniform upper bound $m$ to the cardinality of the fibres of the map ${\rm res}_M$.

As in \cite[Section~10]{Roller}, \cite[Theorem~6.1]{Chepoi}, we can construct a $\CAT$ cube complex $X(M)$ such that $M$ is naturally isomorphic to the median algebra $X(M)^{(0)}$. Given $x,y\in M$, let us denote by $d(x,y)$ and $d_M(x,y)$ their distance in the $1$--skeleta of $X$ and $X(M)$, respectively. 

By construction, $d_M(x,y)$ coincides with the number of walls of $M$ separating $x$ and $y$. It follows from the above discussion on ${\res_M}$ that:
\[d_M(x,y)\leq d(x,y)\leq m\cdot d_M(x,y)\]
for all $x,y\in M$. Thus, the identification between $X(M)^{(0)}$ and $M\cu X^{(0)}$ gives a quasi-isometric embedding $X(M)\ra X$ that is equivariant with respect to the inclusion $H\hookrightarrow G$.

The action $H\acts (M,d_M)$ is cofinite by assumption, and it follows from the above inequalities that it is also proper. This shows that the induced action $H\acts X(M)$ is proper and cocompact, proving part~(2). The Milnor-Schwarz Lemma now implies that the inclusion $H\hookrightarrow G$ is a quasi-isometric embedding, which proves part~(1).
\end{proof}

\begin{proof}[Proof of Theorem~\ref{AMS thm}]
For any vertex $p\in X$, the orbit $H\cdot p$ is an approximate median subalgebra of $X$. By Proposition~\ref{approx subalgebras}, the subalgebra $M:=\langle H\cdot p\rangle$ is at finite Hausdorff distance from $H\cdot p$. Since $X$ is locally finite, it follows that the action $H\acts M$ is cofinite, hence Lemma~\ref{cubical undistortion criterion} shows that $H$ is finitely generated, undistorted and cocompactly cubulated.
\end{proof}

As discussed above, this completes the proof of Theorem~\ref{cmp und intro}. The next example shows that, even for automorphisms of RAAGs, all of the claims in the statement of Theorem~\ref{cmp und intro} can fail if the automorphism does not preserve the coarse median structure.

\begin{ex}\label{mother of all evil}
Here is a recipe to construct automorphisms with unpleasant fixed subgroups. Consider a group $G$ and a homomorphism $\alpha\colon G\ra\Z$. These data define an automorphism $\psi\in\aut(G\x\Z)$ by the formula:
\[\psi(g,n):=(g,n+\alpha(g)).\]
It is clear that $\Fix\psi=\ker\alpha\x\Z$.

Now, consider the situation where $G$ is a right-angled Artin group $\A_{\G}$ and $\alpha\colon\A_{\G}\ra\Z$ takes all standard generators to $+1$. The resulting automorphism $\psi\in\aut(\A_{\G}\x\Z)$ is a product of finitely many twists (as defined in Subsection~\ref{of RAAGs sect}) and we have $\Fix\psi=BB_{\G}\x\Z$, where $BB_{\G}$ denotes the Bestvina--Brady subgroup of $\A_{\G}$ \cite{Bestvina-Brady}.

We apply this construction to obtain examples where $\Fix\psi$ fails to have the three properties provided by Theorem~\ref{cmp und intro}.
\begin{enumerate}
\item The subgroup $BB_{\G}$ is finitely generated if and only if $\G$ is connected \cite{Meier-VanWyck}. For instance, there exists $\psi\in\aut(F_2\x\Z)$ such that $\Fix\psi$ is not finitely generated.
\item If $\A_{\G}$ is freely irreducible, directly irreducible and non-cyclic, then $BB_{\G}$ is finitely generated and quadratically distorted \cite[Theorem~1.1]{Tran-distortion}. This gives examples where $\Fix\psi$ is finitely generated, but distorted.
\item As shown in \cite[Main Theorem]{Bestvina-Brady}, the finiteness properties and homological finiteness properties of $BB_{\G}$ are governed by the homology and homotopy groups of the flag simplicial complex $L_{\G}$ determined by $\G$. The same is true of $\Fix\psi=BB_{\G}\x\Z$ \cite{MMVW,Bux-Gonzalez}. In particular, if $L_{\G}$ is not contractible, then $\Fix\psi$ is not of type $F$ (hence not cocompactly cubulated, since there is no torsion).

This can even be achieved while ensuring that $\Fix\varphi$ is undistorted: by \cite[Theorem~1.1]{Tran-distortion}, it suffices to make sure that $\A_{\G}$ splits as a product.
\end{enumerate}

We emphasise that, by embedding $\A_{\G}\x\Z$ as a parabolic subgroup of a larger RAAG and suitably extending the automorphism $\psi$, we can ensure that all these bad behaviours also occur for automorphisms of \emph{irreducible} RAAGs.
\end{ex}

We conclude this subsection by proving Corollaries~\ref{cmp fi overgroups} and~\ref{Nielsen intro}. All that is required is Proposition~\ref{cmp prop intro}, Proposition~\ref{approx subalgebras} and (part of) Lemma~\ref{cubical undistortion criterion}.

\begin{proof}[Proof of Corollary~\ref{cmp fi overgroups}]
Let $H\leq G$ be a finite-index subgroup with a proper cocompact action on a $\CAT$ cube complex $H\acts X$. Replacing $H$ with a finite-index subgroup, it is not restrictive to suppose that $H\lhd G$. By our assumption, the conjugation action $G\acts H$ preserves the coarse median structure $[\mu]$ induced on $H$ by $H\acts X$.

It is well-known that the concept of induced representation can be generalised to actions on metric spaces (see e.g.\ \cite[Subsection~2.1]{Bridson-rhombic} or \cite[Subsection~2.2]{BBF-QT}). In our context, this yields a proper action $G\acts X_1\x\dots\x X_n$, where $n$ is the index of $H$ in $G$ and each $X_i$ is isomorphic to $X$. Since $H$ is normal, each factor is preserved by $H$ and each action $H\acts X_i$ can be made equivariantly isomorphic to $H\acts X$ by twisting it by an automorphism of $H$ corresponding to a conjugation by an element of $G$.

Since $[\mu]$ is preserved by the conjugation action $G\acts H$, it is the coarse median structure induced by all the cubulations $H\acts X_i$. This implies that, for every finite subset $A\cu X^{(0)}$, the orbit $H\cdot A$ is an approximate median subalgebra. Since $H$ is normal, we can choose a finite subset $A\cu X^{(0)}$ so that $H\cdot A$ is $G$--invariant. Proposition~\ref{approx subalgebras} guarantees that the $G$--invariant median subalgebra $M:=\langle H\cdot A\rangle$ is at finite Hausdorff distance from $H\cdot A$. Since each $X_i$ is locally finite, this implies that the action $G\acts M$ is cofinite.

Since $M$ is a discrete median algebra, the natural $\CAT$ cube complex $X(M)$ with $M$ as its $0$--skeleton (as in \cite[Section~10]{Roller} or \cite[Theorem~6.1]{Chepoi}) gives the required cocompact cubulation of $G$. Here properness of the $G$--action on $X(M)$ can be checked as in the proof of Lemma~\ref{cubical undistortion criterion}, using the fact that $M$ is coarsely connected to conclude that $X(M)$ and $\prod X_i$ induce bi-Lipschitz equivalent metrics on $M$.
% we know the $G$--action has finite vertex-stabiliser, but a priori the cube complex associated to $M$ might not be locally finite; so we need the fact that $d_M$ and $d_X$ are bilipschitz equivalent
\end{proof}

\begin{proof}[Proof of Corollary~\ref{Nielsen intro}]
Let $G=\A_{\G}$ or $G=\W_{\G}$. Consider a finite subgroup $F\leq\out G$ as in the statement. Let $\pi\colon\aut G\ra\out G$ be the quotient projection. Our goal is to construct a proper, cocompact action on a $\CAT$ cube complex $\pi^{-1}(F)\acts Y$. We can then take $Q$ to be the quotient of $Y$ by the finite-index normal subgroup $G\simeq\ker\pi\lhd\pi^{-1}(F)$.

Let $G\acts X$ be the standard action on the universal cover of the Salvetti/Davis complex. In both cases, Proposition~\ref{cmp prop intro} shows that $F$ preserves the coarse median structure on $G$ induced by this action. Thus, Corollary~\ref{cmp fi overgroups} provides the required proper cocompact action $\pi^{-1}(F)\acts Y$.
\end{proof}

\subsection{Staircases in cube complexes.}\label{staircases sect}

In the rest of Section~\ref{fix cmp sect}, our goal is to obtain a \emph{quasi-convexity} criterion for median subalgebras of cube complexes, which will lead to the proof of Theorem~\ref{U_0 cc intro}. Ultimately, we will restrict to universal covers of Davis/Salvetti complexes for right-angled groups and an important point will be that they do not admit infinite \emph{staircases}.

In this subsection, we study staircases in general $\CAT$ cube complexes.

\begin{defn}\label{staircase defn}
Let $M$ be a median algebra.
\begin{enumerate}
\item A \emph{length--$n$ staircase} in $M$ is the data of two chains of halfspaces $\mf{h}_1\supsetneq\dots\supsetneq\mf{h}_n$ and $\mf{k}_1\supsetneq\dots\supsetneq\mf{k}_n$ such that $\mf{h}_i$ is transverse to $\mf{k}_j$ for $j\leq i$, while $\mf{k}_{i+1}\subsetneq\mf{h}_i$. 
\item The \emph{staircase length} of $M$ is the supremum of $n\in\N$ such that $M$ has a length--$n$ staircase.
\end{enumerate}
Figure~\ref{staircase} depicts part of a staircase of length $\geq 5$.
\end{defn}

When speaking of staircases in relation to a $\CAT$ cube complex $X$, we always refer to the median algebra $M=X^{(0)}$. Note that the above notion of staircase seems to be a bit more general than the one in \cite[p.~51]{Hagen-Susse}: given hyperplanes bounding halfspaces as in Definition~\ref{staircase defn}, there might not be a convex subcomplex of $X$ with exactly these hyperplanes.

\begin{figure} 
\begin{tikzpicture}[smooth]
\draw[cyan,thick] plot[tension=0.3] coordinates{ (-1.5,-2) (-1.5,-1.5) (-2.5,-1.5)};
\draw[cyan,thick] plot[tension=0.3] coordinates{ (-0.5,-2) (-0.5,-0.5) (-2.5,-0.5)};
\draw[cyan,thick] plot[tension=0.3] coordinates{ (0.5,-2) (0.5,0.5) (-2.5,0.5)};
\draw[cyan,thick] plot[tension=0.3] coordinates{ (1.5,-2) (1.5,1.5) (-2.5,1.5)};
\draw[fill,magenta,thick] (-2,-2) -- (-2,2);
\draw[fill,magenta,thick] (-1,-2) -- (-1,2);
\draw[fill,magenta,thick] (0,-2) -- (0,2);
\draw[fill,magenta,thick] (1,-2) -- (1,2);
\draw[fill,magenta,thick] (2,-2) -- (2,2);
\draw [->,cyan,thick] (-2.25,-1.5) -- (-2.25, -1.25);
\draw [->,cyan,thick] (-2.25,-0.5) -- (-2.25, -0.25);
\draw [->,cyan,thick] (-2.25,0.5) -- (-2.25, 0.75);
\draw [->,cyan,thick] (-2.25,1.5) -- (-2.25, 1.75);
\draw [->,magenta,thick] (-2,1.75) -- (-1.75, 1.75);
\draw [->,magenta,thick] (-1,1.75) -- (-0.75, 1.75);
\draw [->,magenta,thick] (0,1.75) -- (0.25, 1.75);
\draw [->,magenta,thick] (1,1.75) -- (1.25, 1.75);
\draw [->,magenta,thick] (2,1.75) -- (2.25, 1.75);
\node[left,cyan] at (-2.25,1.75) {$\mf{h}_4$};
\node[left,cyan] at (-2.25,0.75) {$\mf{h}_3$};
\node[left,cyan] at (-2.25,-0.25) {$\mf{h}_2$};
\node[left,cyan] at (-2.25,-1.25) {$\mf{h}_1$};
\node[above,magenta] at (-1.75, 1.75) {$\mf{k}_1$};
\node[above,magenta] at (-0.75, 1.75) {$\mf{k}_2$};
\node[above,magenta] at (0.25, 1.75) {$\mf{k}_3$};
\node[above,magenta] at (1.25, 1.75) {$\mf{k}_4$};
\node[above,magenta] at (2.25, 1.75) {$\mf{k}_5$};
\end{tikzpicture}
\caption{}
\label{staircase} 
\end{figure}

In view of the following discussion, it is convenient to introduce a notation for gate-projections to intervals. Given a median algebra $M$ and points $x,y\in M$, we denote by $\pi_{xy}\colon M\ra I(x,y)$ the map $\pi_{xy}(z)=m(x,y,z)$.

\begin{lem}\label{no staircases +}
Let $M$ be a median algebra of rank $r$ and staircase length $d$. If there exist halfspaces $\mf{k}_1\supsetneq\dots\supsetneq\mf{k}_n$ and points $x,y\in\mf{k}_1^*$ such that $\pi_{xy}(\mf{k}_1)\supsetneq\dots\supsetneq\pi_{xy}(\mf{k}_n)$, then $n\leq 2rd$.
\end{lem}
\begin{proof}
The sets $C_i:=\pi_{xy}(\mf{k}_i)$ are convex, for instance by \cite[Lemma~2.2(1)]{Fio1}. Since $C_{i+1}\subsetneq C_i$, there exist halfspaces $\mf{h}_i\in\mscr{H}(M)$ such that $\mf{h}_i\in\mscr{H}_{C_i}(M)$ and $C_{i+1}\cu\mf{h}_i$. 

Since both $\mf{h}_i$ and $\mf{h}_i^*$ intersect $C_i\cu I(x,y)$, we have $\mf{h}_i\in\mscr{H}(x|y)\sqcup\mscr{H}(y|x)$ for all $i$. Possibly swapping $x$ and $y$, we can assume that at least $n/2$ of the $\mf{h}_i$ lie in $\mscr{H}(x|y)$. By Dilworth's lemma, there exist $k\geq n/2r$ and indices $i_1<\dots<i_k$ such that $\mf{h}_{i_1},\dots,\mf{h}_{i_k}$ lie in $\mscr{H}(x|y)$ and no two of them are transverse. Up to re-indexing, we can assume that these are $\mf{h}_1,\dots,\mf{h}_k$. 

Since $C_j$ is contained in $\mf{h}_i$ if and only if $j>i$, we must have $\mf{h}_1\supsetneq\dots\supsetneq\mf{h}_k$. Note that $y\in\mf{h}_i\cap\mf{k}_j^*$ and $x\in\mf{h}_i^*\cap\mf{k}_j^*$ for all $i,j$. If $j\leq i$, we have $\mf{h}_i\in\mscr{H}_{C_j}(X)$, hence $\mf{h}_i\cap\mf{k}_j$ and $\mf{h}_i^*\cap\mf{k}_j$ are both nonempty. This shows that $\mf{h}_i$ and $\mf{k}_j$ are transverse for $j\leq i$, while the fact that $C_{i+1}\cu\mf{h}_i$ implies that $\mf{k}_{i+1}\cu\mf{h}_i$. In conclusion, the $\mf{h}_i$ and $\mf{k}_j$ form a length--$k$ staircase with $k\geq n/2r$. Since $M$ has staircase length $d$, we have $n\leq 2rk\leq 2rd$.
\end{proof}

\begin{lem}\label{preservation of finite staircase length}
Let $X$ be a $\CAT$ cube complex of dimension $r$ and staircase length $d$. Consider vertices $x,y\in X$ and $z\in I(x,y)$. Let $\alpha\cu I(x,z)$ be a (combinatorial) geodesic from $x$ to $z$. Then the median subalgebra $M=X^{(0)}\cap I(x,y)\cap\pi_{xz}^{-1}(\alpha)$ has staircase length $\leq d(1+2r^2)$.
\end{lem}
\begin{proof}
Since $\pi_{xz}(y)=z$ and $x,z\in\alpha$, the three points $x,y,z$ all lie in $M$. Since $M\cu I(x,y)$, every wall of $M$ separates $x$ and $y$. Recall that we use the notation $\mscr{H}(X)$ and $\mscr{W}(X)$ with the meaning of $\mscr{H}(X^{(0)})$ and $\mscr{W}(X^{(0)})$. 

\smallskip
{\bf Claim~1:} \emph{if $\mf{u},\mf{v}\in\mscr{W}(M)$ separate $x$ and $z$, then $\mf{u}$ and $\mf{v}$ are not transverse.}

\smallskip\noindent
\emph{Proof of Claim~1.}
Pick halfspaces $\hat{\mf{h}},\hat{\mf{k}}\in\mscr{H}(X)\cap\mscr{H}(x|z)$ such that $\mf{h}:=\hat{\mf{h}}\cap M\in\mscr{H}(M)$ is bounded by $\mf{u}$ and $\mf{k}:=\hat{\mf{k}}\cap M$ is bounded by $\mf{v}$; this is possible by Remark~\ref{halfspaces of subsets}(1). The intersections $\hat{\mf{h}}\cap\alpha$ and $\hat{\mf{k}}\cap\alpha$ are subsegments of $\alpha$ containing $z$. Without loss of generality, we have $\hat{\mf{h}}\cap\alpha\cu\hat{\mf{k}}\cap\alpha$. Then $\hat{\mf{h}}\cap\hat{\mf{k}}^*\cap\alpha=\emptyset$, hence $\emptyset=\hat{\mf{h}}\cap\hat{\mf{k}}^*\cap M=\mf{h}\cap\mf{k}^*$, proving the claim.
\hfill$\blacksquare$

\smallskip
{\bf Claim~2:} \emph{if $\hat{\mf{h}},\hat{\mf{k}}\in\mscr{H}(z|y)$ are halfspaces of $X$, then $\hat{\mf{h}}$ and $\hat{\mf{k}}$ are transverse if and only if $\hat{\mf{h}}\cap M$ and $\hat{\mf{k}}\cap M$ are transverse halfspaces of $M$.}

\smallskip\noindent
\emph{Proof of Claim~2.}
The vertex set of the interval $I(z,y)\cu X$ is entirely contained in $M$, since $\pi_{xz}(I(z,y))=\{z\}$. Thus, $I(z,y)$ is a convex subset of both $X$ and $M$. Remark~\ref{halfspaces of subsets}(2) then shows that $\hat{\mf{h}}$ and $\hat{\mf{k}}$ are transverse if and only if $\hat{\mf{h}}\cap I(z,y)$ and $\hat{\mf{k}}\cap I(z,y)$ are transverse, if and only if $\hat{\mf{h}}\cap M$ and $\hat{\mf{k}}\cap M$ are transverse.
\hfill$\blacksquare$

\smallskip
Now, suppose that $M$ contains a length--$n$ staircase. Thus $M$ has halfspaces $\mf{h}_1\supsetneq\dots\supsetneq\mf{h}_n$ and $\mf{k}_1\supsetneq\dots\supsetneq\mf{k}_n$ such that each $\mf{h}_i$ is transverse to all $\mf{k}_j$ with $j\leq i$, while $\mf{k}_{i+1}\cu\mf{h}_i$. 

Since $\mf{k}_n\cu\mf{h}_{n-1}\cu\mf{h}_1$, we have either $\{\mf{h}_1,\mf{k}_n\}\cu\mscr{H}(x|y)$ or $\{\mf{h}_1,\mf{k}_n\}\cu\mscr{H}(y|x)$. If we replace all $\mf{h}_i$ and $\mf{k}_j$ with $\mf{k}^*_{n-i+1}$ and $\mf{h}^*_{n-j+1}$, respectively, we obtain another length--$n$ staircase. Thus, we can assume that $\{\mf{h}_1,\mf{k}_n\}\cu\mscr{H}(x|y)$. It follows that all $\mf{h}_i$ and $\mf{k}_j$ lie in $\mscr{H}(x|y)$.

Let $0\leq a,b\leq n$ be the largest indices such that $z\in\mf{h}_i$ and $z\in\mf{k}_j$ hold for $1\leq i\leq a$ and $1\leq j\leq b$. Since $\mf{h}_1$ and $\mf{k}_1$ are transverse, Claim~1 shows that they cannot both lie in $\mscr{H}(x|z)$. Thus $\min\{a,b\}=0$. Since $\mf{k}_{a+2}\cu\mf{h}_{a+1}$, we have $z\not\in\mf{k}_{a+2}$, hence $b\leq a+1$. In conclusion, either $b=0$, or $(a,b)=(0,1)$.

The halfspaces $\mf{h}_i,\mf{k}_j$ with $i,j>\max\{a,b\}$ all lie in $\mscr{H}(z|y)$ and form a staircase of length $n-\max\{a,b\}$. By Remark~\ref{halfspaces of subsets}(1) and Claim~2, this determines a staircase of halfspaces of $X$. Since $X$ has staircase length $d$, we deduce that $n-\max\{a,b\}\leq d$. 

If $b=1$ and $a=0$, we get $n\leq d+1$ and we are done. If instead $b=0$, then $n\leq a+d$ and the proof is completed with the following claim.

\smallskip
{\bf Claim~3:} \emph{if $b=0$, then $a\leq 2r^2d$.}

\smallskip\noindent
\emph{Proof of Claim~3.}
As a recap, $M$ has halfspaces $\mf{h}_1\supsetneq\dots\supsetneq\mf{h}_a$ in $\mscr{H}(x|z,y)$ and $\mf{k}_1\supsetneq\dots\supsetneq\mf{k}_a$ in $\mscr{H}(x,z|y)$ forming a length--$a$ staircase. By Remark~\ref{halfspaces of subsets}(1), there exist halfspaces $\hat{\mf{h}}_i,\hat{\mf{k}}_j\in\mscr{H}(X)$ such that $\mf{h}_i=\hat{\mf{h}}_i\cap M$ and $\mf{k}_i=\hat{\mf{k}}_i\cap M$. 

By Dilworth's lemma, there exist $a'\geq a/r$ and indices $1\leq j_1<\dots<j_{a'}\leq a$ such that no two among $\hat{\mf{k}}_{j_1},\dots,\hat{\mf{k}}_{j_{a'}}$ are transverse. Thus, up to reindexing, we can assume that $\hat{\mf{k}}_1\supsetneq\dots\supsetneq\hat{\mf{k}}_{a'}$.

Now, since the $\mf{h}_i$ and $\mf{k}_j$ form a staircase in $M$ and $\hat{\mf{h}}_i\in\mscr{H}(x|z)$, we have, for every $1\leq j\leq a'$:
\begin{itemize}
\item $\emptyset=\mf{h}_j^*\cap\mf{k}_{j+1}=\hat{\mf{h}}_j^*\cap\hat{\mf{k}}_{j+1}\cap M$, hence $\pi_{xz}(\hat{\mf{k}}_{j+1})\cap\hat{\mf{h}}_j^*\cap\alpha=\emptyset$;
\item $\emptyset\neq\mf{h}_j^*\cap\mf{k}_j=\hat{\mf{h}}_j^*\cap\hat{\mf{k}}_j\cap M$, hence $\pi_{xz}(\hat{\mf{k}}_j)\cap\hat{\mf{h}}_j^*\cap\alpha\neq\emptyset$.
\end{itemize}
Note moreover that $x,z\in\hat{\mf{k}}_1^*$. If we had $a'>2rd$, Lemma~\ref{no staircases +} would imply that there exists $j$ with $\pi_{xz}(\hat{\mf{k}}_j)=\pi_{xz}(\hat{\mf{k}}_{j+1})$. However, $\pi_{xz}(\hat{\mf{k}}_j)$ intersects $\hat{\mf{h}}_j^*\cap\alpha$ while $\pi_{xz}(\hat{\mf{k}}_{j+1})$ does not.

We conclude that $a\leq ra'\leq 2r^2d$, as required.
\hfill$\blacksquare$

\smallskip
As discussed before Claim~3, this proves the lemma.
\end{proof}

Recall that, if $\G$ is a finite simplicial graph, $\X_{\G}$ and $\mc{Y}_{\G}$ denote the universal covers, respectively, of the Salvetti complex for $\A_{\G}$ and the Davis complex for $\W_{\G}$.

\begin{lem}\label{no staircases}
The staircase length of $\X_{\G}$ and $\mc{Y}_{\G}$ is at most $\#\G^{(0)}$.
\end{lem}
\begin{proof}
We only run the proof for $\X_{\G}$, since the argument for $\mc{Y}_{\G}$ is identical. The important property, shared by both complexes, is that there is a map $\g\colon\mscr{W}(\X_{\G})\ra\G^{(0)}$ such that, if $\mf{u},\mf{v}$ are hyperplanes with intersecting carriers, then $\mf{u}$ and $\mf{v}$ are transverse if and only if $\g(\mf{u})$ and $\g(\mf{v})$ are joined by an edge of $\G$. For simplicity, let us extend the map $\g$ to $\mscr{H}(\X_{\G})$, simply by composing it with the $2$-to-$1$ map $\mscr{H}(\X_{\G})\ra\mscr{W}(\X_{\G})$ pairing each halfspace with its hyperplane.

Consider halfspaces $\mf{h}_1\supsetneq\dots\supsetneq\mf{h}_n$ and $\mf{k}_1\supsetneq\dots\supsetneq\mf{k}_n$ such that $\mf{h}_i$ is transverse to all $\mf{k}_j$ with $j\leq i$, while $\mf{k}_{i+1}\subsetneq\mf{h}_i$. We define the following subsets of $\G^{(0)}$:
\[\G_j:=\g(\mf{k}_1^*)\cup\g(\mscr{W}(\mf{k}_1^*|\mf{k}_j))\cup\{\g(\mf{k}_j)\}.\]
It is clear that $\G_j\cu\G_{j+1}$ for all $j\geq 1$. The lemma is immediate from the following claim:

\smallskip
{\bf Claim:} \emph{we have $\G_j\subsetneq\G_{j+1}$ for all $j\geq 1$.}

\smallskip\noindent
Suppose for the sake of contradiction that, for some $j\geq 1$, we have $\G_{j+1}=\G_j$.

Given $\mf{j}\in\mscr{H}(\mf{h}_j^*|\mf{k}_{j+1})$, we have $\mf{j}\cap\mf{k}_1\supseteq\mf{k}_{j+1}\neq\emptyset$. Moreover, $\mf{j}^*\cap\mf{k}_1\neq\emptyset$ and $\mf{j}^*\cap\mf{k}_1^*\neq\emptyset$, since $\mf{j}^*$ contains $\mf{h}_j^*$, which is transverse to $\mf{k}_1$. Thus, for each $\mf{j}\in\mscr{H}(\mf{h}_j^*|\mf{k}_{j+1})$, there are only two possibilities:
\begin{enumerate}
\item[(a)] either $\mf{j}\cap\mf{k}_1^*=\emptyset$, hence $\mf{j}\cu\mf{k}_1$ and $\mf{j}\in\mscr{H}(\mf{k}_1^*|\mf{k}_{j+1})$;
\item[(b)] or $\mf{j}$ is transverse to $\mf{k}_1$. 
\end{enumerate}
Note that no halfspace of type~(a) can contain a halfspace of type~(b). Moreover, each $\mf{j}$ of type~(b) is also transverse to $\mf{k}_j$: we have $\mf{j}\cap\mf{k}_j\supseteq\mf{k}_{j+1}\neq\emptyset$, $\mf{j}\cap\mf{k}_j^*\supseteq\mf{j}\cap\mf{k}_1^*\neq\emptyset$, $\mf{j}^*\cap\mf{k}_j\supseteq\mf{h}_j^*\cap\mf{k}_j\neq\emptyset$ and $\mf{j}^*\cap\mf{k}_j^*\supseteq\mf{j}^*\cap\mf{k}_1^*\neq\emptyset$. Thus, every $\mf{j}$ of type~(b) is transverse to the set $\mscr{H}(\mf{k}_1^*|\mf{k}_j)\cup\{\mf{k}_1^*,\mf{k}_j\}$.

Now, consider a maximal chain of halfspaces $\mf{j}_1\supsetneq\dots\supsetneq\mf{j}_m$ in $\mscr{H}(\mf{h}_j^*|\mf{k}_{j+1})$ with $m\geq 0$. We can enlarge this chain by adding $\mf{j}_0:=\mf{h}_j$ and $\mf{j}_{m+1}=\mf{k}_{j+1}$, which are, respectively, of type~(b) and~(a). Thus, there exists an index $0\leq k\leq m$ such that $\mf{j}_0,\dots,\mf{j}_k$ are of type~(b) and $\mf{j}_{k+1},\dots,\mf{j}_{m+1}$ are of type~(a). Since the chain is maximal, the set $\mscr{W}(\mf{j}_k^*|\mf{j}_{k+1})$ is empty. Thus, since $\mf{j}_k$ and $\mf{j}_{k+1}$ are not transverse, the labels $\g(\mf{j}_k)$ and $\g(\mf{j}_{k+1})$ are not joined by an edge of $\G$.

However, since $\G_{j+1}=\G_j$, we have: 
\[\g(\mf{j}_{k+1})\in\g(\mscr{W}(\mf{k}_1^*|\mf{k}_{j+1}))\cup\{\g(\mf{k}_1^*),\g(\mf{k}_{j+1})\}=\g(\mscr{W}(\mf{k}_1^*|\mf{k}_j))\cup\{\g(\mf{k}_1^*),\g(\mf{k}_j)\},\] 
while $\mf{j}_k$ is transverse to $\mscr{H}(\mf{k}_1^*|\mf{k}_j)\cup\{\mf{k}_1^*,\mf{k}_j\}$, a contradiction. This proves claim and lemma.
\end{proof}

\subsection{A quasi-convexity criterion for median subalgebras.}\label{wqc sect}

In this subsection, we provide a criterion (Proposition~\ref{wqc->qc}) for when a median subalgebra $M$ of a $\CAT$ cube complex $X$ is quasi-convex. The subalgebra $M$ will be required to satisfy two conditions, \emph{edge-connectedness} and \emph{weak quasi-convexity}, which we study separately in the next two subsubsections.

\subsubsection{Edge-connected median subalgebras.}

Let $X$ be a $\CAT$ cube complex.

\begin{defn}\label{edge-connected defn}
A subset $A\cu X^{(0)}$ is \emph{edge-connected} if, for all $x,y\in A$, there exists a sequence of points $x_1,\dots,x_n\in A$ such that $x_1=x$, $x_n=y$ and, for all $i$, the points $x_i$ and $x_{i+1}$ are joined by an edge of $X$.
\end{defn}

\begin{rmk}\label{intersection with edge-connected}
If $A\cu X^{(0)}$ is edge-connected, then there do not exist distinct halfspaces $\mf{h},\mf{k}\in\mscr{H}_A(X)$ with $\mf{h}\cap A=\mf{k}\cap A$. Indeed, the intersections $\mf{h}\cap\mf{k}$ and $\mf{h}^*\cap\mf{k}^*$ would both be nonempty, so, possibly swapping $\mf{h}$ and $\mf{k}$, we would either have $\mf{h}\subsetneq\mf{k}$ or $\mf{h}$ and $\mf{k}$ would be transverse. However, since $A$ is edge connected and intersects both $\mf{h}\cap\mf{k}$ and $\mf{h}^*\cap\mf{k}^*$, we must have $A\cap\mf{h}^*\cap\mf{k}\neq\emptyset$ if $\mf{h}\subsetneq\mf{k}$, and either $A\cap\mf{h}^*\cap\mf{k}\neq\emptyset$ or $A\cap\mf{h}\cap\mf{k}^*\neq\emptyset$ if $\mf{h}$ and $\mf{k}$ are transverse. This contradicts the fact that $\mf{h}\cap A=\mf{k}\cap A$.
\end{rmk}

\begin{lem}\label{eq characterisations}
For a median subalgebra $M\cu X^{(0)}$, the following are equivalent: 
\begin{enumerate}
\item $M$ is edge-connected;
\item for all $x,y\in M$, there exists a geodesic $\alpha\cu X$ joining $x$ and $y$ such that $\alpha\cap X^{(0)}\cu M$;
\item the restriction map ${\rm res}_M\colon\mscr{H}_M(X)\ra\mscr{H}(M)$ is injective.
\end{enumerate}
\end{lem}
\begin{proof}
%Proof of (1)$\Ra$(2): take a shortest edge path joining $x,y\in M$. If it's not a geodesic in $X$, then there exists a shortest subpath that begins and ends crossing the same hyperplane. Since $M$ is a median algebra, all points needed to avoid crossing the hyperplane are there.
The implication (2)$\Ra$(1) is clear and the implication (1)$\Ra$(3) follows from Remark~\ref{intersection with edge-connected}. Let us show that (3)$\Ra$(2).

Since $M$ is a discrete median algebra, it is isomorphic to the $0$--skeleton of a $\CAT$ cube complex $X(M)$ (see \cite[Theorem~6.1]{Chepoi} or \cite[Section~10]{Roller}). Given $x,y\in M$, let $\beta\cu X(M)$ be a geodesic joining $x$ and $y$, and let $x_1=x,x_2,\dots,x_n=y$ be the elements of $\beta\cap M$ as they appear along $\beta$. Since the restriction map ${\rm res}_M\colon\mscr{H}_M(X)\ra\mscr{H}(M)$ is injective, there is only one hyperplane $\mf{w}_i\in\mscr{W}(X)$ separating $x_i$ and $x_{i+1}$, that is, these two points are joined by an edge of $X$. If $i\neq j$, then $\mf{w}_i\neq\mf{w}_j$, or $\beta$ would cross the corresponding wall of $M$ twice. We conclude that there exists a geodesic $\alpha\cu X$ with $\alpha\cap M=\{x_1,\dots,x_n\}$. 
\end{proof}

By the 3rd characterisation in Lemma~\ref{eq characterisations}, edge-connected subalgebras can be viewed as a middle ground between general median subalgebras and convex subcomplexes (cf.\ part~(2) of Remark~\ref{halfspaces of subsets}).

\begin{lem}\label{generate edge-connected}
If $A\cu X^{(0)}$ is an edge-connected subset, then $\langle A\rangle$ is an edge-connected subalgebra.
\end{lem}
\begin{proof}
Suppose for the sake of contradiction that $\langle A\rangle$ is not edge-connected. Then there exist distinct halfspaces $\mf{h},\mf{k}\in\mscr{H}_{\langle A\rangle}(X)$ with $\mf{h}\cap\langle A\rangle=\mf{k}\cap\langle A\rangle$ by Lemma~\ref{eq characterisations}. Note that $\mf{h},\mf{k}\in\mscr{H}_A(X)$, and $\mf{h}^*\cap\mf{k}\cap A=\emptyset$ and $\mf{h}\cap\mf{k}^*\cap A=\emptyset$. In particular, $\mf{h}\cap A=\mf{k}\cap A$, which violates Remark~\ref{intersection with edge-connected}.
\end{proof}

\begin{lem}\label{edge-connected vs gate-projections}
Let $M\cu X^{(0)}$ be an edge-connected median subalgebra. Let $C\cu X$ be a convex subcomplex with gate-projection $\pi\colon X\ra C$. Then:
\begin{enumerate}
\item $\pi(M)$ is an edge-connected subalgebra of $C^{(0)}$;
\item if $N\cu\pi(M)$ is an edge-connected subalgebra, then $M\cap\pi^{-1}(N)$ is edge-connected as well.
\end{enumerate}
\end{lem}
\begin{proof}
If vertices $x,y\in X$ are joined by an edge, then either $\pi(x)$ and $\pi(y)$ are joined by an edge or they are equal. Thus, part~(1) is immediate from definitions.

Let us address part~(2). Consider two points $x,y\in M\cap\pi^{-1}(N)$. Since $N$ is edge-connected, there exists a geodesic $\alpha\cu C$ joining $\pi(x)$ and $\pi(y)$ with $\alpha\cap C^{(0)}\cu N$ (see Lemma~\ref{eq characterisations}). It suffices to show that $M\cap\pi^{-1}(\alpha)$ is edge-connected.

In fact, since $\pi^{-1}(v)\cap M\neq\emptyset$ for every vertex $v\in\alpha$, it suffices to show that $M\cap\pi^{-1}(e)$ is edge-connected for every edge $e\cu\alpha$. In other words, we can suppose that $\pi(x)$ and $\pi(y)$ are joined by an edge $e\cu C$. Since $M$ is edge-connected, there exists a geodesic $\beta\cu X$ joining $x$ and $y$ with $\beta\cap X^{(0)}\cu M$. Since $\pi$ is a median morphism, the projection $\pi(\beta)$ is the image of a geodesic from $\pi(x)$ to $\pi(y)$, i.e.\ $\pi(\beta)=e$. Thus $\beta\cap X^{(0)}\cu M\cap\pi^{-1}(e)$, concluding the proof.
\end{proof}

\subsubsection{Weakly quasi-convex median subalgebras.}

Let $X$ be a $\CAT$ cube complex. 

\begin{defn}\label{wqc defn}
A subset $A\cu X^{(0)}$ is \emph{weakly quasi-convex} if there exists a function $\eta\colon\N\ra\N$ such that, for all $a,b,p\in X^{(0)}$ with $\mscr{W}(p|a)$ transverse to $\mscr{W}(p|b)$, we have:
\[d(p,A)\leq \eta\big(\max\{d(a,A),d(b,A)\}\big).\]
\end{defn}

\begin{rmk}
\begin{enumerate}
\item[]
\item If $A\cu X^{(0)}$ is quasi-convex in the sense of Definition~\ref{qc defn}, then $A$ is weakly quasi-convex. Indeed, suppose that $\mc{J}(A)\cu\mc{N}_R(A)$ and set $D=\max\{d(a,A),d(b,A)$. If $\mscr{W}(p|a)$ and $\mscr{W}(p|b)$ are transverse, then $p\in I(a,b)$. Thus, $p\in\mc{J}(\mc{N}_D(A))$ and Lemma~\ref{about J} yields $d(p,A)\leq 2D+R=:\eta(D)$.
\item If $A,B\cu X^{(0)}$ have finite Hausdorff distance, then $A$ is weakly quasi-convex if and only if $B$ is. This is straightforward, observing that $\eta$ can always taken to be weakly increasing.
\end{enumerate}
\end{rmk}

The following is the main result of this subsection.

\begin{prop}\label{wqc->qc}
If $X$ has finite dimension and finite staircase length, then every edge-connected, weakly quasi-convex median subalgebra $M\cu X^{(0)}$ is quasi-convex.
\end{prop}

Proposition~\ref{wqc->qc} fails for cube complexes of infinite staircase length, as the next example shows.

\begin{ex}
Consider the standard structure of cube complex on $\R^2$. Let $\alpha$ be the geodesic line through all points $(n,n)$ and $(n+1,n)$ with $n\in\Z$. Let $X\cu\R^2$ be the subcomplex that lies above $\alpha$, including $\alpha$ itself. Note that $X$ is a $2$--dimensional $\CAT$ cube complex of infinite staircase length, and $\alpha\cu X$ is an edge-connected median subalgebra that is not quasi-convex. It is not hard to see that $\alpha$ is weakly quasi-convex with $\eta(t)=2t$.
\end{ex}

The next lemma essentially proves the $2$--dimensional case of Proposition~\ref{wqc->qc}.

\begin{lem}\label{wqc 2d}
Suppose that $\dim X=2$ and that $X$ has staircase length $d$. Let $M\cu X^{(0)}$ be an edge-connected median subalgebra. Consider $x,y\in M$ and $z\in X^{(0)}\cap I(x,y)$. Then there exist $0\leq k\leq d$ and vertices $z_0,z_1,z_2,\dots,z_k\in I(x,y)$ and $w_1,\dots,w_k\in I(x,y)$ such that:
\begin{itemize}
\item $z_0=z$, while $z_k\in M$ and $w_1,\dots,w_k\in M$;
\item the sets $\mscr{W}(z_i|w_{i+1})\cu\mscr{W}(X)$ and $\mscr{W}(z_i|z_{i+1})\cu\mscr{W}(X)$ are transverse for all $0\leq i\leq k-1$.
\end{itemize}
\end{lem}
\begin{proof}
If $z\in M$, we can simply take $k=0$. If $z\not\in M$, we begin with the following observation.

\smallskip
{\bf Claim:} \emph{we can assume that there exist transverse hyperplanes $\mf{u}\in\mscr{W}(x,z|y)$ and $\mf{v}\in\mscr{W}(y,z|x)$ such that $x,z$ lie in the carrier of $\mf{u}$ and $y,z$ lie in the carrier of $\mf{v}$.}

\smallskip\noindent
\emph{Proof of Claim.} 
Up to replacing $x$ and $y$ with other points in the interval $I(x,y)$, we can assume that there do not exist points $x',y'\in I(x,y)$ with $z\in I(x',y')$, except for $\{x',y'\}=\{x,y\}$. 

Since $M$ is edge-connected, there exists a point $x'\in M\cap I(x,y)$ such that $x$ and $x'$ are separated by a single hyperplane $\mf{u}\in\mscr{W}(X)$. By the above assumption on $x$ and $y$, we must have $z\not\in I(x',y)$, hence $\emptyset\neq\mscr{W}(z|x',y)=\mscr{W}(z,x|x',y)\cu\{\mf{u}\}$. It follows that $\mscr{W}(z,x|x',y)=\{\mf{u}\}$. 

Observing that $\mscr{W}(z|\mf{u})\cu\mscr{W}(z|x',y)=\mscr{W}(z,x|x',y)=\{\mf{u}\}$, we conclude that $\mscr{W}(z|\mf{u})$ is empty. This shows that the carrier of $\mf{u}$ contains $z$, while it is clear that it also contains $x$. The existence of $\mf{v}$ is obtained similarly. Finally, since $\mf{v}\in\mscr{W}(y,z|x)$ and $\mf{v}\neq\mf{u}$, we must have $\mf{v}\in\mscr{W}(y,z|x,x')$. Recalling that $\mf{u}\in\mscr{W}(z,x|x',y)$, this shows that $\mf{u}$ and $\mf{v}$ are transverse.
\hfill$\blacksquare$

\smallskip
Now, the sets $\mscr{H}(z|x)$ and $\mscr{H}(z|y)$ are transverse, respectively, to $\mf{u}$ and $\mf{v}$. Since $\dim X=2$, the set $\mscr{H}(z|x)$ is a descending chain $\mf{h}_1\supsetneq\dots\supsetneq\mf{h}_m$, and $\mscr{H}(z|y)$ is a descending chain $\mf{k}_1\supsetneq\dots\supsetneq\mf{k}_n$. Note that $\mf{k}_1$ and $\mf{h}_1$ are bounded, respectively, by $\mf{u}$ and $\mf{v}$, as depicted in Figure~\ref{2d lemma}.

Since $\mf{h}_1$ and $\mf{k}_1$ are transverse, there exists a function $\tau\colon \{1,\dots,m\}\ra\{1,\dots,n\}$ such that $\mf{h}_i$ is transverse to $\mf{k}_j$ if and only if $1\leq j\leq \tau(i)$. Note that $\tau(1)=n$ and that $\tau$ is weakly decreasing. 

Let $1\leq i_1<\dots<i_{k-1}<m$ be all indices $i$ with $\tau(i+1)<\tau(i)$. Also define $i_k:=m$ and set $\tau_s:=\tau(i_s)$ for simplicity. Since the halfspaces $\mf{h}_{i_k}^*,\dots,\mf{h}_{i_1}^*$ and $\mf{k}_{\tau_k},\dots,\mf{k}_{\tau_1}$ form a length--$k$ staircase, while $X$ has staircase length $d$, we must have $k\leq d$.

Set $z_0=z$ and $w_1=y$. For $1\leq s\leq k$, let $z_s\in I(x,y)$ be the point with $\mscr{H}(z|z_s)=\{\mf{h}_1,\dots,\mf{h}_{i_s}\}$. In particular, $z_k=x\in M$. Since $M$ is edge-connected, there exist points 
\[w_{s+1}\in M\cap\mf{h}_{i_s}\cap\mf{h}^*_{i_s+1}\cap\mf{k}^*_{\tau_{s+1}+1}.\] 
Observing that $\mscr{H}(z_s|w_{s+1})\cu\{\mf{k}_1,\dots,\mf{k}_{\tau_{s+1}}\}$ is transverse to $\mscr{H}(z_s|z_{s+1})=\{\mf{h}_{i_s+1},\dots,\mf{h}_{i_{s+1}}\}$, this completes the proof of the lemma.
\end{proof}

\begin{figure} 
\begin{tikzpicture}
\draw[fill] (0,3.5) -- (3.5,3.5);
\draw[fill] (0,3) -- (3.5,3);
\draw[fill] (0,2.5) -- (2.5,2.5);
\draw[fill] (0,2) -- (2.5,2);
\draw[fill] (0,1.5) -- (1.5,1.5);
\draw[fill] (0,1) -- (1.5,1);
\draw[fill] (0,0.5) -- (1,0.5);
\draw[fill] (0,0) -- (0.5,0);
\draw[fill] (0,0) -- (0,3.5);
\draw[fill] (0.5,0) -- (0.5,3.5);
\draw[fill] (1,0.5) -- (1,3.5);
\draw[fill] (1.5,1) -- (1.5,3.5);
\draw[fill] (2,2) -- (2,3.5);
\draw[fill] (2.5,2) -- (2.5,3.5);
\draw[fill] (3,3) -- (3,3.5);
\draw[fill] (3.5,3) -- (3.5,3.5);
\draw[dashed,thick,cyan] (-0.25,3.25) -- (3.75,3.25);
\draw[->,thick,cyan]  (3.75,3.25) --  (3.75,3);
\node[right,cyan] at (3.75,3.25) {$\mf{h}_1$};
\draw[dashed,thick,cyan] (-0.25,2.25) -- (2.75,2.25);
\draw[->,thick,cyan]  (2.75,2.25) --  (2.75,2);
\node[right,cyan] at (2.75,2.25) {$\mf{h}_{i_s}$};
\draw[dashed,thick,cyan] (-0.25,1.75) -- (1.75,1.75);
\draw[->,thick,cyan]  (1.75,1.75) --  (1.75,1.5);
\node[right,cyan] at (1.75,1.75) {$\mf{h}_{i_s+1}$};
\draw[dashed,thick,cyan] (-0.25,1.25) -- (1.75,1.25);
\draw[->,thick,cyan]  (1.75,1.25) --  (1.75,1);
\node[right,cyan] at (1.75,1.25) {$\mf{h}_{i_{s+1}}$};
\draw[dashed,thick,magenta] (0.25,-0.25) -- (0.25,3.75);
\draw[->,thick,magenta]  (0.25,3.75) --  (0.5,3.75);
\node[right,magenta] at (0.5,3.75) {$\mf{k}_1$};
\draw[dashed,thick,magenta] (1.25,0.75) -- (1.25,3.75);
\draw[->,thick,magenta]  (1.25,3.75) --  (1.5,3.75);
\node[right,magenta] at (1.5,3.75) {$\mf{k}_{\tau_{s+1}}$};
\node[right,magenta] at (0.25,-0.25) {$\mf{u}$};
\node[below left] at (0,0) {$z_k=x$};
\draw[fill] (0,0) circle [radius=0.04cm];
\node[above right] at (3.5,3.5) {$y=w_1$};
\draw[fill] (3.5,3.5) circle [radius=0.04cm];
\node[above left] at (0,3.5) {$z_0=z$};
\draw[fill] (0,3.5) circle [radius=0.04cm];
\node[left] at (0,2) {$z_s$};
\draw[fill] (0,2) circle [radius=0.04cm];
\node[left] at (0,3) {$z_1$};
\draw[fill] (0,3) circle [radius=0.04cm];
\end{tikzpicture}
\caption{}
\label{2d lemma} 
\end{figure}

The next lemma allows us to reduce the proof of Proposition~\ref{wqc->qc} to the $2$--dimensional case.

\begin{lem}\label{2d reduction}
Let $X$ have dimension $r$ and staircase length $d$. Let $M\cu X^{(0)}$ be an edge-connected median subalgebra. For all points $x,y\in M$ and $z\in X^{(0)}\cap I(x,y)$, there exists a median subalgebra $N\cu X^{(0)}\cap I(x,y)$ with the following properties:
\begin{itemize}
\item $x,y,z\in N$ and $\rk N\leq 2$;
\item $N$ has staircase length $\leq d(1+2r^2)^2$;
\item $N$ and $N\cap M$ are edge-connected.
\end{itemize}
\end{lem}
\begin{proof}
Let $\pi_{xz}\colon X\ra I(x,z)$ be the gate-projection and note that $\pi_{xz}(y)=z$. By Lemma~\ref{edge-connected vs gate-projections}(1), the projection $\pi_{xz}(M)$ is an edge-connected median subalgebra containing $x$ and $z$. Thus there exists a (combinatorial) geodesic $\alpha\cu I(x,z)$ joining $x$ and $z$ with $\alpha\cap X^{(0)}\cu\pi_{xz}(M)$. 

By Lemma~\ref{edge-connected vs gate-projections}(2), the median subalgebras $N':=\pi_{xz}^{-1}(\alpha)\cap I(x,y)\cap X^{(0)}$ and $M\cap N'$ are edge-connected. Lemma~\ref{preservation of finite staircase length} shows that $N'$ has staircase length $\leq d(1+2r^2)$, while it is clear that $\rk N'\leq\dim X=r$.

Note that $x,y,z\in N'$. Since $\pi_{xz}(I(z,y))=\{z\}$, the entire interval $I(z,y)\cap X^{(0)}$ is contained in $N'$. Consider the projection $\pi_{zy}\colon X\ra I(z,y)$. Since $M\cap N'$ is edge-connected, Lemma~\ref{edge-connected vs gate-projections} again shows that the projection $\pi_{zy}(M\cap N')$ is edge-connected, and we can join $y$ and $z$ by a geodesic $\beta$ with $\beta\cap X^{(0)}\cu\pi_{zy}(M\cap N')$. Repeating the above argument, we see that $N:=N'\cap\pi_{yz}^{-1}(\beta)$ has staircase length $\leq d(1+2r^2)^2$, that $N$ and $N\cap M$ are edge-connected, and that $x,y,z\in N$ (recall that $N'$ is a finite median algebra, so it is naturally identified with the $0$--skeleton of a $\CAT$ cube complex and we can run the above argument in this cube complex).

We are left to show that $\rk N\leq 2$. Since $x,y\in N\cu I(x,y)$, every wall of $N$ either separates $x$ from $y,z$, or it separates $x,z$ from $y$. If two walls of $N$ separate $x$ and $z$, then they are not transverse (cf.\ Claim~1 during the proof of Lemma~\ref{preservation of finite staircase length}). The same is true of walls separating $z$ and $y$. This implies that $\rk N\leq 2$, concluding the proof.
\end{proof}

\begin{proof}[Proof of Proposition~\ref{wqc->qc}]
Let $X$ have dimension $r$ and staircase length $d$. Let $M$ be an edge-connected, weakly quasi-convex subalgebra. We will show that $d_{\rm Haus}(I(x,y),M\cap I(x,y))$ remains uniformly bounded as $x$ and $y$ vary in $M$, which implies that $M$ is quasi-convex.

Consider $x,y\in M$ and $z\in X^{(0)}\cap I(x,y)$. By Lemma~\ref{2d reduction}, the points $x,y,z$ lie in a median subalgebra $N\cu X^{(0)}\cap I(x,y)$ such that $N$ and $N\cap M$ are edge-connected, $\rk N\leq 2$, and $N$ has staircase length $\leq d(1+2r^2)^2$. 

Viewing $N$ as the vertex set of a finite $\CAT$ cube complex and applying Lemma~\ref{wqc 2d} to $M\cap N$, there exist points $z_0=z,z_1,\dots,z_{k-1}\in N$ and $z_k,w_1,\dots,w_k\in N\cap M$ with $k\leq d(1+2r^2)^2$, such that each wall of $N$ separating $z_i$ and $z_{i+1}$ is transverse to every wall of $N$ separating $z_i$ and $w_{i+1}$. The same is true of hyperplanes of $X$ separating these points.

Since $M$ is weakly quasi-convex, it admits a function $\eta$ as in Definition~\ref{wqc defn}. Without loss of generality, we can take $\eta$ to be weakly increasing. Then, since $d(w_i,M)=0$, we have:
\begin{align*}
d(z,M)\leq \max\{\eta(d(z_1,M)),\eta(0)\}&\leq\max\{\eta^2(d(z_2,M)),\eta^2(0),\eta(0)\} \\
&\leq\dots\leq \max\{\eta^k(0),\dots,\eta^2(0),\eta(0)\}.
\end{align*}
The last constant only depends on $d$, $r$ and $\eta$, so this proves that $M$ is quasi-convex.
\end{proof}

\subsection{Fixed subgroups in right-angled groups.}\label{fix cc sect}

In this subsection, we combine the results of the previous two subsections to prove Theorem~\ref{U_0 cc intro}.

Let $\G$ be a finite simplicial graph. Our focus will be on the right-angled Artin group $\A=\A_{\G}$ and the universal cover of its Salvetti complex $\X=\X_{\G}$. Throughout, we will identify $\A\cong\X^{(0)}$.

However, all results and proofs in this subsection (except for Remark~\ref{useless rmk}) immediately extend to right-angled Coxeter groups $\W=\W_{\G}$ and Davis complexes $\mc{Y}_{\G}$, without requiring any adaptations. We suggest that the reader keep track of this as they make their way through the results, in view of Corollary~\ref{fix RACG cc} below. The relevant properties shared by RAAGs and RACGs are:
\begin{itemize}
\item the Cayley graph of $\A$/$\W$ associated to the standard generators (vertices of $\G$) is the $1$--skeleton of a $\CAT$ cube complex (the universal cover of the Salvetti/Davis complex) of finite staircase length (Lemma~\ref{no staircases}); 
\item hyperplanes are labelled by vertices of $\G$ and labels of transverse hyperplanes are joined by an edge of $\G$;
\item elementary automorphisms of $\A$ and $\W$ (as defined in Subsection~\ref{of RAAGs sect}) have the same form with respect to standard generators.
\end{itemize}
We are interested in the subgroups $U_0(\A)\leq U(\A)$ and $\aut_0\W\leq\aut\W$ generated by inversions, folds and partial conjugations, as defined at the end of Subsection~\ref{of RAAGs sect}. 

Given a subset $\Delta\cu\G^{(0)}$, it is convenient to introduce the notation:
\[\Delta^{\perp}=\bigcap_{v\in\Delta}\lk v.\]

\begin{rmk}\label{useless rmk}
It is not hard to observe that a subgroup of $\A$ is an intersection of stabilisers of hyperplanes of $\X$ if and only if it is conjugate to a subgroup of the form $\A_{\Delta^{\perp}}$ for some $\Delta\cu\G$. 
% for Coxeter groups we would need to intersect *stars* instead

Although we will not be using this remark in the present paper, we find it interesting in relation to Lemma~\ref{orthogonals are preserved} below: elements of $U_0(\A)$ permute hyperplane-stabilisers while preserving labels.
\end{rmk}

Statements similar to the next lemma have been widely used in the literature, e.g.\ in \cite[Proposition~3.2]{CCV}, \cite[Proposition~3.2]{CV1} and \cite[Section~3]{CV2}). Compared to these references, we get a slightly stronger result because here we are only concerned with untwisted automorphisms.

\begin{lem}\label{orthogonals are preserved}
For every $\varphi\in U_0(\A)$ and $\Delta\cu\G$, the subgroups $\A_{\Delta^{\perp}}$ and $\varphi(\A_{\Delta^{\perp}})$ are conjugate.
\end{lem}
\begin{proof}
It suffices to prove the lemma for elementary generators. It is clear that it holds for inversions, so we are left to consider folds and partial conjugations.

If $\tau_{v,w}$ is a fold, then $\tau_{v,w}(\A_{\Delta^{\perp}})=\A_{\Delta^{\perp}}$. This is immediate if $v\not\in\Delta^{\perp}$. If instead $v\in\Delta^{\perp}$, we have $\Delta\cu\lk v\cu\lk w$, hence $w\in\Delta^{\perp}$.

If $\kappa_{w,C}$ is a partial conjugation, then $\kappa_{w,C}(\A_{\Delta^{\perp}})$ is either $\A_{\Delta^{\perp}}$ or $w^{-1}\A_{\Delta^{\perp}}w$. This is clear if $\Delta^{\perp}$ intersects at most one connected component of $\G\setminus\St w$. Suppose instead that $\Delta^{\perp}$ intersects two distinct components of $\G\setminus\St w$. Then, for every $a\in\Delta$, the fact that $\Delta^{\perp}\cu\lk a$ implies that $a\in\lk w$. Thus, $w\in\Delta^{\perp}$ and $\kappa_{w,C}(\A_{\Delta^{\perp}})=\A_{\Delta^{\perp}}$ in this case.
\end{proof}

\begin{cor}\label{orthogonals are preserved 2}
For every $\varphi\in U_0(\A)$ and $g\in\A$, we have $\G(\varphi(g))^{\perp}=\G(g)^{\perp}$.
\end{cor}
\begin{proof}
It suffices to show that $\G(\varphi(g))^{\perp}\supseteq\G(g)^{\perp}$ for all $\varphi\in U_0(\A)$ and $g\in\A$. Note that $g$ has a conjugate in $\A_{\G(g)}\leq\A_{\G(g)^{\perp\perp}}$. Thus, Lemma~\ref{orthogonals are preserved} implies that a conjugate of $\varphi(g)$ lies in $\A_{\G(g)^{\perp\perp}}$. This shows that $\G(\varphi(g))\cu\G(g)^{\perp\perp}$, hence $\G(\varphi(g))^{\perp}\supseteq\G(g)^{\perp\perp\perp}=\G(g)^{\perp}$, as required.
\end{proof}

For the next results, recall that we are identifying elements of $\A$ and vertices of $\X$.

\begin{lem}\label{few labels outside}
For every $\varphi\in U_0(\A)$, there exists a constant $K(\varphi)$ with the following property. For all $x,y\in\A$, at most $K(\varphi)$ among the hyperplanes in $\mscr{W}(\varphi(x)|\varphi(y))$ have label outside $\g(\mscr{W}(x|y))^{\perp\perp}$.
\end{lem}
\begin{proof}
It suffices to show that, for every $g\in\A$, at most $K(\varphi)$ among the hyperplanes in $\mscr{W}(1|\varphi(g))$ have label outside $\g(\mscr{W}(1|g))^{\perp\perp}$.

Since $\G$ has only finitely many subsets, Lemma~\ref{orthogonals are preserved} shows that there exists a constant $K'(\varphi)$ with the following property. For every $\Delta\cu\G$ there exists $x_{\Delta}\in\A$ with $\varphi(\A_{\Delta^{\perp}})=x_{\Delta}\A_{\Delta^{\perp}}x_{\Delta}^{-1}$ and $|x_{\Delta}|\leq K'(\varphi)$. Here $|\cdot|$ denotes word length with respect to the standard generators.

Now, consider $g\in\A$ and set $\Delta(g):=\g(\mscr{W}(1|g))^{\perp}$. Then $g\in\A_{\Delta(g)^{\perp}}$ and the above observation shows that all but $2|x_{\Delta(g)}|$ hyperplanes in $\mscr{W}(1|\varphi(g))$ have label in $\Delta(g)^{\perp}$. Taking $K(\varphi):=2K'(\varphi)$, this concludes the proof.
\end{proof}

\begin{prop}\label{fix wqc}
If $\varphi\in U_0(\A)$, the subgroup $\Fix\varphi$ is a weakly quasi-convex subset of $\X^{(0)}\cong\A$.
\end{prop}
\begin{proof}
Consider vertices $a,b,p\in\X$ with $\mscr{W}(p|a)$ transverse to $\mscr{W}(p|b)$. Set: 
\[D:=\max\{d(a,\Fix\varphi),d(b,\Fix\varphi)\}.\] 
Let $K=K(\varphi)$ be as in Lemma~\ref{few labels outside}, let $\zeta_1,\zeta_2$ be the functions provided by Lemma~\ref{displacement vs fix dist} (without loss of generality, strictly increasing), and let $C$ be a constant such that
\[\varphi(m(x,y,z))\approx_Cm(\varphi(x),\varphi(y),\varphi(z)),\ \forall x,y,z\in\X.\] 

Let us write $a',b',p'$ for $\varphi(a),\varphi(b),\varphi(p)$. Since $\mscr{W}(p|a)$ and $\mscr{W}(p|b)$ are transverse, we have $p\in I(a,b)$ hence $\mscr{W}(p|a,b)=\emptyset$. Observing that $m(a',b',p')\approx_C\varphi(m(a,b,p))=p'$, we also have $\#\mscr{W}(p'|a',b')\leq C$. Finally, by the first inequality in Lemma~\ref{displacement vs fix dist}, we have $a'\approx_{D'}a$ and $b'\approx_{D'}b$, where $D':=\zeta_1^{-1}(D)$.

Putting together these inequalities, we obtain:
\begin{align*}
\#\mscr{W}(p|p')&= \#\mscr{W}(p|a',b',p')+\#\mscr{W}(p,a'|b',p')+\#\mscr{W}(p,b'|a',p')+\#\mscr{W}(p,a',b'|p') \\ 
&\leq \#\mscr{W}(p|a,b)+2D'+\#\mscr{W}(p,a'|b,p')+D'+\#\mscr{W}(p,b'|a,p')+D'+\#\mscr{W}(a',b'|p') \\
&\leq \#\mscr{W}(p,a'|b,p')+\#\mscr{W}(p,b'|a,p')+C+4D'.
\end{align*}

By Lemma~\ref{few labels outside}, at most $K$ elements of $\mscr{W}(a'|p')$ have label in $\g(\mscr{W}(a|p))^{\perp}$. Since $\mscr{W}(p|a)$ and $\mscr{W}(p|b)$ are transverse, we deduce that $\#\mscr{W}(p,a'|b,p')\leq K$ and, similarly, $\#\mscr{W}(p,b'|a,p')\leq K$. We conclude that:
\[d(p,\varphi(p))=\#\mscr{W}(p|p')\leq 2K+C+4D'.\]
Lemma~\ref{displacement vs fix dist} gives $d(p,\Fix\varphi)\leq\zeta_2(2K+C+4\cdot\zeta_1^{-1}(D))$, as required by Definition~\ref{wqc defn}. 
\end{proof}

\begin{cor}\label{fix RAAG cc}
For every $\varphi\in U_0(\A)$, the subgroup $\Fix\varphi$ is convex-cocompact in $\A\acts\X$.
\end{cor}
\begin{proof}
Set $H:=\Fix\varphi$. By Theorem~\ref{cmp und intro}, $H$ is finitely generated, so there exists $R\geq 0$ such that $\mc{N}_R(H)$ is edge-connected, viewed as a subset of $\X$. By Lemma~\ref{generate edge-connected}, the median subalgebra $M:=\langle\mc{N}_R(H)\rangle$ is edge-connected. Since $H$ is an approximate median subalgebra by Lemma~\ref{fix approximate subalgebra}, Proposition~\ref{approx subalgebras} shows that $M$ is at finite Hausdorff distance from $H$. Since $H$ is weakly quasi-convex by Proposition~\ref{fix wqc}, so is $M$.

Finally, $\X$ has finite staircase length by Lemma~\ref{no staircases}. We have shown that $M\cu\X^{(0)}$ is edge-connected and weakly quasi-convex, so Proposition~\ref{wqc->qc} implies that $M$ is quasi-convex. By Lemma~\ref{about J}, $\hull M$ is at finite Hausdorff distance from $M$, which is at finite Hausdorff distance from $H$. This implies that $H$ acts cocompactly on the convex subcomplex $\hull M\cu\X$.
\end{proof}

The discussion in this subsection immediately extends to right-angled Coxeter groups $\W$ and the finite-index subgroup $\aut_0\W\leq\aut\W$ generated by folds and partial conjugations.

\begin{cor}\label{fix RACG cc}
For every $\varphi\in\aut_0\W$, the subgroup $\Fix\varphi$ is convex-cocompact in $\W\acts\mc{Y}$, where $\mc{Y}$ is the universal cover of the Davis complex.
\end{cor}

Recalling Lemma~\ref{equiv cc} and Remark~\ref{right-angled qc rmk}, the previous two corollaries prove Theorem~\ref{U_0 cc intro}.

\section{Invariant splittings of RAAGs.}\label{invariant splittings sect}
 
This section only contains the proofs of Proposition~\ref{intro invariant splitting} and Corollary~\ref{intro fix graph}, which are independent from all other results mentioned in the Introduction. 

Let $\G$ be a finite simplicial graph and let $\A=\A_{\G}$ be the corresponding right-angled Artin group. All results and proofs in this section immediately extend to the right-angled Coxeter group $\W_{\G}$ and automorphisms in $\aut_0\W_{\G}$. We encourage the reader to verify this as they go through the material, emphasising that only Lemmas~\ref{preservative representative} and~\ref{splitting preserved by U} and Corollary~\ref{fix graph} require any kind of attention, as all other results in this section are purely about the finite graph $\G$.

The following is Proposition~\ref{intro invariant splitting} from the Introduction.

\begin{prop}\label{invariant splitting prop}
Let $\A$ be directly irreducible, freely irreducible and non-cyclic. Then there exists an amalgamated product splitting $\A=\A_+\ast_{\A_0}\A_-$, with $\A_{\pm}$ and $\A_0$ parabolic subgroups of $\A$, such that the corresponding Bass--Serre tree $\A\acts T$ is $U_0(\A)$--invariant. That is: for every $\varphi\in U_0(\A)$, there exists an isometry $f\colon T\ra T$ satisfying $f\o g=\varphi(g)\o f$ for all $g\in\A$.
\end{prop}

Proposition~\ref{invariant splitting prop} follows from Corollary~\ref{splitting preserved by U} and Proposition~\ref{goods exist} below. The latter will be proved right after Lemma~\ref{graph reduction}.

Given a partition $\G^{(0)}=\L^+\sqcup\L\sqcup\L^-$, we write $\A_+:=\A_{\L\sqcup\L^+}$ and $\A_-:=\A_{\L\sqcup\L^-}$ for simplicity. If $\L^{\pm}$ are nonempty and $d(\L^+,\L^-)\geq 2$ (where $d$ denotes the graph metric on $\G$), then the partition corresponds to a splitting as amalgamated product: 
\[\A=\A_+\ast_{\A_{\L}}\A_-.\] 
We denote by $\A\acts T_{\L}$ the Bass--Serre tree of this splitting. This will not cause any ambiguity related to possible different choices of the sets $\L^{\pm}$ in the following discussion.

We are interested in partitions of $\G^{(0)}$ that satisfy a certain list of properties.

\begin{defn}\label{good defn}
A partition $\G^{(0)}=\L^+\sqcup\L\sqcup\L^-$ into three nonempty subsets is \emph{good} if:
\begin{enumerate}
\item[(i)] $d(\L^+,\L^-)\geq 2$, where $d$ is the graph metric on $\G$;
\item[(ii)] for every $\eps\in\{\pm\}$ and $w\in\L^{\eps}$, there does not exist $v\in \L\sqcup\L^{-\eps}$ with $\lk v\cu\lk w\cup\L^{\eps}$;
\item[(iii)] for every $\eps\in\{\pm\}$ and $w\in\L^{\eps}$, the subgraph of $\G$ spanned by $(\L\sqcup\L^{-\eps})\setminus\St w$ is connected.
% condition~(iii) implies condition~(ii) except when $v\in\St w$ or $(\L\sqcup\L^{-\eps})\setminus\St w=\{v\}$
\end{enumerate}
We will simply write $\G=\L^+\sqcup\L\sqcup\L^-$, rather than $\G^{(0)}=\L^+\sqcup\L\sqcup\L^-$.
\end{defn}

The motivation for Definition~\ref{good defn} comes from the next lemma and the subsequent corollary. Definition~\ref{good defn} actually contains slightly stronger requirements than what is strictly necessary to the two results: this will facilitate the inductive construction of good partitions of graphs $\G$.

\begin{lem}\label{preservative representative}
Let $\G=\L^+\sqcup\L\sqcup\L^-$ be a good partition. For every $\psi\in U_0(\A)$, there exists $\varphi\in U_0(\A)$ representing the same outer automorphism and simultaneously satisfying $\varphi(\A_+)=\A_+$ and $\varphi(\A_-)=\A_-$ (hence also $\varphi(\A_{\L})=\A_{\L}$).
\end{lem}
\begin{proof}
Inversions preserve $\A^+$ and $\A^-$. Given vertices $v,w\in\G$ with $\lk v\cu\lk w$, condition~(ii) implies that either $w\in\L$, or $\{v,w\}\cu\L^+$, or $\{v,w\}\cu\L^-$. Thus, folds also preserve $\A^+$ and $\A^-$. 

We are left to prove the lemma in the case when $\psi$ is a partial conjugation $\kappa_{w,C}$. If $w\in\L$, it is clear that $\kappa_{w,C}$ preserves $\A^+$ and $\A^-$. Thus, let us assume without loss of generality that $w\in\L^+$. By condition~(iii), the set $\L\cup\L^-$ intersects a unique connected component $K\cu\G\setminus\St w$. 

If $K\neq C$, then $\kappa_{w,C}$ is the identity on $\A^-$, so $\A^{\pm}$ are both preserved. If $K=C$, then $\kappa_{w,C}$ represents the same outer automorphism as $\kappa_{w^{-1},K_1}\cdot\ldots\cdot\kappa_{w^{-1},K_k}$, where $K_1,\dots,K_k$ are the connected components of $\G\setminus\St w$ other than $K$. Again, the latter is the identity on $\A^-$, so $\A^{\pm}$ are preserved.
\end{proof}

This shows that $T_{\L}$ is invariant under twisting by elements of $U_0(\A)$:

\begin{cor}\label{splitting preserved by U}
Let $\G=\L^+\sqcup\L\sqcup\L^-$ be a good partition. For every $\varphi\in U_0(\A)$, there exists an automorphism $f\colon T_{\L}\ra T_{\L}$ satisfying $f\o g=\varphi(g)\o f$ for all $g\in\A$.
\end{cor}
\begin{proof}
If $\varphi$ is inner, we can take $f$ to coincide with an element of $\A$. If $\varphi(\A_+)=\A_+$ and $\varphi(\A_-)=\A_-$, the statement is also clear, since the Bass--Serre tree can be defined in terms of cosets of $\A^{\pm}$. By Lemma~\ref{preservative representative}, every element of $U_0(\A)$ is a product of two automorphisms of these two types. 
\end{proof}

Our next goal is to show that good partitions (almost) always exist. We say that $\G$ is \emph{irreducible} if it does not split as a nontrivial join (equivalently, the opposite graph $\G^o$ is connected).

\begin{prop}\label{goods exist}
If $\G$ is connected, irreducible and not a singleton, then $\G$ admits a good partition.
\end{prop}

Proposition~\ref{goods exist} and Corollary~\ref{splitting preserved by U} immediately imply Proposition~\ref{invariant splitting prop}, as well as the analogous result for right-angled Coxeter groups.

Before proving Proposition~\ref{goods exist}, we need to obtain a few lemmas.

\begin{lem}\label{excellent construction}
If $\G$ is connected and $\diam\G^{(0)}\geq 3$, there exists a good partition of $\G$.
\end{lem}
\begin{proof}
Let $x,y\in\G$ be arbitrary vertices with $d(x,y)\geq 3$. Let $C_y$ be the connected component of $\G\setminus\St x$ that contains $y$. Similarly, let $C_x$ be the connected component of $\G\setminus\St y$ that contains $x$. 

Since $d(x,y)\geq 3$, we have $\St x\cap\St y=\emptyset$, hence $\St y\cu C_y$ and $\St x\cu C_x$. Since $\G$ is connected, $\G\setminus C_x$ is also connected. Note that $\St x$ and $\G\setminus C_x$ are disjoint and $y\in\G\setminus C_x$. This implies that $\G\setminus C_x\cu C_y$. In conclusion, $\G=C_x\cup C_y$.

Note that, if $z\in\G^{(0)}$ and $\lk z\cap C_y=\emptyset$, we cannot have $z\in C_y$. Indeed, this would imply that $C_y=\{z\}$ and $\lk z\cu\St x$. Since $y\in C_y$, we would then have $y=z$ and $\lk y\cu\St x$, contradicting the fact that $\G$ is connected and $d(x,y)\geq 3$. 

Thus, we can define:
\begin{align*}
\L^+&:=\{z\in\G^{(0)} \mid \St z\cap C_y=\emptyset\}=\{z\in\G^{(0)} \mid \lk z\cap C_y=\emptyset\}, \\
\L^-&:=\{z\in\G^{(0)} \mid \St z\cap C_x=\emptyset\}=\{z\in\G^{(0)} \mid \lk z\cap C_x=\emptyset\}, \\
\L&:=\G^{(0)}\setminus(\L^+\sqcup\L^-).
\end{align*}
Note that $x\in\L^+$ and $y\in\L^-$. If $z\in\L^+$ and $w\in\L^-$, we have $\St z\cap\St w=\emptyset$, since $\G=C_x\cup C_y$. This shows that $d(\L^+,\L^-)\geq 3$. Since $\G$ is connected, we also conclude that $\L\neq\emptyset$. We are left to verify conditions~(ii) and~(iii) of Definition~\ref{good defn}.

If $v\in\L$, then $\lk v$ intersects both $C_x$ and $C_y$. Since $C_y$ is disjoint from $\lk w\cup\L^+$ for every $w\in\L^+$ (and similarly with $C_x$ and $\L^-$), this implies condition~(ii) when $v\in\L$. On the other hand, the case with $v\in\L^{-\eps}$ is immediate from the fact that $d(\L^+,\L^-)\geq 3$ and $\G$ is connected.

Finally, let us check condition~(iii). Without loss of generality, we can suppose that $w\in\L^+$. Note that $C_y$ is connected, contained in $(\L\sqcup\L^-)\setminus\St w$, and it intersects the link of every point of $\L$. Moreover, since $\L^-\cap C_x=\emptyset$ and $\G=C_x\cup C_y$, we have $\L^-\cu C_y$. This shows that $(\L\sqcup\L^-)\setminus\St w$ is connected, concluding the proof.
\end{proof}

When the previous lemma cannot be applied, we will construct a good partition of $\G$ inductively, extending good partitions on subgraphs. We now prove a sequence of three lemmas aimed precisely at this, after which we will give the argument for Proposition~\ref{goods exist}.

For $x\in\G^{(0)}$, let $\G\setminus x$ be the graph obtained by removing $x$ and all open edges incident to $x$.

\begin{lem}\label{ext**}
Let $\G\setminus x=\Delta^+\sqcup\Delta\sqcup\Delta^-$ be a good partition. Then one of the following happens:
\begin{enumerate}
\item there exist $w\in\Delta^+$ and $z\in\Delta^-$ with $\lk x\cu\lk z\cap\lk w$; %e.g. if $\lk x$ is empty...
\item the partition of $\G$ with $\L^+=\Delta^+\sqcup\{x\}$, $\L=\Delta$, $\L^-=\Delta^-$ is good;
\item the partition of $\G$ with $\L^+=\Delta^+$, $\L=\Delta\sqcup\{x\}$, $\L^-=\Delta^-$ is good;
\item the partition of $\G$ with $\L^+=\Delta^+$, $\L=\Delta$, $\L^-=\Delta^-\sqcup\{x\}$ is good.
\end{enumerate}
\end{lem}
\begin{proof}
We begin with the following observation.

\smallskip
{\bf Claim}: \emph{if there exists $w\in\Delta^+$ such that $\lk x\cu\lk w\cup\Delta^+$, we are either in case~(1) or in case~(2).}

\smallskip\noindent
\emph{Proof of Claim.} We assume that we are not in case~(1) and show that the partition of $\G$ in case~(2) is good. We need to verify conditions~(i)--(iii) from Definition~\ref{good defn}.

Since $d(\Delta^+,\Delta^-)\geq 2$ (both in $\G\setminus x$ and in $\G$), the set $\Delta^-$ is disjoint from $\lk w\cup\Delta^+$. Since $\lk x\cu\lk w\cup\Delta^+$, it follows that $\Delta^-\cap\St x=\emptyset$, hence $d(\L^+,\L^-)\geq 2$. This proves condition~(i).

If condition~(ii) fails, there exist $u\in\L^{\eps}$ and $v\in \L\sqcup\L^{-\eps}$ with $\lk v\cu\lk u\cup\L^{\eps}$. Since the partition of $\G\setminus x$ is good, we must have either $v=x$ or $u=x$. If $v=x$, then $u\in\Delta^-$ and 
\[\lk x\cu (\lk u\cup\Delta^-)\cap(\lk w\cup\Delta^+)=\lk u\cap\lk w,\] 
which lands us in case~(1). If instead $u=x$, we have $v\in\Delta\sqcup\Delta^-$ with:
\[\lk v\cu\lk x\cup\L^+\cu\lk w\cup\Delta^+\cup\{x\}.\] 
This violates condition~(ii) for the partition of $\G\setminus x$.

Finally, suppose that condition~(iii) fails. Thus, there exists $u\in\L^{\eps}$ such that $(\L\sqcup\L^{-\eps})\setminus\St u$ is disconnected. Since the partition of $\G\setminus x$ is good, this can happen only in two ways: either $u=x$, or $u\in\L^-$ and $x$ is isolated in $(\L\sqcup\L^+)\setminus\St u$. In the latter case, we have $\lk x\cu\lk u\cup\Delta^-$, which again leads to case~(1).

Suppose instead that $u=x$ and let us show that $(\L\sqcup\L^-)\setminus\St x=(\Delta\sqcup\Delta^-)\setminus\lk x$ is connected. Since $\lk x\cu\lk w\cup\Delta^+$, the set $(\Delta\sqcup\Delta^-)\setminus\lk x$ contains $(\Delta\sqcup\Delta^-)\setminus\lk w$. The latter is connected, as the partition of $\G\setminus x$ satisfies condition~(iii). Since condition~(ii) is satisfied, every point of $(\Delta\sqcup\Delta^-)\cap\lk w=\Delta\cap\lk w$ is joined by an edge to a point of $\G\setminus(\lk w\cup\L^+)=(\Delta\sqcup\Delta^-)\setminus\lk w$. Thus, the star of every point of $(\Delta\sqcup\Delta^-)\setminus\lk x$ intersects the connected set $(\Delta\sqcup\Delta^-)\setminus\lk w$, proving that $(\Delta\sqcup\Delta^-)\setminus\lk x$ is connected. This completes the proof of the Claim. 
\hfill$\blacksquare$

\smallskip
By the Claim, if there exist either $w\in\Delta^+$ with $\lk x\cu\lk w\cup\Delta^+$ or $z\in\Delta^-$ with $\lk x\cu\lk z\cup\Delta^-$, then we are in cases~(1),~(2) or~(4). In order to conclude the proof of the lemma, let us suppose that neither of the two inclusions is satisfied. We will show that the partition in case~(3) is good.

Condition~(i) is clear. Condition~(ii) is immediate from the corresponding condition for $\G\setminus x$ and our assumption that $\lk x$ be not contained in any subsets as in the previous paragraph. 

Suppose that condition~(iii) fails. Then there exists $u\in\L^{\eps}$ such that $(\L\sqcup\L^{-\eps})\setminus\St u$ is disconnected. Without loss of generality, we have $u\in\L^+$. Since the partition of $\G\setminus x$ satisfies condition~(iii), the point $x$ must be isolated in $(\L\sqcup\L^-)\setminus\St u$. Hence $\lk x\cu\lk u\cup\Delta^+$, again violating our assumption.
\end{proof}

\begin{lem}\label{G-x reducible}
Let $\G$ be an irreducible graph. Let $x\in\G$ be a vertex such that there does not exist $y\in\G^{(0)}\setminus\{x\}$ with $\lk x\cu\lk y$. Suppose that $\G\setminus x$ is reducible. Then the partition of $\G$ given by $\L^+=\{x\}$, $\L=\lk x$, $\L^-=\G\setminus\St x$ is good.
\end{lem}
\begin{proof}
Write $\G\setminus x$ as a join of nonempty subgraphs $\G_1$ and $\G_2$. Since $\G$ is irreducible, there exist points $a_1\in\G_1\setminus\lk x$ and $a_2\in\G_2\setminus\lk x$. Condition~(i) is clear.

In order to verify condition~(ii), we need to exclude the existence of $w\in\L^{\eps}$ and $v\in \L\sqcup\L^{-\eps}$ with $\lk v\cu\lk w\cup\L^{\eps}$. If $\eps=-$ and $v\in\L$, then $x$ lies in $\lk v$, but not in $\lk w\cup\L^-$. If $\eps=-$ and $v=x$, then $\lk x$ is disjoint from $\L^-$, and it cannot be contained in the link of any point of $\G\setminus x$ by our hypotheses. If $\eps=+$, then $\lk w\cup\L^{\eps}=\St x$, which cannot contain the link of any point of $\G\setminus x$, as it does not contain $a_1$ and $a_2$.

Finally, let us show that, for every $w\in\L^{\eps}$, the set $(\L\sqcup\L^{-\eps})\setminus\St w$ is connected. If $\eps=+$, this amounts to showing that $\G\setminus\St x$ is connected. This is immediate, since every point of $\G\setminus x$ is joined by an edge to either $a_1$ or $a_2$, and these two points are themselves joined by an edge. If instead $\eps=-$, we need to show that $\St x\setminus\St w$ is connected for every $w\in\G\setminus\St x$. This is also clear since this set is a cone over $x$.
\end{proof}

Consider the equivalence relation on $\G^{(0)}$ where $v\sim w$ if and only if $\lk v=\lk w$. We define a graph $\overline\G$ with a vertex for every $\sim$--equivalence class $[v]\cu\G$ and an edge joining $[v]$ and $[w]$ exactly when $v$ and $w$ are joined by an edge (this is independent of the chosen representatives). 

It is clear that $\overline\G$ is again a simplicial graph, with at most as many vertices as $\G$. We denote by $r\colon\G\ra\overline\G$ the natural morphism of graphs.

\begin{lem}\label{graph reduction}
\begin{enumerate}
\item[]
\item $\G$ is irreducible if and only if $\overline\G$ is irreducible.
\item If $\G$ has at least one edge, then $\G$ is connected if and only if $\overline\G$ is connected.
\item If $\overline\G=\L^+\sqcup\L\sqcup\L^-$ is a good partition, then so is $\G=r^{-1}(\L^+)\sqcup r^{-1}(\L)\sqcup r^{-1}(\L^-)$.
\end{enumerate}
\end{lem}
\begin{proof}
Parts~(1) and~(2) are straightforward, so we only prove part~(3). 

Consider a good partition $\overline\G=\L^+\sqcup\L\sqcup\L^-$. It is clear that the partition of $\G$ satisfies condition~(i), while condition~(ii) follows from the observation that $\lk r(x)=r(\lk x)$ for every $x\in\G$. 

Finally, we verify condition~(iii). Given $w\in r^{-1}(\L^{\eps})$, observe that $r$ maps the subgraph $(r^{-1}(\L)\sqcup r^{-1}(\L^{-\eps}))\setminus\St w$ onto the connected graph $(\L\sqcup\L^{-\eps})\setminus\St r(w)$. As in part~(2), this shows that $(r^{-1}(\L)\sqcup r^{-1}(\L^{-\eps}))\setminus\St w$ is connected, possibly except the case when $(\L\sqcup\L^{-\eps})\setminus\St r(w)$ is a singleton. The latter is ruled out by the fact that the partition of $\overline\G$ satisfies condition~(ii).
\end{proof}

\begin{proof}[Proof of Proposition~\ref{goods exist}]
We proceed by induction on the number of vertices of $\G$. Since no graph with at most $3$ vertices satisfies the hypotheses of the proposition, the base step is trivially satisfied. For the inductive step, we consider a connected irreducible graph $\G$ with at least $4$ vertices, and assume that the proposition is satisfied by all graphs with fewer vertices than $\G$.

If $\diam\G^{(0)}\geq 3$, we can simply appeal to Lemma~\ref{excellent construction}. If the graph $\overline\G$ defined above has fewer vertices than $\G$, then we can use the inductive hypothesis and Lemma~\ref{graph reduction}. Thus, we can assume that $\G=\overline\G$ and $\diam\G^{(0)}=2$.

Pick a vertex $x\in\G$ whose link is maximal under inclusion. Since $\G=\overline\G$, there does not exist $y\in\G^{(0)}\setminus\{x\}$ with $\lk x=\lk y$. If $\G\setminus x$ is reducible, Lemma~\ref{G-x reducible} then shows that $\G$ admits a good partition. If $\G\setminus x$ were disconnected, then the fact that $\diam\G^{(0)}=2$ would imply that $\lk x=\G\setminus x$, contradicting the assumption that $\G$ is irreducible.

In conclusion, $\G\setminus x$ is connected, irreducible, not a singleton, and it has fewer vertices than $\G$. We conclude by applying the inductive hypothesis and Lemma~\ref{ext**} (case~(1) of the latter is ruled out by our choice of $x$).
\end{proof}

The previous results prove Proposition~\ref{invariant splitting prop}. The following is Corollary~\ref{intro fix graph} from the Introduction.

\begin{cor}\label{fix graph}
Consider $\varphi\in U_0(\A)$. 
\begin{enumerate}
\item If $\A$ splits as a direct product $\A_1\x\A_2$, then $\varphi(\A_i)=\A_i$ and $\Fix\varphi=\Fix\varphi|_{\A_1}\x\Fix\varphi|_{\A_2}$.
\item If $\A$ is directly irreducible, then the subgroup $\Fix\varphi\leq\A$ splits as a (possibly trivial) finite graph of groups with vertex and edge groups of the form $\Fix\varphi|_{P}$, for proper parabolic subgroups $P\leq\A$ with $\varphi(P)=P$ and $\varphi|_{P}\in U_0(P)$.
\end{enumerate}
\end{cor}
\begin{proof}
For simplicity, set $H:=\Fix\varphi$. We distinguish three cases.

\smallskip
{\bf Case~1}: \emph{$\A$ is not directly irreducible.}

\noindent
Let us write $\A=A\x\A_1\x\dots\x\A_m$, where $A$ is a free abelian group and $\A_i$ are directly-irreducible (non-cyclic) right-angled Artin groups. This corresponds to a splitting of $\G$ as a join of a complete subgraph and irreducible subgraphs $\G_1,\dots,\G_m$. 

Since $\varphi\in U_0(\A)$, we have $\varphi(\A_k)=\A_k$ and $\varphi|_{\A_k}\in U_0(\A_k)$ for every $1\leq k\leq m$, and $\varphi|_A$ is a product of inversions. Indeed, this is clear for inversions, folds and partial conjugations. 

Thus $H=A'\x H_1\x\dots\x H_m$, where $H_i=\Fix(\varphi|_{\A_i})$ and $A'$ is a standard direct factor of $A$. This proves part~(1) of the corollary.

\smallskip
{\bf Case~2:} \emph{$\A$ is not freely irreducible.} 

\noindent
Write $\A=F\ast\A_1\ast\dots\ast\A_m$, where $F$ is a free group and $\A_i$ are freely-irreducible (non-cyclic) right-angled Artin groups of lower complexity. Since $H$ is finitely generated by Theorem~\ref{cmp und intro}, Kurosh's theorem guarantees that $H$ decomposes as a free product $H=L\ast H_1\ast\dots\ast H_n$, where $L$ is a finitely generated free group and each $H_i$ is a finitely generated subgroup of some $g_i\A_{k_i}g_i^{-1}$ with $g_i\in\A$ and $1\leq k_i\leq m$.

By Grushko's theorem, the subgroup $\varphi(\A_k)$ is conjugate to some $\A_{k'}$ for every $1\leq k\leq m$. Since $\varphi$ fixes the nontrivial subgroup $H_i\leq g_i\A_{k_i}g_i^{-1}$ pointwise, we must have $\varphi(g_i\A_{k_i}g_i^{-1})=g_i\A_{k_i}g_i^{-1}$ for $1\leq i\leq n$. 

Consider the automorphism $\psi_i\in U_0(\A)$ defined by $\psi_i(x)=g_i^{-1}\varphi(g_ixg_i^{-1})g_i$. Note that $\psi_i(\A_{k_i})=\A_{k_i}$ and $\Fix\psi_i|_{\A_{k_i}}=g_i^{-1}H_ig_i$. By Lemma~\ref{restriction of U_0}, we have $\psi_i|_{\A_{k_i}}\in U_0(\A_{k_i})$. This proves part~(2) of the corollary in the freely reducible case.

\smallskip
{\bf Case~3}: \emph{$\A$ is freely and directly irreducible.}

\noindent
We can assume that $\A\not\simeq\Z$. By Proposition~\ref{goods exist}, $\G$ admits a good partition $\G=\L^+\sqcup\L\sqcup\L^-$. By Corollary~\ref{splitting preserved by U}, there exists $f\in\aut T_{\L}$ satisfying $f\o g=\varphi(g)\o f$ for all $g\in\A$.  

If $H$ is elliptic in $T_{\L}$, we have $H\leq V$, where $V$ is the $\A$--stabiliser of some vertex of $T_{\L}$. The existence of the automorphism $f\in\aut T_{\L}$ guarantees that all subgroups $\varphi^n(V)$ with $n\in\Z$ are $\A$--stabilisers of vertices of $T_{\L}$; in particular, they are all conjugate to either $\A_+$ or $\A_-$. We conclude that $H$ is contained in the $\langle\varphi\rangle$--invariant parabolic subgroup $P:=\bigcap_{n\in\Z}\varphi^n(V)$. Thus, we have $H=\Fix\varphi|_P$ and, by Lemma~\ref{restriction of U_0}, $\varphi|_P\in U_0(P)$. This proves the corollary in this case, with $H$ splitting as a trivial graph of groups.

Suppose instead that $H$ is not elliptic in $T_{\L}$ and denote by $T_H\cu T_{\L}$ the $H$--minimal subtree. Since $H$ is finitely generated, the action $H\acts T_H$ is cocompact and gives a splitting of $H$ as a (nontrivial) finite graph of groups. We are left to understand vertex-stabilisers of the action $H\acts T_H$. 

As $f$ normalises $H$ in $\aut T_{\L}$, we have $f(T_H)=T_H$. It is convenient to distinguish two subcases.

\smallskip
{\bf Case~3a}: \emph{$f$ is elliptic in $T_{\L}$.}

\noindent
Since $f$ commutes with every element of $H$, the tree $T_H$ is fixed pointwise by $f$. For every $v\in T_H$, its $\A$--stabiliser $\A_v$ satisfies $\varphi(\A_v)=\A_v$ and is conjugate to either $\A_+$ or $\A_-$. By Lemma~\ref{restriction of U_0}, we have $\varphi|_{\A_v}\in U_0(\A_v)$, proving the corollary in this case.

\smallskip
{\bf Case~3b}: \emph{$f$ is loxodromic in $T_{\L}$.}

\noindent
Let $\alpha\cu T_{\L}$ be the axis of $f$. Since $f$ commutes with every element of $H$, the geodesic $\alpha$ must be $H$--invariant and every non-loxodromic element of $H$ fixes $\alpha$ pointwise. Note that $T_H$ cannot be a singleton, or $f$ would be elliptic. Thus, $T_H=\alpha$ and $H$ contains a shortest loxodromic element $h\in H$. Moreover, $H=H_0\rtimes\langle h\rangle$, where $H_0$ is the kernel of the action $H\acts\alpha$.

Let $Q\leq\A$ be the intersection of the $\A$--stabilisers of the vertices of $\alpha$. Being an intersection of parabolic subgroups, $Q$ is itself a (possibly trivial) parabolic subgroup of $\A$. Since $f(\alpha)=\alpha$, we have $\varphi(Q)=Q$ and $H_0=\Fix\varphi|_Q$. Lemma~\ref{restriction of U_0} guarantees that $\varphi|_Q\in U_0(Q)$. Thus, the HNN splitting $H=H_0\rtimes\langle h\rangle$ is as required by the the corollary.
\end{proof}

\begin{rmk}\label{better 3b}
In Case~3b of the proof of Corollary~\ref{fix graph}, we can actually say more on the structure of $H=\Fix\varphi$. Specifically, $H=H_0\x\langle h\rangle$ and $h$ can be taken to be label-irreducible.

Indeed, since $h\alpha=\alpha$, the element $h$ lies in the normaliser of $Q$ in $\A$, which is a subgroup of the form $Q\x Q'$ (since $Q$ is parabolic in $\A$). If $h=h_1\cdot\ldots\cdot h_k$ is the decomposition of $h$ into label-irreducible components, every $h_i$ lies in either $Q$ or $Q'$. Since $\varphi$ is coarse-median preserving and fixes $h$, it must permute the $h_i$; Corollary~\ref{orthogonals are preserved 2} then shows that $\varphi(h_i)=h_i$ for every $i$. Thus, all the label-irreducible components of $h$ that lie in $Q$ actually lie in $H_0$. Up to replacing $h$, we can assume that all $h_i$ lie in $Q'$; in particular, $h$ lies in $Q'$, hence it commutes with $H_0$. Since $H=\Fix\varphi$ is generated by $H_0$ and $h$, we must then have $k=1$, i.e.\ $h$ is label-irreducible.
\end{rmk}

In relation to Theorem~\ref{U_0 cc intro}, it is natural to wonder if the proof of Corollary~\ref{fix graph} can be used to give an alternative, inductive argument showing that $\Fix\varphi$ is convex-cocompact in $\A$ for every $\varphi\in U_0(\A)$. In light of Remark~\ref{better 3b}, the only problematic situation is the one in Case~3a. 

Unfortunately, cubical convex-cocompactness does not seem to be well-behaved with respect to graph-of-groups constructions, as the next example shows.

\begin{figure} 
\begin{tikzpicture}
\draw[fill] (-1,0) -- (0,0.5);
\draw[fill] (1,0) -- (0,0.5);
\draw[fill] (-1,0) circle [radius=0.04cm];
\draw[fill] (1,0) circle [radius=0.04cm];
\draw[fill] (0,0.5) circle [radius=0.04cm];
\draw[fill] (0,-0.5) circle [radius=0.04cm];
\node[below] at (0,-0.5) {$x$};
\node[above] at (0,0.5) {$y$};
\node[left] at (-1,0) {$a$};
\node[right] at (1,0) {$b$};
\end{tikzpicture}
\caption{}
\label{cc bad graphs fig} 
\end{figure}

\begin{ex}\label{cc bad graphs}
Let $\G$ be the graph in Figure~\ref{cc bad graphs fig}. Consider the subgroup $H=\langle ayx^{-1},xby\rangle\leq\A_{\G}$. We have an amalgamated product splitting $\A_{\G}=\langle a,x,y\rangle\ast_{\langle x,y\rangle}\langle b,x,y\rangle$, which induces a splitting $H=\langle ayx^{-1}\rangle\ast\langle xby\rangle\simeq F_2$. The subgroups $\langle ayx^{-1}\rangle$ and $\langle xby\rangle$ are convex-cocompact, as they are each generated by a single label-irreducible element.

However, $H$ is not convex-cocompact in $\A$: the element $aby^2$ lies in $H$, but no power of its label-irreducible components $ab$ and $y^2$ does (which, for instance, violates Lemma~\ref{label-irreducibles don't escape}).
\end{ex}

\section{Projectively invariant metrics on finite-rank median algebras.}\label{compatible sect}

In this section, we initiate the lengthy proof of Theorem~\ref{Q2 thm intro}, which will be completed in Section~\ref{ultra sect}. Our main goal here is to formulate a criterion, for a group $U$ and a subgroup $G\leq U$, guaranteeing that a $U$--action on a finite-rank median algebra admits a $G$--invariant compatible pseudo-metric for which $U$ acts by homotheties (Corollary~\ref{homothety 1}). An important tool will be the Lefschetz fixed point theorem for compact ANRs.

Throughout the section, $M$ denotes a fixed median algebra of finite rank $r$.

\subsection{Multi-bridges.}\label{multi-bridge sect}

The \emph{bridge} of two gate-convex sets was first studied in \cite{Behrstock-Charney,CFI} for $\CAT$ cube complexes and in \cite[Section~2.2]{Fio3} for general median algebras. We will need an extension of this concept to arbitrary finite collections of gate-convex subsets: \emph{multi-bridges}.

We briefly motivate why. As a recurring setup in the rest of the paper (especially in Subsections~\ref{WNE subsubsect} and~\ref{meaty sect}), we will often find ourselves studying a group $G\leq\aut M$ with a finite generating subset $S\cu G$ and a $G$--invariant compatible pseudo-metric $\eta$ on $M$. It will be important to understand which points of $M$ are moved as little as possible by all elements of $S$ (i.e.\ which points realise the quantity $\overline\tau_S^{\eta}$ from Subsection~\ref{identities sect}). It turns out that the set of such points does not depend much on the specific pseudo-metric $\eta$, and can instead be characterised purely in terms of the median-algebra structure on $M$, using the notion of multi-bridge (Propositions~\ref{multi-bridge of S} and~\ref{properties of B(S)}).

Let $C_1,\dots,C_k\cu M$ be gate-convex subsets, with gate-projections $\pi_i\colon M\ra C_i$. Let $\mc{H}\cu\mscr{H}(M)$ be the set of halfspaces that contain at least one $C_i$ and intersect each $C_i$. Then we have a partition:
\[\mscr{H}(M)=\Big(\mc{H}\sqcup\mc{H}^*\Big)\sqcup\Big(\bigcap_{1\leq i\leq k}\mscr{H}_{C_i}(M)\Big)\sqcup\Big(\bigcup_{1\leq i,j\leq k}\mscr{H}(C_i|C_j)\Big).\]
If $i\neq j$, the sets $\mscr{H}_{C_i}(M)\cap\mscr{H}_{C_j}(M)$ and $\mscr{H}(C_i|C_j)$ are transverse. Thus, every halfspace in the second set of the above partition of $\mscr{H}(M)$ is transverse to every halfspace in the third set.

\begin{lem}
The intersection of all halfspaces in $\mc{H}$ is a nonempty convex subset of $M$.
\end{lem}
\begin{proof}
We will prove this by appealing to Lemma~\ref{useful criterion}(1). It is clear that the elements of $\mc{H}$ intersect pairwise. Let us show that, for every chain $\mscr{C}\cu\mc{H}$, the set $\mf{k}:=\bigcap\mscr{C}$ is again an element of $\mc{H}$. 

Note that there exist $1\leq i_0\leq k$ and a cofinal subset $\mscr{C}'\cu\mscr{C}$ consisting of halfspaces containing $C_{i_0}$. Thus, $C_{i_0}\cu\mf{k}$ and $\mf{k}$ is nonempty. Since $\mf{k}$ is the intersection of a chain of halfspaces, both $\mf{k}$ and $\mf{k}^*$ are convex. It follows that $\mf{k}$ is a halfspace of $M$. 

For every $\mf{h}\in\mscr{C}\cu\mc{H}$, the fact that $\mf{h}$ intersects each $C_i$ implies that $\pi_i(\mf{h})=\mf{h}\cap C_i$ (see e.g.\ \cite[Lemma~2.2(1)]{Fio1}). Recalling that $\mf{k}=\bigcap\mscr{C}$, we deduce that $\pi_i(\mf{k})\cu\mf{k}\cap C_i$ for $1\leq i\leq k$, hence $\mf{k}$ intersects all $C_i$. Since we have already seen that $C_{i_0}\cu\mf{k}$, we conclude that $\mf{k}\in\mc{H}$, as required.
\end{proof}

\begin{defn}
The intersection $\mc{B}=\mc{B}(C_1,\dots,C_k)\cu M$ of all halfspaces in $\mc{H}$ is the \emph{multi-bridge} of the gate-convex sets $C_1,\dots,C_k$. 
\end{defn}

For every $\mf{k}\in\mscr{H}(M)\setminus\mc{H}^*$, the set $\mc{H}\sqcup\{\mf{k}\}$ is again pairwise-intersecting. Hence, Lemma~\ref{useful criterion}(1) yields: 
\[\mscr{H}_{\mc{B}}(M)=\mscr{H}(M)\setminus(\mc{H}\sqcup\mc{H}^*)=\Big(\bigcap\mscr{H}_{C_i}(M)\Big)\sqcup\Big(\bigcup\mscr{H}(C_i|C_j)\Big).\]
We have already observed that the two sets in this partition are transverse. By Remark~\ref{halfspaces of subsets}(2) and Lemma~\ref{product lem}, we obtain a natural product splitting:
\begin{align*}
\mc{B}&=\mc{B}_{/\mkern-5mu/}\x\mc{B}_{\perp}, \text{ \hspace{.1cm} where:}&  \mscr{H}_{\mc{B}_{/\mkern-5mu/}}(M)&=\bigcap\mscr{H}_{C_i}(M), & \mscr{H}_{\mc{B}_{\perp}}(M)&=\bigcup\mscr{H}(C_i|C_j).
\end{align*}
We can view $\mc{B}_{/\mkern-5mu/}$ and $\mc{B}_{\perp}$ as subsets of $M$ by identifying them with any fibre of the splitting of $\mc{B}$. 

\begin{lem}
The sets $\mc{B}$, $\mc{B}_{/\mkern-5mu/}$, $\mc{B}_{\perp}$ are gate-convex in $M$.
\end{lem}
\begin{proof}
Since each $C_i$ is gate-convex, Lemma~\ref{useful criterion}(2) shows that, for every chain $\mscr{C}\cu\bigcap\mscr{H}_{C_i}(M)$, either $\bigcap\mscr{C}$ is empty in $M$, or $\bigcap\mscr{C}\in\bigcap\mscr{H}_{C_i}(M)$. Hence $\mc{B}_{/\mkern-5mu/}$ is gate-convex in $M$.

If $\mscr{C}\cu\bigcup\mscr{H}(C_i|C_j)$ is a chain, a cofinal subset of $\mscr{C}$ is contained in a single $\mscr{H}(C_i|C_j)$. Hence $\bigcap\mscr{C}\in\mscr{H}(C_i|C_j)$. Invoking again Lemma~\ref{useful criterion}(2), this shows that $\mc{B}_{\perp}$ is gate-convex. 

Every chain in $\mscr{H}_{\mc{B}}(M)$ has a cofinal subset contained in either $\bigcap\mscr{H}_{C_i}(M)$ or $\bigcup\mscr{H}(C_i|C_j)$. One last application of Lemma~\ref{useful criterion}(2) shows that $\mc{B}$ is gate-convex.
\end{proof}

\begin{cor}\label{multibridge cor}
If $C_1,\dots,C_k\cu M$ are gate-convex subsets, their multi-bridge $\mc{B}=\mc{B}(C_1,\dots,C_k)$ is a gate-convex subset of $M$ enjoying the following properties:
\begin{enumerate}
\item $\mc{B}$ splits as a product $\mc{B}_{/\mkern-5mu/}\x\mc{B}_{\perp}$ with $\mscr{H}_{\mc{B}_{/\mkern-5mu/}}(M)=\bigcap\mscr{H}_{C_i}(M)$ and $\mscr{H}_{\mc{B}_{\perp}}(M)=\bigcup\mscr{H}(C_i|C_j)$;
\item each fibre $\{\ast\}\x\mc{B}_{\perp}$ intersects all of the $C_i$.
\end{enumerate}
\end{cor}
\begin{proof}
The only statement that has not already been proved is part~(2). If it were false, there would exist an index $i$ and $\mf{h}\in\mscr{H}(M)$ such that $C_i\cu\mf{h}$ and $\{\ast\}\x\mc{B}_{\perp}\cu\mf{h}^*$. Since $C_i\cu\mf{h}$, we have $\mf{h}\not\in\mscr{H}_{\mc{B}_{/\mkern-5mu/}}(M)$, so $\mc{B}\cu\mf{h}^*$. Hence $\mf{h}^*\in\mc{H}$, contradicting the fact that $C_i\cu\mf{h}$.
\end{proof}

Recall the notation $\mc{PD}(M)$ and $\mc{D}(M)$ for compatible (pseudo-)metrics, as in Subsection~\ref{compatible subsec}.

\begin{rmk}\label{constant distance on fibres}
If $\eta\in\mc{PD}(M)$ and $x,y\in\mc{B}$ lie in the same fibre $\mc{B}_{/\mkern-5mu/}\x\{\ast\}$, then $\eta(x,C_i)=\eta(y,C_i)$ for all $1\leq i\leq k$. Indeed, since $\mscr{H}(x|y)\cu\mscr{H}_{\mc{B}_{/\mkern-5mu/}}(M)=\bigcap\mscr{H}_{C_i}(M)$, we have $\mscr{W}(x|C_i)=\mscr{W}(y|C_i)$ and it follows (e.g.\ by Remark~\ref{from pseudo-metrics to measures}) that $\eta(x,\pi_i(x))=\eta(y,\pi_i(y))$ for every $\eta\in\mc{PD}(M)$.
\end{rmk}

\begin{rmk}\label{distance to multi-bridge}
If $\eta\in\mc{PD}(M)$, then $\eta(x,\mc{B})\leq r\cdot\max_i\eta(x,C_i)$ for every $x\in M$.

In order to see this, let $\mf{h}_1,\dots,\mf{h}_k$ be the minimal elements of $\mscr{H}(x|\mc{B})$. Since the $\mf{h}_i$ are pairwise transverse and $\rk M=r$, we have $k\leq r$. Note that each $\mf{h}_i$ must lie in $\mc{H}$, hence there exists an index $j_i$ such that $C_{j_i}\cu\mf{h}_i$. It follows that:
\[\mscr{H}(x|\mc{B})\cu\bigcup\mscr{H}(x|\mf{h}_i)\cu\bigcup\mscr{H}(x|C_{j_i}).\]
Hence $\eta(x,\mc{B})\leq k\cdot\max_i\eta(x,C_i)\leq r\cdot\max_i\eta(x,C_i)$.
% reference Remark~\ref{from pseudo-metrics to measures}?
\end{rmk}

\begin{rmk}\label{compact parallel multi-bridge}
If $\delta\in\mc{D}(M)$ and $(M,\delta)$ is complete, then $\mc{B}_{\perp}$ is compact in $(M,\delta)$.

In order to prove this, let $x_{i,j}\in C_i$ and $x_{j,i}\in C_j$ be a pair of gates for all distinct $1\leq i,j\leq k$. Let $K$ be the convex hull of the finite set $F=\{x_{i,j}\mid 1\leq i,j\leq k\}$. Recall that $K=\mc{J}^r(F)$ by Remark~\ref{Bow J}, so it follows from \cite[Corollary~2.20]{Fio1} that $K$ is compact.

We have $K\cap\mc{B}\neq\emptyset$. Otherwise, the set $\mscr{H}(K|\mc{B})$ would be nonempty and contained in $\mc{H}$. However, each element of $\mc{H}$ contains some $C_i$, so it cannot be disjoint from $K$. 

Finally, observing that $\mscr{H}_K(M)$ contains the set
\[\bigcup\mscr{H}(x_{i,j}|x_{j,i})=\bigcup\mscr{H}(C_i|C_j)=\mscr{H}_{\mc{B}_{\perp}}(M),\]
we deduce that $K\cap\mc{B}$ must contain a fibre $\{\ast\}\x\mc{B}_{\perp}$. Since $\mc{B}_{\perp}$ is gate-convex, it must be a closed subset of $K$, hence it is compact too.
\end{rmk}

Now, let $S\cu\aut M$ be a finite set of automorphisms acting non-transversely and stably without inversions. By Theorem~\ref{all from 10b -2}(1), the reduced cores $\overline{\mc{C}}(s)$ of $s\in S$ are all gate-convex. Let $\mc{B}(S)$ be their multi-bridge. 

\begin{defn}\label{multi-bridge of S defn}
We refer to $\mc{B}(S)$ as the \emph{multi-bridge} of the finite set $S\cu\aut M$.
\end{defn}

Recalling the notation introduced in Subsection~\ref{identities sect}, we have:

\begin{prop}\label{multi-bridge of S}
Let $S\cu\aut M$ be a finite set of automorphisms acting non-transversely and stably without inversions. The multi-bridge $\mc{B}(S)$ is gate-convex and, for all $\eta\in\mc{PD}(M)^{\langle S\rangle}$:
\begin{enumerate}
\item we have $\tau_S^{\eta}(\pi_{\mc{B}}(x))\leq\tau_S^{\eta}(x)$ for all $x\in M$, where $\pi_{\mc{B}}\colon M\ra\mc{B}(S)$ is the gate-projection;
\item $\tau_S^{\eta}(\cdot)$ is constant on each fibre $\mc{B}_{/\mkern-5mu/}(S)\x\{\ast\}$;
\item if $\delta\in\mc{D}(M)^{\langle S\rangle}$ and $(M,\delta)$ is complete, then there exists $z\in\mc{B}(S)$ with $\tau_S^{\delta}(x)=\overline\tau_S^{\delta}$.
\end{enumerate}
\end{prop}
\begin{proof}
Since the multi-bridge $\mc{B}(S)$ intersects each $\overline{\mc{C}}(s)$, we have $\mscr{H}(\pi_{\mc{B}}(x)|\overline{\mc{C}}(s))\cu\mscr{H}(x|\overline{\mc{C}}(s))$ for all $x\in M$. Hence $\eta(\pi_{\mc{B}}(x),\overline{\mc{C}}(s))\leq\eta(x,\overline{\mc{C}}(s))$. %, since $\overline{\mc{C}}(s)$ is gate-convex. 
Theorem~\ref{all from 10b -2}(2) now implies that $\tau_S^{\eta}(\pi_{\mc{B}}(x))\leq\tau_S^{\eta}(x)$, proving part~(1). By Remark~\ref{constant distance on fibres}, if $x,y\in\mc{B}(S)$ lie in the same fibre $\mc{B}_{/\mkern-5mu/}(S)\x\{\ast\}$, then $\eta(x,\overline{\mc{C}}(s))=\eta(y,\overline{\mc{C}}(s))$. This proves part~(2). Finally, part~(3) follows from the previous two parts and Remark~\ref{compact parallel multi-bridge}.
\end{proof}

\begin{ex}
Let $G=\langle a,b\rangle$ be the free group over two generators. Let $T$ be the standard Cayley graph of $G$, with all edges of length $1$. Let $(X,\delta)$ be the (incomplete) median space obtained by removing from $T$ all midpoints of edges. Then, taking $S=\{a,bab^{-1}\}\cu G\cu\isom X$, there is no point $x\in X$ with $\tau^{\delta}_S(x)=\overline\tau_S^{\delta}=2$.
\end{ex}

Our interest in multi-bridges is due to the following result, which helps us understand the behaviour on $M$ of the functions $\tau_S^{\eta}(\cdot)$ for $\eta\in\mc{PD}(M)^{\langle S\rangle}$.

\begin{prop}\label{properties of B(S)}
Let $S\cu\aut M$ be a finite set of automorphisms acting non-transversely and stably without inversions. Recall that $r=\rk M$. Then, the following hold for every $\eta\in\mc{PD}(M)^{\langle S\rangle}$.
\begin{enumerate}
\item If $s_1,s_2\in S$, then $\eta(\overline{\mc{C}}(s_1),\overline{\mc{C}}(s_2))\leq \overline\tau^{\eta}_S$.
\item If $s\in S$ and $x\in\mc{B}(S)$, then $\eta(x,\overline{\mc{C}}(s))\leq r\overline\tau^{\eta}_S$.
\item If $x\in\mc{B}(S)$, then $\tau^{\eta}_S(x)\leq (2r+1)\overline\tau^{\eta}_S$.
\item The $\eta$--diameter of each fibre $\{\ast\}\x\mc{B}_{\perp}(S)$ is at most $r^2\overline\tau_S^{\eta}$.
\item If $x\in M$, then $\eta(x,\mc{B}(S))\leq\tfrac{r}{2}\tau_S^{\eta}(x)$.
\item For any $x\in M$ and any fibre $P=\mc{B}_{/\mkern-5mu/}(S)\x\{\ast\}$, we have $\eta(x,P)\leq 2r^2\tau_S^{\eta}(x)$.
\end{enumerate}
\end{prop}
\begin{proof}
We begin with part~(1). For every $x\in M$, we have:
\[\mscr{W}(\overline{\mc{C}}(s_1)|\overline{\mc{C}}(s_2))=\mscr{W}(x,\overline{\mc{C}}(s_1)|\overline{\mc{C}}(s_2))\sqcup\mscr{W}(\overline{\mc{C}}(s_1)|\overline{\mc{C}}(s_2),x)\cu \mscr{W}(x|\overline{\mc{C}}(s_1))\sqcup\mscr{W}(x|\overline{\mc{C}}(s_2)). \]
Along with Theorem~\ref{all from 10b -2}(2), this implies that:
\[\tfrac{1}{2}\eta(\overline{\mc{C}}(s_1),\overline{\mc{C}}(s_2))\leq\max\{\eta(x,\overline{\mc{C}}(s_1)),\eta(x,\overline{\mc{C}}(s_2))\}\leq\tfrac{1}{2}\max\{\eta(x,s_1x),\eta(x,s_2x)\}\leq\tfrac{1}{2}\tau_S^{\eta}(x).\]
Part~(1) follows by taking an infimum over $x\in M$.

Let us prove part~(2). If $x\in\mc{B}(S)$ and $s\in S$, Corollary~\ref{multibridge cor}(2) implies that $\mscr{H}(x|\overline{\mc{C}}(s))$ is contained in the union of the sets $\mscr{H}(\overline{\mc{C}}(t)|\overline{\mc{C}}(s))$ with $t\in S\setminus\{s\}$. The maximal halfspaces in $\mscr{H}(x|\overline{\mc{C}}(s))$ are pairwise-transverse, so there are at most $r$ of them. Hence, there exist $t_1,\dots,t_r\in S$ such that $\Om:=\bigcup_i\mscr{H}(\overline{\mc{C}}(t_i)|\overline{\mc{C}}(s))$ contains every maximal element of $\mscr{H}(x|\overline{\mc{C}}(s))$. In particular, $\mscr{H}(x|\overline{\mc{C}}(s))\cu\Om$ and part~(1) yields $\eta(x,\overline{\mc{C}}(s))\leq r\overline\tau^{\eta}_S$. 

Part~(3) of the proposition now follows from Theorem~\ref{all from 10b -2}(2):
\[ \tau^{\eta}_S(x)=\max_{s\in S}[\ell(s,\eta)+2\eta(x,\overline{\mc{C}}(s))]\leq\max_{s\in S}[\overline\tau^{\eta}_S+2r\overline\tau^{\eta}_S]=(2r+1)\overline\tau^{\eta}_S.\]

Regarding part~(4), consider two points $x,y$ lying in the same fibre $\{\ast\}\x\mc{B}_{\perp}(S)$. Let $\mf{h}_1,\dots,\mf{h}_k$ be the minimal elements of $\mscr{H}(x|y)$. Since $\rk M=r$, we have $k\leq r$. By definition of $\mc{B}_{\perp}(S)$, there exist elements $s_i\in S$ with $\overline{\C}(s_i)\cu\mf{h}_i$. Thus: 
\[\mscr{H}(x|y)\cu\bigcup\mscr{H}(x|\mf{h}_i)\cu\bigcup\mscr{H}(x|\overline{\C}(s_i)).\]
Using part~(2) of the proposition, it follows that $\eta(x,y)\leq k\cdot\max_s\eta(x,\overline{\C}(s))\leq kr\overline\tau_S^{\eta}\leq r^2\overline\tau_S^{\eta}$.

Finally, part~(5) is a consequence of Remark~\ref{distance to multi-bridge} and the fact, due to Theorem~\ref{all from 10b -2}(2), that $\tau_S^{\eta}(x)\geq 2\eta(x,\overline{\C}(s))$ for every $s\in S$. Part~(6) is obtained by combining parts~(4) and~(5):
\[\eta(x,P)\leq\eta(x,\mc{B}(S))+r^2\overline\tau_S^{\eta}\leq \tfrac{r}{2}\tau_S^{\eta}(x)+r^2\overline\tau_S^{\eta}\leq 2r^2\tau_S^{\eta}(x). \qedhere\]
\end{proof}

\subsection{Promoting median automorphisms to homotheties.}\label{metrics on algebras}

Recall that $M$ is a median algebra of finite rank $r$. In this subsection, we consider subgroups $G\lhd U\leq\aut M$, with the goal of constructing $G$--invariant compatible pseudo-metrics $\eta\in\mc{PD}^G(M)$ with respect to which $U$ acts by homotheties and $G$ is non-elliptic. In general, this will only be possible after passing to a subalgebra of $M$. The final result in this direction is Corollary~\ref{homothety 1}.

Our main technical tools are the notion of multi-bridge (exploited in Lemma~\ref{c(x) lemma new}) and the Lefschetz fixed point theorem applied to projectivisations of certain cones $\mc{C}$ in the topological vector space $\mc{PD}^G(M)$ (Proposition~\ref{finding eigenvectors new}). Some extra work is required in order to ensure that our cones $\mc{C}$ have compact projectivisation and that they only contain pseudo-metrics $\eta$ for which $G$ acts non-elliptically (i.e.\ $\overline\tau_S^{\eta}>0$ for some/any generating set $S\cu G$).

\subsubsection{Preliminaries on normed spaces and ARs.}

\begin{defn}
Let $V$ be a real vector space. 
\begin{enumerate}
\item A \emph{cone} is a convex subset $\mc{C}\cu V$ that is closed under multiplication by scalars in $[0,+\infty)$. 
\item A \emph{positive cone} is a cone $\mc{C}\cu V$ for which $\mc{C}\setminus\{0\}$ is convex. Equivalently, $\mc{C}\cap(-\mc{C})=\{0\}$.
\item The \emph{projectivisation} $\mbb{P}(\mc{C})$ of a cone $\mc{C}$ is the quotient of $\mc{C}\setminus\{0\}$ obtained by identifying points that differ by multiplication by a scalar. 
\end{enumerate}
\end{defn}

Given a countable probability space $(\Om,\s)$ and a function $f\colon\Om\ra\R$, recall that:
\begin{align*}
\|f\|_1&=\sum_{\om\in\Om}|f(\om)|\s(\om), & \|f\|_{\infty}&=\sup_{\om\in\Om}|f(\om)|.
\end{align*}
We denote by $\ell^1(\Om,\s)$ and $\ell^{\infty}(\Om)$ the spaces of functions where $\|\cdot\|_1$ and $\|\cdot\|_{\infty}$ are finite, respectively. 

The next result collects a few simple observations that will be useful later in this subsection. In particular, part~(3) will be our compactness criterion for projectivised cones: we only need to ensure that $\|\cdot\|_1$ and $\|\cdot\|_{\infty}$ are bi-Lipschitz equivalent on the cone. This is one of the reasons why we are forced to work with both norms $\|\cdot\|_1$ and $\|\cdot\|_{\infty}$.

\begin{lem}\label{general L^1}
Let $(\Om,\s)$ be a countable set with a fully-supported probability measure.
\begin{enumerate}
\item We have $\ell^{\infty}(\Om)\cu \ell^1(\Om,\s)$ and $\|\cdot\|_1\leq\|\cdot\|_{\infty}$.
\item The topology of $(\ell^1(\Om,\s),\|\cdot\|_1)$ is finer than the topology of pointwise convergence on $\Om$. The converse holds on those subsets of $\ell^1(\Om,\s)$ where $\|\cdot\|_{\infty}$ is bounded.
\item Let $\mc{C}\cu \ell^1(\Om,\s)$ be a positive cone that is closed in the topology of $\|\cdot\|_1$.
Suppose that there exists $c>0$ such that $\|f\|_{\infty}\leq c\cdot\|f\|_1$ for all $f\in\mc{C}$. Then $\mbb{P}(\mc{C})$ is compact with respect to the quotient topology induced by $\|\cdot\|_1$. 
\end{enumerate}
\end{lem}
\begin{proof}
Part~(1) is clear. The two halves of part~(2) respectively follow from the inequalities:
\begin{align*}
|f(\om)|\s(\{\om\})&\leq\|f\|_1, & \|f\|_1&\leq\sum_{x\in F}|f(x)|\s(\{x\}) + \|f\|_{\infty}\cdot \s(\Om\setminus F),
\end{align*}
which hold for all $f\in \ell^1(\Om,\s)$, all $\om\in\Om$ and every finite subset $F\cu\Om$.

Finally, let us prove part~(3). If $S$ is the unit sphere in $\ell^1(\Om,\s)$, then $\mbb{P}(\mc{C})$ is homeomorphic to $\mc{C}\cap S$. 
% here we use positivity of the cone to ensure that the map from $\mc{C}$ to $\mbb{P}(\mc{C})$ is injective
Since the latter is metrisable, it suffices to show that every sequence $(f_k)_k\cu\mc{C}\cap S$ has a converging subsequence. Since $\|f_k\|_{\infty}\leq c\cdot\|f_k\|_1=c$, the sequence $(f_k(\om))_k$ takes values in the compact interval $[-c,c]$ for all $\om\in\Om$.  Since $\Om$ is countable, a diagonal argument allows us to replace $(f_k)_k$ with a subsequence that converges pointwise to a function $f\colon\Om\ra [-c,c]$. Thus, part~(2) shows that $\|f_k-f\|_1\ra 0$. Since $\mc{C}$ is closed in $\ell^1(\Om,\s)$, we have $f\in\mc{C}\cap S$, as required.
\end{proof}

\begin{defn}
A metrisable topological space $X$ is an \emph{absolute retract (AR)} if it enjoys the following property. For every metrisable topological space $Y$ and every closed subset $A\cu Y$ homeomorphic to $X$, there exists a continuous retraction $Y\ra A$.
\end{defn}

The following summarises the key properties of ARs that we will need. 

\begin{thm}\label{all about ARs}
\begin{enumerate}
\item[]
\item Let $X$ be a compact AR. Then every continuous map $f\colon X\ra X$ has a fixed point.
\item Let $(E,\|\cdot\|)$ be a normed space. If $\mc{C}\cu E$ is any positive cone, then $\mbb{P}(\mc{C})$ is an AR (with the quotient of the norm topology of $E$).
\end{enumerate}
\end{thm}
\begin{proof}
Part~(1) is a consequence of the Lefschetz fixed point theorem for compact ANRs \cite{Lefschetz1,Lefschetz2}. See e.g.\ Theorem~III.7.4 and Section~I.6 in \cite{Hu} for a clear statement.
% also see Proposition~3.1.3 and Theorem~3.1.5 in Gorniewicz--Rozploch-Nowakowska ``On the Schauder fixed point theorem''
% or Granas ``Generalizing the Hopf-Lefschetz fixed point theorem for non-compact ANRs''

If $S$ is the unit sphere in the normed space $E$, then $\mbb{P}(\mc{C})$ is homeomorphic to $\mc{C}\cap S$. %which is metrisable
Recall that every convex subset of a normed space is an AR (see e.g.\ \cite[Corollary~4.2]{Dugundji} or Corollary~II.14.2 and Theorem~III.3.1 in \cite{Hu}). 
%Dugundji explicitly remarks that the convex set needs not be closed right after the corollary
Every retract of an AR is again an AR \cite[Proposition~7.7]{Hu}. Thus, part~(2) is immediate from the observation that $\mc{C}\cap S$ is a retract of the convex set $\mc{C}\setminus\{0\}$.
\end{proof}
% NOTE: we could have used the Schauder fixed point thm (for compact convex subsets of normed spaces) or even the Tychonoff fixed point thm (in locally convex topological vector spaces). See https://en.wikipedia.org/wiki/Locally_convex_topological_vector_space. 
% I guess we would have to cut the cone by an affine subspace so that we get a convex thing in there homeomorphic to the projectivisation, which might require a bit of care, but should be doable.
% In any case, it's not worth writing it this way because the reasons why we use two norms are different: it's mostly needed to ensure compactness of the projectivisation and that $\overline\tau$ is positive.

\subsubsection{Finding a projectively invariant metric.}\label{middle subsub}
\hfill \smallskip \\ 
Let $M$ be a {\bf countable}, finite-rank median algebra. Consider a finite set $S\cu\aut M$ and let $G\leq\aut M$ be the subgroup that it generates. Let $\alpha\in\aut M$ be an element that normalises $G$. 

Consider the locally convex real vector space $\mc{E}(M)=\R^{M\x M}$, endowed with the topology of pointwise convergence on $M\x M$. We have a continuous linear action $\aut M\acts\mc{E}(M)$ given by
\[(\psi\cdot f)(x,y)=f(\psi^{-1}(x),\psi^{-1}(y)),\ \forall\psi\in\aut M,\ \forall f\in\mc{E}(M),\ \forall x,y\in M.\]

\begin{rmk}\label{PD closed}
The sets $\mc{PD}(M)$ and $\mc{PD}^G(M)$ (introduced in Subsection~\ref{compatible subsec}) are closed positive cones in $\mc{E}(M)$. In addition, $\mc{PD}(M)$ is $(\aut M)$--invariant and $\mc{PD}^G(M)$ is $\langle\alpha\rangle$--invariant. 

Although $\mc{D}(M)\cup\{0\}$ also is a positive cone, it is only closed when $M$ is a single point. % can collapse walls ad libitum
\end{rmk}

Given a function $\cc\colon M\x M\ra(0,+\infty)$, consider the (not necessarily convex) subset:
\[\mc{PD}_{\cc}^G(M):=\{\eta\in\mc{PD}^G(M) \mid \eta(x,y)\leq\cc(x,y)\cdot\overline\tau_S^{\eta},\ \forall x,y\in M\}.\]
As we shall see, this serves two purposes: on the one hand all closed cones in $\mc{PD}_{\cc}^G(M)$ have compact projectivisation; on the other, they only contain pseudo-metrics with $\overline\tau_S^{\eta}>0$ (except for $\eta=0$).

Our main aim in this subsubsection is to prove the following result:

\begin{prop}\label{finding eigenvectors new}
Suppose that, for some $\cc\colon M\x M\ra(0,+\infty)$, there exists a nontrivial % i.e. not just \{0\}
 $\langle\alpha\rangle$--invariant cone $\mc{C}\cu\mc{PD}_{\cc}^G(M)$ that is closed in $\mc{E}(M)$ with respect to the topology of pointwise convergence. Then there exists $\eta\in\mc{C}\setminus\{0\}$ such that $\overline\tau_S^{\eta}>0$ and $\alpha\cdot\eta=\l\eta$ for some $\l>0$.
\end{prop}

In order to prove the proposition, let us fix a probability measure $\s$ on $M$ with full support. Given a function $\cc\colon M\x M\ra(0,+\infty)$, we define for $f\in\mc{E}(M)$:
\begin{align*}
\|f\|_1^{\cc}:=&\sum_{x,y\in M}\frac{|f(x,y)|}{\cc(x,y)}\s(x)\s(y), & \|f\|_{\infty}^{\cc}:=&\sup_{x,y\in M}\frac{|f(x,y)|}{\cc(x,y)}.
\end{align*}
Note that $\|f\|_1^{\cc}$ is a norm on the subspace $\mc{E}^1_{\cc}(M)\cu\mc{E}(M)$ where it is finite (the same is true of $\|f\|_{\infty}^{\cc}$, but this will not be relevant to us).

\begin{rmk}
Rescaling functions $f\in\mc{E}(M)$ by $\cc$, we map $\mc{E}^1_{\cc}(M)$ linearly isometrically onto $\ell^1(M\x M,\s\otimes\s)$ while taking $\|f\|_{\infty}^{\cc}$ to $\|f\|_{\infty}$. Thus, we can apply Lemma~\ref{general L^1} in this context.
\end{rmk}

\begin{lem}\label{controlling each other new}
Consider a function $\cc\colon M\x M\ra (0,+\infty)$.
\begin{enumerate}
\item The subset $\mc{PD}_{\cc}^G(M)\cu\mc{E}(M)$ is closed under pointwise convergence.
\item There exists a constant $c>0$ (depending on $\cc$ and $\s$) such that, for every $\eta\in\mc{PD}_{\cc}^G(M)$:
\[\|\eta\|_1^{\cc}\leq\|\eta\|_{\infty}^{\cc}\leq\overline\tau_S^{\eta}\leq c\cdot\|\eta\|_1^{\cc}.\]
\end{enumerate}
\end{lem}
\begin{proof}
We begin with part~(1). First, observe that the function $\eta\mapsto\overline\tau^{\eta}_S$ is upper semicontinuous. Indeed, if $\eta_n\in\mc{PD}^G(M)$ converge pointwise to some $\eta\in\mc{PD}^G(M)$, then, for every $x\in M$:
\[\max_{s\in S}\eta(x,sx)=\lim_{n\ra+\infty}\max_{s\in S}\eta_n(x,sx)\geq\limsup_{n\ra+\infty}\overline\tau_S^{\eta_n}.\]
Hence $\overline\tau_S^{\eta}\geq\limsup\overline\tau_S^{\eta_n}$, which proves upper semicontinuity. Now, if $\eta_n\in\mc{PD}_{\cc}^G(M)$, then
\[\eta(x,y)=\lim_{n\ra+\infty}\eta_n(x,y)\leq\limsup_{n\ra+\infty}\cc(x,y)\cdot\overline\tau_S^{\eta_n}\leq\cc(x,y)\cdot\overline\tau_S^{\eta},\]
for all $x,y\in M$. Along with Remark~\ref{PD closed}, this yields $\eta\in\mc{PD}_{\cc}^G(M)$, proving part~(1).

Regarding part~(2), the first inequality is in Lemma~\ref{general L^1}(1) and the second is immediate from the fact that $\eta\in\mc{PD}_{\cc}^G(M)$. In order to prove the third one, choose any point $x_0\in M$. Then:
\[\overline\tau_S^{\eta}=\inf_{x\in M}\max_{s\in S}\eta(x,sx)\leq\max_{s\in S}\eta(x_0,sx_0)\leq \|\eta\|_1^{\cc}\cdot \max_{s\in S}\frac{\cc(x_0,sx_0)}{\s(\{x_0\})\s(\{sx_0\})}.\]
The constant appearing on the rightmost side is positive and well-defined, since $\cc$ takes positive values and $\s$ has full support. This concludes the proof.
\end{proof}

\begin{proof}[Proof of Proposition~\ref{finding eigenvectors new}]
We want to apply the Lefschetz fixed point theorem to $\alpha\colon\mbb{P}(\mc{C})\ra\mbb{P}(\mc{C})$.

Since $\mc{C}\cu\mc{PD}^G(M)$, the cone $\mc{C}$ is actually a positive cone. By Lemma~\ref{controlling each other new}(2), the set $\mc{C}$ is contained in $\mc{E}^1_{\cc}(M)$. Thus, Theorem~\ref{all about ARs}(2) shows that the projectivisation $\mbb{P}(\mc{C})$, endowed with the quotient topology induced by $\|\cdot\|^{\cc}_1$, is an AR. 

Since $\mc{C}\cu\mc{E}^1_{\cc}(M)$ is closed in the topology of pointwise convergence, the first half of Lemma~\ref{general L^1}(2) guarantees that $\mc{C}$ is also closed in the topology of $\|\cdot\|^{\cc}_1$. Thus, by Lemma~\ref{controlling each other new}(2) and Lemma~\ref{general L^1}(3), the projectivisation $\mbb{P}(\mc{C})$ is compact. 

We are left to show that the action $\langle\alpha\rangle\acts\mc{C}$ is continuous with respect to the topology of $\|\cdot\|^{\cc}_1$. Note that, by Lemma~\ref{controlling each other new}(2), $\alpha$ takes $\|\cdot\|^{\cc}_1$--bounded subsets of $\mc{C}\cu\mc{PD}_{\cc}^G(M)$ to $\|\cdot\|^{\cc}_1$--bounded subsets of $\mc{C}$:
\[\|\alpha\cdot\eta\|_1^{\cc}\leq\overline\tau_S^{\alpha\cdot\eta}=\inf_{x\in M}\max_{s\in S}\eta(\alpha^{-1}x,\alpha^{-1}sx)=\overline\tau_{\alpha^{-1}S\alpha}^{\eta}\leq |\alpha^{-1}S\alpha|_S\cdot\overline\tau_S^{\eta}\leq c|\alpha^{-1}S\alpha|_S\cdot\|\eta\|_1^{\cc}.\]
% We cannot use this to conclude that $\alpha$ is Lipschitz, as differences of elements of $\mc{C}$ might not lie in $\mc{C}$

Since the topology given by $\|\cdot\|^{\cc}_1$ is metrisable, it suffices to show that $\alpha\colon\mc{C}\ra\mc{C}$ is sequentially continuous. Let $\eta_n\in\mc{C}$ be a sequence that $\|\cdot\|^{\cc}_1$--converges to $\eta\in\mc{C}$. By Lemma~\ref{general L^1}(2), $\eta_n$ converges to $\eta$ pointwise. Since the action $\aut M\acts\mc{E}(M)$ is continuous, the sequence $\alpha\cdot\eta_n$ converges to $\alpha\cdot\eta$ pointwise. Note that the set $\{\eta_n\}_{n\geq 0}\cup\{\eta\}$ is $\|\cdot\|^{\cc}_1$--bounded and, by the above observation, so must be $\{\alpha\cdot\eta_n\}_{n\geq 0}\cup\{\alpha\eta\}$. By Lemma~\ref{controlling each other new}(2), this set is also $\|\cdot\|^{\cc}_{\infty}$--bounded, so Lemma~\ref{general L^1}(2) shows that $\alpha\cdot\eta_n$ $\|\cdot\|^{\cc}_1$--converges to $\alpha\cdot\eta$, as required.

In conclusion, $\alpha$ induces a homeomorphism of the compact AR $\mbb{P}(\mc{C})$. Theorem~\ref{all about ARs}(1) yields an $\langle\alpha\rangle$--fixed point $[\eta]\in\mbb{P}(\mc{C})$. The fact that $\overline\tau_S^{\eta}>0$ is clear since $\eta\in\mc{PD}_{\cc}^G(M)\setminus\{0\}$. 
\end{proof}

In fact, Proposition~\ref{finding eigenvectors new} can be easily generalised to extensions of $G$ by abelian groups.

\begin{cor}\label{finding eigenvectors abelian}
Let $U\leq\aut M$ be a countable subgroup such that $G\lhd U$, with abelian quotient $U/G$; let $p\colon U\ra A$ be the quotient projection. Suppose that, for some $\cc$, there exists a nontrivial, $U$--invariant, closed cone $\mc{C}\cu\mc{PD}_{\cc}^G(M)$. Then there exists $\eta\in\mc{C}\setminus\{0\}$ with $\overline\tau_S^{\eta}>0$ and a homomorphism $\l\colon A\ra (\R_{>0},\ast)$ such that $u\cdot\eta=\l(p(u))\eta$ for all $u\in U$.
\end{cor}
\begin{proof}
Let $\{a_i\}_{i\geq 0}$ be a generating set for $A$. Consider the subgroups $A_n:=\langle a_i\mid i<n\rangle$ and $U_n:=p^{-1}(A_n)$; in particular, $A_0=\{1\}$ and $U_0=G$. We will show by induction on $n\geq 0$ that there exist nontrivial, $U$--invariant, closed cones $\mc{C}_n\cu\mc{PD}_{\cc}^G(M)$ and homomorphisms $\l_n\colon A_n\ra(\R_{>0},\ast)$ such that $u\cdot\eta=\l_n(p(u))\eta$ for all $\eta\in\mc{C}_n$ and $u\in U_n$. As base step, set $\mc{C}_0:=\mc{C}$. 

Regarding the inductive step, suppose that we have constructed $\mc{C}_n$ and $\l_n$. By Proposition~\ref{finding eigenvectors new}, there exists a point $[\eta_{n+1}]\in\mbb{P}(\mc{C}_n)$ fixed by $p^{-1}(a_{n+1})$. In fact, since $U_n$ acts trivially on $\mbb{P}(\mc{C}_n)$, the entire group $U_{n+1}$ fixes $[\eta_{n+1}]$ and there exists a homomorphism $\l_{n+1}\colon A_{n+1}\ra(\R_{>0},\ast)$ such that $u\cdot\eta_{n+1}=\l_{n+1}(p(u))\eta_{n+1}$ for all $u\in U_{n+1}$. We can then define $\mc{C}_{n+1}$ as the closed cone: 
\[\{\eta\in\mc{C}_n\mid u\cdot\eta=\l_{n+1}(p(u))\eta,\ \forall u\in U_{n+1}\}.\] 
Since $U\acts\mc{C}$ factors through the abelian group $A$, this cone is $U$--invariant, as required.

Finally, when $A$ is not finitely generated, note that the intersection of the descending chain $\mc{C}_n$ is not just $\{0\}$. This is because, as we observed in the proof of Proposition~\ref{finding eigenvectors new}, the sets $\mbb{P}(\mc{C}_n)$ are compact. This concludes the proof.
\end{proof}

\subsubsection{Universal uniform non-elementarity.}\label{WNE subsubsect}
\hfill \smallskip \\ 
Let $G\acts M$ be an action by automorphisms on a median algebra of finite rank $r$. Consider the following strengthening of Definition~\ref{UNE defn} in the context of compatible metrics on median algebras:

\begin{defn}\label{WNE defn}
The action $G\acts M$ is \emph{universally uniformly non-elementary (WNE)} if there exists a constant $c>0$ such that, for every $\eta\in\mc{PD}^G(M)$, the action $G\acts(M,\eta)$ is $c$--UNE. 
\end{defn}

This may seem an impossibly strong requirement to impose on $G\acts M$, but we will see in Corollary~\ref{from sub-ultras to X new cor} that many actions arising from ultralimits of Salvetti complexes are WNE.

\begin{lem}\label{c(x) lemma new}
Let $G\leq\aut M$ be generated by a finite set $S$ of automorphisms acting non-transversely and stably without inversions. Let $G\lhd U\leq\aut M$. Pick a point $q$ in the multi-bridge $\mc{B}(S)\cu M$ and let $\mf{M}\cu M$ be the median subalgebra generated by the orbit $U\cdot q$. 
Then:
\begin{enumerate}
\item there exists $\cc_1\colon\mf{M}\ra (0,+\infty)$ such that $\tau_S^{\eta}(x)\leq\cc_1(x)\cdot\overline\tau^{\eta}_S$ for all $\eta\in\mc{PD}^G(M)$ and $x\in\mf{M}$;
\item if $G\acts M$ is WNE, there exists $\cc_2\colon\mf{M}\x\mf{M}\ra (0,+\infty)$ such that $\eta(x,y)\leq \cc_2(x,y)\cdot\overline\tau^{\eta}_S$ for all $\eta\in\mc{PD}^G(M)$ and $x,y\in\mf{M}$.
\end{enumerate}
\end{lem}
\begin{proof}
We only prove part~(1), since part~(2) then follows by setting $\cc_2(x,y):=c\cdot(\cc_1(x)+\cc_1(y))$, for a constant $c$ as in Definition~\ref{WNE defn}. 

If part~(1) holds for points $x,y,z\in\mf{M}$, then it holds for their median $m(x,y,z)$. Indeed, we can take $\cc_1(m(x,y,z))=\cc_1(x)+\cc_1(y)+\cc_1(z)$ and we have:
\begin{align*}
\tau_S^{\eta}(m(x,y,z))&=\max_{s\in S}\eta(m(x,y,z),m(sx,sy,sz))\leq\max_{s\in S}[\eta(x,sx)+\eta(y,sy)+\eta(z,sz)] \\
&\leq\tau_S^{\eta}(x)+\tau_S^{\eta}(y)+\tau_S^{\eta}(z)\leq [\cc_1(x)+\cc_1(y)+\cc_1(z)]\cdot\overline\tau_S^{\eta}.
\end{align*}

Thus, it suffices to prove part~(1) for $x\in U\cdot q$. Since $q\in\mc{B}(S)$, we have $uq\in\mc{B}(uSu^{-1})$ for all $u\in U$. Moreover, since $U$ normalises $G$, the set $uSu^{-1}$ is just another generating set of $G$. By Proposition~\ref{properties of B(S)}(3), we have: 
\begin{align*}
\tau_S^{\eta}(uq)\leq|S|_{uSu^{-1}}\cdot \tau^{\eta}_{uSu^{-1}}(uq)&\leq |S|_{uSu^{-1}}\cdot(2r+1)\overline\tau^{\eta}_{uSu^{-1}} \\
&\leq |S|_{uSu^{-1}}\cdot(2r+1)\cdot|uSu^{-1}|_{S}\cdot\overline\tau^{\eta}_{S}.
\end{align*}
So we can take $\cc_1(uq)=(2r+1)\cdot|S|_{uSu^{-1}}\cdot|uSu^{-1}|_{S}$. This concludes the proof.
\end{proof}

\begin{cor}\label{homothety 1}
Let $G\leq\aut M$ be generated by a finite set $S$ of automorphisms acting non-transversely and stably without inversions. Suppose that $G\acts M$ is WNE and that $\mc{D}^G(M)\neq\emptyset$. Consider a countable subgroup $U\leq\aut M$ such that $G\lhd U$ and $U/G$ is abelian. Then there exist a nonempty, countable, $U$--invariant, median subalgebra $\mf{M}\cu M$, a pseudo-metric $\eta\in\mc{PD}^G(\mf{M})\setminus\{0\}$ with $\overline\tau_S^{\eta}>0$, and a homomorphism $\l\colon U\ra (\R_{>0},\ast)$ (trivial on $G$) with $u\cdot\eta=\l(u)\eta$ for all $u\in U$.
\end{cor}
\begin{proof}
Define the median subalgebra $\mf{M}\cu M$ as in the statement of Lemma~\ref{c(x) lemma new}. Since $\mf{M}$ is generated by a countable set, it is itself countable. The restriction map
\[{\rm res}_{\mf{M}}\colon\mc{PD}(M)\ra\mc{PD}(\mf{M})\]
takes $\mc{PD}^G(M)$ into $\mc{PD}^G(\mf{M})$ without decreasing the value of $\overline\tau_S^{\bullet}$. Thus, in the notation of Subsubsection~\ref{middle subsub}, Lemma~\ref{c(x) lemma new}(2) yields:
\[{\rm res}_{\mf{M}}(\mc{PD}^G(M))\cu\mc{PD}^G_{\cc_2}(\mf{M}).\]

Choose $\delta\in\mc{D}^G(M)$ and let $\mc{C}\cu\mc{D}^G(M)$ be the smallest cone containing the $U$--orbit of $\delta$. In other words, $\mc{C}$ is the convex hull of $U\cdot \delta$, saturated under multiplication by nonnegative scalars. Then ${\rm res}_{\mf{M}}(\mc{C})$ is a nontrivial $U$--invariant cone contained in $\mc{PD}_{\cc_2}^G(\mf{M})$.

Its closure $\overline{{\rm res}_{\mf{M}}(\mc{C})}\cu\mc{E}(\mf{M})$ in the topology of pointwise convergence is also a $U$--invariant cone. By Lemma~\ref{controlling each other new}(1), this is still contained in the set $\mc{PD}_{\cc_2}^G(\mf{M})$. We can thus apply Corollary~\ref{finding eigenvectors abelian}, obtaining $\eta\in\overline{{\rm res}_{\mf{M}}(\mc{C})}\setminus\{0\}$ with $\overline\tau_S^{\eta}>0$, and a homomorphism $\l\colon U\ra (\R_{>0},\ast)$ such that $u\cdot\eta=\l(u)\eta$ for all $u\in U$.
\end{proof}

\section{Ultralimits and coarse-median preserving automorphisms.}\label{ultra sect}

In this section we prove Theorem~\ref{special UNE thm intro} (Corollary~\ref{UNE cor}) and complete the proof of Theorem~\ref{Q2 thm intro} (Theorem~\ref{homothety 2}). Both results will follow quickly once we prove Theorem~\ref{from sub-ultras to X new new} in Subsection~\ref{meaty sect}, which can be viewed as the main goal of this entire section.

This theorem claims that, in many cases, if $G\acts M$ is an action of a special group on a median algebra, $\eta$ is a $G$--invariant compatible pseudo-metric and $C$ is a large $k$--cube in $M$, then any subset of $G$ that moves all points in $C$ by a lot less than the ``size'' of $C$ must commute with a copy of $\Z^k$ sitting inside $G$. This result holds, for instance, for co-special cubulations of $G$, ultralimits of these, and subalgebras thereof, with uniform constants that are independent of the specific choice of $\eta$. 

The case $k=1$ thus implies that all these actions are WNE (Definition~\ref{WNE defn}) and that centreless special groups are UNE (Definition~\ref{UNE defn}). The cases with $k>1$ ensure that the actions on median spaces that we will construct for Theorem~\ref{Q2 thm intro} are \emph{moderate}, as defined in the Introduction.

\subsection{The Bestvina--Paulin construction.}\label{BP sect}

As sketched in the Introduction, the first step in the proof of Theorem~\ref{Q2 thm intro} will involve a standard Bestvina--Paulin construction, with some additional issues caused by the lack of hyperbolicity. In this subsection, we discuss the role played by UNE groups (Definition~\ref{UNE defn}) in addressing these issues.

Consider a group $G$, a geodesic metric space $(X,d)$, and a homomorphism $\rho\colon G\ra\isom X$ inducing a proper cocompact action $G\acts X$ (we simply write $gx$ rather than $\rho(g)\cdot x$). 

\subsubsection{The classical Bestvina--Paulin construction.}
\hfill \smallskip \\ 
Fix a finite generating set $S\cu G$ and let $|\cdot|_S$ be the induced word length on $G$. Denote by $\pi\colon\aut G\ra\out G$ the quotient projection. Given $g,h\in G$, we write $\mf{c}[g](h):=ghg^{-1}$.

Every group automorphism $\varphi\colon G\ra G$ is bi-Lipschitz with respect to $|\cdot|_S$. By the Milnor--Schwarz lemma, $\varphi$ induces a quasi-isometry $\wt\varphi\colon X\ra X$ satisfying $\wt\varphi\o\rho(g)=\rho(\varphi(g))\o\wt\varphi$ for all $g\in G$.

Consider a sequence $\varphi_n\in\aut G$ and set $\rho_n:=\rho\o\varphi_n$ for all $n\geq 0$. Pick basepoints $p_n\in X$ with: 
\[\tau^{\rho_n}_S(p_n)-\overline\tau^{\rho_n}_S\leq 1.\] 
We introduce the quantities $\eps_n:=1/\overline\tau^{\rho_n}_S$ to simplify the notation.

\begin{ass}\label{tau diverges}
In the rest of Subsection~\ref{BP sect}, we assume that no two elements of the sequence $\pi(\varphi_n)\in\out G$ coincide. A classical argument due to Bestvina and Paulin (see e.g.\ \cite{Bestvina-Duke} and \cite[p.\ 338]{Paulin-arboreal}) then guarantees that $\eps_n\ra 0$ for $n\ra+\infty$.
\end{ass}

Fix a non-principal ultrafilter $\om$ and consider the ultralimit $(X_{\om},d_{\om},p_{\om})=\lim_{\om}(X,\eps_nd,p_n)$. We have a homomorphism $\rho_{\om}\colon G\ra\isom X_{\om}$ obtained as ultralimit of the actions $\rho_n$, namely: 
\[\rho_{\om}(g)\cdot (x_n)=(\rho_n(g)\cdot x_n)=(\varphi_n(g)x_n),\] 
for all $g\in G$ and $(x_n)\in X_{\om}$. This is well-defined since:
\begin{align*}
\lim_{\om}\eps_nd(\varphi_n(g)x_n,p_n)&\leq \lim_{\om}\eps_n[d(\varphi_n(g)x_n,\varphi_n(g)p_n)+d(\varphi_n(g)p_n,p_n)] \\
&\leq\lim_{\om}\eps_n\left[d(x_n,p_n)+|g|_S\cdot \tau_S^{\rho_n}(p_n)\right]=d_{\om}((x_n),p_{\om})+|g|_S<+\infty.
\end{align*}
One easily checks that $\tau_S^{\rho_{\om}}(p_{\om})=\overline\tau_S^{\rho_{\om}}=1$, so the action $G\acts X_{\om}$ induced by $\rho_{\om}$ does not have a global fixed point. 

\subsubsection{Automorphisms of UNE groups.}
\hfill \smallskip \\ 
Suppose for a moment that we are in the special case where there exists $\varphi\in\aut G$ such that $\varphi_n=\varphi^n$ for all $n\geq 0$ (thus $\rho_n=\rho\o\varphi^n$). We want to show that $\varphi$ induces a map $\Phi\colon X_{\om}\ra X_{\om}$ with the property that $\Phi\o\rho_{\om}(g)=\rho_{\om}(\varphi(g))\o\Phi$ for all $g\in G$. A natural attempt is setting $\Phi((x_n))=(\wt\varphi(x_n))$ for all $(x_n)\in X_{\om}$. However, for this to be well-defined we need $\lim_{\om}\eps_nd(\wt\varphi(p_n),p_n)<+\infty$. 

We are actually interested in the following more general setting. 

\begin{ass}\label{ass BP 2}
Let $N\leq\out G$ be a subgroup with infinite centre $Z(N)$. Let $\varphi_n\in\aut G$ be a sequence that is mapped by the projection $\pi\colon\aut G\ra\out G$ to a sequence of pairwise distinct elements in $Z(N)$. Consider again $\rho_n=\rho\o\varphi_n$ as above.
\end{ass}

If $\psi\in\pi^{-1}(N)$, then $\pi(\psi)$ commutes with each $\pi(\varphi_n)$. For every $n\in\Z$, choose $g_{n,\psi}\in G$ with:
\[\varphi_n\o\psi=\mf{c}[g_{n,\psi}]\o\psi\o\varphi_n.\]
We are about to prove that, if $G$ is UNE, $\psi$ induces a well-defined map $\zeta(\psi)\colon X_{\om}\ra X_{\om}$ given by:
\[\zeta(\psi)((x_n))=(g_{n,\psi}\wt\psi(x_n)),\] 
(recall that $\wt\psi\colon X\ra X$ is the quasi-isometry induced by $\psi$). We essentially use the same argument as \cite[pp.\ 154--156]{Paulin-ENS}, replacing hyperbolicity with the UNE condition. 

The proof of this result is quite technical. On a first read, we suggest restricting to the situation where $N\simeq\Z$ and the automorphisms $\psi$ and $\varphi_n$ are all powers of a given automorphism, in which case the elements $g_{n,\psi}$ can all be taken to be the identity and our strategy boils down to what is described right before Assumption~\ref{ass BP 2}. This case is sufficient for Theorem~\ref{Q2 thm intro}, though not for the more general Theorem~\ref{homothety 2} below.

\begin{prop}\label{from UNE to X_om}
Suppose that $G$ is UNE. Let $N\leq\out G$ and $\varphi_n\in\aut G$ be as in Assumption~\ref{ass BP 2}. Then there exists a homomorphism $\zeta\colon \pi^{-1}(N)\ra\homeo X_{\om}$ that extends $\rho_{\om}$, in the sense that $\zeta(\cc[g])=\rho_{\om}(g)$ for every $g\in G$. Every homeomorphism in the image of $\zeta$ is bi-Lipschitz.
\end{prop}
\begin{proof}
Consider an element $\psi\in \pi^{-1}(N)$. Let $L\geq 1$ be a constant such that $\wt\psi\colon X\ra X$ is an $(L,L)$--quasi-isometry and such that $\psi\colon G\ra G$ is $L$--bi-Lipschitz with respect to $|\cdot|_S$.

\smallskip
{\bf Step~1:} \emph{the map $\zeta(\psi)$ described above is a well-defined bi-Lipschitz homeomorphism of $X_{\om}$.} \\
Since $\wt\psi$ is a quasi-isometry and $\eps_n\ra 0$, it suffices to show that $\zeta(\psi)$ is a well-defined map, i.e.\ that
$\lim_{\om}\eps_nd(g_{n,\psi}\wt\psi(p_n),p_n)$ is finite. 

We begin by observing that, since $\varphi_n\o\psi=\mf{c}[g_{n,\psi}]\o\psi\o\varphi_n$ and $\wt\psi\o\rho(g)=\rho(\psi(g))\o\wt\psi$:
\begin{align*}
\tau^{\rho_n}_S(g_{n,\psi}\wt\psi(p_n))&=\max_{s\in S}d(\varphi_n(s)g_{n,\psi}\wt\psi(p_n),g_{n,\psi}\wt\psi(p_n))=\max_{s\in S}d((\mf{c}[g_{n,\psi}]^{-1}\varphi_n)(s)\wt\psi(p_n),\wt\psi(p_n)) \\
&=\max_{s\in S}d(\wt\psi((\psi^{-1}\mf{c}[g_{n,\psi}]^{-1}\varphi_n)(s)p_n),\wt\psi(p_n))=\max_{s\in S}d(\wt\psi(\varphi_n\psi^{-1}(s)p_n),\wt\psi(p_n)) \\
&\leq L\cdot\max_{s\in S}d(\varphi_n\psi^{-1}(s)p_n,p_n) +L=L\cdot\max_{s\in S}d(\rho_n(\psi^{-1}(s))\cdot p_n,p_n)+L \\
&\leq L\cdot\max_{s\in S} |\psi^{-1}(s)|_S\cdot\tau^{\rho_n}_S(p_n)+L\leq L^2\cdot\tau^{\rho_n}_S(p_n)+L.
\end{align*}

Now, since $G$ is UNE, there exists a constant $c>0$ such that, for every generating set $T\cu G$ and all $x,y\in X$, we have $d(x,y)\leq c\cdot(\tau_T^{\rho}(x)+\tau_T^{\rho}(y))$. For $T=\varphi_n(S)$, we obtain:
\begin{align*}
\lim_{\om}\eps_nd(g_{n,\psi}\wt\psi(p_n),p_n) &\leq c\cdot\lim_{\om}\eps_n(\tau_{\varphi_n(S)}^{\rho}(g_{n,\psi}\wt\psi(p_n))+\tau_{\varphi_n(S)}^{\rho}(p_n)) \\
&=c\cdot\lim_{\om}\eps_n(\tau^{\rho_n}_S(g_{n,\psi}\wt\psi(p_n))+\tau^{\rho_n}_S(p_n)) \leq c(L^2+1)\cdot\lim_{\om}\eps_n\tau^{\rho_n}_S(p_n)<+\infty.
\end{align*}

{\bf Step~2:} \emph{$\zeta$ is a homomorphism.} \\
Since $G$ is UNE, Example~\ref{UNE example}(3) shows that the centre $Z(G)\leq G$ is finite. Then, since $G$ acts cocompactly on $X$, there exists a constant $M$ such that $d(x,zx)\leq M$ for all $x\in X$ and $z\in Z(G)$. Given $\psi_1,\psi_2\in N$, we can take $\wt{\psi_1\psi_2}=\wt\psi_1\wt\psi_2$. Moreover:
\begin{align*}
\mf{c}[g_{n,\psi_1\psi_2}]\psi_1\psi_2\varphi_n&=\varphi_n\psi_1\psi_2=\mf{c}[g_{n,\psi_1}]\psi_1\varphi_n\psi_2 \\
&=\mf{c}[g_{n,\psi_1}]\psi_1\mf{c}[g_{n,\psi_2}]\psi_2\varphi_n=\mf{c}[g_{n,\psi_1}]\mf{c}[\psi_1(g_{n,\psi_2})]\psi_1\psi_2\varphi_n.
\end{align*}
Hence $g_{n,\psi_1\psi_2}$ and $g_{n,\psi_1}\psi_1(g_{n,\psi_2})$ differ by multiplication by an element of $Z(G)$. It follows that, for every $x\in X$, we have $d(g_{n,\psi_1\psi_2}x,g_{n,\psi_1}\psi_1(g_{n,\psi_2})x)\leq M$.
Thus, for every $(x_n)\in X_{\om}$:
\begin{align*}
\zeta(\psi_1\psi_2)((x_n))&=(g_{n,\psi_1\psi_2}\wt{\psi_1\psi_2}(x_n))=(g_{n,\psi_1}\psi_1(g_{n,\psi_2})\wt\psi_1(\wt\psi_2(x_n))) \\
&=(g_{n,\psi_1}\wt\psi_1(g_{n,\psi_2}\wt\psi_2(x_n)))=\zeta(\psi_1)((g_{n,\psi_2}\wt\psi_2(x_n)))=\zeta(\psi_1)\zeta(\psi_2)((x_n)).
\end{align*}

\smallskip
{\bf Step~3:} \emph{we have $\zeta(\cc[g])=\rho_{\om}(g)$ for all $g\in G$.} \\
Since $\cc[g]\colon G\ra G$ is at bounded distance from left multiplication by $g$, the quasi-isometry $\wt{\cc[g]}$ is at bounded distance from $\rho(g)$. Moreover, observing that
\[\cc[\varphi_n(g)]\o\varphi_n=\varphi_n\o\cc[g]=\cc[g_{n,\cc[g]}]\o\cc[g]\o\varphi_n,\]
we deduce that $\cc[\varphi_n(g)g^{-1}]=\cc[g_{n,\cc[g]}]$, hence $g_{n,\cc[g]}\in Z(G)\varphi_n(g)g^{-1}$. Thus, for every $(x_n)\in X_{\om}$:
\[\zeta(\cc[g])((x_n))=(g_{n,\cc[g]}\wt{\cc[g]}(x_n))=(g_{n,\cc[g]}gx_n)=(\varphi_n(g)g^{-1}gx_n)=(\varphi_n(g)x_n)=\rho_{\om}(g)((x_n)).\]
This concludes the proof of the proposition.
\end{proof}

In the special case where there exists $\varphi\in\aut G$ such that $\varphi_n=\varphi^n$ and $N=\langle\pi(\varphi)\rangle$, we have $\pi^{-1}(N)\simeq(G/Z(G))\rtimes_{\varphi}\Z$ and we obtain:

\begin{cor}\label{from UNE to X_om 2}
Suppose that $G$ is UNE and that $\pi(\varphi)\in\out G$ has infinite order. Take $\varphi_n=\varphi^n$. Then the map $\Phi\colon X_{\om}\ra X_{\om}$ given by $\Phi((x_n))=(\wt\varphi(x_n))$ is a well-defined bi-Lipschitz homeomorphism of $X_{\om}$ satisfying $\Phi\o\rho_{\om}(g)=\rho_{\om}(\varphi(g))\o\Phi$ for all $g\in G$.
\end{cor}

\subsubsection{Coarse-median preserving automorphisms of UNE groups.}
\hfill \smallskip \\ 
Suppose now that $X$ admits a coarse median $\mu$ of finite rank $r$. We can define a map $\mu_{\om}\colon X_{\om}^3\ra X_{\om}$ by setting $\mu_{\om}((x_n),(y_n),(z_n))=(\mu(x_n,y_n,z_n))$. It was shown in \cite[Section~9]{Bow-cm} that $\mu_{\om}$ is well-defined and the pair $(X_{\om},\mu_{\om})$ is a median algebra of rank $\leq r$. 

If the coarse median structure $[\mu]$ is fixed by $G\acts X$, then the action $G\acts X_{\om}$ is by automorphisms of the median algebra $(X_{\om},\mu_{\om})$. Moreover, if an automorphism $\psi\in\pi^{-1}(N)\leq\aut G$ is such that $\wt\psi$ fixes $[\mu]$, then $\zeta(\psi)\in\aut(X_{\om},\mu_{\om})$. Note that, although the metric $d_{\om}$ on $X_{\om}$ is $G$--invariant, it needs not be preserved by $\zeta(\psi)$.

\begin{rmk}
If the space $X$ is coarse median but not median, the metric $d_{\om}$ may not be compatible with $\mu_{\om}$ (in the sense of Definition~\ref{compatible defn}). However, it was shown by Zeidler \cite[Proposition~3.3]{Zeidler} that there always exists a metric $\delta\in\mc{D}^G(X_{\om},\mu_{\om})$ such that $(X_{\om},\delta)$ is complete, geodesic, and bi-Lipschitz equivalent to $(X_{\om},d_{\om})$. Theorem~\ref{all from 10b}(2) and the fact that $G$ does not fix a point in $X_{\om}$ then imply that $G$ acts on $(X_{\om},\delta)$ with unbounded orbits (alternatively, one can appeal to \cite{Bow4}).

This is only tangentially relevant to us as we will only be interested in ultralimits of $\CAT$ cube complexes in the forthcoming subsections.
\end{rmk}

Summing up the above discussion:

\begin{cor} 
Let $G$ be a UNE group. Let $N\leq\out G$ be a subgroup with infinite centre. Let $(X,[\mu])$ be a geodesic coarse median space of finite rank $r$. Let $G\acts X$ be a proper cocompact action fixing the coarse median structure $[\mu]$. Suppose that the quasi-isometries of $X$ induced by the elements of $\pi^{-1}(N)$ also preserve $[\mu]$.

Then there exists a complete, geodesic %=connected, since complete
median space $X_{\om}$ of rank $\leq r$, and an action $\pi^{-1}(N)\acts X_{\om}$ by bi-Lipschitz homeomorphisms that preserve the underlying median-algebra structure. The composition $G\ra G/Z(G)\hookrightarrow\pi^{-1}(N)\acts X_{\om}$ is an isometric $G$--action with unbounded orbits.
\end{cor}

\subsection{Equivariant embeddings in products of $\R$--trees.}\label{treeable ultralimits sect}

Let $M$ be a median algebra and $G\acts M$ an action by median automorphisms. In the rest of Section~\ref{ultra sect}, we will be interested in situations where $M$ can be embedded $G$--equivariantly into a finite product of $\R$--trees. We reserve this subsection for a few general remarks on this setting.

\begin{defn}
An \emph{$\R$--tree} is a geodesic, rank--$1$ median space. 
\end{defn}

This is equivalent to the usual definition of $\R$--trees as geodesic metric spaces where every geodesic triangle is a tripod. We stress that $\R$--trees are not required to be complete.

The next remark collects various simple observations for later use.

\begin{rmk}\label{tree factors vs X}
Consider isometric $G$--actions on $\R$--trees $T_1,\dots,T_k$. Equip $T_1\x\dots\x T_k$ with the diagonal $G$--action. Let $f=(f_i)\colon M\hookrightarrow\prod T_i$ be a $G$--equivariant, injective median morphism.
\begin{enumerate}
\item The image $f(M)$ is a median subalgebra of $\prod_iT_i$. The set of halfspaces of the median algebra $\prod_iT_i$ is naturally identified with the disjoint union $\bigsqcup_i\mscr{H}(T_i)$. 

Every halfspace of $T_i$ is either open or closed. Open halfspaces are precisely the single connected components of the sets $T_i\setminus\{p\}$, as $p$ varies through all points of $T_i$ (including when $T_i\setminus\{p\}$ is connected). Closed halfspaces are precisely the complements of open halfspaces.

If we let $\mscr{H}_i\cu\mscr{H}(M)$ be the set of halfspaces of the form $f_i^{-1}(\mf{h})$ with $\mf{h}\in\mscr{H}(T_i)$, then the $\mscr{H}_i$ cover $\mscr{H}(M)$ by Remark~\ref{halfspaces of subsets}(1). However, the $\mscr{H}_i$ are usually not pairwise disjoint. 

\item Since the sets $\mscr{H}_i$ are $G$--invariant and no two halfspaces in the same $\mscr{H}_i$ are transverse, we see that each $g\in G$ must act non-transversely on $M$.

\item Suppose that, for all $i$, all $x\in T_i$ and all $g\in G$, we have $g^2x=x$ if and only if $gx=x$. Then the action $G\acts M$ has no wall inversions.

Indeed, suppose instead that there exists $\mf{h}\in\mscr{H}(M)$ such that $g\mf{h}=\mf{h}^*$. Pick $i$ such that $\mf{h}\in\mscr{H}_i$, and choose $\mf{k}\in\mscr{H}(T_i)$ with $f_i^{-1}(\mf{k})=\mf{h}$. Then $g\mf{k}\cap\mf{k}$ and $g\mf{k}^*\cap\mf{k}^*$ are disjoint from the $\langle g\rangle$--invariant median subalgebra $f_i(M)$. Note that we cannot have $g\mf{k}\cu\mf{k}$ or $g\mf{k}\supseteq\mf{k}$, so, without loss of generality, $g\mf{k}\cap\mf{k}=\emptyset$. It follows that $f_i(M)\cu\mf{k}\cup g\mf{k}$, hence $g$ is elliptic and fixes a unique point $p$ in the convex hull of $\mf{k}\cup g\mf{k}$. We conclude that $g^2\mf{k}=\mf{k}$, hence the points on the arc connecting $p$ to $\mf{k}$ are fixed by $g^2$, but not by $g$. This is a contradiction.
%Indeed, suppose instead that there exists $\mf{h}\in\mscr{H}(M)$ such that $g\mf{h}=\mf{h}^*$. Pick $i$ such that $\mf{h}\in\mscr{H}_i$, and choose $\mf{k}\in\mscr{H}(T_i)$ with $f_i^{-1}(\mf{k})=\mf{h}$. Then, we must have $g\mf{k}\cap\mf{k}=\emptyset$ and $g\mf{k}^*\cap\mf{k}^*\cap f_i(M)=\emptyset$. In particular, the convex hull of $g\mf{k}\sqcup\mf{k}$ contains a unique $\langle g\rangle$--fixed point $p$. Since $f_i(M)$ is a median subalgebra of $T_i$, the halfspaces $\mf{k},g\mf{k},g^2\mf{k}$ cannot be pairwise disjoint, otherwise $p\in g\mf{k}^*\cap\mf{k}^*$ would lie in $f_i(M)$. We conclude that $g^2\mf{k}=\mf{k}$, hence the points on the arc connecting $p$ to $\mf{k}$ are fixed by $g^2$, but not by $g$.

\item Suppose that $g$ acts on $M$ stably without wall inversions. Then, by Remark~\ref{from 10b rmk}(2) and Theorem~\ref{all from 10b}(3), a halfspace $\mf{h}\in\mscr{H}(M)$ lies in the set $\mscr{H}_{\overline{\C}(g)}(M)$ if and only if either $\mf{h}\subsetneq g\mf{h}$, or $\mf{h}\subsetneq g^{-1}\mf{h}$, or $\mf{h}=g\mf{h}$.  

It follows that, for every $i$, either $g$ is loxodromic in $T_i$ and $f_i(\overline{\C}(g,M))$ is contained in its axis, or $g$ is elliptic in $T_i$ and fixes $f_i(\overline{\C}(g,M))$ pointwise. 
\end{enumerate}
\end{rmk}

% because of the following discussion, we need the number of vertices of $\G$ to be bounded above in our right-angled Artin group. So, as far as $\om$ is concerned, the graph is constant. Hence no point in varying the Salvetti complex. 
% However, note that we *could* have varied salvetti if we had used the result that non-transverse actions suffice (weaker than embedding in products of trees) and the fact that they are characterised by the equality $d(x,gx)=\ell(g,X)+2d(x,\mc{M}(g))$, which *does* pass to ultralimits.
Now, let us fix a non-principal ultrafilter $\om$. Let the group $G$ be generated by a finite subset $S$. Consider a sequence of actions by automorphism on median algebras $G\acts M_n$, along with metrics $\delta_n\in\mc{D}^G(M_n)$ and basepoints $p_n\in M_n$. Suppose moreover that:
\[\max_{s\in S}\sup_n \delta_n(sp_n,p_n)<+\infty.\]

Define $(M_{\om},\delta_{\om},p_{\om}):=\lim_{\om}(M_n,\delta_n,p_n)$. The set $M_{\om}$ becomes a median algebra if we endow it with the operator $m((x_n),(y_n),(z_n))=(m(x_n,y_n,z_n))$. We have an action by median automorphisms $G\acts M_{\om}$ given by $g(x_n)=(gx_n)$. Finally, note that $\delta_{\om}\in\mc{D}^G(M_{\om})$, and that $(M_{\om},\delta_{\om})$ is a complete median space (every ultralimit of metric spaces is complete).

Given a sequence of subsets $A_n\cu M_n$, we will employ the notation:
\[\lim_{\om}A_n:=\{(x_n)\in M_{\om} \mid x_n\in A_n \text{ for $\om$-all $n$}\}=\{(y_n)\in M_{\om} \mid \lim_{\om} \delta_n(y_n,A_n)=0\}.\]
Note that $\lim_{\om}A_n$ is a (possibly empty) \emph{closed} subset of $(M_{\om},\delta_{\om})$ for \emph{any} sequence of subsets $A_n\cu M_n$. It is also clear that $\lim_{\om}A_n\cu M_{\om}$ is convex as soon as $A_n\cu M_n$ is convex for $\om$--all $n$.
% from the description with medians

\smallskip
Fix an integer $k\geq 1$. Suppose that each action $G\acts M_n$ is equipped with a $G$--equivariant, $\delta_n$--isometric embedding $f_n=(f_n^i)\colon M_n\hookrightarrow\prod_iT_n^i$, where $\prod_iT_n^i$ is a product of $k$ $\R$--trees endowed with an isometric, diagonal $G$--action as in Remark~\ref{tree factors vs X}. (We have switched the index ``$i$'' from subscript to superscript to avoid confusion.)

It is straightforward to check that the ultralimits $\lim_{\om}(T_n^i,f_n^i(p_n))$ yield isometric $G$--actions on $\R$--trees $T_{\om}^i$ and a $G$--equivariant, $\delta_{\om}$--isometric embedding $f_{\om}=(f_{\om}^i)\colon M_{\om}\hookrightarrow\prod_iT_{\om}^i$.

\begin{lem}\label{ultra-Salvetti 1}
Consider the above setting. For every $g\in G$, we have:
\begin{enumerate}
\item $\ell(g,T^i_{\om})=\lim_{\om}\ell(g,T^i_n)$ and $\overline{\C}(g,T^i_{\om})=\lim_{\om}\overline{\C}(g,T^i_n)$ for all $1\leq i\leq k$.
\end{enumerate}
If, in addition, $(M_n,\delta_n)$ is a geodesic space for $\om$--all $n$, then $(M_{\om},\delta_{\om})$ is geodesic and:
\begin{enumerate}
\setcounter{enumi}{1}
\item $\ell(g,\delta_{\om})=\lim_{\om}\ell(g,\delta_n)$ and $\overline{\C}(g,M_{\om})=\lim_{\om}\overline{\C}(g,M_n)$.
\end{enumerate}
\end{lem}
\begin{proof}
We only prove part~(2), since part~(1) is a special case of it. 

By Remark~\ref{connected inversions} and Remark~\ref{tree factors vs X}(2), each $g\in G$ acts on $M_{\om}$ stably without inversions and non-transversely; the same is true of the action on $\om$--all $M_n$. Theorem~\ref{all from 10b -2}(2) shows that, for every $x=(x_n)\in M_{\om}$, we have:
\[\delta_{\om}(x,gx)=\lim_{\om}\delta_n(x_n,gx_n)=\lim_{\om}~[\ell(g,\delta_n)+2\delta_n(x_n,\overline{\C}(g,M_n))]\geq\lim_{\om}\ell(g,\delta_n).\]
Hence $\ell(g,\delta_{\om})\geq\lim_{\om}\ell(g,\delta_n)$. By Theorem~\ref{all from 10b -2}(1), the sets $\overline{\C}(g,M_n)$ are gate-convex. If $y_n$ is the gate-projection of the basepoint $p_n\in M_n$ to $\overline{\C}(g,M_n)$, we have:
\[\lim_{\om}\delta_n(y_n,p_n)=\lim_{\om}\delta_n(p_n,\overline{\C}(g,M_n))\leq\lim_{\om}\tfrac{1}{2}\delta_n(p_n,gp_n)<+\infty.\]
It follows that we have a well-defined point $y=(y_n)\in M_{\om}$ and that $\delta_{\om}(y,gy)=\lim_{\om}\ell(g,\delta_n)$. This shows that $\ell(g,\delta_{\om})=\lim_{\om}\ell(g,\delta_n)$.

Finally, since $\overline{\C}(g,M_{\om})$ is gate-convex, it is a closed subset of the complete median space $(M_{\om},\delta_{\om})$. Thus a point $x=(x_n)\in M_{\om}$ lies in $\overline{\C}(g,M_{\om})$ if and only if $\delta_{\om}(x,\overline{\C}(g,M_{\om}))=0$, which happens if and only if $\delta_{\om}(x,gx)=\ell(g,\delta_{\om})$ (again by Theorem~\ref{all from 10b -2}). Equivalently, $x$ lies in $\overline{\C}(g,M_{\om})$ if and only if $\lim_{\om}\delta_n(x_n,\overline{\C}(g,M_n))=0$, i.e.\ if and only if $x\in\lim_{\om}\overline{\C}(g,M_n)$. This concludes the proof.
\end{proof}

\begin{lem}\label{walls from om to n}
Consider again the above setting, with $(M_n,\delta_n)$ geodesic for $\om$--all $n$. Consider two elements $g,h\in G$ and $s\geq 1$.
\begin{enumerate}
\item Suppose that, for some $\mf{w}\in\mscr{W}(M_{\om})$, we have $\{\mf{w},g^s\mf{w}\}\cu\W_1(g,M_{\om})\cap\W_1(h,M_{\om})$. Then, for $\om$--all $n$, there exists $\mf{w}_n\in\mscr{W}(M_n)$ such that $\{\mf{w}_n,g^s\mf{w}_n\}\cu\W_1(g,M_n)\cap\W_1(h,M_n)$.
\item If there exist walls $\mf{u},\mf{v}\in\W_1(g,M_{\om})$ such that $\{\mf{u},g^s\mf{u}\}$ is transverse to $\{\mf{v},g^s\mf{v}\}$, then, for $\om$--all $n$, there exist $\mf{u}_n,\mf{v}_n\in\W_1(g,M_n)$ such that $\{\mf{u}_n,g^s\mf{u}_n\}$ is transverse to $\{\mf{v}_n,g^s\mf{v}_n\}$.
\end{enumerate}
\end{lem}
\begin{proof}
We begin with some general observations. We have already noted in Lemma~\ref{ultra-Salvetti 1} that $(M_{\om},\delta_{\om})$ is connected, hence $g,h$ act stably without inversions. By parts~(1) and~(4) of Remark~\ref{tree factors vs X}, each wall of $M_{\om}$ arises from a wall of (at least) one of the trees $T^i_{\om}$. Moreover, each projection $f_{\om}^i(\overline{\C}(g,M_{\om}))$ is either fixed pointwise by $g$ or it is a $\langle g\rangle$--invariant geodesic (and similarly for $h$). 

We now prove part~(1). By the above discussion, there exist an index $i$ and $\mf{v}\in\mscr{W}(T^i_{\om})$ such that $\{\mf{v},g^s\mf{v}\}\cu\W_1(g,T^i_{\om})\cap\W_1(h,T^i_{\om})$. Thus, $g$ and $h$ are both loxodromic in $T^i_{\om}$, which implies that they are loxodromic in $\om$--all $T^i_n$. Let $\alpha_{\om},\alpha_n$ and $\beta_{\om},\beta_n$ be the axes in $T^i_{\om},T^i_n$ of $g$ and $h$, respectively. By Lemma~\ref{ultra-Salvetti 1}, we have $\alpha_{\om}=\lim_{\om}\alpha_n$ and $\beta_{\om}=\lim_{\om}\beta_n$. Since $\alpha_{\om}$ and $\beta_{\om}$ both cross $\mf{v}$ and $g^s\mf{v}$, they must share a segment of length $\eps+s\cdot\ell(g,T^i_{\om})$ for some $\eps>0$. 

If $y$ and $z$ are the endpoints of this segment, we can write $y=(y_n)=(y_n')$ and $z=(z_n)=(z_n')$ with $y_n,z_n\in\alpha_n$ and $y_n',z_n'\in\beta_n$. Denoting by $\delta_n^i$ the metric of $T_n^i$, we have:
\begin{align*}
\lim_{\om}\delta_n^i(y_n,y_n')=&\lim_{\om}\delta_n^i(z_n,z_n')=0, &\lim_{\om}\delta_n^i(y_n,z_n)=&\lim_{\om}\delta_n^i(y_n',z_n')=\eps+s\cdot \lim_{\om}\ell(g,T_n^i).
\end{align*}
Hence $\alpha_n$ and $\beta_n$ share a segment $\s_n$ of length $>s\cdot\ell(g,T_n^i)$ for $\om$--all $n$. It follows that there exists a wall $\mf{v}_n\in\mscr{W}(T^i_n)$ such that $\s_n$ crosses $\mf{v}_n$ and $g^s\mf{v}_n$. Hence $\{\mf{v}_n,g^s\mf{v}_n\}\cu\W_1(g,T^i_n)\cap\W_1(h,T^i_n)$, and it is clear that $\mf{v}_n$ determines a wall $\mf{w}_n$ of $M$ with $\{\mf{w}_n,g^s\mf{w}_n\}\cu\W_1(g,M_n)\cap\W_1(h,M_n)$.

Let us now prove part~(2). By Remark~\ref{tree factors vs X}(4), $\mf{u}$ and $\mf{v}$ determine halfspaces $\mf{h},\mf{k}\in\mscr{H}(M_{\om})$ satisfying $g\mf{h}\subsetneq\mf{h}$ and $g\mf{k}\subsetneq\mf{k}$. Since $\{\mf{u},g^s\mf{u}\}$ and $\{\mf{v},g^s\mf{v}\}$ are transverse, Helly's lemma implies that there exist points:
\begin{align*} 
x&\in g^s\mf{h}\cap g^s\mf{k}\cap\overline{\C}(g,M_{\om}), & y&\in g^s\mf{h}\cap\mf{k}^*\cap\overline{\C}(g,M_{\om}), \\
z&\in \mf{h}^*\cap g^s\mf{k}\cap\overline{\C}(g,M_{\om}), & w&\in\mf{h}^*\cap\mf{k}^*\cap\overline{\C}(g,M_{\om}).
\end{align*} 

Suppose that $\mf{u}$ and $\mf{v}$ arise from trees $T^i_{\om}$ and $T^j_{\om}$, where $g$ has axes $\alpha^i$ and $\alpha^j$, respectively. Then the points $f_{\om}^i(x),f_{\om}^i(y),f_{\om}^i(z),f_{\om}^i(w)$ lie on $\alpha^i$, and $\{f_{\om}^i(x),f_{\om}^i(y)\}$ is separated from $\{f_{\om}^i(z),f_{\om}^i(w)\}$ by a segment of length $>s\cdot\ell(g,T^i_{\om})$. Similarly, $\{f_{\om}^j(x),f_{\om}^j(z)\}$ and $\{f_{\om}^j(y),f_{\om}^j(w)\}$ are separated by a subsegment of $\alpha^j$ of length $>s\cdot\ell(g,T^j_{\om})$. 

Writing $x=(x_n),y=(y_n),z=(z_n),w=(w_n)$, it follows that, for $\om$-all $n$, there exist walls $\mf{u}_n'\in\W_1(g,T^i_n)$ and $\mf{v}_n'\in\W_1(g,T^j_n)$ such that:
\begin{align*} 
\{\mf{u}_n',g^s\mf{u}_n'\}&\cu\mscr{W}(f_n^i(x_n),f_n^i(y_n)|f_n^i(z_n),f_n^i(w_n)), & \{\mf{v}_n',g^s\mf{v}_n'\}&\cu\mscr{W}(f_n^j(x_n),f_n^j(z_n)|f_n^j(y_n),f_n^j(w_n)).
\end{align*}
Thus $\mf{u}_n',\mf{v}_n'$ induce $\mf{u}_n,\mf{v}_n\in\W_1(g,M_n)$ with $\{\mf{u}_n,g^s\mf{u}_n\}$ transverse to $\{\mf{v}_n,g^s\mf{v}_n\}$ (cf.\ Lemma~\ref{3pwt}).
\end{proof}

\subsection{Ultralimits of convex-cocompact actions on Salvettis.}\label{ultra-Salvetti sect}

Let $\G$ be a finite simplicial graph, $\A=\A_{\G}$ the associated right-angled Artin group, and $\X=\X_{\G}$ the universal cover of its Salvetti complex. Denote by $d$ the $\ell^1$ metric on $\X$ and set $r=\dim\X$. Fix a non-principal ultrafilter $\om$.

When we speak \emph{convex-cocompactness} in $\A$ from now on (Definition~\ref{cc defn}), this is always meant with respect to the standard action $\A\acts\X$. Note that a group $G$ is isomorphic to a convex-cocompact subgroup of a right-angled Artin group if and only if $G$ is the fundamental group of a compact special cube complex \cite{Haglund-Wise-GAFA}. In particular, $G$ must be torsion-free and finitely generated.

In the rest of Section~\ref{ultra sect} we make the following assumption.

\begin{ass}
Let $G\leq\A$ be a convex-cocompact subgroup. Let $Y\cu\X$ be a $G$--invariant, convex subcomplex on which $G$ acts with exactly $q$ orbits of vertices. Let $[\mu]$ be the induced coarse median structure on $G$. Consider a sequence $\varphi_n\in\aut(G,[\mu])$. Denote by $\rho\colon G\hookrightarrow\A$ the standard inclusion and set $\rho_n=\rho\o\varphi_n$.
\end{ass}

We say for simplicity that $g\in G$ is \emph{label-irreducible} if $\rho(g)$ is a label-irreducible element of $\A$.

\begin{rmk}\label{invariance of label-irreducibility}
If $g\in G$ is label-irreducible, then Corollary~\ref{cmp preserve cc} and Lemma~\ref{label-irreducibles are cc 1}(2) show that $\rho_n(g)\in\A$ is label-irreducible for all $n\geq 0$.
\end{rmk}

Let $S\cu G$ be a finite generating set. Choose basepoints $p_n\in Y_n$ with $\tau_S^{\rho_n}(p_n)=\overline\tau_S^{\rho_n}$ and define $\delta_n:=d/\overline\tau_S^{\rho_n}\in\mc{D}^G(\X)$. For ease of notation, let us write $G\acts\X_n$ and $G\acts Y_n$ for the actions of $G$ on $\X$ and $Y$ induced by the homomorphism $\rho_n$.  

Recall that $\g\colon\mscr{W}(\X)\ra\G^{(0)}$ is the map pairing each hyperplane with its label. For every $v\in\G^{(0)}$, the hyperplanes in $\g^{-1}(v)$ are pairwise disjoint. Hence there is a natural simplicial tree $\T^v$ (usually locally infinite) that is dual to the collection $\g^{-1}(v)$. In the terminology of Subsection~\ref{CCC prelims}, the tree $\T^v$ is the restriction quotient of $\X$ associated to $\g^{-1}(v)\cu\mscr{W}(\X)$. 

In particular, we have an $\A$--equivariant, surjective median morphism $\pi^v\colon \X\ra\T^v$ taking cubes to edges or vertices, and an $\A$--equivariant, isometric median morphism $(\pi^v)\colon\X\hookrightarrow\prod_{v\in\G}\T^v$.

Let $\T_n^v$ denote the tree $\T^v$ equipped with the twisted $G$--action induced by $\rho_n$ and with its graph metric rescaled by $\overline\tau_S^{\rho_n}$. We obtain a $G$--equivariant, $\delta_n$--isometric embedding $(\pi_n^v)\colon\X_n\hookrightarrow\prod_{v\in\G}\T_n^v$.

Thus, our setting is a special case of the one in the second part of Subsection~\ref{treeable ultralimits sect} (after Remark~\ref{tree factors vs X}). If the automorphisms $\varphi_n$ are pairwise distinct in $\out G$, then we are also in a special case of Subsection~\ref{BP sect}, but we do not make this assumption for the moment.

As in Subsection~\ref{treeable ultralimits sect}, the sequence of actions $G\acts\X_n$ with metrics $\delta_n$ and basepoints $p_n$ yields a limit action $G\acts\X_{\om}$, along with a metric $\delta_{\om}\in\mc{D}^G(\X_{\om})$, a basepoint $p_{\om}\in\X_{\om}$, and a $G$--equivariant, $\delta_{\om}$--isometric embedding $(\pi_{\om}^v)\colon\X_{\om}\hookrightarrow\prod_{v\in\G}\T_{\om}^v$. The pair $(\X_{\om},\delta_{\om})$ is a complete, geodesic median space of rank $\leq r$. 

We now prove a sequence of fairly straightforward lemmas regarding the action of $G$ on $\X_{\om}$ and its median subalgebras. After that comes the most important part of this subsection, which is concerned with the notion of \emph{cubical configurations} (Definition~\ref{cubical config defn}).

\begin{lem}\label{from om to commutation new}
Consider label-irreducible elements $g,h\in G$. 
\begin{enumerate}
\item If there exist walls $\mf{u},\mf{w}$ with $\{\mf{u},\mf{w},h^{4r}\mf{u},g^{4r}\mf{w}\}\cu\W_1(g,\X_{\om})\cap\W_1(h,\X_{\om})$, then $\langle g,h\rangle\simeq\Z$.
\item There do not exist walls $\mf{u},\mf{w}\in\W_1(g,\X_{\om})$ such that $\{\mf{u},g^{4r}\mf{u}\}$ is transverse to $\{\mf{w},g^{4r}\mf{w}\}$.
\end{enumerate}
\end{lem}
\begin{proof}
We begin with part~(1). By Lemma~\ref{walls from om to n}(1), there exist hyperplanes $\mf{u}_n,\mf{w}_n\in\mscr{W}(\X_n)$, for some $n$, such that $\{\mf{u}_n,\mf{w}_n,h^{4r}\mf{u}_n,g^{4r}\mf{w}_n\}\cu\W_1(g,\X_n)\cap\W_1(h,\X_n)$. Since $\rho_n(g)$ and $\rho_n(h)$ are label-irreducible by Remark~\ref{invariance of label-irreducibility}, Lemma~\ref{friendly label-irreducibles} guarantees that $\langle g,h\rangle\simeq\Z$.

Regarding part~(2), if there existed such walls, then Lemma~\ref{walls from om to n}(2) would yield hyperplanes $\mf{u}_n,\mf{w}_n\in\mscr{W}(\X_n)$, for some $n$, such that the sets $\{\mf{u}_n,g^{4r}\mf{u}_n\}$ and $\{\mf{w}_n,g^{4r}\mf{w}_n\}$ are transverse and both contained in $\W_1(g,\X_n)$. This would violate Lemma~\ref{irreducible elements new}, since $\rho_n(g)$ is label-irreducible.
\end{proof}

\begin{lem}\label{fixsets of powers}
For every $G$--invariant median subalgebra $M\cu\X_{\om}$, we have:
\begin{enumerate}
\item the action $G\acts M$ has no wall inversions;
\item each element $g\in G$ is elliptic (resp.\ loxodromic) in $M$ if and only if it is in $\X_{\om}$.
\end{enumerate}
\end{lem}
\begin{proof}
Part~(2) quickly follows from part~(1). Indeed, note that $\mc{H}_1(g,M)=\emptyset$ if and only if $\mc{H}_1(g,\X_{\om})=\emptyset$, for instance by Remark~\ref{from 10b rmk}(3). Since the action $G\acts M$ has no inversions, Theorem~\ref{all from 10b}(2) then shows that $g$ is elliptic/loxodromic in $M$ if and only if it is $\X_{\om}$.

Regarding part~(1), we will need the following observation.

\smallskip
{\bf Claim}: \emph{Let an action $G\acts (T_{\om},d_{\om})$ be the ultralimit of a sequence of actions on $\R$--trees $G\acts (T_n,d_n)$. Suppose in addition that $g\in G$ is loxodromic in $\om$--all $T_n$. Then, for all $k\in\Z\setminus\{0\}$ and all $x\in T_{\om}$, the point $x$ is fixed by $g^k$ if and only if it is fixed by $g$.}

\smallskip\noindent
\emph{Proof of Claim.} 
Let $\alpha_n$ be the axis of $g$ in $T_n$ and consider a point $y=(y_n)\in T_{\om}$. Then:
\[d_n(y_n,g^ky_n)=\ell(g^k,T_n)+2d_n(y_n,\alpha_n)\geq\ell(g,T_n)+2d_n(y_n,\alpha_n)=d_n(y_n,gy_n).\]
It follows that $d_{\om}(y,g^ky)\geq d_{\om}(y,gy)$ for all $k\in\Z\setminus\{0\}$, which proves the Claim.
\hfill$\blacksquare$

\smallskip
Now, we will deduce that the action $G\acts M$ has no wall inversions from Remark~\ref{tree factors vs X}(3). We need to show that, for every $v\in\G$, every $x\in\T_{\om}^v$ and every $g\in G$, we have $g^2x=x$ if and only if $gx=x$. If $\rho_n(g)$ is loxodromic in $\T_n^v$ for $\om$--all $n$, this follows from the Claim. If instead $\rho_n(g)$ is elliptic in $\T_n^v$ for $\om$--all $n$, then it follows from the observation that edge-stabilisers for the action $G\acts\T_n^v$ are closed under taking roots in $G$ (since they are hyperplane-stabilisers for $G\acts\X_n$).
\end{proof}

\begin{lem}\label{label-irreducibles in sub-ultras}
Consider $g\in G$ such that its label-irreducible components $g_1,\dots,g_k$ also lie in $G$ (in general, they only lie in $\A$). Then, for every $G$--invariant median subalgebra $M\cu\X_{\om}$:
\begin{enumerate}
\item we have a partition $\W_1(g,M)=\W_1(g_1,M)\sqcup\dots\sqcup\W_1(g_k,M)$; 
\item each wall in $\W_1(g_i,M)$ is preserved by each $g_j$ with $j\neq i$;
\item the sets $\W_1(g_1,M),\dots,\W_1(g_k,M)$ are transverse to each other;
\item we have $\overline{\mc{C}}(g,M)=\overline{\mc{C}}(g_1,M)\cap\dots\cap\overline{\mc{C}}(g_k,M)$ and $\overline{\mc{C}}(g^m,M)=\overline{\mc{C}}(g,M)$ for all $m\geq 1$;
\item for every $\eta\in\mc{PD}^G(M)$, we have $\ell(g,\eta)=\ell(g_1,\eta)+\dots+\ell(g_k,\eta)$.
\end{enumerate}
%Note that some (or all) of the $g_i$ might act elliptically on $M$, in which case $\W_1(g_i,M)$ is empty.
\end{lem}
\begin{proof}
Let us prove parts~(1) and~(2) first, except for disjointness of the sets $\W_1(g_i,M)$, which will follow from part~(3). Note that it suffices to consider the case when $M=\X_{\om}$. Indeed, by Remark~\ref{halfspaces of subsets}, we have a surjection $\res_M\colon\mscr{W}_M(\X_{\om})\ra\mscr{W}(M)$
and, by Remark~\ref{from 10b rmk}(3), a wall $\mf{w}\in\mscr{W}_M(\X_{\om})$ lies in $\W_1(g,\X_{\om})$ if and only if $\res_M(\mf{w})$ lies in $\W_1(g,M)$. 

In fact, Remark~\ref{tree factors vs X}(1) shows that it suffices to prove parts~(1) and~(2) ``for the trees $\T_{\om}^v$'', i.e.\ that, for every $v\in\G$, we have $\W_1(g,\T_{\om}^v)=\W_1(g_1,\T_{\om}^v)\cup\dots\cup\W_1(g_k,\T_{\om}^v)$, and that $g_j$ fixes the set $\W_1(g_i,\T_{\om}^v)$ pointwise for $j\neq i$. 

Note that distinct components $g_i$ cannot be loxodromic in the same tree $\T_{\om}^v$. Otherwise they would have the same axis, since they commute, and Lemma~\ref{from om to commutation new}(1) would give a contradiction. Thus, at most one of the sets $\W_1(g_1,\T_{\om}^v),\dots,\W_1(g_k,\T_{\om}^v)$ can be nonempty, for each $v$. 

Recalling that $g=g_1\cdot\ldots\cdot g_k$ and that the $g_i$ commute pairwise, we conclude that either $\W_1(g,\T_{\om}^v)$ is empty, or it coincides with $\W_1(g_{i_v},\T_{\om}^v)$, where $g_{i_v}$ is the only label-irreducible component that is loxodromic in $\T_{\om}^v$. If $j\neq i_v$, then $g_j$ is elliptic in $\T_{\om}^v$ and, since it commutes with $g_{i_v}$, it must fix pointwise its entire axis. In particular, $g_j$ preserves every wall in the set $\W_1(g_{i_v},\T_{\om}^v)$. This proves parts~(1) and~(2), except for disjointness of the sets $\W_1(g_i,M)$.

In order to prove part~(3), note that part~(2) shows that $\W_1(g_i,M)\cu\W_0(g_j,M)$ for $i\neq j$. By Lemma~\ref{fixsets of powers}(1), the action $G\acts M$ has no wall inversions. Thus $\W_1(g_i,M)$ and $\W_1(g_j,M)$ are transverse by Theorem~\ref{all from 10b}(3). In particular, $\W_1(g_i,M)$ and $\W_1(g_j,M)$ are disjoint, which completes the proof of part~(1).

Regarding part~(4), it suffices to prove the statements for $M=\X_{\om}$. Indeed, since $G$ acts non-transversely on $\X_{\om}$ and without inversions on $M$, we have $\overline{\C}(g,M)=M\cap\overline{\C}(g,\X_{\om})$, for instance by \cite[Proposition~3.40]{Fio10b}. The same holds for the $g_i$. Now, Lemma~\ref{ultra-Salvetti 1}(2) implies that $\overline{\C}(g,\X_{\om})$ coincides with $\bigcap_i\overline{\C}(g_i,\X_{\om})$ and $\overline{\C}(g^m,\X_{\om})$, since this is true for convex cores in all $\X_n$ (recall Lemma~\ref{label-irreducible decomposition}(2) and Remark~\ref{invariance of label-irreducibility}).

%note that Lemma~\ref{label-irreducible decomposition}(2) and Lemma~\ref{ultra-Salvetti 1}(2) imply that $\overline{\C}(g,\X_{\om})$ coincides with $\bigcap_i\overline{\C}(g_i,\X_{\om})$. Now, since $G$ acts non-transversely on $\X_{\om}$ and without inversions on $M$, we have $\overline{\C}(g,M)=M\cap\overline{\C}(g,\X_{\om})$ and $\overline{\C}(g_i,M)=M\cap\overline{\C}(g_i,\X_{\om})$, for instance by \cite[Proposition~3.40]{Fio10b}.

Finally, we prove part~(5). Parts~(1) and~(2) imply that a $\mscr{B}$--measurable fundamental domain for the action $\langle g\rangle\acts\mc{H}_1(g,M)$ can be constructed as the disjoint union of $\mscr{B}$--measurable fundamental domains for the actions $\langle g_i\rangle\acts\mc{H}_1(g_i,M)$. Since $G\acts M$ has no wall inversions, translation lengths coincide with a measure of these fundamental domains (Remark~\ref{length from fundamental domains}) and part~(5) follows.
\end{proof}

\begin{lem}\label{displacement lem}
Let $M\cu\X_{\om}$ be a $G$--invariant median subalgebra with a pseudo-metric $\eta\in\mc{PD}^G(M)$. Consider an element $g\in G$ and a point $x\in M$.
\begin{enumerate}
\item For every $m\geq 1$, we have $\eta(x,gx)\leq\eta(x,g^mx)$.
\item If $h\in G$ is a label-irreducible component of $g$, then $\eta(x,hx)\leq\eta(x,gx)$.
\end{enumerate}
\end{lem}
\begin{proof}
Recall that the action $G\acts M$ is non-transverse by Remark~\ref{tree factors vs X}(2), and without inversions by Lemma~\ref{fixsets of powers}(1). Thus, Theorem~\ref{all from 10b -2}(2) guarantees that $\eta(x,gx)=\ell(g,\eta)+2\eta(x,\overline{\mc{C}}(g,M))$.

Now, part~(1) is obtained by observing that $\ell(g^m,\eta)=m\cdot\ell(g,\eta)$ and $\overline{\mc{C}}(g^m,M)=\overline{\mc{C}}(g,M)$, which follow from Remark~\ref{length from fundamental domains} and Lemma~\ref{label-irreducibles in sub-ultras}(4), respectively. For part~(2), it suffices to recall that $\overline{\mc{C}}(g,M)\cu\overline{\mc{C}}(h,M)$ and $\ell(h,\eta)\leq\ell(g,\eta)$, as shown in parts~(4) and~(5) of Lemma~\ref{label-irreducibles in sub-ultras}.
\end{proof}

We now introduce \emph{cubical configurations}, which will be important in the proof of Theorem~\ref{from sub-ultras to X new new}, hence in those of Theorems~\ref{Q2 thm intro} and~\ref{special UNE thm intro}. The idea is that large cubes in $\X_{\om}$ that are moved very little by a subset $F\cu G$ will give rise to cubical configurations in $\X_{\om}$ (Lemma~\ref{from cubes to cubical config}). 

After the definition, we will see how to transfer cubical configurations from $\X_{\om}$ to $\X$ (Lemma~\ref{transferring cubical config}) and how to use them to construct large abelian subgroups in the centraliser $Z_G(F)$ (Lemma~\ref{commutation from cubical config}).

\begin{defn}\label{cubical config defn}
Consider an action on a median algebra $G\acts M$ and a finite subset $F\cu G$. An \emph{$(s,t,F)$--cubical configuration of width $m\geq 1$} in $M$ is the datum of nonempty subsets $\mc{U}_1,\dots,\mc{U}_s\cu\mscr{W}(M)$, walls $\mf{v}_1,\dots,\mf{v}_t\in\mscr{W}(M)$ and a partition $F=F_0\sqcup\{g_1,\dots,g_t\}$ such that:
\begin{enumerate}
\item the sets $\mc{U}_1,\dots,\mc{U}_s,\{\mf{v}_1,g_1^m\mf{v}_1\},\dots,\{\mf{v}_t,g_t^m\mf{v}_t\}$ are transverse to each other and their union is contained in $\W_0(f,M)$ for every $f\in F_0$; 
\item for each $1\leq j\leq t$, we have $\{\mf{v}_j,g_j^m\mf{v}_j\}\cu\W_1(g_j,M)$, while $\W_0(g_j,M)$ contains $\mc{U}_1,\dots,\mc{U}_s$ and all sets $\{\mf{v}_{j'},g_{j'}^m\mf{v}_{j'}\}$ with $j'\neq j$.
\end{enumerate}
We refer to $\mc{U}_1,\dots,\mc{U}_s$ as the \emph{static sets} and to $g_1,\dots,g_t$ as the \emph{skewering elements}.
\end{defn}

The proof of the next result is quite similar in spirit to that of Lemma~\ref{ultra-Salvetti 1}, but we repeat it for the reader's convenience, since it is a bit more technical.

We denote by $Y_{\om}\cu\X_{\om}$ the convex subset obtained as $\lim_{\om}Y_n$. A subset $\mscr{C}\cu\mscr{W}(Y)$ is a \emph{chain} if it is the set of hyperplanes associated to a set of halfspaces that is totally ordered by inclusion.

\begin{lem}\label{transferring cubical config}
Suppose that the sequence $\varphi_n$ is not $\om$--constant. Let $F\cu G$ be a finite subset of label-irreducible elements such that no two of them generate a cyclic subgroup. Suppose that $Y_{\om}$ admits an $(s,t,F)$--cubical configuration of width $\geq 4r$ with skewering elements $g_1,\dots,g_t$.

Then, for $\om$--all $n$, there exists a $(\s,\tau,\varphi_n(F))$--cubical configuration of width $\geq 4r$ in $Y$ such that $\s+\tau=s+t$ and the $\varphi_n(g_i)$ are skewering elements (hence $\tau\geq t$ and $\s\leq s$). In addition, the static sets of this configuration can be taken to be arbitrarily long chains of hyperplanes.
\end{lem}
\begin{proof}
Let the cubical configuration in $Y_{\om}$ consist of static sets $\mc{U}_1,\dots,\mc{U}_s$, walls $\mf{v}_1,\dots,\mf{v}_t$ and the partition $F=F_0\sqcup\{g_1,\dots,g_t\}$. It suffices to assume that each $\mc{U}_i$ is a singleton $\{\mf{u}_i\}$. Recall that the action $G\acts Y_{\om}$ is non-transverse by Remark~\ref{tree factors vs X}(2), and without inversions by Remark~\ref{connected inversions}.

By Remark~\ref{tree factors vs X}(1), there exist vertices $u_1,\dots,u_s,v_1,\dots,v_t\in\G$ such that the walls $\mf{u}_i$ and $\mf{v}_j$ arise from walls $\overline{\mf{u}}_i\in\mscr{W}(\T_{\om}^{u_i})$ and $\overline{\mf{v}}_j\in\mscr{W}(\T_{\om}^{v_j})$, respectively. Note that each $\overline{\mf{u}}_i$ is preserved by all elements of $F$, while $\overline{\mf{v}}_j$ and $g_j^{4r}\overline{\mf{v}}_j$ are preserved by $F\setminus\{g_j\}$ and cross the axis of $g_j$ in $\T_{\om}^{v_j}$.

\smallskip
{\bf Claim:} \emph{there exist nontrivial arcs $\alpha_i\cu\T_{\om}^{u_i}$ and $\beta_j\cu\T_{\om}^{v_j}$, with endpoints $\alpha_i^{\pm}$ and $\beta_j^{\pm}$, such that:}
\begin{itemize}
\item[\emph{(a)}] \emph{each $\alpha_i$ is fixed by $F$ and each $\beta_j$ is fixed by $F\setminus\{g_j\}$;}
\item[\emph{(b)}] \emph{$\beta_j$ is contained in the axis of $g_j$ in $\T_{\om}^{v_j}$ and it has length $>4r\cdot\ell(g_j,\T_{\om}^{v_j})$;}
\item[\emph{(c)}] \emph{these arcs induce transverse sets of walls of $Y_{\om}$. More precisely, for every $(\eps,\zeta)\in\{\pm\}^s\x\{\pm\}^t$, there exists a point $x^{\eps,\zeta}\in Y_{\om}$ such that, for all $i$ and $j$, the nearest-point projection of $\pi_{\om}^{u_i}(x^{\eps,\zeta})$ to $\alpha_i$ is $\alpha_i^{\eps_i}$ and the nearest-point projection of $\pi_{\om}^{v_j}(x^{\eps,\zeta})$ to $\beta_j$ is $\beta_j^{\zeta_j}$.}
\end{itemize}

\smallskip\noindent
\emph{Proof of Claim.}
The walls $\mf{u}_i,\mf{v}_j\in\mscr{W}(Y_{\om})$ correspond to halfspaces $\mf{u}_i^{\pm},\mf{v}_j^{\pm}\in\mscr{H}(Y_{\om})$, which we label so that, for each $j$, the halfspaces $\mf{v}_j^-$ and $g_j^{4r}\mf{v}_j^+$ are disjoint. Since the sets $\{\mf{u}_i\}$ and $\{\mf{v}_j,g_j^{4r}\mf{v}_j\}$ are all transverse to each other, Helly's lemma allows us to find points $x^{\eps,\zeta}\in Y_{\om}$ so that $x^{\eps,\zeta}$ lies in $\mf{u}_i^{\eps_i}$ for all $i$ and so that, for all $j$, it lies in $\mf{v}_j^-$ if $\zeta_j=-$ and in $g_j^{4r}\mf{v}_j^+$ if $\zeta_j=+$. 

For each $i$, there exists a point $q_i\in\T_{\om}^{u_i}$ such that one of the two halfspaces associated to $\overline{\mf{u}}_i$ is a connected component $\kappa_i$ of $\T_{\om}^{u_i}\setminus\{q_i\}$ (in particular, $\kappa_i$ is open). Since $G$ acts on $Y_{\om}$ without inversions, $F$ fixes $q_i$ and leaves $\kappa_i$ invariant. Since $F$ is finite, it fixes nontrivial arc of $\T_{\om}^{u_i}$ with one endpoint equal to $q_i$ and the other lying in $\kappa_i$. We let $\alpha_i$ be this arc, possibly shrinking it a bit to ensure that the finitely many points $\pi_{\om}^{u_i}(x^{\eps,\zeta})$ have the correct projections to $\alpha_i$.

Finally, consider an index $1\leq j\leq t$. The set of points $x^{\eps,\zeta}$ with $\zeta_j=-$ is contained in $\mf{v}_j^-$, whereas the set of points $x^{\eps,\zeta}$ with $\zeta_j=+$ is contained in $g_j^{4r}\mf{v}_j^+$. Note that either $\pi_{\om}^{v_j}(\mf{v}_j^-)$ or $\pi_{\om}^{v_j}(g_j^{4r}\mf{v}_j^+)$ is an \emph{open} halfspace of $\T_{\om}^{v_j}$. It follows that the set of points $\pi_{\om}^{v_j}(x^{\eps,\zeta})$ with $\zeta_j=-$ is separated by the set of points $\pi_{\om}^{v_j}(x^{\eps,\zeta})$ with $\zeta_j=+$ by an arc $\beta_j$ that is contained in the axis of $g_j$ in $\T_{\om}^{v_j}$ and has length $>4r\cdot\ell(g_j,\T_{\om}^{v_j})$. Since $F\setminus\{g_j\}$ preserves the halfspaces $\mf{v}_j^-$ and $g_j^{4r}\mf{v}_j^+$, we can shrink $\beta_j$ a bit to ensure that it is fixed by $F\setminus\{g_j\}$, while retaining length $>4r\cdot\ell(g_j,\T_{\om}^{v_j})$. This concludes the proof of the Claim.
\hfill$\blacksquare$ 

\smallskip
Now, it is straightforward to approximate, for $\om$--all $n$, the data provided by the Claim by arcs $\alpha_i(n)\cu\T_n^{u_i}$, $\beta_j(n)\cu\T_n^{v_j}$ and points $x^{\eps,\zeta}_n\in Y_n$ satisfying analogous conditions. 

Here we need to account for the fact that some elements of $F$ that are elliptic in one of the trees $\T_{\om}^{\bullet}$ might be loxodromic in the trees $\T_n^{\bullet}$, with translation lengths converging to zero. In any case, we can ensure that the following are satisfied:
\begin{enumerate}
\item[$(a^{\prime})$] for all $f\in F$ and $1\leq i\leq s$, either $f$ fixes $\alpha_i$ pointwise, or $\alpha_i$ is contained in the axis of $f$ in $\T_n^{u_i}$ and has length $>4r\cdot\ell(f,\T_n^{u_i})$; the same holds for $f\in F\setminus\{g_j\}$ and $\beta_j$;
\item[$(b^{\prime})$] $\beta_j(n)$ is contained in the axis of $g_j$ in $\T_n^{v_j}$ and it has length $>4r\cdot\ell(g_j,\T_n^{v_j})$;
\item[$(c^{\prime})$] the nearest-point projection of $\pi_n^{u_i}(x_n^{\eps,\zeta})$ to $\alpha_i(n)$ is $\alpha_i^{\eps_i}(n)$, and the nearest-point projection of $\pi_n^{v_j}(x_n^{\eps,\zeta})$ to $\beta_j(n)$ is $\beta_j^{\zeta_j}(n)$.
\end{enumerate}

Fix a value of $n$ such that the above are satisfied. Condition~$(c^{\prime})$ implies that the subsets of $\mscr{W}(Y_n)$ corresponding to the arcs $\alpha_i(n)$ and $\beta_j(n)$ are all transverse to each other. Hence:
\begin{itemize}
\item \emph{Each $f\in F$ fixes all arcs $\alpha_i(n),\beta_j(n)$ except at most one.} Otherwise, Conditions~$(a^{\prime})$ and~$(b^{\prime})$ would yield $\mf{w},\mf{w}'\in\W_1(f,Y_n)$ such that $\{\mf{w},f^{4r}\mf{w}\}$ and $\{\mf{w}',f^{4r}\mf{w}'\}$ are transverse. Along with Lemma~\ref{irreducible elements new}, this would contradict label-irreducibility of $\rho_n(f)$ (Remark~\ref{invariance of label-irreducibility}).

\smallskip
\item \emph{Each of the arcs $\alpha_i(n),\beta_j(n)$ is fixed by all elements of $F$ except at most one.} Indeed, if neither of $f_1,f_2\in F$ fixed a given arc, then the same conditions would yield $\mf{w}_1,\mf{w}_2$ such that $\{\mf{w}_1,f_1^{4r}\mf{w}_1,\mf{w}_2,f_2^{4r}\mf{w}_2\}\cu\W_1(f_1,Y_n)\cap\W_1(f_2,Y_n)$. Since $\rho_n(f_1)$ and $\rho_n(f_2)$ are label-irreducible, Lemma~\ref{friendly label-irreducibles} would then imply that $\langle f_1,f_2\rangle\simeq\Z$, contradicting our assumptions.
\end{itemize}

In conclusion, up to reordering, there exists $0\leq\s\leq s$ such that, for $1\leq i\leq\s$, the arcs $\alpha_i(n)$ are fixed by the whole $F$, while, for $\s<i\leq s$, there exists $f_i\in F$ such that $f_i$ contains $\alpha_i(n)$ in its axis and $\alpha_i(n)$ is fixed by $F\setminus\{f_i\}$. We obtain a $(\s,\tau,F)$--cubical configuration of width $\geq 4r$ in $Y_n$, where the static sets are given by the hyperplanes of $Y_n$ originating from the arcs $\alpha_i(n)$ with $i\leq\s$, while the skewering elements are $f_{\s+1},\dots,f_s$ and $g_1,\dots,g_t$ (thus, $\tau=s+t-\s$).

This immediately translates into a $(\s,\tau,\varphi_n(F))$--cubical configuration of width $\geq 4r$ in $Y$. The fact that the static sets can be taken with arbitrarily large cardinality is also immediate, recalling that, since $\varphi_n$ is not $\om$--constant, the scaling factors $\overline\tau_S^{\rho_n}$ diverge (cf.\ Assumption~\ref{tau diverges} above). 
\end{proof}

The next result only requires the material in Subsection~\ref{more on raag cc} for its proof. However, it is best stated in terms of cubical configurations, as defined above. 

Recall that $q$ is the number of orbits of vertices for the action $G\acts Y$.

\begin{lem}\label{commutation from cubical config}
Let $F\cu G$ be a finite set of label-irreducible elements. Suppose that there is an $(s,t,F)$--cubical configuration of width $\geq 4r$ in $Y$, where all the static sets are chains containing each $\geq q$ hyperplanes. Then the centraliser $Z_G(F)$ contains a copy of $\Z^k$ with $k=s+t$.
\end{lem}
\begin{proof}
Let the cubical configuration consist of static sets $\mc{U}_1,\dots,\mc{U}_s$, hyperplanes $\mf{v}_1,\dots,\mf{v}_t$ and the partition $F=F_0\sqcup\{g_1,\dots,g_t\}$. 

Form a set $\mc{U}_i'$ by adding to $\mc{U}_i$ all hyperplanes of $Y$ that separate hyperplanes of $\mc{U}_i$. This guarantees that there exist vertices $x_i,y_i\in Y$ such that $\mscr{W}(x_i|y_i)=\mc{U}_i'$. The sets $\mc{U}_i'$ and $\{\mf{v}_j,g_j^{4r}\mf{v}_j\}$ are still all transverse to each other and the elements of $F$ still fix each $\mc{U}_i'$ pointwise.

Since $\#\mscr{W}(x_i|y_i)\geq\#\mc{U}_i\geq q$, any geodesic joining $x_i$ to $y_i$ must contain two points in the same $G$--orbit. Thus, there exist $z_i\in Y$ and $h_i\in G\setminus\{1\}$ with $\mscr{W}(z_i|h_iz_i)\cu\mc{U}_i'$ for all $1\leq i\leq s$. 

Lemma~\ref{another commutation criterion new} shows that the $h_i$ commute pairwise. In addition, for every $f\in F$, we have $\mscr{W}(z_i|h_iz_i)\cu\mc{U}_i'\cu\W_0(f)$, which is transverse to $\W_1(f)$. Since $\W_1(f)$ contains $\mscr{W}(z|fz)$ for any $z\in\overline{\mc{C}}(f)$, another application of Lemma~\ref{another commutation criterion new} guarantees that $h_i$ and $f$ commute. Finally, for $1\leq j\leq t$, the hyperplanes  $\mf{v}_j$ and $g_j^{4r}\mf{v}_j$ are preserved by all elements of $F\setminus\{g_j\}$. Since $g_j$ is label-irreducible, Corollary~\ref{label-irreducible cor}(1) implies that $g_j$ commutes with every element of $F$.

In conclusion, we have shown that the subgroup generated by $A:=\{h_1,\dots,h_s,g_1,\dots,g_t\}$ is abelian and contained in $Z_G(F)$. We are left to show that $A$ is a basis for $\langle A\rangle$.

Observe that $\mf{v}_j$ is preserved by all elements of $A\setminus\{g_j\}$, but lies in $\W_1(g_j)$. Similarly, there exist hyperplanes $\mf{u}_i\in\mscr{W}(z_i|h_iz_i)$ that are preserved by $A\setminus\{h_i\}$, but lie in $\W_1(h_i)$. 
% to be precise, we should repeat the argument in the proof of Lemma~\ref{another commutation criterion new}...
If a product $h_1^{m_1}\cdot\ldots\cdot h_s^{m_s}\cdot g_1^{n_1}\cdot\ldots\cdot g_t^{n_t}$ represents the identity, then it must preserve all hyperplanes $\mf{u}_i$ and $\mf{v}_j$, which implies that $m_i=0$ and $n_j=0$ for all $i,j$. This concludes the proof.
\end{proof}

\subsection{Ultralimits of Salvettis and the WNE property.}\label{meaty sect}

This subsection is devoted to the proof of Theorems~\ref{Q2 thm intro} and~\ref{special UNE thm intro}. We keep the exact same setting as the previous subsection:

\begin{ass}\label{assumption}
Let $G\leq\A$ be a convex-cocompact subgroup. Let $Y\cu\X$ be a $G$--invariant, convex subcomplex on which $G$ acts with $q$ orbits of vertices. Set $r=\dim\X$. Denote by $d$ the $\ell^1$ metric on $\X$ and $Y$. Let $[\mu]$ be the induced coarse median structure on $G$. 

Consider a sequence $\varphi_n\in\aut(G,[\mu])$. Denote by $\rho\colon G\hookrightarrow\A$ the standard inclusion and set $\rho_n=\rho\o\varphi_n$. Fixing a non-principal ultrafilter $\om$, define $\X_{\om},\X_n,Y_n$ and $Y_{\om}=\lim_{\om}Y_n$ as in Subsection~\ref{ultra-Salvetti sect}.
\end{ass}

The following result is the coronation of our efforts from Subsection~\ref{more on raag cc} and the previous portion of Section~\ref{ultra sect}. Its second part (with $k=1$) proves Theorem~\ref{special UNE thm intro}, while its first part is the last remaining ingredient in the proof of Theorem~\ref{Q2 thm intro} (together with Corollary~\ref{homothety 1}).

\begin{thm}\label{from sub-ultras to X new new}
Let $F\cu G$ be a finite subset and suppose that {\bf one} of the following holds.
\begin{enumerate}
\item Let $\varphi_n$ not be $\om$--constant. Let $M\cu Y_{\om}$ be a $G$--invariant median subalgebra and consider $\eta\in\mc{PD}^G(M)$. There exists a $k$--cube $C\cu M$ such that, for any two distinct points $x,y\in C$:
\[\eta(x,y)>4r^2q\cdot[\tau^{\eta}_F(x)+\tau^{\eta}_F(y)].\]
\item There exists a (generalised) $k$--cube $C\cu Y^{(0)}$ such that, for any two distinct points $x,y\in C$:
\[d(x,y)>(2r^2q+\tfrac{rq}{2}\cdot\max\{4r,q\})\cdot[\tau^d_F(x)+\tau^d_F(y)].\]
\end{enumerate}
%Let $M\cu Y_{\om}$ be a $G$--invariant median subalgebra. Let $\eta\in\mc{PD}^G(M)$ be arbitrary if $\varphi_n$ is not $\om$--constant, or the restriction of $\delta_{\om}$ if $\varphi_n$ is $\om$--constant. Suppose that there exist a finite subset $F\cu G$ and a $k$--cube $C\cu M$ such that, for any two distinct points $x,y\in C$, we have:
%\[\eta(x,y)>(2r^2q+\tfrac{rq}{2}\cdot\max\{4r,q\})\cdot[\tau^{\eta}_F(x)+\tau^{\eta}_F(y)].\]
Then the centraliser $Z_G(F)$ contains a copy of $\Z^k$. 
\end{thm}

The theorem will follow quickly from Lemma~\ref{from cubes to cubical config} below, which constructs a cubical configuration in $Y_{\om}$ (in Case~(1)) or directly in $Y$ (in Case~(2)). Indeed, we can then use Lemma~\ref{transferring cubical config} to always obtain a cubical configuration in $Y$ and this yields the required copy of $\Z^k$ in $Z_G(F)$ by Lemma~\ref{commutation from cubical config}.

\begin{lem}\label{from cubes to cubical config}
Consider the setting of Theorem~\ref{from sub-ultras to X new new}. There exists an $(s,t,F')$--cubical configuration of width $\geq 4r$ in $Y_{\om}$ (in Case~(1)) or in $Y$ (in Case~(2)), where $s+t=k$ and $F'\cu G$ is a finite subset with $Z_G(F')=Z_G(F)$. All elements of $F'$ are label-irreducible and no two of them generate a cyclic subgroup.

%Consider the setting of Theorem~\ref{from sub-ultras to X new new}. There exists an $(s,t,F')$--cubical configuration of width $\geq 4r$ in $Y_{\om}$, where $s+t=k$ and $F'\cu G$ is a finite subset with $Z_G(F')=Z_G(F)$. All elements of $F'$ are label-irreducible and no two of them generate a cyclic subgroup. 
In addition, in Case~(2), the static sets are chains of hyperplanes of cardinality $\geq q$.
\end{lem}
\begin{proof}
We prove the lemma simultaneously in the two cases of the theorem. In fact, in this proof it is irrelevant whether the $\varphi_n$ are $\om$--constant or not, so we can view Case~(2) as a special instance of Case~(1) by taking $\varphi_n\equiv\id_G$, $Y_{\om}=Y$, $M=Y^{(0)}$, $\eta=d$.

Recall that, by Remark~\ref{tree factors vs X}(2) and Lemma~\ref{fixsets of powers}(1), the action $G\acts M$ is non-transverse and without inversions. We begin by constructing the finite subset $F'\cu G$.

\smallskip
{\bf Claim~1:} \emph{there exists $F'\cu G$ such that $Z_G(F')=Z_G(F)$, all elements of $F'$ are label-irreducible and no two of them generate a cyclic subgroup. In addition, $\tau_{F'}^{\eta}(x)\leq q\cdot\tau_F^{\eta}(x)$ for all $x\in M$.}

\smallskip\noindent
\emph{Proof of Claim~1.}
Let $a_1,\dots,a_N\in\A$ be a choice of generator for each maximal cyclic subgroup of $\A$ that contains a label-irreducible component of an element of $F$. Let $m_i\geq 1$ be the smallest integer such that $a_i^{m_i}$ lies in $G$; by Lemma~\ref{label-irreducibles don't escape}, $m_i$ is well-defined and, by Remark~\ref{uniform exponent}, we have $1\leq m_i\leq q$. Define $F':=\{a_i^{m_i} \mid 1\leq i\leq N\}$.

It is clear that every element of $F'$ is label-irreducible and that any two elements of $F'$ generate a non-cyclic subgroup. Since all nontrivial powers of any given element of $\A$ have the same centraliser, Lemma~\ref{label-irreducible decomposition}(3) implies that $Z_G(F')=Z_G(F)$. 

For every $a_i$, there exist $n\geq 1$ and $f\in F$ such that $a_i^n$ is a label-irreducible component of $f$. Thus $a_i^{nm_i}$ is a label-irreducible component of $f^{m_i}$. Applying Lemma~\ref{displacement lem}, it follows that:
\[\eta(x,a_i^{m_i}x)\leq\eta(x, a_i^{nm_i}x)\leq\eta(x,f^{m_i}x)\leq m_i\cdot\eta(x,fx)\leq q\cdot\eta(x,fx)\leq q\cdot\tau_F^{\eta}(x).\] 
Hence $\tau_{F'}^{\eta}(x)\leq q\cdot\tau_F^{\eta}(x)$ for all $x\in M$, as required.
\hfill$\blacksquare$ 

\smallskip
Now, consider the multi-bridge $\mc{B}(F')\cu M$ introduced in Definition~\ref{multi-bridge of S defn}. Pick any fibre $P=\mc{B}_{/\mkern-5mu/}(F')\x\{\ast\}$. Let $\pi_P\colon M\ra P$ be the gate-projection.

\smallskip
{\bf Claim~2:} \emph{the set $C':=\pi_P(C)$ is again a $k$--cube and, for all distinct points $x',y'\in C'$, we have $\eta(x',y')\geq 4r^2\cdot\overline\tau_{F'}^{\eta}$. Under the assumptions of Case~(2), we further have $\eta(x',y')\geq rq\cdot\overline\tau_{F'}^{\eta}$.}

\smallskip\noindent
\emph{Proof of Claim~2.} 
If $x,y\in C$ are distinct, note that we have $\eta(x,y)>4r^2q\cdot[\tau^{\eta}_F(x)+\tau^{\eta}_F(y)]$ under the assumptions of both Case~(1) and Case~(2). Also recall that, by Proposition~\ref{properties of B(S)}(6), we have $\eta(x,\pi_P(x))\leq 2r^2\cdot\tau_{F'}^{\eta}(x)$ for all $x\in M$. Combining these inequalities with Claim~1, we obtain:
\begin{align*}
\eta(\pi_P(x),\pi_P(y))&\geq \eta(x,y)-2r^2\cdot[\tau^{\eta}_{F'}(x)+\tau^{\eta}_{F'}(y)] \\
&>4r^2q\cdot[\tau^{\eta}_F(x)+\tau^{\eta}_F(y)]-2r^2\cdot[\tau^{\eta}_{F'}(x)+\tau^{\eta}_{F'}(y)] \\
&\geq 2r^2\cdot[\tau^{\eta}_{F'}(x)+\tau^{\eta}_{F'}(y)]\geq 4r^2\cdot\overline\tau_{F'}^{\eta}.
\end{align*}
In particular, distinct points of $C$ project to distinct points of $C'$, which guarantees that $C'$ is again a $k$--cube. Moreover, $\eta(x',y')\geq 4r^2\cdot\overline\tau_{F'}^{\eta}$ whenever $x',y'$ are distinct points of $C'$. 

Under the assumptions of Case~(2), we also have $\eta(x,y)>(2r^2q+rq^2/2)\cdot[\tau^{\eta}_F(x)+\tau^{\eta}_F(y)]$ for all $x,y\in C$. Using this instead of $\eta(x,y)>4r^2q\cdot[\tau^{\eta}_F(x)+\tau^{\eta}_F(y)]$ in the above chain of inequalities, we obtain $\eta(\pi_P(x),\pi_P(y))\geq rq\cdot\overline\tau_{F'}^{\eta}$, as required.
\hfill$\blacksquare$ 

\smallskip
Let $\{C_{i,-}',C_{i,+}'\}$ be the $k$ pairs of opposite codimension--$1$ faces of $C'$. Setting $\mc{H}_i:=\mscr{H}(C_{i,-}'|C_{i,+}')$, we obtain sets of halfspaces $\mc{H}_1,\dots,\mc{H}_k\cu\mscr{H}(M)$ that are transverse to each other. If $\nu_{\eta}$ is the measure introduced in Remark~\ref{from pseudo-metrics to measures}, we have $\nu_{\eta}(\mc{H}_i)>4r^2\cdot\overline\tau_{F'}^{\eta}$ by Claim~2.

The set $\mc{H}_i$ can be partitioned into at most $r$ measurable subsets such that no two halfspaces in the same subset are transverse; this follows from Corollary~\ref{measurable partitions of intervals new}, proved in the Appendix (note that $\mc{D}(M)\neq\emptyset$ since $\mc{D}(\X_{\om})\neq\emptyset$, even though $\eta$ is just a pseudo-metric). Define $\mc{H}_i'\cu\mc{H}_i$ as the subset with the largest measure among those in this partition. No two halfspaces in $\mc{H}_i'$ are transverse and:
\[\nu_{\eta}(\mc{H}_i')\geq\tfrac{1}{r}\cdot\nu_{\eta}(\mc{H}_i)>4r\cdot\overline\tau_{F'}^{\eta}.\]

Let $\mc{U}_i'\cu\mscr{W}(M)$ be the set of walls associated to $\mc{H}_i'$. Recall that: 
\[\mc{U}_i'\cu\mscr{W}_{C'}(M)\cu\mscr{W}_P(M)\cu\bigcap_{f\in F'}\mscr{W}_{\overline{\C}(f)}(M)=\bigcap_{f\in F'}\left(\W_1(f,M)\sqcup\W_0(f,M)\right).\]
Since the sets $\W_1(f,M)$ and $\W_0(f,M)$ are transverse, while no two walls in $\mc{U}_i'$ are transverse, we must have either $\mc{U}_i'\cu\W_1(f,M)$ or $\mc{U}_i'\cu\W_0(f,M)$ for every index $i$ and element $f\in F'$. Consider the partitions $F'=F_i\sqcup F_i^{\perp}$ such that $\mc{U}_i'\cu\W_1(f,M)$ if $f\in F_i$ and $\mc{U}_i'\cu\W_0(f,M)$ if $f\in F_i^{\perp}$.

\smallskip
{\bf Claim~3:} \emph{we have $\#F_i\leq 1$ for all $1\leq i\leq k$, and $F_i\cap F_j=\emptyset$ for $i\neq j$.}

\smallskip\noindent
\emph{Proof of Claim~3}
For every $f\in F'$, we have:
\[\nu_{\eta}(\mc{H}_i')>4r\cdot\overline\tau_{F'}^{\eta}\geq 4r\cdot\ell(f,\eta)=\ell(f^{4r},\eta).\]
If $\mc{U}_i'\cu\W_1(f,M)$, it follows that there exists a wall $\mf{w}\in\mc{U}_i'$ such that $f^{4r}\mf{w}\cu\mc{U}_i'$. 

Thus, if $f,g\in F_i$, there exist walls $\mf{w},\mf{w}'$ such that $\{\mf{w},f^{4r}\mf{w},\mf{w}',g^{4r}\mf{w}'\}\cu\W_1(f,M)\cap\W_1(g,M)$. By Remark~\ref{halfspaces of subsets}(1) and Remark~\ref{from 10b rmk}(3), there is an analogous inclusion involving walls of $\X_{\om}$, so Lemma~\ref{from om to commutation new}(1) yields $\langle f,g\rangle\simeq\Z$. Since $f,g\in F'$, this means that $f=g$. Hence $\#F_i\leq 1$.

Suppose towards a contradiction that there exists $f\in F_i\cap F_j$. Then there are walls $\mf{w}_i,\mf{w}_j$ such that $\{\mf{w}_i,f^{4r}\mf{w}_i\}\cu\mc{U}_i'$ and $\{\mf{w}_j,f^{4r}\mf{w}_j\}\cu\mc{U}_j'$. In particular, the subsets $\{\mf{w}_i,f^{4r}\mf{w}_i\}$ and $\{\mf{w}_j,f^{4r}\mf{w}_j\}$ are transverse to each other and contained in $\W_1(f,M)$. Again, Remark~\ref{halfspaces of subsets}(1) and Remark~\ref{from 10b rmk}(3) give walls of $\X_{\om}$ with the same properties, which contradicts Lemma~\ref{from om to commutation new}(2).
\hfill$\blacksquare$

\smallskip
Up to permuting the sets $\mc{U}_i'$, we can assume that there exists an index $0\leq s\leq k$ such that $F_i=\emptyset$ for $1\leq i\leq s$, while $F_i=\{f_i\}$ for $s<i\leq k$ and pairwise-distinct elements $f_i\in F'$. This all follows from Claim~3. Each $f\in F'$ preserves every wall in $\mc{U}_i'$ except if $i>s$ and $f=f_i$. In addition, the proof of Claim~3 gives walls $\mf{v}_i$, for $i>s$, such that $\{\mf{v}_i,f_i^{4r}\mf{v}_i\}\cu\mc{U}_i'\cu\W_1(f_i,M)$. Also recall that the sets $\mc{U}_i'$ are all transverse to each other.

This gives an $(s,k-s,F)$--cubical configuration of width $\geq 4r$ in $M$, where $\mc{U}_1',\dots,\mc{U}_s'$ are the static sets and $f_{s+1},\dots,f_k$ are the skewering elements. Since $M\cu Y_{\om}$, it is straightforward to transfer this to a cubical configuration in $Y_{\om}$ with the same parameters using Remark~\ref{halfspaces of subsets}(1) and Remark~\ref{from 10b rmk}(3). This completes the proof of the lemma in Case~(1).

In Case~(2), we are left to show that the static sets $\mc{U}_1',\dots,\mc{U}_s'$ contain at least $q$ hyperplanes each. Recall that $M=Y^{(0)}$ and $\eta=d$, so $\nu_{\eta}$ is just the counting measure. By Claim~2, we have $\#\mc{H}_i'=\nu_{\eta}(\mc{H}_i')\geq\tfrac{1}{r}\nu_{\eta}(\mc{H}_i)\geq q\cdot\overline\tau_{F'}^{\eta}\geq q$, which concludes the proof since $\#\mc{H}_i'=\#\mc{U}_i'$.
\end{proof}

Combining Lemmas~\ref{transferring cubical config},~\ref{commutation from cubical config} and~\ref{from cubes to cubical config}, we can finally prove Theorem~\ref{from sub-ultras to X new new}.

\begin{proof}[Proof of Theorem~\ref{from sub-ultras to X new new}]
Our goal is to construct an $(s,t,F'')$--cubical configuration of width $\geq 4r$ in $Y$, where: $s+t=k$, the centraliser $Z_G(F'')$ is isomorphic to $Z_G(F)$, all elements of $F''$ are label-irreducible, and all static sets are chains of hyperplanes of cardinality $\geq q$. If we manage to do this, then Lemma~\ref{commutation from cubical config} guarantees that $Z_G(F)\simeq Z_G(F'')$ contains the required copy of $\Z^k$.

In Case~(2) of the theorem, a cubical configuration with these properties is provided by Lemma~\ref{from cubes to cubical config}. In Case~(1), we first apply Lemma~\ref{from cubes to cubical config} to obtain an $(s,t,F')$--cubical configuration of width $\geq 4r$ in $Y_{\om}$, where $Z_G(F')=Z_G(F)$. Then we obtain the required cubical configuration in $Y$ from Lemma~\ref{transferring cubical config}, with $F''=\varphi_n(F')$ for some $n$ (this is the only place where it is important that $\varphi_n$ is not $\om$--constant). Since $Z_G(F'')=\varphi_n(Z_G(F'))\simeq Z_G(F)$, this proves the theorem.
\end{proof}

The following two corollaries collect the key takeaways from Theorem~\ref{from sub-ultras to X new new} that we will need in the rest of the paper.

\begin{cor}\label{UNE cor}
Every special group with trivial centre is UNE (Definition~\ref{UNE defn}).
\end{cor}
\begin{proof}
Let $G$ be a special group with trivial centre. Embed $G$ as a convex-cocompact subgroup of a RAAG and apply Theorem~\ref{from sub-ultras to X new new}(2), taking $k=1$ and letting $F$ be an arbitrary generating set for $G$. This shows that the proper cocompact action $G\acts Y$ is UNE, hence $G$ is a UNE group.
\end{proof}

\begin{cor}\label{from sub-ultras to X new cor}
Consider the setting of Assumption~\ref{assumption}. Suppose that the $\varphi_n$ are pairwise distinct.
\begin{enumerate}
\item If $C\cu Y_{\om}$ is a $k$--cube and $H\leq G$ fixes $C$ pointwise, then $Z_G(H)$ contains a copy of $\Z^k$.
\item Let $G$ have trivial centre. Then, for every $G$--invariant median subalgebra $M\cu Y_{\om}$ the action $G\acts M$ is WNE (in the sense of Definition~\ref{WNE defn}).
\end{enumerate}
\end{cor}
\begin{proof} 
Note that by Remark~\ref{noetherian centralisers}, it suffices to prove part~(1) under the additional assumption that $H$ be finitely generated. So let us suppose that $H$ is generated by a finite set $F$ that fixes the $k$--cube $C$. We have observed in Subsection~\ref{ultra-Salvetti sect} that $\mc{D}(\X_{\om})^G\neq\emptyset$. Applying Theorem~\ref{from sub-ultras to X new new}(1) to any choice of $\eta\in\mc{D}(Y_{\om})^G$, we obtain the required copy of $\Z^k$ inside $Z_G(F)=Z_G(H)$.

Part~(2) also follows from Theorem~\ref{from sub-ultras to X new new}(1), setting $k=1$ and letting $F$ generate $G$.
\end{proof}

The following implies parts~(1) and~(2) of Theorem~\ref{Q2 thm intro} as a special case (parts~(3) and~(4) are obtained below in Remark~\ref{addendum to Q2 thm}). Note that the essentiality requirement in Theorem~\ref{homothety 2}(3) is equivalent to the minimality requirement in Theorem~\ref{Q2 thm intro}(2), because of \cite[Theorem~C]{Fio10b}.

Recall that we denote by $\pi\colon\aut G\ra\out G$ the quotient projection. If $G$ has trivial centre and $A\leq\out G$ is a subgroup, we have $G\lhd\pi^{-1}(A)$ and $\pi^{-1}(A)/G\simeq A$. 

\begin{thm}\label{homothety 2}
Let $G\leq\A$ be a convex-cocompact subgroup with trivial centre. Let $[\mu]$ be the induced coarse median structure on $G$. Let $A\leq\out (G,[\mu])$ be an infinite abelian subgroup. Then there exists an action $\pi^{-1}(A)\acts X$ with the following properties:
\begin{enumerate}
\item $X$ is a geodesic median space $X$ with $\rk X\leq r$;
\item $\pi^{-1}(A)\acts X$ is an action by homotheties;
\item the restriction $G\acts X$ is isometric, essential and with unbounded orbits;
\item if $C\cu X$ is a $k$--cube and $H\leq G$ fixes $C$ pointwise, then $Z_G(H)$ contains a copy of $\Z^k$.
\end{enumerate}
\end{thm}
\begin{proof}
Consider a sequence of pairwise distinct automorphisms $\varphi_n\in A$ and set $\rho_n=\rho\o\varphi_n$. Choose a finite generating set $S\cu G$ and consider the action $G\acts Y_{\om}$ as in Subsection~\ref{ultra-Salvetti sect}. 

Corollary~\ref{UNE cor} shows that $G$ is UNE. Thus, denoting by $\aut Y_{\om}$ the group of automorphisms of the underlying median algebra, Proposition~\ref{from UNE to X_om} yields a homomorphism $\zeta\colon\pi^{-1}(A)\ra\aut Y_{\om}$ that extends the isometric action $G\acts Y_{\om}$.

By Corollary~\ref{from sub-ultras to X new cor}(2), the action $G\acts Y_{\om}$ is WNE. Thus, Corollary~\ref{homothety 1} yields a nonempty, countable, $\pi^{-1}(A)$--invariant, median subalgebra $\mf{M}\cu Y_{\om}$, and a pseudo-metric $\eta\in\mc{PD}^G(\mf{M})\setminus\{0\}$ for which $\overline\tau_S^{\eta}>0$ and $\pi^{-1}(A)\acts (\mf{M},\eta)$ is homothetic. 
%Note that the median algebra $\mf{M}$ in Corollary~\ref{homothety 1} will not be closed in general.

Let $(\mf{M}_{\o},\delta)$ be the quotient median space obtained by identifying points $x,y\in\mf{M}$ with $\eta(x,y)=0$. By Remark~\ref{median homo rmk}, we have $\rk\mf{M}_{\o}\leq\rk\mf{M}\leq\rk X_{\om}\leq r$. Since $\overline\tau_S^{\delta}=\overline\tau_S^{\eta}>0$, the action $G\acts\mf{M}_{\o}$ does not have a global fixed point. Moreover, since the action $G\acts\mf{M}$ has no wall inversions by Lemma~\ref{fixsets of powers}(1), the action $G\acts\mf{M}_{\o}$ also has no inversions. Theorem~\ref{all from 10b}(2) then guarantees that $G$ acts on $\mf{M}_{\o}$ with unbounded orbits. 

Theorem~\ref{from sub-ultras to X new new}(1) (applied to the pseudo-metric $\eta$ on $\mf{M}$) shows that $G\acts\mf{M}_{\o}$ satisfies part~(4). Thus, we are only left to ensure that the median space be geodesic and the action essential.

In order to make our space geodesic, note that the homothetic $\pi^{-1}(A)$--action extends to the metric completion $\overline{\mf{M}}_{\o}$ of $\mf{M}_{\o}$. This is a median space of rank $\leq r$ by \cite[Proposition~2.21]{CDH} and \cite[Lemma~2.5]{Fio1}. Note that $G\acts \overline{\mf{M}}_{\o}$ still satisfies part~(4) because of Theorem~\ref{from sub-ultras to X new new}(1). Now, ``filling in cubes'' as in \cite[Corollary~2.16]{Fio2}, the space $\overline{\mf{M}}_{\o}$ embeds into a complete, connected median space $Z$ of the same rank. By \cite[Lemma~4.6]{Bow4}, the space $Z$ is geodesic. The isometric $G$--action extends to $Z$ and one can similarly check that so does the homothetic $\pi^{-1}(A)$--action. 

Summing up, we have constructed an action $\pi^{-1}(A)\acts Z$ that satisfies conditions (1)--(4), possibly except essentiality of the $G$--action (in addition, $Z$ is complete). By Theorem~\ref{all from 10b}(4), there exists a $\pi^{-1}(A)$--invariant, nonempty, convex subset $K\cu Z$ and a $\pi^{-1}(A)$--invariant splitting $K=K_0\x K_1$ such that the action $G\acts K_1$ is essential. We conclude by taking $X=K_1$. (Note that $K$ is not closed in $Z$ in general, so we may have lost completeness along the way).
\end{proof}

\begin{rmk}
In Theorem~\ref{homothety 2}, we cannot both require the space $X$ to be complete and the action $G\acts X$ to be essential. There is a very good reason for this. 

Consider the special case where $G$ is hyperbolic. Then $Y_{\om}$ is an $\R$--tree, which forces $X$ to also be an $\R$--tree. Note that an isometric action on an $\R$--tree is essential if and only if it is minimal. 

Let us show that, if $G$ is a finitely generated group and $G\acts T$ is a minimal action on a complete $\R$--tree not isometric to $\R$, then no homothety $\Phi\colon T\ra T$ with factor $\l\neq 1$ can normalise $G$. 

If $G$ is generated by $s_1,\dots,s_k$ and $x\in T$ is any point, the union of all segments $g[x,s_ix]$ with $g\in G$ is a $G$--invariant subtree. Since $G\acts T$ is minimal, $T$ must be covered by the segments $g[x,s_ix]$. In particular, the action $G\acts T$ is cocompact. If $\Phi$ normalised $G$, then every orbit of $G\acts T$ would be dense (see e.g.\ \cite[Proposition~3.10]{Paulin-ENS}). Since $T\not\simeq\R$, this implies that each segment $g[x,s_ix]$ is nowhere-dense. This violates Baire's theorem, since a complete metric space cannot be covered by countably many nowhere-dense subsets. 

I learned this argument from \cite[Example~II.6]{Gaboriau-Levitt}.
\end{rmk}

The following proves parts~(3) and~(4) of Theorem~\ref{Q2 thm intro}.

\begin{rmk}\label{addendum to Q2 thm}
Consider the special case of Theorem~\ref{homothety 2} with $A=\Z$, generated by an outer automorphism $\phi\in\out(G,[\mu])$. Picking a representative $\varphi\in\aut (G,[\mu])$, we have $\pi^{-1}(A)=G\rtimes_{\varphi}\Z$. The theorem gives an isometric action $G\acts X$ and a homothety $H\colon X\ra X$ of factor $\l$ such that $H\o g=\varphi(g)\o H$ for all $g\in G$. We keep the notation of the proof of Theorem~\ref{homothety 2}.
\begin{enumerate}
\item Each $g\in\Fix\varphi$ is elliptic in $X$. Indeed, Lemma~\ref{ultra-Salvetti 1} shows that $g$ is elliptic in $\X_{\om}$, since $\ell(\varphi^n(g),\X)$ does not diverge. Lemma~\ref{fixsets of powers}(2) then implies that $g$ is elliptic in $\mf{M}$, and it is clear that a fixed point in $\mf{M}$ will translate into a fixed point in $X$.

Recalling that $\Fix\varphi$ is finitely generated by Theorem~\ref{cmp und intro}, Theorem~\ref{all from 10b}(2) actually implies that $\Fix\varphi$ has a global fixed point $x_0\in X$. This proves Theorem~\ref{Q2 thm intro}(3).
% Can we show that if g does *not* fix a point of $\X_{\om}$, then it does *not* fix a point of $X$?
% the problem is what will get killed by the pseudo-metric $\eta$...
% is it maybe true that, for all $g\in G$, we have $G\cdot \mc{W}_1(g,\X_{\om})=\mscr{W}(\X_{\om})$??
\item Fix a finite generating set $S\cu G$. Recall from Subsection~\ref{identities sect}, that we denote conjugacy length by $\|\cdot\|_S$. Let $\Lambda(\varphi)$ be the maximal exponential growth rate of the quantity $\|\varphi^n(g)\|_S^{1/n}$:
\[\Lambda(\varphi):=\sup_{g\in G}\limsup_{n\ra+\infty}\|\varphi^n(g)\|_S^{1/n}.\]
Note that $\Lambda(\varphi)$ is independent of the generating set $S$. For every $g\in G$, we have: 
\[\l^n\ell(g,X)=\ell(H^ngH^{-n},X)=\ell(\varphi^n(g),X)\leq\|\varphi^n(g)\|_S\cdot\overline\tau_S^X,\]
where the last inequality follows from the identities in Subsection~\ref{identities sect}. Since there exist elements $g\in G$ with $\ell(g,X)>0$, we deduce that $\l\leq\Lambda(\varphi)$ and, similarly, $\l^{-1}\leq\Lambda(\varphi^{-1})$. 

If $\varphi$ has \emph{sub-exponential growth} (in the sense that $\Lambda(\varphi)=\Lambda(\varphi^{-1})=1$), then these inequalities force $\l=1$. Hence the homothetic action $G\rtimes_{\varphi}\Z\acts X$ provided by Theorem~\ref{homothety 2} is actually isometric, which proves Theorem~\ref{Q2 thm intro}(4).
\end{enumerate}
\end{rmk}

\appendix
\section{Measurable partitions of halfspace-intervals.}\label{app1}

This appendix is devoted to the proof of Corollary~\ref{measurable partitions of intervals new} below. This is needed in the proof of Theorem~\ref{from sub-ultras to X new new} in order to get the exact constant $4r^2q$, and could be avoided if we contented ourselves with the worse bound $4rq\cdot\#\G^{(0)}$. However, Corollary~\ref{measurable partitions of intervals new} is important in the general theory of median spaces and we think it is likely to prove useful elsewhere.

Let $M$ be a median algebra. Given a subset $P\cu M\x M$, let us write $\mc{H}_P:=\bigcup_{(x,y)\in P}\mscr{H}(x|y)$.

\begin{lem}\label{countably many pairs}
Every subset $P\cu [0,1]^n\x [0,1]^n$ contains a countable subset $\Delta\cu P$ with $\mc{H}_{\Delta}=\mc{H}_P$.
\end{lem}
\begin{proof}
First, we prove the case $n=1$. We can assume that $x<y$ for every $(x,y)\in P$. 

Let $\Om(P)\cu [0,1]$ be the union of the closed arcs $I(x,y)$ with $(x,y)\in P$. Let $\mc{D}(P)$ be the set of points that lie in the interior of $\Om(P)$, but not in the interior of any arc $I(x,y)$ with $(x,y)\in P$. Thus each point of $\Om(P)$ lies either in the frontier of $\Om(P)$, or in the interior of some $I(x,y)$, or in the set $\mc{D}(P)$, and these three possibilities are disjoint. There is a unique partition of $\Om(P)$ into maximal segments $J_i$ (closed, open, or half-open) such that:
\begin{itemize}
\item the interior of $J_i$ does not intersect $\mc{D}(P)$;
\item if $J_i$ intersects the interior of $I(x,y)$ for some $(x,y)\in P$, then $I(x,y)\cu J_i$.
\end{itemize} % essentially the J_i have endpoints either in D or at an endpoint of a connected component of \Om
Observe that $\mc{H}_{P}=\bigsqcup_i\mscr{H}_{J_i}([0,1])\cap\mscr{H}(0|1)$.

It is classical to see that there exists a countable subset $\Delta\cu P$ with $\Om(\Delta)=\Om(P)$. Note that $\mc{D}(\Delta)$ is countable and it contains $\mc{D}(P)$. Adding to $\Delta$ countably many pairs in $P$, we can thus ensure that $\mc{D}(\Delta)=\mc{D}(P)$. Hence, $P$ and $\Delta$ determine the same the segments $J_i$, and $\mc{H}_{P}=\mc{H}_{\Delta}$.

Now consider a general $n\geq 1$. Let $I_i\cu[0,1]^n$ be the segment where all coordinates but the $i$--th vanish. Let $\pi_i\colon [0,1]^n\ra I_i$ be the coordinate projections. Setting $P_i:=(\pi_i\x\pi_i)(P)\cu [0,1]^n\x [0,1]^n$, we have $\mc{H}_{P}=\bigcup_i\mc{H}_{P_i}$. By the case $n=1$, there exist countable subsets $\Delta_i\cu P_i$ with $\mc{H}_{\Delta_i}=\mc{H}_{P_i}$. Choosing countable sets $\Delta_i'\cu P$ with $(\pi_i\x\pi_i)(\Delta_i')=\Delta_i$, we have $\mc{H}_{\Delta_i}\cu\mc{H}_{\Delta_i'}\cu\mc{H}_{P}$. Hence, taking $\Delta:=\bigcup_i\Delta_i'$, we obtain $\mc{H}_{P}=\mc{H}_{\Delta}$.
\end{proof}

Recall that $\mscr{B}(M)$ is the $\s$--algebra generated by halfspace-intervals, as in Remark~\ref{from pseudo-metrics to measures}.

\begin{lem}\label{measurable projections}
Let $M\cu [0,1]^n$ be a median subalgebra containing the points $\underline 0=(0,\dots,0)$ and $\underline 1=(1,\dots,1)$. Let $\pi_i\colon M\ra [0,1]$ denote the coordinate projections. Then the induced maps $\pi_i^*\colon\mscr{H}([0,1])\ra\mscr{H}(M)$ (as in Remark~\ref{median homo rmk}) map $\mscr{B}$--measurable sets to $\mscr{B}$--measurable sets.
% THIS IS DIFFERENT FROM SAYING THAT THE MAP $\pi_i^*$ IS $\mscr{B}$--MEASURABLE, which would be about preimages!!
\end{lem}
\begin{proof}
Since $\pi_i^*$ is injective, we have:
\[\pi_i^*(\mscr{H}([0,1])\setminus E)=\pi_i^*(\mscr{H}(0|1))\cup\pi_i^*(\mscr{H}(1|0))\setminus\pi_i^*(E),\] 
for every $E\cu\mscr{H}([0,1])$. Thus, it suffices to show that, for all $0\leq a<b\leq 1$, the set $\pi_i^*\mscr{H}(a|b)$ is $\mscr{B}$--measurable. 

Let $a'$ and $b'$ be, respectively, the infimum and the maximum of $\pi_i(M)\cap[a,b]$. Pick sequences of elements $a'\leq a_{n+1}<a_n<b_n<b_{n+1}\leq b'$ so that $a_n,b_n\in\pi_i(M)$ and $a_n\ra a'$, $b_n\ra b'$. These sequences can be empty if $\pi_i(M)\cap[a,b]=\emptyset$, or consist of single elements if $a',b'\in\pi_i(M)$. Then:
\[\pi_i^*\mscr{H}(a|b)=\bigcup\pi_i^*\mscr{H}(a_n|b_n)\cup\{\pi_i^{-1}((a,1])\}\cup\{\pi_i^{-1}([b,1])\}.\]
Observing that singletons are $\mscr{B}$--measurable, 
% this would really require a proof, but it's not hard and I want to keep this short
 it suffices to show that, for every $x,y\in M$, the set $\pi_i^*\mscr{H}(\pi_i(x)|\pi_i(y))$ is $\mscr{B}$--measurable. 

This means that it actually suffices to prove that the sets $\pi_i^*\mscr{H}(0|1)$ are $\mscr{B}$--measurable. We will achieve this by showing that each set $\mscr{H}(M)\setminus\pi_i^*\mscr{H}(0|1)$ is a countable union of halfspace-intervals.

Note that $\mf{h}\in\mscr{H}(M)$ lies in $\pi_i^*\mscr{H}([0,1])$ if and only if the projections $\pi_i(\mf{h})$ and $\pi_i(\mf{h}^*)$ are disjoint. Thus, $\mf{h}$ lies in $\mscr{H}(M)\setminus\pi_i^*\mscr{H}(0|1)$ if and only if there exist $x,y\in M$ such that $\mf{h}\in\mscr{H}(x|y)$ and $\pi_i(x)\geq\pi_i(y)$. This gives a subset $P\cu M\x M$ with $\mscr{H}(M)\setminus\pi_i^*\mscr{H}(0|1)=\mc{H}_P$. 

In view of Lemma~\ref{countably many pairs} and Remark~\ref{halfspaces of subsets}(1), there exists a countable subset $\Delta\cu P$ with $\mc{H}_{\Delta}=\mc{H}_P$. This concludes the proof.
\end{proof}

The following would be an immediate consequence of Dilworth's lemma, were it not for the measurability requirement.

\begin{cor}\label{measurable partitions of intervals new}
Let $X$ be a median space of finite rank $r$. For all $x,y\in X$, there exists a $\mscr{B}$--meas\-ur\-a\-ble partition $\mscr{H}(x|y)=\mc{H}_1\sqcup\dots\sqcup\mc{H}_r$ so that no two halfspaces in the same $\mc{H}_i$ are transverse. 
\end{cor}
\begin{proof}
Taking the metric completion of $X$ and applying \cite[Proposition~2.19]{Fio1}, we obtain an isometric embedding $\iota\colon I(x,y)\hookrightarrow\R^r$. The image of $\iota$ is contained in a product $J_1\x\dots\x J_r$ of compact intervals $J_i\cu\R$, which is isomorphic to the median algebra $[0,1]^r$. Let $\pi_i\colon M\ra J_i$ be the composition of $\iota$ with the projection to $J_i$, and set $\mc{H}_i':=\mscr{H}(x|y)\cap\pi_i^*(\mscr{H}(J_i))$. We have $\mscr{H}(x|y)=\mc{H}_1'\cup\dots\cup\mc{H}_r'$, no two halfspaces in the same $\mc{H}_i'$ are transverse, and each $\mc{H}_i'$ is $\mscr{B}$--measurable by Lemma~\ref{measurable projections}. We conclude by taking $\mc{H}_i:=\mc{H}_i'\setminus(\mc{H}_1'\cup\dots\cup\mc{H}_{i-1}')$.
\end{proof}

\bibliography{../../mybib}
\bibliographystyle{alpha}

\end{document}